\documentclass[12pt]{article}
\usepackage{makeidx}
\usepackage{amsmath,amsthm,amssymb,amsfonts,mathrsfs}
\usepackage[hypertex, bookmarks, colorlinks=true, plainpages = false, citecolor = green, urlcolor = blue, filecolor = blue]{hyperref}
\newtheorem{theorem}{Theorem}[section]
\newtheorem{lemma}[theorem]{Lemma}
\newtheorem{proposition}[theorem]{Proposition}
\newtheorem{corollary}[theorem]{Corollary}
\newtheorem{definition}[theorem]{Definition}
\newtheorem{example}[theorem]{Example}

\newtheorem{remark}[theorem]{Remark}
\newtheorem{fact}[theorem]{Fact}

\def\a{\bar{a}}\def\b{\bar{b}}\def\c{\bar{c}}
\def\d{\bar{d}}\def\x{\bar{x}}\def\y{\bar{y}}
\def\z{\bar{z}}\def\t{\bar {t}}
\def\s{\bar{s}}\def\e{\bar{e}}

\def\Nn{\mathbb N}

\def\Qn{\mathbb Q}
\def\Rn{\mathbb R}

\def\Sc{\mathcal S}
\def\0{\sf 0}

\makeatletter
\def\dotminussym#1#2{%
  \setbox0=\hbox{$\m@th#1-$}%
  \kern.5\wd0%
  \hbox to 0pt{\hss\hbox{$\m@th#1-$}\hss}%
  \raise.8\ht0\hbox to 0pt{\hss$\m@th#1.$\hss}%
  \kern.5\wd0}

\makeindex
\begin{document}
\begin{center}
{\Large\sc{Elements of affine model theory\\}}\vspace{10mm}

{\bf Seyed-Mohammad Bagheri}
\vspace{4mm}

{\footnotesize {\footnotesize Department of Pure Mathematics, 
Tarbiat Modares University,\\ Tehran, Iran, P.O. Box 14115-134}}
\vspace{1mm}

bagheri@modares.ac.ir 
\end{center}

\newpage

{\hypersetup{linkcolor=blue}{\footnotesize\tableofcontents}


\newpage\section{Introduction}
A general trend in model theory is to extend first order logic in order to deal with notions from other parts of mathematics.
Among the most interesting such notions are those arising from topology.
Topological notions are usually second order properties and this makes the problem sufficiently difficult.
The approach which handles such notions using variations of first order logic is called topological model theory \cite{Ziegler}.
A related approach in this respect is to incorporate continuity in logic.
This notion is ubiquitous in mathematics and deserves particular attention.
C.C. Chang and H.J. Keisler introduced continuous model theory \cite{CK} in its most general form.
They used an arbitrary compact Hausdorff space $X$ as the space of truth values.
In this logic, only continuous operations on $X$ are allowed to be used as logical connectives.
Since then, improvements of continuous logic has been introduced by logicians.
The focus has been mostly on the notion of metric, aiming to put metric algebraic structures in the realm of model theory.
Retaining some sort of compactness theorem similar to what holds in first order logic is an essential part of these essays.
Fortunately, this goal is usually achieved by using the ultraproduct method.
One of the most successful approaches up to now has been the model theory of metric structures \cite{BBHU}.
This is a relatively direct generalization of first order logic where the value space $\{0,1\}$ is replaced with a closed interval say $[0,1]$.
Of course, with some precautions, one can shift to the whole real line and use $\{+,-,\times,0,1\}$ as a system of connectives.
Other connectives are then approximated by their combinations in the light of the Stone-Weierstrass theorem.
In the resulting framework, one handles metric spaces equipped with a family of uniformly continuous relations and operations.
The expressive power of continuous logic is so strong that compact structures are characterized up to isomorphism.

By Lindstr\"{o}m's theorems, the expressive power of first order logic (and similarly continuous logic)
is not strengthened without losing some interesting property.
Weakening it, is however less harmless and has been payed attention by some authors (see e.g. \cite{Poizat}).
Affine continuous logic is a weakening of continuous logic obtained by avoiding $\times$
(or correspondingly $\wedge,\vee$) as connective.
This reduction leads to the affinization of most basic tools and technics of continuous logic
such as the ultraproduct construction, compactness theorem, type, saturation etc.
The affine variant of the ultraproduct construction is the ultramean construction where
ultrafilters are replaced with maximal finitely additive probability measures.
A consequence of this relaxation is that compact structures with at least
two elements have now proper elementary extensions.
In particular, they have non-categorical theories in the new setting.
Thus, a model theoretic framework for study of such structures is provided.
A more remarkable aspect of this logic is that the type spaces are compact convex sets.
The extreme types then play a crucial role in the study of affine theories.

In first order logic, types have several features depending on the ambient theory,
e.g. Dedekind cuts in linearly ordered sets or prime ideals of polynomial rings in algebraically closed fields.
In the continuous logic framework, formulas form a real Banach algebra and types are generally Banach algebra homomorphisms.
In the affine part, types correspond to norm one positive linear functionals and they form compact convex sets.
Surprisingly, in first order models, ordinary types correspond to extremal affine types.
In contrast, Keisler measures are affine types in disguise and they are realized in the affine combinations of first order models.
So, Keisler measures are better understood in the perspective provided by affine continuous logic.

The main object of study in full or affine continuous logic is metric structures.
Examples include first order models, pure metric spaces (e.g. Hadamard spaces), metric groups (or rings etc.) and probability algebras.
Structures based on Banach spaces such as Banach algebras, $C^*$-algebras, Hilbert spaces etc. are
considered as many sorted metric structures which are of main interest in continuous logic \cite{Henson}.
In the affine part, first order theories find now new models consisting of affine combinations
(or more generally ultrameans) of other first order models.
These structures are interesting and can be studied in this framework.
Also, as stated above, compact structures have non-compact elementary extensions in the affine logic
and hence may be payed attention from a new point of view now. An important
example is the classical dynamical systems which are based on compact metric spaces.
Another source of examples is Lie groups some of which carry invariant metrics
making them metric structures in the continuous logic sense.

The purpose of the present article is to initiate the model theory of
metric structures in the framework of affine continuous logic.
Parts of the work owe to R. Safari, M. Malekghasemi and F. Fadai and has been published previously
in \cite{Bagheri-Lip, Bagheri-Extreme, Bagheri-Safari, Malekghasemi, Safari-Bagheri}.
Affine variants of some standard notions of model theory including quantifier-elimination, saturation,
types, definability etc. are considered. There are also new notions raised from the affinity.
There are interesting connections between logic and Choquet theory.
The ultramean construction has new features which is studied in brief but which also needs further work.
A more general form of the affine logic is obtained by using powers of the metric relation.
It provides a framework for axiomatizing structures like Hadamard spaces.
This is discussed in \S \ref{ALp}.
The appendix contains some standard mathematical notions and results used in the text.
It also contains a selected list of questions concerning different aspects of affine model theory.

Recently, a long preprint of I. Ben-Yaacov, T. Ibarluc\'{i}a, T. Tsankov \cite{Ibarlucia} appeared.
They give a detailed study of extreme types, extremal models and representations
of models of affine theories by extremal models. They extended the subject and found its
deep connections with analysis and ergodic theory.
In contrast, the present text is rather introductory and emphasizes more on the general model theoretic aspects of affine logic.

\newpage\section{Logic and structure}
The formal presentation of continuous logic \cite{BBHU} is similar to first order logic.
The space of truth values is the unit interval and every continuous $\alpha:[0,1]^n\rightarrow[0,1]$
can be used as a connective. Also, the quantifiers $\inf$ and $\sup$ replace $\forall$ and $\exists$.
In practice, a generating set of connectives consisting of $\wedge,\vee$, $1-x$ etc. is used.
On the other hand, all operations and relations on the structures are assumed to be uniformly continuous.
Uniform continuity is preserved by ultraproducts if the modulus of continuity is fixed.
The setting of continuous logic has been well regulated to fit in with the ultraproduct construction.
The Keisler-Shelah isomorphism theorem is a natural consequence of this accordance.

To achieve the affine fragment of continuous logic, in the first step, we replace $[0,1]$ with
$\Rn$ and use its algebraic operations as connectives.
There are several options such as $\{1,+,\wedge,r\cdot\}_{r\in\Rn}$ and $\{+,\times,r\}_{r\in\Qn}$.
Thanks to the various forms of the Stone-Weierstrass theorems \cite{Royden},
these systems of connectives are complete and the resulting logics are equivalent to the standard one stated above.
In affine continuous logic, ultrafilters are replaced with finitely additive probability measures
and ultralimit is replaced with integration.
So, in the second step, we keep the linear connectives and remove the rest.
More precisely, we use $(\Rn,1,+,r\cdot,\leqslant)_{r\in\Rn}$ as value structure for defining this fragment.
Uniform continuity is also replaced with Lipschitzness.
Recall that every bounded uniformly continuous real valued function on a metric space
can be uniformly approximated by Lipschitz functions. So, Lipschitzness is not a serious restriction.
Also, the order relation $\leqslant$ is used in the formation of conditions (or statements).

A \emph{Lipshcitz language} \index{Lipschitz language} is a set $$L=\{c,e,...;\ F,G,...;\ P,R,...\}$$ consisting of constant symbols
as well as function and relation symbols of various arities. $L$ is allowed to have arbitrary cardinality.
The arity of a function (or relation) symbol is a number $n\geqslant1$ which indicates the number of variables in the domain of its interpretation.
Also, to each function symbol $F$ (resp. relation symbol $R$) is assigned a Lipschitz constant
$\lambda_F\geqslant0$ (resp. $\lambda_R\geqslant0$).
It is always assumed that $L$ contains a distinguished binary relation symbol $d$ for metric and $\lambda_d=1$.
However, exceptionally it is put in the list of logical symbols.
In brief, the following list of symbols are used in affine logic:
\vspace{1mm}

- logical and auxiliary symbols $1$, $+$, $r\cdot$, $\inf$, $\sup$, $d$, $($, $)$

- individual variables $x,y,z,...$

- symbols of the language $L$.
\bigskip

Let $L$ be a Lipschitz language.

\begin{definition}
\em{\emph{$L$-terms} and their Lipschitz constants (denoted by $\lambda_t$ for the term $t$) are inductively defined as follows:

- Constant symbols and variables are terms with Lipschitz constants resp. $0$ and $1$.

- If $F$ is a $n$-ary function symbol and $t_1,...,t_n$
are terms, then $t=F(t_1,...,t_n)$ is a
term with Lipschitz constant $\lambda_t=\lambda_F\cdot\sum_{i=1}^n\lambda_{t_i}$.}
\end{definition}

\begin{definition}\label{formulas and bounds}
\em{\emph{$L$-formulas} and their Lipschitz constants and bounds (denoted resp. by $\lambda_\phi$ and
$\mathbf{b}_\phi$ for the formula $\phi$) are inductively defined as follows:

- $1$ is an \emph{atomic formula} with Lipschitz constant $0$ and bound $1$.

- If $t_1,t_2$ are terms, $d(t_1,t_2)$ is an atomic formula with Lipschitz constant $\lambda_{t_1}+\lambda_{t_2}$ and bound $1$.

- If $R$ is a $n$-ary relation symbol and $t_1,...,t_n$ are terms, then $R(t_1,...,t_n)$ is an
\emph{atomic formula} with Lipschitz constant $\lambda_R\cdot\sum_{i=1}^n\lambda_{t_i}$ and bound $1$.

- If $\phi,\psi$ are formulas and $r\in\Rn$, then $\phi+\psi$ and $r\phi$ are formulas with
Lipschitz constants resp. $\lambda_\phi+\lambda_\psi$,\ $|r|\lambda_\phi$ and bounds
resp. $\mathbf{b}_{\phi}+\mathbf{b}_\psi$,\ $|r|\mathbf{b}_\phi$.

- If $\phi$ is a formula and $x$ is a variable, then $\inf_x\phi$ and $\sup_x\phi$ are formulas with
the same Lipschitz constant and bound as $\phi$.}
\end{definition}
\vspace{2mm}

Free and bounded variables are defined in the usual way. For example, in the formula $\sup_x d(x,y)$,
the variable $y$ is free while $x$ bounded. The notation $\phi(x_1,...,x_n)$ (resp. $t(\x)$)
indicates that all free variables of the formula $\phi$ (resp. the term $t$) are included in the list $x_1,...,x_n$.
A \emph{sentence} is a formula without free variable.
If $d$ is a pseudometric on $M$, we put the pseudometric $\sum_{i=1}^nd(x_i,y_i)$ on $M^n$
and denote it again by $d$. So, $d$ has several meanings which must not be confused.

\begin{definition} \label{prestructure}
{\em A {\em prestructure} \index{prestructure} in $L$ is a pseudometric space $(M,d)$ with diameter at most $1$ equipped with:

- for each constant symbol $c\in L$, an element $c^{M}\in M$

- for each $n$-ary function symbol $F$ a function $F^{M}:M^n\rightarrow M$ such that
$$d(F^{M}(\a),F^{M}(\b))\leqslant\lambda_F d(\a,\b)\ \ \ \ \ \ \ \forall\a,\b$$

- for each $n$-ary relation symbol $R$ a function $R^{M}:M^n\rightarrow[0,1]$ such that
$$R^{M}(\a)-R^{M}(\b)\leqslant\lambda_R d(\a,\b)\ \ \ \ \ \ \ \ \forall\a,\b.$$}
\end{definition}
\bigskip

The metric symbol is interpreted by the metric function, i.e. $d^M=d$ and the last clause holds for it.
Let $M$ be an $L$-prestructure.
\vspace{1mm}

\begin{definition}
\em{The value of a term $t(x_1,...,x_k)$ in $\a\in M^k$, denoted by $t^M(\a)$, is inductively defined as follows:

- $c^M(\a)=c$

- $x_i^M(\a)=a_i$

- $\big(F(t_1(\x),...,t_n(\x))\big)^M(\a)=F^M(t_1^M(\a),...,t_n^M(\a))$ \ \ \ \ \ (for $n$-ary $F$).}
\end{definition}
\vspace{1mm}

\begin{definition}
\em{The value of a formula $\phi(\x)$ in $\a\in M^k$, denoted by $\phi^M(\a)$, is inductively defined as follows:

  - $1^M=1$

  - $(d(t_1,t_2))^M(\a)=d(t^M_1(\a),t^M_2(\a))$

  - $(R(t_1,...,t_n))^M(\a)=R^M(t^M_1(\a),...,t_n^M(\a))$

  - $(\phi+\psi)^M(\a)=\phi^M(\a)+\psi^M(\a)$

  - $(r\phi)^M(\a)=r\phi^M(\a)$

  - $(\inf_y\phi)^M(\a)=\inf\{\phi^M(\a,b):\ b\in M\}$ \ \ \ \ \ (for $\phi(\x,y)$)

  - $(\sup_y\phi)^M(\a)=\sup\{\phi^M(\a,b):\ b\in M\}$\ \ \ \ (for $\phi(\x,y)$).}
\end{definition}
\vspace{1mm}

\begin{proposition} \label{completion of prestructures}
Let $M$ be an $L$-prestructure and $|\x|=n$.
For each term $t(\x)$, the map $t^M:M^n\rightarrow M$ is $\lambda_t$-Lipschitz.
For each formula $\phi(\x)$, the map $\phi^M:M^n\rightarrow \Rn$ is $\lambda_{\phi}$-Lipschitz
and $|\phi^M(\a)|\leqslant\mathbf{b}_{\phi}$ for every $\a$.
\end{proposition}

A prestructure $M$ is called \emph{structure} \index{structure} (or \emph{model}) if $d^M$ is a complete metric.
This is the standard terminology in continuous logic and we normally follow it.
Nevertheless, we encounter in this article structures which are not complete.
For example, the ultramean of complete structures may be incomplete and the union of a chain of models is incomplete.
Generally, every prestructure can be transformed to a complete structure in a natural process.
For this purpose, first one identifies $a,b\in M$ if $d(a,b)=0$. The quotient metric is denoted again by $d$.
The resulting metric is then completed by adding all Cauchy sequences
and the interpretations of function and relation symbols are extended to this space
in the natural way (see also \cite{BBHU}).

\begin{proposition}\label{exist}
Let $M$ be an $L$-prestructure. Then, there is a complete $L$-structure $\bar M$
and a dense mapping $\pi:M\rightarrow \bar M$ such that for every formula $\phi(\x)$
$$\phi^{M}(a_1,...,a_k)=\phi^{\bar M}(\pi a_1,...,\pi a_k) \ \ \ \ \ \ \forall a_1,...,a_k\in M.$$
\end{proposition}

The logic based on the family of formulas defined above is called \emph{affine continuous logic} (abbreviated AL).
In full continuous logic (abbreviated CL), one further allows $\phi\wedge\psi$ and $\phi\vee\psi$ to be formulas.
The expressive power is then extended and the above mentioned propositions hold for the class of CL-formulas too.
Most fundamental results in classical model theory have in parallel a CL variant (see \cite{BBHU}) and an AL variant.
The later one is the subject of study of this article. Except otherwise stated,
\begin{quote}{\sf every notion or notation used in this text is in the AL sense}.\end{quote}
We may sometimes emphasize this by adding the adverb `affine' (or AL) to that notion.
However, we also mention similar notions in the CL sense mostly for comparison between the two notions.
In this case the full one is indicated by the acronym CL.
\vspace{1mm}

\begin{definition}
\em{A map $f:M\rightarrow N$ is called an \emph{embedding}\index{embedding} if

  - for every constant symbol $c\in L$, $f(c^M)=c^N$

  - for every $n$-ary function symbol $F$ and $a_1,...,a_n\in M$ $$f(F^M(a_1,...,a_n))=F^N(f(a_1),...,f(a_n))$$

  - for every $n$-ary relation symbol $R$ (including $d$) and $a_1,...,a_n\in M$ $$R^M(a_1,...,a_n)=R^N(f(a_1),...,f(a_n)).$$}
\end{definition}
\bigskip

So, by the third clause for every $a,b\in M$ we have that $$d^M(a,b)=d^N(f(a),f(b))$$
which means that $f$ is distance preserving. If $M$ is a subset of $N$ and the inclusion map $id$ is an embedding,
$M$ is called a \emph{substructure}\index{substructure} of $N$. This is denoted by $M\subseteq N$.
A formula is \emph{quantifier-free}\index{quantifier-free} if it does not contain the quantifiers $\sup$ and $\inf$,
i.e. it is of the form $\sum r_i\phi_i$ where $\phi_i$ is atomic and $r_i\in\Rn$.
It is not hard to verify that

\begin{lemma}
$f:M\rightarrow N$ is an embedding if and only if for every quantifier-free formula $\phi(\x)$ and $\a\in M$,
one has that\ \ $\phi^M(\a)=\phi^N(f(\a))$.
\end{lemma}

\begin{definition}
\em{ $M$ and $N$ are \emph{elementarily equivalent}\index{elementary equivalent}, denoted by
$M\equiv N$\index{$\equiv$},
if for every sentence $\sigma$ one has that $\sigma^M=\sigma^N$.
A function $f:M\rightarrow N$ is called an \emph{elementary embedding}\index{elementary embedding} if
for every formula $\phi(\x)$ and $\a\in M$ one has that $\phi^M(\a)=\phi^N(f(\a))$.
If $f=id$ is an elementary embedding, $M$ is called an \emph{elementary substructure}\index{elementary substructure}
of $N$. This is denoted by $M\preccurlyeq N$.\index{$\preccurlyeq$}}
\end{definition}


\begin{proposition} \label{tarski} \emph{(Tarski-Vaught's test)}
Assume $M\subseteq N$. Then, $M\preccurlyeq N$ if and only if for
every formula $\phi(\bar{x},y)$ and $\bar{a}\in M$ one has that
$$\inf\{\phi^N(\bar{a},b)\ | \ b\in N\}=\inf\{\phi^N(\bar{a},c) \ | \ c\in M\}.$$
\end{proposition}

\begin{proposition}\label{chain}
Assume $(M_i: i\in \kappa)$ is an elementary chain of $L$-structures (i.e. $M_i\preccurlyeq M_j$ for $i<j$).
Then $M_j\preccurlyeq M$ for each $j<\kappa$ where $M$ is the completion of $\bigcup_i{M_i}$.
\end{proposition}

The \emph{density character}\index{density character} of a metric space $M$ is \index{$\mathsf{dc}(M)$}
$$\mathsf{dc}(M)=\min\{|X|:\ X\subseteq M \ \mbox{is\ dense}\}.$$

\begin{proposition}\label{downward} \emph{(Downward)}
Let $|L|+\aleph_0\leqslant\kappa$. Assume $N$ is an $L$-structure and $X\subseteq N$
has density character $\leqslant\kappa$.
Then, there exists $X\subseteq M\preccurlyeq N$ such that $\mathsf{dc}(M)\leqslant\kappa$.
\end{proposition}

A \emph{condition}\index{condition} is an expression of the form $\phi\leqslant\psi$ where
$\phi$ and $\psi$ are formulas. If $\phi$, $\psi$ are sentences, it is called a \emph{closed condition}.
The expression $\phi=\psi$ abbreviates $\{\phi\leqslant\psi, \psi\leqslant\phi\}$.\ \
$M$ is a \emph{model} of $\phi\leqslant\psi$ if $\phi^M\leqslant\psi^M$.
One may also write this by $M\vDash\phi\leqslant\psi$ and say that $\phi\leqslant\psi$ is satisfied by $M$.
A set of closed conditions is called a \emph{theory}.
$M$ is a model of a theory $T$, written $M\vDash T$, if it is a model of every condition in it.
Finally, one writes $T\vDash\phi\leqslant\psi$ if for every $M$,\ \ $M\vDash T$ implies $M\vDash\phi\leqslant\psi$.
\bigskip

\begin{example}
\em{1) The axioms of metric spaces as well as the Lipchitzness of relations and operations are stated by affine conditions.

2) In first order logic, an equation is an expression of the form $$\forall\x(t_1(\x)=t_2(\x)).$$
This can be restated by the condition $$\sup_{\x} d(t_1(\x),t_2(\x))=0.$$
So, equational theories are special kinds of affine theories.

3) Some metric variants of first order structures can be stated by affine conditions.
For example, a metric $d$ on a group $G$ is bi-invariant if $d(ax,ay)=d(x,y)=d(xa,ya)$ for all $a$.
This property is stated by AL-conditions and such structures are Lipschitz.
As a special case, a compact connected simple Lie group has a unique bi-invariant metric which makes
it a Lipschitz structure.

4) A classical dynamical system is a pair $(M,f)$ where $M$ is a compact metric space and
$f:M\rightarrow M$ is continuous. If $f$ is Lipschitz, then $(M,d,f)$ is a metric structure.
Continuous flows are other form of such structures.
A dynamical system may be further equipped with a group structure and $f$ may be a group automorphism or an affine transformation.
All such structures can be axiomatized in appropriate languages.

5) Measure algebras and measure algebras equipped with automorphisms can be axiomatized in
appropriate Lipschitz languages (see also \S \ref{Examples}).

6) Banach spaces can be axiomatized in the many sorted variant of AL. Axiomatizing Hilbert spaces
requires the multiplication connective which is absent in AL.
However, it can be axiomatized in the variations of AL using the parallelogram identity.
Similarly, abstract $L^p$ spaces can be axiomatized in variations of AL (see \S \ref{ALp}).

7) O. Gross \cite{Morris} has proved that for every compact connected metric space $M$ there is a unique $r$
such that for each $x_1,...,x_n$ there exists $y$ with $\frac{1}{n}\sum_id(x_i,y)=r$.
This is called the \emph{rendez-vous number} of $M$. That the rendez-vous number of $M$ is $r$
can be stated by the affine conditions
$$\sup_{x_1\ldots x_n}\inf_y\frac{1}{n}\sum_id(x_i,y)\leqslant r
\leqslant\inf_{x_1\ldots x_n}\sup_y\frac{1}{n}\sum_id(x_i,y).$$
Using this notion, we can prove that the unit circle and unit sphere are not elementarily
equivalent as they have different rendez-vous numbers.}
\end{example}

\newpage\section{Basic technics} \label{Basic technics}
In first order logic, the compactness theorem is proved either by the ultraproduct
construction or by Henkin's method. In this section we give the affine variants of these constructions.
A theory $T$ is \emph{affinely satisfiable}\index{affinley satisfiable} if for every
conditions $\phi_1\leqslant\psi_1,\ ...,\ \phi_n\leqslant\psi_n$ in $T$
and $0\leqslant r_1,...,r_n$, the condition $\sum_ir_i\phi_i\leqslant\sum_ir_i\psi_i$ is satisfiable.
The set of all such combinations is called the \emph{affine closure}\index{affine closure} of $T$.
So, $T$ is affinely satisfiable if every condition in its affine closure is satisfiable.
To prove the affine compactness theorem by ultrameans, we need some basic facts on
integration with respect to finitely additive measures which can be seen in the appendix \S \ref{Appendix}.

\subsection{Affine compactness}\label{Affine compactness}

\noindent{\bf The ultramean construction} \label{The ultramean construction}\\
Affine variant of the ultraproduct construction is the \emph{ultramean \index{ultramean} construction}.
An ultracharge\index{ultracharge} on an index set $I$ is a finitely additive probability
measure on the power set of $I$. Let $\mu$ be an ultracharge on $I$.
Let $L$ be a Lipschitz language and $(M_{i}, d_{i})_{i\in I}$ be a family of $L$-structures.
First define a pseudo-metric on $\prod_{i\in I}M_i$ by setting $$d((a_i),(b_i))=\int d_{i}(a_i,b_i)d\mu.$$
Clearly, $d((a_i),(b_i))=0$ is an equivalence relation. The equivalence class of $(a_i)$ is denoted by $[a_i]$.
Let $M$ be the set of these equivalence classes. Then, $d$ induces a metric on $M$ which is denoted again by $d$.
So, $$d([a_i],[b_i])=\int d_{i}(a_i,b_i)d\mu.$$
Define an $L$-structure on $(M,d)$ as follows:
$$c^M=[c^{M_i}]$$
$$F^M([a_i],...)=[F^{M_i}(a_i,...)]$$
$$R^M([a_i],...)=\int R^{M_i}(a_i,...)d\mu$$ where $c,F,R\in L$.
One verifies that $F^M$ and $R^M$ are well-defined as well as $\lambda_F$-Lipschitz
and $\lambda_R$-Lipschitz respectively.
For example, assume $R$ is unary. For $a=[a_i]$ and $b=[b_i]$ one has that
$$R^{M_i}(a_i)-R^{M_i}(b_i)\leqslant\lambda_R\ d(a_i,b_i)\ \ \ \ \ \ \ \ \forall i\in I.$$
So, by integrating,
$$R^{M}(a)-R^{M}(b)\leqslant\lambda_R\ d(a,b).$$
Similarly, for each $i$ one has that $$-\mathbf{b}_R\leqslant R^{M_i}(a_i)\leqslant\mathbf{b}_R.$$
Hence, by integrating, $$-\mathbf{b}_R\leqslant R^{M}(a)\leqslant\mathbf{b}_R.$$

The structure $M$ is called the ultramean of the structures $M_i$ and is denoted by $\prod_\mu M_i$.
Note that it may be incomplete. Any ultrafilter $\mathcal{D}$ on $I$ can be regarded as a $\{0,1\}$-valued ultracharge.
Then, the ultraproduct $\prod_{\mathcal D} M_i$ coincides with the ultramean $\prod_{\mu_{\mathcal D}}M_i$.
In the ultracharge case, one has the following variant of {\L}o\'{s} theorem.

\begin{theorem} \emph{(Ultramean)} \label{th1}
For each AL-formula $\phi(x_{1}, \ldots, x_{n})$
and $[a^{1}_{i}], \ldots, [a^{n}_{i}]\in M$
$$\phi^{M}([a^{1}_{i}],\ldots, [a^{n}_{i}])=
\int\phi^{M_{i}}(a^{1}_{i},\ldots, a^{n}_{i})d\mu.$$
\end{theorem}
\begin{proof}
We prove the claim by induction on the complexity of formulas.
The atomic and connective cases are obvious. Assume the claim holds for $\phi(x)$.
We show that $$\sup_x\phi^{M}(x)=\int\sup_x\phi^{M_i}(x)d\mu.$$
Given $\epsilon>0$, for suitable $[a_i]$ we have
$$\sup_x\phi^{M}(x)-\epsilon<\phi^{M}([a_i])=\int\phi^{M_i}(a_i)d\mu\leqslant\int\sup_x\phi^{M_i}(x)d\mu.$$
Since $\epsilon$ is arbitrary, one direction of the claim is obtained.
Conversely, for each $i$ take $a_i$ such that $$\sup_x\phi^{M_i}(x)-\epsilon\leqslant\phi^{M_i}(a_i).$$
Let $a=[a_i]$. Then, by integrating,
$$\int\sup_x\phi^{M_i}(x)d\mu-\epsilon\leqslant\int\phi^{M_i}(a_i)d\mu=\phi^{M}(a)\leqslant\sup_x\phi^M(x).$$
The case $\inf_x\phi$ is similar.
\end{proof}

If $M_i=N$ for every $i$, $\prod_{\mu}M_i$ is denoted by $N^\mu$ and is called the
\emph{powermean}\index{powermean} of $N$. The map $a\mapsto[a]$ is the diagonal embedding.

\begin{proposition}
The diagonal embedding is an elementary embedding.
\end{proposition}

For example, if $I=\{1,2\}$ and $\mu(\{1\})=\mu(\{2\})=\frac{1}{2}$,
the ultramean of $(M_i)_{i\in I}$ is denoted by $M=\frac{1}{2}M_1+\frac{1}{2}M_2$.
In this case, for each affine sentence $\sigma$ one has that
$$\sigma^M=\frac{1}{2}\sigma^{M_1}+\frac{1}{2}\sigma^{M_2}.$$

Note that if $|N|\geqslant2$ and $\mu$ is not an ultrafilter,
then $N^\mu$ is a proper extension of $N$. Since no compact model
has a proper elementary extension in the CL sense, we conclude that
$\equiv_{\mathrm{AL}}$ is strictly weaker than $\equiv_{\mathrm{CL}}$.

The following theorem was originally proved by R. Safari (see \cite{Bagheri-Lip}).
An other variant of the proof is given here.
Let $\mathbb{D}(L)$ be the vector space of $L$-sentences equipped with the norm and order defined by:
$$\|\sigma\|=\sup\{|\sigma^M|:\ \ M\ \mbox{is\ an\ $L$-structure}\}$$
$$0\leqslant\sigma\ \ \ \mbox{if}\ \ \ \vDash0\leqslant\sigma$$
where $\sigma,\eta$ are identified if $\vDash\sigma=\eta$.

\begin{theorem} \label{compactness}\index{affine compactness}
\emph{(Affine compactness)} Every affinely satisfiable theory is satisfiable.
\end{theorem}
\begin{proof}
Let $\Sigma$ be an affinely satisfiable theory which we may further assume it is maximal with this property.
For every sentence $\sigma$, let $$q(\sigma)=\inf\{r:\sigma\leqslant r\in\Sigma\}.$$
By maximality, $q$ is defined for every $\sigma$ and it is sublinear, i.e.
$$\ \ \ \ \ \ q(r\sigma)=rq(\sigma)\ \ \ \ \ \ \mbox{for}\ r\geqslant0$$
$$q(\sigma+\eta)\leqslant q(\sigma)+q(\eta).$$
Let $T_0$ be the identity map on the linear subspace $\Rn\subseteq\mathbb{D}(L)$.
Then, $T_0\leqslant q$ on $\Rn$ and by the Hahn-Banach extension theorem (\cite{Aliprantis-Inf}, Th 8.30),
$T_0$ extends to a linear map $T$ on $\mathbb{D}(L)$ such that $T(\sigma)\leqslant q(\sigma)$ for every $\sigma$.
Note that $T$ is positive. In particular, if $\vDash\sigma\leqslant0$ then $\sigma\leqslant0\in\Sigma$
and hence $T(\sigma)\leqslant q(\sigma)\leqslant0$.

Let $\{M_i\}_{i\in I}$ be a set containing a model from each equivalence class of the relation $M\equiv N$.
We put the discrete topology on $I$.
So, we may write $\Rn\subseteq\mathbb{D}(L)\subseteq\mathbf{C}_b(I)$ if we identify
$\sigma\in\mathbb{D}(L)$ with the map $i\mapsto\sigma^{M_i}$.
Since $\mathbb{D}(L)$ majorizes $\mathbf{C}_b(I)$, by the Kantorovich extension theorem (\S \ref{Appendix}),
$T$ is extended to a positive linear functional $\bar{T}$ on $\mathbf{C}_b(I)$.
By (a variant of) Riesz representation theorem (\S \ref{Appendix}) 
there is a (maximal) probability charge $\mu$ on $I$
such that for every $\sigma$ $$\bar{T}(\sigma)=\int\sigma^{M_i}d\mu.$$
Let $M=\prod_{\mu}M_i$. Then, for each $\sigma$
$$\sigma^M=\int\sigma^{M_i}d\mu=T(\sigma)\leqslant q(\sigma).$$
Finally, if $\sigma\leqslant\eta\in\Sigma$, then $\sigma^M-\eta^M\leqslant q(0)=0$
and hence $M\vDash\Sigma$. We may also complete $M$ by Proposition
\ref{completion of prestructures} to obtain a complete model of $\Sigma$.
\end{proof}
\vspace{2mm}

The proof can be arranged to show that if $T$ is affinely satisfiable in a class
$\mathcal K$ of $L$-structures, then $T$ has a model of the form $\prod_\mu M_i$
where every $M_i$ belongs to $\mathcal{K}$.
For this purpose, it is sufficient to replace $\mathbb{D}(L)$ with $\mathbb{D}(\mathcal{K})$ consisting
of equivalence classes of $L$-sentences where $\sigma\equiv\eta$ if $\sigma^M=\eta^M$ for all $M\in\mathcal{K}$.
Also, $\sigma\leqslant\eta$ if $\sigma^M\leqslant\eta^M$ for all $M\in\mathcal{K}$
and $$\|\sigma\|=\sup\{|\sigma^M|:\ M\in\mathcal{K}\}.$$
Then, let $\{M_i\}_{i\in I}$ contain a model from each equivalence class of the relation
$M\equiv N$ when restricted to $\mathcal{K}$ and redo the proof in the new situation.

A class $\mathcal{K}$ of $L$-structures is elementary if there is a theory $T$ such that $\mathcal{K}=Mod(T)$.
It is closed under ultramean if $\prod_{\mu}M_i$ (or rather its completion) belongs to $\mathcal{K}$ whenever
$M_i\in\mathcal{K}$ for every $i$. Being closed under $\equiv$ is defined similarly.

\begin{theorem} \emph{(Axiomatizability)}
A class $\mathcal K$ of $L$-structures is elementary if and only if it is closed under
ultramean and elementary equivalence.
\end{theorem}
\begin{proof} We prove the nontrivial direction. Assume $\mathcal K$ is closed under ultramean
and elementary equivalence. Let $T=Th(\mathcal K)$, i.e. the set of conditions holding
in every $M\in\mathcal K$. We show that $\mathcal K=Mod(T)$. Clearly, $\mathcal K\vDash T$.
Conversely, assume $M$ is an arbitrary model of $T$.
Assume $M\vDash0\leqslant\sigma$. By the definition, for each $\epsilon>0$ there exists $N\in\mathcal K$
such that $N\vDash-\epsilon\leqslant\sigma$. Therefore, since $\mathcal K$ is closed under ultraproduct, there is
$N\in\mathcal K$ such that $N\vDash0\leqslant\sigma$.
This shows that every condition in $Th(M)$ is satisfied by some model in $\mathcal K$.
Hence, $Th(M)$ is satisfied by an ultramean of members of $\mathcal K$,
i.e. $M$ is elementarily equivalent to an ultramean of members of $\mathcal K$.
So, by the assumptions, $M\in\mathcal K$.
\end{proof}
\bigskip

\noindent{\bf Henkin's method}\index{Henkin's method}\\
The affine compactness theorem was first proved by Henkin's method.
An $L$-theory $T$ has the \emph{witness property}\index{witness property} if for each $L$-formula $\phi(x)$,
there is a constant symbol $c\in L$ such that $\sup_x\phi(x)\leqslant\phi(c)\in T$.
Using ultraproducts, one verifies that if $T,-r\leqslant\sigma$ is affinely
satisfiable for each $r>0$, then $T,0\leqslant\sigma$ is affinely satisfiable.

\begin{lemma} \label{witness}
Let $T$ be an affinely satisfiable theory in $L$. Then there are $\bar L\supseteq L$
and maximal affinely satisfiable $\bar T\supseteq T$ in $\bar L$ having the witness property.
\end{lemma}
\begin{proof}
We may assume that $T$ is affinely closed. Let $\phi(x)$ be a formula and $c$ be a new constant symbol.
We claim that $T\cup\{\sup_x\phi(x)\leqslant\phi(c)\}$ is affinely satisfiable. Suppose not.
Then, there are $\sigma\leqslant\eta$ in $T$ and $r,s>0$ such that
$$r\sup_x\phi(x)+s\sigma\leqslant r\phi(c)+s\eta$$
is not satisfiable. So, for some $\epsilon>0$
$$\vDash r\phi(c)+s\eta\leqslant r\sup_x\phi(x)+s\sigma-\epsilon.$$
So, $$\vDash r\sup_x\phi(x)+s\eta\leqslant r\sup_x\phi(x)+s\sigma-\epsilon$$
which implies that
$$\vDash \eta\leqslant \sigma-\frac{\epsilon}{s}.$$
This is a contradiction. Now, an easy transfinite induction shows that there is a language $\bar L\supseteq L$
and an affinely satisfiable $\bar L$-theory $T_1\supseteq T$ which has the witness property.
Let $\bar T\supseteq T_1$ be a maximal affinely satisfiable $\bar L$-theory containing $T_1$.
Clearly, $\bar T$ has the required properties.
\end{proof}

\begin{theorem} \emph{(Affine compactness)}
Every affinely satisfiable theory $T$ is satisfiable.
\end{theorem}
\begin{proof}
We just give sketch of the proof. By Lemma \ref{witness}, we may further assume
that $T$ has the witness property and is maximal.
Let $M$ be the set of all constant symbols of the language.
For $c,e\in M$ define $d(c,e)=r$ if the condition $d(c,e)=r$ belongs to $T$.
This is a pseudometric on $M$.
Also, for $c,R,F\in L$ define
$$c^M=c$$
$$R^M(c_1,...,c_n)=r \ \ \ \ \ \ \mbox{if} \ \ \ \ \ \ \ \ \ \ R(c_1,...,c_n)\in T$$
$$\ \ \ \ \ \ \ F^M(c_1,...,c_n)=e \ \ \ \ \ \ \ \  \mbox{if} \ \ \ \ \ \ d(F(c_1,...,c_n),e)=0\in T.$$
This defines a well-defined $L$-prestructure on $M$. Then, one verifies by induction on the complexity of $\phi$
that for every $c_1,...,c_n$ and $r\in\Rn$
$$\phi^M(c_1,...,c_n)=r\ \ \ \ \mbox{iff} \ \ \ \ \phi(c_1,...,c_n)=r\in T.$$
Finally, use Proposition \ref{exist} to find a model $\bar{M}$ of $T$.
\end{proof}

In the next sections, we use a modified form of Henkin's method to prove the existence of extremal models.

\subsection{Upward, JEP and AP} 
An immediate consequence of the affine compactness theorem is that if both $\sigma\leqslant0$
and $0\leqslant\sigma$ are satisfiable, then so is $\sigma=0$.
This is because the set $\{0\leqslant\sigma,\sigma\leqslant0\}$ is affinely satisfiable.
It is also easy to prove the following special form of the affine compactness theorem.
Let $\Sigma$ be a set of closed conditions of the form $\sigma=0$ such that $r\sigma+s\eta=0\in\Sigma$
whenever $\sigma=0, \eta=0\in\Sigma$ and $r,s\in\Rn$. If every $\sigma=0\in\Sigma$
is satisfiable, then $\Sigma$ is satisfiable.
Finite satisfiability is stronger than affine satisfiability.
So, one verifies that if $T\vDash0\leqslant\theta$ then for each $\epsilon>0$ there is a finite
$\Delta\subseteq T$ such that $\Delta\vDash-\epsilon\leqslant\theta$.
Using affine compactness, we can further prove the following:

\begin{lemma} \label{proofs}
Let $S$ be a set of conditions of the form $0\leqslant\sigma$.

\emph{(i)} If $T\cup S\vDash0\leqslant\theta$, then for each $\epsilon>0$ there exists
$0\leqslant\sigma$ in the affine closure of $S$ such that $T\vDash\sigma\leqslant\theta+\epsilon$.
In particular, $T,0\leqslant\sigma\vDash0\leqslant\theta+\epsilon$.

\emph{(ii)} If $T,0\leqslant\sigma\vDash0\leqslant\theta$, for each $\epsilon>0$ there exists
$\delta>0$ such that $T,-\delta\leqslant\sigma\vDash-\epsilon\leqslant\theta$.
\end{lemma}
\begin{proof} (i): Otherwise, there exists $\epsilon>0$ such that
for every $0\leqslant\sigma$ in the affine closure of $S$, the theory
$T\cup\{\theta+\epsilon\leqslant\sigma\}$ is satisfiable.
This implies that $T\cup S\cup\{\theta+\epsilon\leqslant0\}$ is affinely satisfiable.\
(ii): By (i), there exists $r\geqslant0$
such that $T\vDash r\sigma\leqslant\theta+\frac{\epsilon}{2}$.
If $r=0$ we have that $T\vDash-\epsilon\leqslant\theta$.
Otherwise, $T,\frac{-\epsilon}{2r}\leqslant\sigma\vDash-\epsilon\leqslant\theta$.
\end{proof}

Let $M$ be an $L$-structure. The \emph{diagram} and \emph{elementary diagram} of $M$
are respectively defined as follows:
$$diag(M):=\{0\leqslant\phi(\a):\ 0\leqslant\phi^M(\a),\ \phi\ \mbox{is\ quantifier-free},\ \a\in M\}$$
$$ediag(M):=\{0\leqslant\phi(\a):\ 0\leqslant\phi^M(\a),\ \phi\ \mbox{is\ arbitrary}, \a\in M\}.$$
These are theories in $L(M)$, where we add a constant symbol for each $a\in M$.
It is then clear that $N\vDash diag(M)$ if and only if there is an embedding from $M$ into $N$.
Similarly, $N\vDash ediag(M)$ if and only if there is an elementary embedding from $M$ into $N$.

\begin{proposition} \label{upward} \emph{(Upward)}
Every model $M$ having at least two elements has arbitrarily large elementary extensions.
\end{proposition}
\begin{proof} Let $a,b\in M$ be such that $d(a,b)=r>0$. Let $\kappa\geqslant2$ be a cardinal
and $\{c_i:i<\kappa\}$ a set of new (and distinct) constant symbols. We show that
$$\Sigma=ediag(M)\cup\{\frac{r}{2}\leqslant d(c_i,c_j) \ : \ \ \  i<j<\kappa\}.$$
is affinely satisfiable.
In fact, we must show that any condition of the form
$$\frac{r}{2}\sum_{i<j\leqslant n}\alpha_{ij}\leqslant\sum_{i<j\leqslant n}\alpha_{ij} d(c_i,c_j)=\sigma_n\ \ \ \ \ \ \ \ \ \alpha_{ij}\geqslant0$$
is satisfiable in $M$. Suppose this has been done for $n-1$ and
$c_0^M,...,c_{n-1}^M$ are interpreted by $a$ or $b$ such that
$$\frac{r}{2}\sum_{i<j\leqslant n-1}\alpha_{ij}\leqslant\sigma_{n-1}^M.$$
Then, $\sigma_n=\sigma_{n-1}+\sum_{i<n}\alpha_{in}d(c_i,c_n)$ and
for at least one of $c_n^M=a$ and $c_n^M=b$ we must have that
$$\frac{r}{2}\sum_{i<j\leqslant n}\alpha_{ij}\leqslant\sigma_n^M.$$
Now, any model of $\Sigma$ extends $M$ elementarily and has cardinality at least $\kappa$.

Alternatively, one may use the ultramean method. Let $\wp_0$ be a probability charge on a set $I$ whose
corresponding algebra of events has cardinality at least $\kappa$, i.e. there are
sets $A_i\subseteq I$,\ \ $i<\kappa$, such that $\wp_0(A_i\Delta A_j)\neq0$ for $i\neq j$.
For this purpose, the representation theorem for probability algebras (\cite{BBHU}, Th. 16.1) may be used.
Extend $\wp_0$ to a maximal probability charge on $I$ by Theorem \ref{ultracharge-extension}.
Then $M^\wp$ has cardinality at least $\kappa$.
\end{proof}
\vspace{2mm}

Combining the downward and upward L\"{o}wenheim-Skolem theorems, one proves that

\begin{proposition}
If $|L|+\aleph_0\leqslant\kappa$, every theory which has a model of cardinality
at least $2$ has a model of density character $\kappa$.
\end{proposition}
\vspace{2mm}

\begin{proposition} \emph{(Elementary JEP and AP)}

\emph{(i)} If $M\equiv N$, then there are $K$ and elementary embeddings $f:M\rightarrow K$, $g:N\rightarrow K$.

\emph{(ii)} Let $f:A\rightarrow M$ and $g:A\rightarrow N$ be elementary.
Then there are $K$ and elementary embeddings $f':M\rightarrow K$
and $g':N\rightarrow K$ such that $f'f=g'g$.
\end{proposition}
\begin{proof} (i) Let $\Sigma=ediag(M)\cup ediag(N)$ and assume $0\leqslant\phi^N(\b)$.
Then $0\leqslant\sup_{\x}\phi(\x)$ holds in $N$ and hence in $M$. Since, $ediag(N)$ is affinely closed, we conclude by the
affine compactness theorem that $\Sigma$ is satisfiable. Now, every model of $\Sigma$ does the job.
The second part is proved similarly.
\end{proof}

\subsection{Quantifier-elimination} \label{Quantifier-elimination}
A condition is called universal if it of the form $0\leqslant\inf_{\x}\phi(\x)$ where $\phi$ is quantifier-free.
A theory $T$ is \emph{universal}\index{universal theory} if it can be axiomatized by a set of universal conditions.
$T_{\forall}$ denotes the set of universal consequences of $T$.
A theory $T$ is \emph{closed under substructure} if whenever $A\subseteq M\vDash T$,
one has that $A\vDash T$.

\begin{lemma} $A\vDash T_{\forall}$ if and only if there exists $M$ such that $A\subseteq M\vDash T$.
\end{lemma}
\begin{proof}  Assume $A\vDash T_{\forall}$.
We claim that $T\cup diag(A)$ has a model.
For this purpose it is sufficient to show that for each $0\leqslant\phi(\a)$
in $diag(A)$, $T\cup\{0\leqslant\phi(\a)\}$ is satisfiable.
Suppose not. Then, for some $0\leqslant\phi(\a)\in diag(A)$ and $\epsilon>0$
one has that $T\vDash\phi(\a)\leqslant-\epsilon$.
So, $T\vDash\sup_{\x}\phi(\x)\leqslant -\epsilon$ and hence
$\epsilon\leqslant\inf_{\x}-\phi(\x)$ belongs to $T_{\forall}$.
However this condition does not hold in $A$. Now, let $M$ be a model of $T\cup diag(A)$.
Then $A\subseteq M\vDash T$. The other direction is obvious.
\end{proof}

For $|\x|=n\geqslant0$, \ $\mathbb{D}_{n}(T)$\index{$\mathbb{D}_n(T)$} denotes the real vector space consisting
of formulas with free variables $\x$ where $\phi(\x)$, $\psi(\x)$ are identified if $T\vDash\phi=\psi$.
This is a partially ordered normed space by
$$\|\phi(\x)\|_T=\sup\big\{|\phi^M(\a)|:\ M\vDash T,\ \a\in M\big\}$$
$$0\leqslant_T\phi\ \ \ \ \mbox{if} \ \ \ \ T\vDash0\leqslant\inf_{\x}\phi(\x).$$
So, $\phi$ and $\psi$ are $T$-equivalent if and only if $\|\phi-\psi\|_T=0$.

\begin{definition}
\emph{$T$ has \emph{quantifier-elimination}\index{quantifier-elimination} if for each $n\geqslant1$ and $\x$ with $|\x|=n$,
the subspace consisting of quantifier-free formulas $\phi(\x)$ is dense in $\mathbb{D}_n(T)$.}
\end{definition}

For complete $T$, every sentence is $T$-equivalent to some $r$.
So, in this case, the definition of quantifier-elimination is the same if we require it for every $n\geqslant0$.
A set $\Gamma$ of formulas is called \emph{convex} if for each nonnegative $r,s$ with $r+s=1$
and $\phi,\psi\in\Gamma$, one has that $r\phi+s\psi\in\Gamma$.
By $\Gamma\leqslant r$ we mean the set $\{\theta\leqslant r: \theta\in\Gamma\}$.

\begin{proposition} \label{quant-el-loc}
For each $\phi(\x)$ with $|\x|\geqslant1$, the following are equivalent:

\emph{(i)} For each $\epsilon>0$ there is a quantifier-free $\theta(\x)$ such that $\|\phi-\theta\|_T\leqslant\epsilon$.

\emph{(ii)} If $A\subseteq M\vDash T$,\ \ $A\subseteq N\vDash T$ and $\a\in A$ then $\phi^M(\a)=\phi^N(\a)$.
\end{proposition}
\begin{proof} (i)$\Rightarrow$(ii) is obvious. (ii)$\Rightarrow$(i):
Let $$\Gamma(\x)=\{\theta(\x):\ \ T\vDash\theta\leqslant\phi\ \ \& \ \ \theta\ \mbox{is\ quantifier-free}\}.$$
$\Gamma$ is convex. We have only to show that for each $\epsilon>0$ there exists $\theta\in\Gamma$ such that
$T\vDash\phi-\epsilon\leqslant\theta$. Assume not. Then, by convexity, for some $\epsilon>0$,
$$T\cup\{\theta(\d)+\epsilon\leqslant\phi(\d):\ \theta\in\Gamma\}$$
is affinely satisfiable, where $\d$ is a tuple of new constant symbols replacing $\x$.
Let $M$ be a model for it and $\gamma=\phi^M(\d)$. Let $A$ be the (incomplete) substructure of $M$ generated by $\d^M$.
Then $$A\vDash \Gamma(\d)\leqslant\gamma-\epsilon\ \ \ \ \ (*).$$

\noindent{\sc Claim}: $\Sigma=T\cup diag(A)\cup\{\phi(\d)\leqslant \gamma-\epsilon\}$ is satisfiable.\\
\noindent{\sc Proof}: Assume not. Then for some $r,s\geqslant0$
and $0\leqslant\xi(\d)\in diag(A)$, the theory
$$T, s\phi(\d)\leqslant r\xi(\d)+s(\gamma-\epsilon)$$ is unsatisfiable.
This means that for some $\delta>0$,
$$T\vDash r\xi(\d)+s(\gamma-\epsilon)\leqslant s\phi(\d)-\delta.$$
Since $s$ cannot be zero, we must have that
$$T\vDash\frac{r}{s}\xi(\x)+\gamma-\epsilon+\frac{\delta}{s}\leqslant\phi(\x).$$
Therefore, $$\frac{r}{s}\xi(\x)+\gamma-\epsilon+\frac{\delta}{s}\in\Gamma$$ and hence by $(*)$,
$A$ satisfies $\frac{r}{s}\xi(\d)+\frac{\delta}{s}\leqslant0$ which is a contradiction.

Finally, let $N\vDash\Sigma$. Then, $A\subseteq N\vDash T$ while
$\phi^N(\d)\leqslant\gamma-\epsilon$. This is a contradiction.
\end{proof}
\vspace{1mm}

It is proved by induction on the complexity of formulas that a theory $T$ has
quantifier-elimination if and only if every formula $\inf_y\phi(\x,y)$,
where $\phi$ is quantifier-free, is approximated by quantifier-free formulas.

\begin{corollary} \label{quant-el} The following are equivalent:

\emph{(i)} $T$ has quantifier-elimination

\emph{(ii)} If $M,N\vDash T$, then every embedding of a substructure of $M$ into $N$ extends
to an embedding of $M$ into an elementary extension of $N$.
\end{corollary}
\begin{proof}
(i)$\Rightarrow$(ii): Assume $A\subseteq M$ and $f:A\rightarrow N$ is an embedding
(which may be assumed to be the inclusion map).
It is sufficient to show that the theory $$T\cup diag(M)\cup ediag(N)$$ is satisfiable.
In fact, we have only to show that every $0\leqslant\phi(\a,\bar m)+\psi(\a,\bar n)$
is satisfiable in a model of $T$, where $\a\in A$, $0\leqslant\phi^M(\a,\bar m)$,
$0\leqslant\psi^N(\a,\bar n)$ and $\phi$ is quantifier-free.
Let $\epsilon>0$. By quantifier-elimination, there is a quantifier-free
$\theta(\x)$ which is $\epsilon$-close to $\sup_{\y}\phi(\x,\y)$ in models of $T$.
So, it is sufficient to check that $0\leqslant\theta(\a)+\epsilon+\psi(\a,\bar n)$
is satisfied in $N$ which is clearly true since
$$0\leqslant\phi^M(\a,\bar m)\leqslant\sup_{\x}\phi^M(\a,\x)\leqslant\theta^M(\a)+\epsilon=\theta^N(\a)+\epsilon.$$

(ii)$\Rightarrow$(i): We must verify the condition (ii) of Proposition
\ref{quant-el-loc} for every $\inf_y\phi(\x,y)$ where $\phi$ is quantifier-free.
Assume $\a\in A\subseteq M,N\vDash T$. Let $f:M\rightarrow N'$ be an embedding where
$N\preccurlyeq N'$ and $f|_A=id$.
Then $$\inf_y\phi^N(\a,y)=\inf_y\phi^{N'}(\a,y)\leqslant\inf_y\phi^M(\a,y).$$
The reverse inequality holds similarly. Hence $T$ has quantifier-elimination.
\end{proof}

As in first order logic (see \cite{Marker} p.78), we say that a theory $T$ has \emph{algebraically prime models} of for each
$A\vDash T_\forall$ there are $M\vDash T$ and embedding $i:A\rightarrow M$ such that for every
$N\vDash T$ and embedding $j:A\rightarrow N$, there is an embedding $h:M\rightarrow N$ with $j=h\circ i$.
For $M,N\vDash T$ and $M\subseteq N$, we say $M$ is \emph{simply closed} in $N$, denoted by $M\prec_s N$,
if for every quantifier-free formula $\phi(\x,y)$ and $\a\in M$ one has that
$$\inf_{y\in M}\phi^M(\a,y)=\inf_{y\in N}\phi^N(\a,y).$$
The following is then routine.

\begin{proposition}
If $T$ satisfies the following conditions, it has quantifier-elimination.

\emph{(i)} $T$ has algebraically prime models

\emph{(ii)} for every $M,N\vDash T$, if $M\subseteq N$ then $M\prec_s N$.
\end{proposition}

\subsection{Preservation theorems}
Preservation theorems relate syntactical properties of theories to the category theoretic properties of their models.
In this subsection we prove appropriate variants of some of these theorems in affine continuous logic.

Let $\Delta$ be an affinely closed set conditions of the form $0\leqslant\phi$ and
assume $-\epsilon\leqslant\phi\in\Delta$ whenever $0\leqslant\phi\in\Delta$ and $\epsilon>0$.
A $\Delta$-theory is a theory $T$ which can be axiomatized by $\Delta$-conditions.

\begin{lemma}\label{localization}
Assume $\Delta$ is as above. Then, the following are equivalent:

\emph{(i)} $T$ is a $\Delta$-theory.

\emph{(ii)} If $M\vDash T$ and every $\Delta$-condition which holds in $M$ holds in $N$, then $N\vDash T$.
\end{lemma}
\begin{proof} (i)$\Rightarrow$(ii) is obvious. (ii)$\Rightarrow$(i): Let
$$T_\Delta=\{0\leqslant\phi\in\Delta:\ T\vDash0\leqslant\phi\}.$$
Every model of $T$ is a model of $T_\Delta$. Conversely assume $N\vDash T_\Delta$.
Let $$\Sigma=\{-\epsilon\leqslant\phi:\ 0\leqslant\phi\in\Delta,\ 0<\epsilon\ \mbox{and}\
N\vDash\phi\leqslant-\epsilon\}\cup\{0\leqslant0\}.$$
Note that $\Sigma$ is affinely closed. We show that $\Sigma\cup T$ is satisfiable.
It is sufficient to show that $T\cup\{\phi\leqslant -\epsilon\}$
is satisfiable for each $\phi\leqslant-\epsilon\in\Sigma$.
Suppose not. Then for some $\phi\leqslant-\epsilon\in\Sigma$
and $0<\delta<\epsilon$, one has that $T\vDash-\epsilon+\delta\leqslant\phi$.
Hence $-\epsilon+\delta\leqslant\phi\in T_\Delta$. This is a contradiction as
this condition does not hold in $N$.
Let $M$ be a model of $\Sigma\cup T$. Then by (ii) we must have that $N\vDash T$.
\end{proof}

A $\forall\exists$-condition is a condition of the form
$0\leqslant\inf_{\x}\sup_{\y}\phi(\x,\y)$ where $\phi$ is quantifier-free.
A theory axiomatized by such conditions is called a $\forall\exists$-theory.
A theory $T$ is \emph{inductive}\index{inductive theory} if the union of every chain of
models of $T$ is a model of $T$.

\begin{proposition}
A theory $T$ is inductive if and only if it is a $\forall\exists$-theory.
\end{proposition}
\begin{proof} The `if' part is obvious. We prove the `only if' part. Assume $T$ is inductive.
Let $\Delta$ be the set of all consequence of $T$ which are $\emptyset$-equivalent to a $\forall\exists$-condition.
Then, $\Delta$ is affinely closed and satisfies the requirements of Lemma \ref{localization}.
We just need to verify the part (ii) of this lemma.
So, assume $M\vDash T$ and every $\forall\exists$-condition which holds in $M$ holds in $N$.
We must prove that $N\vDash T$. For this purpose, we first use this assumption to prove the following subclaim:
\begin{quote}
There are $M_1,N_1$ such that $N\subseteq M_1\subseteq N_1$,\ \ $M_1\equiv M$\ \ and $N\preccurlyeq N_1$.
\end{quote}
Let $\Sigma$ be the set of all universal $L(N)$-conditions
holding in $N$. Then $Th(M)\cup\Sigma$ must be satisfiable.
Indeed, since $\Sigma$ is affinely closed, it is sufficient to show that
every condition $0\leqslant\inf_{\x}\phi(\x,\a)\in\Sigma$ is satisfiable in $M$.
But, the condition $0\leqslant\sup_{\y}\inf_{\x}\phi(\x,\y)$
holds in $N$ and hence in $M$. Let $M_1$ be a model of $Th(M)\cup\Sigma$.
So, $$N\subseteq M_1\equiv M.$$
Moreover, every universal condition holding in $N$ holds in $M_1$.
This implies easily that $diag(M_1)\cup ediag(N)$ is satisfiable and has a model say $N_1$.
We have therefore that $$M_1\subseteq N_1,\ \ \ \ N\preccurlyeq N_1$$
and the subclaim is proved.

Now, note that every $\forall\exists$-condition holding in $M_1$ holds in $N_1$. 
So, we may iterate the construction to build a chain of structures
$$N=N_0\subseteq M_1\subseteq N_1\subseteq M_2\subseteq N_2\ \cdots$$ such that for each $i$,
$$N_i\preccurlyeq N_{i+l} \ \ \ \ \&\ \ \ \ M_i\equiv M.$$
Let $\bar N=\bigcup N_i=\bigcup M_i$.
Since $T$ is inductive and each $M_i$ is a model of $T$, one has that $\bar N\vDash T$.
Since $N\preccurlyeq\bar N$ we conclude that $N\vDash T$.
\end{proof}
\vspace{1mm}

A theory $T$ is \emph{model-complete} if for every $M,N\vDash T$,
if $M\subseteq N$ then $M\preccurlyeq N$. This is equivalent to saying that
for every $M\vDash T$, the theory $T\cup diag(M)$ is complete.
A formula is \emph{infimal}\index{infimal formula} (resp. \emph{supremal}\index{supremal formula})
if it of the form $\inf_{\x}\phi$ (resp. $\sup_{\x}\phi$) where $\phi$ is quantifier-free.

\begin{theorem} Let $T$ be a satisfiable theory. Then the following are equivalent:

\emph{(i)} $T$ is model-complete.

\emph{(ii)} If $M, N$ are models of $T$ and $M\subseteq N$, then for every infimal formula
$\phi(\x)$ and $\a\in M$, one has that $\phi^N(\a)=\phi^M(\a)$.

\emph{(iii)} Every infimal formula is $T$-approximated by supremal formulas.

\emph{(iv)} Every formula is $T$-approximated by supremal formulas.
\end{theorem}
\begin{proof} We only prove that (ii) implies (iii). The other parts are straightforward.
Let $\phi(\x)$ be an infimal formula and set
$$\Gamma(\x)=\{\theta(\x):\ \ T\vDash\theta(\x)\leqslant\phi(\x)\ \ \mbox{and}
\ \ \theta(\x)\ \mbox{is\ a\ supremal\ formula}\}.$$
Note that $\Gamma$ is convex. We show that for each $\epsilon>0$ there is a $\theta\in\Gamma$ such that
$$T\vDash\phi(\x)-\epsilon\leqslant\theta(\x).$$
Assume not. Then, by convexity, for some $\epsilon>0$, the theory
$$T\cup\{\theta(\d)+\epsilon\leqslant\phi(\d):\ \theta(\x)\in\Gamma\}$$
where $\d$ is a tuple of new constant symbols, is affinely satisfiable.
Let $M$ be model for it.
Let $\gamma=\phi^M(\d)$ and set $$\Sigma=T\cup diag(M)\cup\{\phi(\d)\leqslant \gamma-\epsilon\}.$$

\noindent{\sc Claim}: $\Sigma$ is satisfiable.\\
\noindent{\sc Proof}: Otherwise, for some $r,s\geqslant0$
and $0\leqslant\xi(\d,\a)\in diag(M)$,
the theory $$T\cup\{s\phi(\d)\leqslant r\xi(\d,\a)+s(\gamma-\epsilon)\}$$ is unsatisfiable.
This means that for some $\delta>0$,
$$T\vDash r\xi(\d,\a)+s(\gamma-\epsilon)\leqslant s\phi(\d)-\delta.$$
Since $s$ can not be zero, we must have that
$$T\vDash\frac{r}{s}\sup_{\y}\xi(\x,\y)+\gamma-\epsilon+\frac{\delta}{s}\leqslant\phi(\x).$$
Therefore, $$\frac{r}{s}\sup_{\y}\xi(\x,\y)+\gamma-\epsilon+\frac{\delta}{s}\ \in\Gamma$$ and hence,
$M\vDash\frac{r}{s}\sup_{\y}\xi(\d,\y)+\frac{\delta}{s}\leqslant0$ which is
a contradiction. The claim is proved.

Finally, let $N\vDash\Sigma$. Then, $M\subseteq N\vDash T$ while
$\phi^N(\d)\leqslant \gamma-\epsilon$. This is a contradiction.
\end{proof}

Two $L$-theories $U,T$ are called \emph{cotheory} if every model of one is embedded
in a model of the other one.
$U$ is \emph{model-companion} of $T$ if $U,T$ are cotheory and $U$ is model-complete.
The following propositions are proved as in the first order case.

\begin{proposition}
If $U_1,U_2$ are model-companion of $T$, then $U_1\equiv U_2$.
\end{proposition}

\begin{proposition}
If $U$ is model-complete and has the amalgamation property,
then it has elimination of quantifiers. As a consequence, if $T$
has the amalgamation property, then its model-companion, if it exists and contains $T$,
has elimination of quantifiers.
\end{proposition}
\bigskip

A model $M\vDash T$ is called \emph{existentially closed} for $T$ if
for every extension $M\subseteq N\vDash T$ and quantifier-free formula $\phi(\x)$
with parameters in $M$, one has that
$$\inf_{\b\in M}\phi^M(\b)=\inf_{\c\in N}\phi^N(\c).$$
Using chain arguments, one proves that if $T$ is an inductive theory, every model of $T$ is embedded in
an existentially closed model of $T$.
The following proposition is then proved as in the classical case (see \cite{CK1}).

\begin{proposition}
A $\forall\exists$-theory $T$ has a model-companion if
and only if the class of existentially closed models of $T$ is axiomatizable.
Moreover, the model-companion of $T$, if it exists, is the common theory of such models.
\end{proposition}

The following proposition is obvious.

\begin{proposition} \emph{(Prenex form)}
Every affine formula is $\emptyset$-equivalent to a formula of the form $\mathcal{Q}_1x_1\cdots\mathcal{Q}_kx_k\phi$
where $\phi$ is quantifier-free and $\mathcal{Q}_i$ is either $\sup$ or $\inf$.
\end{proposition}

The affine hierarchy of formulas is defined similar to first order logic.
A formula $\phi$ is $\Sigma_0=\Pi_0$ if it is quantifier-free.
A formula $\phi$ is $\Sigma_{n+1}$ (resp. $\Pi_{n+1}$)
if it is of the form $\sup_{\x}\psi$ (resp. $\inf_{\x}\psi$) where $\psi$ is $\Pi_n$ (resp. $\Sigma_n$).
We may extend a bit the terminology and say that $\phi$ is $\Sigma_n$
(resp. $\Pi_n$) if it is equivalent to a $\Sigma_n$ (resp. $\Pi_n$) formula.
It is then clear that $$\Sigma_n\ \subseteq\ \Sigma_{n+1}\cap\Pi_{n+1},\ \ \ \ \ \ \ \Pi_n\ \subseteq\ \Sigma_{n+1}\cap\Pi_{n+1}.$$

The notion of  $\Sigma_n$-extension generalizes the notion of embedding.
If $M\subseteq N$, then $N$ is a $\Sigma_n$-extension of $M$ if for each
$\Sigma_n$-formula $\phi(\x)$ and $\a\in M$ one has that $\phi^M(\a)\leqslant\phi^N(\a)$.
So, $\Sigma_0$-extension is the same as embedding.

\begin{lemma} \label{preserve2}
Let $M_0\subseteq M_1\subseteq\cdots$ be a $\Sigma_n$-chain of $L$-structures and $M=\cup_{k<\omega} M_k$.
Then

\emph{(i)} $M$ is a $\Sigma_n$-extension of each $M_k$.

\emph{(ii)} For each $\Pi_{n+1}$-sentence $\phi$, if $r\leqslant\phi^{M_k}$ for all $k$, then $r\leqslant\phi^M$.
\end{lemma}
\begin{proof}
(i) The claim holds for $n=0$. Assume it holds for $n-1$.
Let $\phi(\x)=\sup_{\y}\psi(\x,\y)$ where $\psi$ is $\Pi_{n-1}$.
Let $\phi^{M_k}(\a)=r$. Then, for each $\epsilon>0$, there exists $\b\in M_{k}$
such that $$r-\epsilon\leqslant\psi^{M_k}(\a,\b).$$
Consider the $\Sigma_n$-chain
$$(M_k,\a,\b)\subseteq (M_{k+1},\a,\b)\subseteq\cdots.$$
Since $r-\epsilon\leqslant\psi(\a,\b)$ holds in every model of this chain,
by the induction hypothesis, it holds in $(M,\a,\b)$.
Hence, $$r-\epsilon\leqslant\sup_{\y}\psi^M(\a,\y).$$
Since $\epsilon$ is arbitrary, one has that $r\leqslant\sup_{\y}\psi^M(\a,\y)$.

(ii) Let $\phi=\inf_{\x}\psi(\x)$ where $\psi$ is $\Sigma_n$.
Assume $r\leqslant\phi^{M_k}$ for all $k$.
Let $\a\in M$. Then $\a\in M_\ell$ for some $\ell$ and $r\leqslant\psi^{M_\ell}(\a)$.
So, by (i), $r\leqslant\psi^M(\a)$. We conclude that $r\leqslant\phi^M$.
\end{proof}
\bigskip

The following result is the affine variant of Theorem 3.1.11 of \cite{CK1}:

\begin{proposition}
For a sentence $\phi$, the following are equivalent (where $n\geqslant1$):

\emph{(i)} $\phi$ is approximated by both $\Sigma_{n+1}$-sentences and $\Pi_{n+1}$-sentences.

\emph{(ii)} $\phi$ is approximated by linear combinations of $\Sigma_n$-sentences.
\end{proposition}
\begin{proof}
(ii)$\Rightarrow$(i) is obvious.
\noindent(i)$\Rightarrow$(ii): We first prove the following claim for each $M, N$.
\bigskip

\noindent{\sc Claim}:
If $\theta^M=\theta^N$ for each $\Sigma_n$-sentence $\theta$, then $\phi^M=\phi^N$.

\noindent{\sc Proof of the claim}:
Assume $M,N$ satisfy the hypothesis of the claim. We construct a $\Sigma_n$-chain
$$M=M_0\subseteq N_0\subseteq M_1\subseteq N_1\subseteq\cdots$$
such that for all $k$
$$M_k\equiv M,\ \ \ \ \  \ N_k\equiv N.\ \ \ \ \ \ \ \ \ \ \ (1)$$
Suppose that
$$M_0\subseteq N_0\subseteq\cdots\subseteq M_m\subseteq N_m$$
has been constructed such that (1) holds for $k\leqslant m$.
Let $T$ be the set of all conditions $0\leqslant\sigma$ holding in $N_k$
where $\sigma$ is a $\Sigma_n$-sentence in $L(N_m)$.
Clearly, $T$ is affinely closed.
For each $0\leqslant\sigma(\b)$ in $T$, the condition $0\leqslant\sup_{\y}\sigma(\y)$
holds in $N_m$ and hence in $M_m$ by $(1)$ and assumption of the claim.
So, $T\cup Th(M)$ has a model, say $M_{m+1}$.
Therefore, $N_m\subseteq M_{m+1}$ is a $\Sigma_n$-extension and $M_{m+1}\equiv M$.
Similarly, there is a $\Sigma_n$-extension $M_{m+1}\subseteq N_{m+1}$ such that $N_{m+1}\equiv N$.
The required infinite chain is obtained.

Now, let $r\leqslant\phi$ hold in $M$. Then, it holds in every $M_k$.
Since $\phi$ is approximated by $\Pi_{n+1}$-sentences, by Lemma \ref{preserve2} (ii),
$r\leqslant\phi$ holds in $\cup M_k=\cup N_k$.
Suppose $r\leqslant\phi$ does not hold in $N$. Then $-r+\epsilon\leqslant-\phi$
holds in $N$ for some $\epsilon>0$. Since $-\phi$ is approximated by $\Pi_{n+1}$ formulas,
again by Lemma \ref{preserve2}, $-r+\epsilon\leqslant-\phi$ must hold in $\cup_kN_k$
which is a contradiction. Similarly, if $r\leqslant\phi$ holds in $N$,
it must hold in $M$ too. We conclude that $\phi^M=\phi^N$.

\noindent{\sc Proof of the proposition}:\\
Let $\Gamma$ be the vector space generated by $\Sigma_n$-sentences and
$K$ be the set of all norm $1$ positive linear functionals on $\Gamma$.
So, for each $T\in K$, there exists $M$ such that $T(\sigma)=\sigma^M$ for all $\sigma\in\Gamma$.
It is clear that $K$ is compact and convex.
For $\sigma\in\Gamma$ set $$f_\sigma(T)=T(\sigma).$$
Let $$X=\{f_\sigma:\ \sigma\in\Gamma\}.$$
Then, $X$ is a linear subspace of $\mathbf{A}(K)$, the space of affine continuous real functions on $K$
(\S \ref{Appendix}), which contains constant functions.
Assume $T_1\neq T_2$. Then, there is a $\Sigma_n$-sentence $\sigma$ such that $T_1(\sigma)\neq T_2(\sigma)$.
So, $f_\sigma(T_1)\neq f_\sigma(T_2)$.
This shows that $X$ separates points and hence is dense in $\mathbf{A}(K)$.

Define similarly $f_\phi(T)=\phi^M$ where $M\vDash T$.
By the above claim, $f_\phi$ is well-defined. It is also affine and continuous.
Hence $f_\phi\in \mathbf{A}(K)$.
We conclude that for each $\epsilon>0$ there is a $\sigma\in\Gamma$
such that for every $T\in K$,\ \ $|f_\phi(T)-f_\sigma(T)|\leqslant\epsilon$.
In other words, for every $M$,\ \ $|\phi^M-\sigma^M|\leqslant\epsilon$.
\end{proof}
\vspace{2mm}

\begin{definition}
{\em An \emph{expanding} (resp. \emph{contracting}) homomorphism is a map
$f:M\rightarrow N$ such that

\noindent - for each $c\in L$, $f(c^M)=c^N$

\noindent- for each $F\in L$ and $\a\in M$, $f(F^M(\a))=F^N(f(\a))$

\noindent- for each $R\in L$ (including $d$) and $\a\in M$, $R^M(\a)\leqslant R^N(f(\a))$
(resp. $R^N(f(\a))\leqslant R^M(\a)$).}
\end{definition}

The family of \emph{positive} formulas is inductively defined as follow:
$$r,\ \ \ d(t_1,t_2),\ \ \ R(t_1,...,t_n),\ \ \ \phi+\psi,
\ \ \ s\phi,\ \ \ \sup_x\phi,\ \ \ \inf_x\phi$$
where $r\in\Rn$ and $s\in\Rn^+$.
$\phi$ is negative if $-\phi$ is equivalent to a positive formula.
A \emph{positive axiom} (resp. negative axiom) is a condition of the form $0\leqslant\sigma$
where $\sigma$ is positive (resp. $\sigma$ is negative).
It is not hard to check that a surjective function $f:M\rightarrow N$ is an expanding
(resp. contracting) homomorphism if and only if for every positive formula $\phi(\x)$ and $\a\in M$
one has that $\phi^M(\a)\leqslant\phi^N(f(\a))$ (resp. $\phi^N(f(\a))\leqslant\phi^M(f(\a))$). Let
$$ediag^+(M)=\{0\leqslant\phi(\a): \ 0\leqslant\phi^M(\a),
\ \a\in M,\ \phi(\x)\ \textrm{is positive}\}$$
$$ediag^-(M)=\{0\leqslant\phi(\a): \ 0\leqslant\phi^M(\a),
\ \a\in M,\ \phi(\x)\ \textrm{is negative}\}$$
Let write $M\vartriangleleft^+ N$ if $\sigma^M\leqslant\sigma^N$ for every positive sentence $\sigma$.

\begin{proposition} \label{pos-axioms}
\emph{(i)} A theory $T$ is preserved under surjective expanding homomorphisms if and only if it has a set of positive axioms.

\emph{(ii)} $T$ is preserved under surjective contracting homomorphisms if and only if it has a set of negative axioms.
\end{proposition}
\begin{proof}
(i) We prove the nontrivial direction which is a linearized variant of the proof
of Theorem 3.2.4. in \cite{CK1}. Assume $T$ is preserved by surjective expanding homomorphisms.
One first proves that if $M\vartriangleleft^+ N$ then there is an elementary extension $N\preccurlyeq N'$
and a mapping $f:M\rightarrow N'$ such that $$(M,a)_{a\in M}\vartriangleleft^+ (N', f(a))_{a\in M}.$$
For this purpose one checks that $ediag^+(M)\cup ediag(N)$
is linearly satisfiable.
Similarly, if $M\vartriangleleft^+ N$, then there is an elementary extension
$M\preccurlyeq M'$ and a mapping $g:N\rightarrow M'$ such that
$(M',g(b))_{b\in N}\vartriangleleft^+ (N,b)_{b\in N}$.
For this purpose, one checks that $ediag^-(N) \cup ediag(M)$ is linearly satisfiable.
Now assume $M_0\vDash T$ and $M_0\vartriangleleft^+ N_0$.
Iterate the arguments to find chains
$$M_0\preccurlyeq M_1\preccurlyeq\ldots, \ \ \ \ \ \ N_0\preccurlyeq N_1\preccurlyeq\ldots$$
and maps $$f_i:M_i\rightarrow N_{i+1},\ \ \ \ \ \ g_i:N_i\rightarrow M_i$$
such that
$$(M_0,a)_{a\in A_0} \vartriangleleft^+ (N_1, f_0a)_{a\in A_0}$$
$$(M_1,a,g_1b)_{a\in A_0,\ b\in B_1}\vartriangleleft^+ (N_1, f_0a, b)_{a\in A_0, b\in B_0}$$
and so forth. In particular, $f_i:M_i\rightarrow N_{i+1}$ is an expanding homomorphism
and $f_i\subseteq f_{i+1}$, $g_{i+1}^{-1}\subseteq f_{i+1}$.
Set $\bar M=\cup_i M_i$ and $\bar N=\cup_i N_i$.
Then $M_0\preccurlyeq\bar M$, $N_0\preccurlyeq\bar N$ and $\cup f_i:M\rightarrow N$
is a surjective expanding homomorphism. By the assumption of
proposition, we must have that $N_0\vDash T$.
Let $\Delta$ be the set of all positive $L$-conditions.
Thus, we have proved that the clause (ii) of Lemma \ref{localization} holds for $\Delta$.
We conclude $T$ is axiomatized by a set of positive conditions. (ii) is proved similarly.
\end{proof}

\newpage\section{Types} \label{Types}
Let $T$ be a satisfiable theory.

\begin{definition}
\em{A partial \emph{$n$-type} for $T$ is a set $\Sigma(\x)$
of conditions $\phi(\x)\leqslant\psi(\x)$ such that $T\cup\Sigma(\x)$ is satisfiable.
Maximal partial types are called \emph{types}\index{type} (or complete types).}
\end{definition}

Every partial $n$-types can be extended to a complete $n$-type.
If $p$ is a $n$-type, then for each $\phi(\x)$ there is a unique real number $p(\phi)$
such that $\phi=p(\phi)$ belongs to $p$.
Moreover, the map $\phi\mapsto p(\phi)$ is a positive linear functional on
$\mathbb{D}_n(T)$ (the space of $T$-equivalence classes of formulas $\phi(\x)$,
$|\x|=n$, \S \ref{Quantifier-elimination}) with $p(1)=1$. The converse of this observation is true.

\begin{lemma}
Let $p:\mathbb{D}_n(T)\rightarrow\Rn$ be a positive linear functional such that $p(1)=1$.
Then the set of conditions of the form $\phi=p(\phi)$ is a $n$-type of $T$.
\end{lemma}
\begin{proof}
The set of conditions $\phi(\x)=p(\phi)$ is closed under linear combinations.
We only need to show that every such condition is satisfiable with $T$.
Assume $\phi\leqslant p(\phi)$ is not satisfiable with $T$.
Then, for some $\epsilon>0$ one has that $T\vDash p(\phi)+\epsilon\leqslant\phi$.
Since $p$ is positive linear and normal, one has that $p(\phi)+\epsilon\leqslant p(\phi)$.
This is a contradiction. Similarly, $p(\phi)\leqslant\phi$ and hence $\phi=p(\phi)$ is satisfiable with $T$.
\end{proof}

So, $n$-types correspond to normal (i.e. of norm $1$) positive linear maps $p:\mathbb{D}_n(T)\rightarrow\Rn$.
The set of $n$-types of $T$ is denoted by $K_n(T)$\index{$K_n(T)$}.
For a complete $T$, types over a set of parameters $A\subseteq M\vDash T$ are defined similarly. They are types of $Th(M,a)_{a\in A}$.
Let $\mathbb{D}_n(A)$ be the vector space of formulas $\phi(\x)$ with parameters from $A$
where $\phi$, $\psi$ are identified if $\phi^M(\x)=\psi^M(\x)$ for all $\x\in M$.
Then, a type over $A$ is a normal positive linear functional $p:\mathbb{D}_n(A)\rightarrow\Rn$.
The set of $n$-types over $A$ is denoted by $K_n(A)$\index{$K_n(A)$}.

$K_n(T)$ is a closed subset of the unit ball of $\mathbb{D}_n(T)^*$.
So, by the Banach-Alaoglu theorem (\cite{Aliprantis-Inf}, Th. 6.21), it is compact (and Hausdorff).
In model theory, this topology is called \emph{logic topology}\index{logic topology}
and is generated by the sets of the form $$\{p\in K_n(T):\ 0<p(\phi)\}.$$
$K_n(T)$ is also convex, i.e. if $p,q\in K_n(T)$ and $\gamma\in[0,1]$, then $\gamma p+(1-\gamma)q\in K_n(T)$.
Similar properties hold for $K_n(A)$ if $T$ is complete.

\begin{theorem}  \label{Conway} \emph{(\cite{Conway} p.125)} 
Let $V$ be a normed linear space. Then the weak* dual of $V^*$ coincides with $V$ itself,
in other words $(V^*,\sigma(V^*,V))^*=V$.
\end{theorem}

\begin{theorem} \label{Rudin} \emph{(\cite{Rudin} Theorem 3.4)} 
Let $A$ and $B$ be disjoint nonempty closed convex subsets of a locally convex
topological vector space $V$ with $B$ compact. Then there exists $f\in V^*$ and $r,s\in\Rn$ such that $f(A)<r<s<f(B)$.
\end{theorem}

\begin{corollary} \label{closed convex}
Closed convex subsets of $K_n(T)$ are exactly the sets defined by partial types, i.e.
sets of the form $\{p\in K_n(T):\ \ \Gamma\subseteq p\}$ where $\Gamma$ is a set of conditions.
\end{corollary}
\begin{proof}
Every set of the form $\{p\in K_n(T)\ |\ \ p\supseteq \Gamma\}$ is closed and convex.
For the converse, first note that every formula $\phi(\x)$ defines a map
$$\hat\phi:K_n(T)\rightarrow\Rn, \ \ \ \ \ \ \hat\phi(p)=p(\phi).$$
Assume $X\subseteq K_n(T)$ is closed and convex.
By Theorem \ref{Conway}, the dual of $\mathbb D_n(T)^*$ with respect to the topology
$\sigma(\mathbb D_n(T)^*,\mathbb D_n(T))$ is $\mathbb D_n(T)$.
By Theorem \ref{Rudin}, for each $q\in K_n(T)-X$ there exist $\phi_q(\x)$ and $r_q,s_q$ such that
$$\hat\phi_q(q)\leqslant r_q<s_q\leqslant \hat\phi_q(p)\ \ \ \ \ \ \ \forall p\in K.$$
We conclude that
$$X=\big\{p\in K_n(T)\ | \  p\vDash s_q\leqslant\phi_q(\x) \ \ \ \ \forall q\in K_n(T)-K\big\}.$$
\end{proof}

\subsection{Saturation and homogeneity} \label{Saturation and homogeneity}
Let $T$ be a complete theory and $A\subseteq M$. A type $p(\x)\in K_n(A)$ is \emph{realized}\index{realized} by $\a\in M$
if $p(\phi)=\phi^M(\a)$ for every $\phi(\x)$.

\begin{lemma}\label{realize}
Let $A\subseteq M$ and $p(\x)\in K_n(A)$. Then $p$ is realized in some $M\preccurlyeq N$.
\end{lemma}
\begin{proof} Let $\Sigma=\mbox{ediag}(M)\cup p(\x)$ and $0\leqslant\phi^M(\a,\b)$
where $\a\in A$, $\b\in M-A$. By the definition, the condition $0\leqslant\sup_{\y}\phi(\a,\y)$ belongs to $p$.
This shows that $\Sigma$ is affinely satisfiable. Let $N$ be a model of $\Sigma$.
Then, $M\preccurlyeq N$ and $p$ is realized in $N$.
\end{proof}

Saturated models and related notions are defined as in first order (or full continuous) logic.
We review these notions and prove some relations between them in the AL setting.
Proofs are mostly adaptations of the usual first order proofs (see \cite{Buechler, Marker, Poizat-Model}).
As in full continuous logic, instead of cardinality, the density character of models is adopted.

To simplify the proofs, we assume that $T$ is a complete theory in a countable language.
In this case, for each infinite $\kappa$ and $A\subseteq M\vDash T$ with $|A|\leqslant\kappa$,
one has that $|K_n(A)|\leqslant2^{\kappa}$.
If $A\subseteq M$, a map $f:A\rightarrow N$ is called a \emph{partial elementary embedding}
\index{partial elementary map} if $\phi^M(\a)=\phi^N(f(\a))$ for every $\a\in A$.
For $\a,\b\in M$ of equal (possibly infinite) length we write $\a\equiv\b$ if the partial map defined by
$f(a_i)=b_i$ is elementary. If we extend the notion of type to infinite tuples, this means that $tp(\a)=tp(\b)$.

\begin{definition}
\emph{Let $\kappa$ be an infinite cardinal.\\
\emph{(i)} $M$ is called $\kappa$-\emph{saturated}\index{saturated} if for each $A\subseteq M$
with $|A|<\kappa$, every type $p\in K_1(A)$ (and hence every $p\in K_n(A)$) is realized in $M$.\\
\emph{(ii)} $M$ is $\kappa$-\emph{homogeneous}\index{homogeneous} if for every $\a,\b,c\in M$
with $|\a|=|\b|<\kappa$ and $\a\equiv\b$, there exists $e\in M$ such that $\a c\equiv\b e$.
It is \emph{strongly $\kappa$-homogeneous}\index{strongly homogeneous} if for $\a,\b$ as above there is an automorphism $f$ of $M$ such that $f(\a)=\b$.\\
\emph{(iii)} $M$ is $\kappa$-\emph{universal} if every $N\equiv M$ with $\mathsf{dc}(N)<\kappa$ is elementarily embedded in $M$.}
\end{definition}

By repeated use of Lemma \ref{realize} and forming chains of elementary extensions, one proves that

\begin{proposition}
Every model has $\kappa$-saturated elementary extensions for every $\kappa$.
\end{proposition}

\begin{proposition}
Assume $M\equiv N$ are $\kappa$-saturated where $\kappa=\mathsf{dc}(M)=\mathsf{dc}(N)$. Then $M\simeq N$.
\end{proposition}
\begin{proof}
Let $A=\{a_i: i<\kappa\}$, $B=\{b_i: i<\kappa\}$ be dense sets in $M$ and $N$ respectively.
We construct a chain $$f_0\subseteq f_1\subseteq\cdots\subseteq f_i\subseteq\cdots\ \ \ \ \ \ \ \ \ i<\kappa$$
of partial elementary maps such that $a_i\in Dom(f_{i+1})$ and $b_i\in Im(f_{i+1})$.
Let $f_0=\emptyset$ and assume $f_i$ is defined. Let $\bar{c},\bar{e}$ enumerate $Dom(f_i)$ and $Im(f_i)$ respectively
so that $\bar{c}\equiv\bar{e}$.
Let $$p(x)=\{\phi(x,\bar{e})=0:\ \phi^M(a_i,\bar{c})=0\}.$$
Then, $p$ is realized by say $b\in N$ and hence $f_i\cup\{(a_i,b)\}$ is a partial elementary map.
Similarly, there exists a partial elementary map $f_{i+1}=f_i\cup\{(a_i,b),(a,b_i)\}$.
For limit $i$ we let $f_i=\cup_{j<i}f_j$. Finally, $f=\cup_{i<\kappa}f_i$ is a partial elementary embedding
with $A\subseteq Dom(f)$ and $B\subseteq Im(f)$. Then, $f$ extends to an isomorphism from $M$ to $N$ in the natural way.
\end{proof}

\begin{proposition}
$M$ is $\kappa$-saturated if and only if it is $\kappa$-homogeneous and $\kappa^+$-universal.
For $\aleph_1\leqslant\kappa$, $M$ is $\kappa$-saturated if and only if it is $\kappa$-homogeneous and $\kappa$-universal.
\end{proposition}
\begin{proof}
Assume $M$ is $\kappa$-saturated. For $\kappa$-homogeneity, suppose that
$|\bar{a}|=|\bar{b}|=\alpha<\kappa$, $\bar{\a}\equiv\bar{b}$ and $c$ is given.
Then the type $$\{\phi(\b,y)=0: \phi^M(\a,c)=0\}$$ is realized in $M$ by say $e$. We then have that $\a c\equiv\b e$.
For $\kappa^+$-universality, assume $N\equiv M$ and $\mathsf{dc}(N)\leqslant\kappa$. Let $\{b_i: i<\kappa\}\subseteq N$ be dense.
We construct a chain $$f_0\subseteq f_1\subseteq\cdots f_i\subseteq\cdots\ \ \ \ \ \ \ \ \ \ \ i<\kappa$$
of partial elementary maps from $N$ to $M$ such that $b_i\in Dom(f_{i+1})$.
Set $f_0=\emptyset$ and assume $f_i$ is constructed. Let $\bar{c}$ enumerate the domain of $f_i$.
Then $|\bar{c}|<\kappa$ and the type $$\{\phi(f_i(\bar{c}),y)=0: \phi^N(\bar{c},b_i)=0\}$$ is realized by say $e\in M$.
So, $f_{i+1}=f_i\cup\{(b_i,e)\}$ is a partial elementary map. For limit $i$ set $f_i=\cup_{j<i}f_j$.
By completeness of $M$,\ \ $\cup_{i<\kappa} f_i$ extends to an elementary embedding of $N$ into $M$.

Conversely, assume $M$ is $\kappa$-homogeneous and $\kappa^+$-universal.
Let $A\subseteq M$, $|A|<\kappa$ and $p(x)\in K_1(A)$. There exists $M\preccurlyeq\bar{N}$ realizing $p(x)$
by say $b\in\bar{N}$.
Let $$A\cup\{b\}\subseteq N\preccurlyeq\bar{N}$$ where $\mathsf{dc}(N)\leqslant\kappa$. Then, there exists an elementary embedding
$f:N\rightarrow M$. Assume $\a$ enumerates $A$. Then, by $\kappa$-homogeneity, there exists $c\in M$ such that
$\a c\equiv f(\a)f(b)$. It is then clear that $c$ realizes $p(x)$.
The second part is a rearrangement of the proof for uncountable $\kappa$.
\end{proof}

\begin{proposition}\label{strong homogeneity}
Every $\kappa$-homogeneous model $M$ of density character $\kappa$ is strongly $\kappa$-homogeneous.
\end{proposition}
\begin{proof} Let $\a\equiv\b$ where $|\a|=|\b|<\kappa$ and $\{c_i:i<\kappa\}$ be a dense subset of $M$.
The map $f_0:\a\mapsto\b$ is partial elementary. It is then routine to construct a chain
$$f_0\subseteq f_1\subseteq\cdots\subseteq f_j\subseteq\cdots\ \ \ \ j<\kappa$$
of partial elementary maps such that $c_j$ belongs to both $Dom(f_i)$ and $Im(f_j)$.
Then, $\cup_j f_j$ extends to an automorphism of $M$.
\end{proof}

Again using chain arguments one easily proves the following.

\begin{proposition} \label{saturated2}
For each infinite $\kappa$, $T$ has a $\kappa^+$-saturated model of cardinality $2^{\kappa}$.
If $\kappa\geqslant\aleph_1$ is strongly inaccessible, $T$ has a $\kappa$-saturated model of cardinality $\kappa$.
\end{proposition}

\begin{lemma} \label{aut-ext}
Let $f:A\subseteq M\rightarrow M$ be a partial elementary map where $\mathsf{dc}(M)\leqslant\kappa$.\\
\indent\emph{(i)} There are $M\preccurlyeq N$ and automorphism $f\subseteq g:N\rightarrow N$
such that $\mathsf{dc}(N)\leqslant\kappa$.

\emph{(ii)} There are $M\preccurlyeq N$ and automorphism $f\subseteq g:N\rightarrow N$ such that
$N$ is $\kappa$-saturated and $|N|\leqslant2^\kappa$.
\end{lemma}
\begin{proof}
(i) We give sketch of the proof which has two steps. 

\noindent{\sf Step 1}.
Let $\{c_i: i<\kappa\}$ be a dense subset of $M$. Construct a chain
$$M=M_0\preccurlyeq M_1\preccurlyeq\cdots\preccurlyeq M_i\preccurlyeq\cdots \ \ \ \ \ i<\kappa$$
and elementary maps $$f_i:A\cup\{c_j:j<i\}\rightarrow M_i$$
such that $\mathsf{dc}(M_i)\leqslant\kappa$, $f_0=f$ and $f_i\subseteq f_j$ whenever $i<j$.
Let $N_0$ be the completion of $\cup_iM_i$ and $\bar{f}=\cup_{i<\kappa} f_i$.
Then, $M\preccurlyeq N_0$, $\mathsf{dc}(N_0)\leqslant\kappa$ and $\bar{f}$ extends to an elementary map $f\subseteq g_0:M\rightarrow N_0$.
\vspace{1mm}

\noindent{\sf Step 2}. Iterating the step $1$, construct a chain
$$M=M_0\preccurlyeq N_0\preccurlyeq M_1\preccurlyeq N_1\preccurlyeq M_2\preccurlyeq N_2\preccurlyeq\cdots$$
and elementary maps $g_i:M_i\rightarrow N_i$ such that
$$f\subseteq g_0\subseteq g_1\subseteq\cdots,$$
$$N_i\subseteq Im(f_{i+1}), \ \ \ \ \ \ \ \ \ \ \ \mathsf{dc}(N_i)\leqslant\kappa.$$
Let $N$ be the completion of $\cup_iM_i=\cup_iN_i$. Then, $\cup_ig_i$ is an automorphism
of $\cup_i M_i$ which can be extended to an automorphism $f\subseteq g:N\rightarrow N$.
We have also that $\mathsf{dc}(N)\leqslant\kappa$.

(ii) We may assume $\kappa$ is regular. In the proof of (i), we may first replace each
$M_{i+1}$ in the step $1$ with a $\kappa$-saturated model of cardinality $\leqslant2^\kappa$.
We then set $N_0=\cup_i M_i$ which is a $\kappa$-saturated model of cardinality $\leqslant2^\kappa$.
Also, in the step 2, we continue the construction to obtain a longer chain
$$M=M_0\preccurlyeq N_0\preccurlyeq M_1\preccurlyeq N_1\preccurlyeq\cdots
\preccurlyeq M_i\preccurlyeq N_i\preccurlyeq\cdots\ \ \ \ \ \ \ \ i<\kappa$$
with $M_i$, $N_i$ being $\kappa$-saturated of cardinality $\leqslant2^\kappa$ and maps $f_i$
satisfying the mentioned requirements. Then set $N=\cup_i M_i=\cup_i N_i$ and $g=\cup_i g_i$.
\end{proof}

\begin{proposition} \label{homogeneous2}
Assume $M\vDash T$ and $\mathsf{dc}(M)\leqslant2^{\kappa}$.
Then, there exists a $\kappa$-saturated strongly $\kappa$-homogeneous $M\preccurlyeq N$ such that $|N|\leqslant2^{\kappa}$.
\end{proposition}
\begin{proof}
We may assume $\kappa$ is regular. Let $\mathscr{C}$ be the family of all partial elementary maps $f:A\rightarrow M$
where $A\subseteq M$ and $|A|\leqslant\kappa$. Then $|\mathscr{C}|\leqslant2^{\kappa}$.
The proof of Lemma \ref{aut-ext} (ii) can be easily extended to show that
there is a $\kappa$-saturated $M\preccurlyeq M_1$ such that $|M_1|\leqslant2^{\kappa}$ and every
$f\in\mathscr{C}$ extends to automorphism of $M_1$ (the step $1$ must be simultaneously carried out for all $f\in\mathscr{C}$).
Repeating the argument, there is a chain
$$M=M_0\preccurlyeq M_1\preccurlyeq\cdots \preccurlyeq M_i\preccurlyeq\cdots\ \ \ \ \ \ \ \ i<\kappa$$
of $\kappa$-saturated models such that $\mathsf{dc}(M_i)\leqslant2^{\kappa}$ and every partial elementary map $f:A\rightarrow M_i$,
where $A\subseteq M_i$ and $|A|<\kappa$, extends to an automorphism of $M_{i+1}$.
Let, $N=\cup_{i<\kappa} M_i$. Then $M\preccurlyeq N$ is $\kappa$-saturated strongly $\kappa$-homogeneous
and $|N|\leqslant2^{\kappa}$.
\end{proof}

If we do not limit the cardinality of the extension, we can prove the following stronger result (see also \cite{BBHU} Prop. 7.12)

\begin{proposition}\label{saturated strong homogeneous}
For each $\kappa$ and $L$-structure $M$ there exists $M\preccurlyeq N$ which is $\kappa$-saturated
and strongly $\kappa$-homogeneous on every sublanguage of $L$.
\end{proposition}

\begin{proposition}\label{homogeneous}
Let $T$ be complete and $M,N\vDash T$ be separable and $\aleph_0$-homogeneous.
Then $M\simeq N$ if and only if they realize the same types in every $K_n(T)$.
\end{proposition}
\begin{proof}
We prove the non-trivial part. Assume $\a\in M$, $\b\in N$ are such that $tp(\a)=tp(\b)$.
Given $c\in M$, let $(\bar{u},v)\in N$ realize $tp(\a,c)$. Then, $tp(\a)=tp(\b)=tp(\bar{u})$.
By homogeneity, there exists $e\in N$ such that $tp(\a,c)=tp(\bar{u},v)=tp(\b,e)$.
Similarly, given $e\in N$, one finds $c\in M$ such that $tp(\a,c)=tp(\b,e)$.
Now, by a back and forth argument one defines a partial elementary map $f$ from a dense subset of $M$
onto a dense subset of $N$. Then, by completeness of $M$ and $N$,\ $f$ is extended to an isomorphism from $M$ onto $N$.
\end{proof}

A similar argument shows that 
Proposition \ref{homogeneous} holds in larger cardinalities too.
So. if $M\equiv N$ are complete $\kappa$-homogeneous models of density character $\kappa$
which realize the same types in every $K_n(T)$, then $M\simeq N$.

\begin{proposition}\label{homogeneous 3}
$M\vDash T$ is $\aleph_0$-saturated if and only if it is $\aleph_0$-homogeneous and realizes all types in every $K_n(T)$.
\end{proposition}
\begin{proof}
The `only if' part is obvious. Assume $p(\a,y)$ is a type over $\a\in M$ where $|\a|<\omega$.
Assume $p(\x,y)$ is realized by $(\b,c)\in M$. Then, $\a\equiv\b$ and by homogeneity, there exists $e$
such that $\a e\equiv \b c$. So, $p(\a,y)$ is realized by $e$. We conclude that $M$ is $\aleph_0$-saturated.
\end{proof}

It is also proved that a $\kappa$-homogeneous model which realizes all types in every
$K_n(T)$ is $\kappa$-saturated. 

\begin{proposition}\label{connected range}
Let $M$ be $\aleph_0$-saturated and $\phi_1(\x),...,\phi_n(\x)$ be formulas with parameters.
Then, the range of $f(\x)=(\phi^M_1(\x),...,\phi^n_n(\x))$ is compact and convex in $\mathbb R^n$.
\end{proposition}
\begin{proof}
Let $f(\a)=\bar r$ and $f(\b)=\bar s$ and $\lambda\in(0,1)$ be fixed.
We show that $$\{\phi_k(\x)\leqslant\lambda r_k+(1-\lambda)s_k\}_{k=1}^n
\cup\{\lambda r_k+(1-\lambda)s_k\leqslant\phi_k(x)\}_{k=1}^n$$ is satisfiable.
For this, one must show that for any $\alpha_k,\beta_k\geqslant0$ the condition
$$\sum_k\alpha_k\phi_k(\x)+\sum_k\beta_k(\lambda r_k+(1-\lambda)s_k)\leqslant
\sum_k\alpha_k(\lambda r_k+(1-\lambda)s_k)+\sum_k\beta_k\phi_k(\x)$$
or equivalently
$$\sum_k(\alpha_k-\beta_k)\phi_k(\x)\leqslant\sum_k(\alpha_k-\beta_k)(\lambda r_k+(1-\lambda)s_k)$$
is satisfiable. One verifies easily that this holds for at least one of $\a$ and $\b$.
Therefore there exists $\bar c$ such that $f(\bar c)=\lambda\bar r+(1-\lambda)\s$.
For closedness, let $\bar u$ be in the closure of the range of $f$.
Then the set
$$\big\{u_k-\epsilon\leqslant\phi_k(\x)\leqslant u_k+\epsilon|\ \ \ \epsilon>0, \ k=1,...,n\big\}$$
is finitely satisfiable. Hence satisfiable in $M$.
\end{proof}
\vspace{1mm}

\begin{proposition} \label{midpoints}
Let $M$ be $\aleph_1$-saturated and $a_1,a_2,...\in M$.
Assume $d(a_i,a_j)=r_{ij}$ and $\sum_{i=1}^{\infty}u_i=1$ where $u_i\geqslant0$.
Then, there exists $x\in M$ such that\ $d(x,a_j)=\sum_{i=1}^\infty u_ir_{ij}$\ for every $j$.
\end{proposition}
\begin{proof} It is sufficient show that for every $s_1,...,s_n\in\Rn$ the condition
$$\phi(x)=\sum_{j=1}^ns_j\ d(x,a_j)=s$$ is satisfiable in $M$ where
$$s=\sum_{j=1}^ns_j\sum_{i=1}^\infty u_ir_{ij}.$$
Note that there must exist $k$ such that $\sum_{j=1}^ns_j r_{kj} \leqslant s$.
Since otherwise, $$s=\sum_{i=1}^\infty u_i\sum_{j=1}^ns_j r_{ij}>\sum_{i=1}^\infty u_is=s$$
which is impossible. For such $k$ we will have that $\phi^M(a_k)\leqslant s$.
Similarly, there exists $\ell$ such that $\sum_js_j r_{\ell j} \geqslant s$
and hence $\phi^M(a_\ell)\geqslant s$.
This shows that both $\phi(x)\leqslant s$ and $\phi(x)\geqslant s$
are satisfiable in $M$. We conclude that $\phi(x)=s$ is satisfiable in $M$.
\end{proof}

In particular, if $d(a_i,a_j)=\delta_{ij}$ and $u_i=\frac{1}{n}$ for
$1\leqslant i,j\leqslant n$, then there is $x$ such that $d(x,a_i)=\frac{n-1}{n}$.
For $n=2$, $x$ is the midpoint of $a_1$ and $a_2$.
A {\em geodesic} in a metric space $M$ is a distance preserving map
$\gamma: [0,|\gamma|]\rightarrow M$.
A {\em geodesic space} is a metric space wherein any two points are joined by a geodesic.
If $M$ is $\aleph_0$-saturated, any two points have a midpoint.
This implies (by completeness of the metric) that any two points are joined by a geodesic.
So, every $\aleph_0$-saturated model is a geodesic space.
This also shows that the class of geodesic metric spaces is not elementary.

\subsection{Metrics of the type spaces} \index{metric topology of types}}
The space of $n$-types $K_n(T)$ is equipped with several topologies the most important of
which is the logic topology induced by the weak* topology on $\mathbb{D}_n(T)^*$.
This topology is compact Hausdorff by the Banach-Alaoglu theorem.
It is also is equipped with the \emph{norm metric} defined by
$$\|p-q\|=\inf\big\{\lambda:\ \ \forall\phi\in\mathbb{D}_n(T),\ \ |p(\phi)-q(\phi)|\leqslant\lambda\|\phi\|_\infty\big\}.$$
Also, if $T$ is complete, $K_n(T)$ inherits a metric from the models of $T$ defined below.
Let $M\vDash T$ be $\aleph_0$-saturated. If $\a$ realizes $p$ we write $\a\vDash p$.
Then $$\mathbf{d}(p,q)=\inf\{d(\a,\b):\ \a,\b\in M,\ \a\vDash p,\ \b\vDash q\}.$$
Note that if $s<{\bf d}(p, q)$, the partial type
$$p(\x)\cup q(\y)\cup \{d(\x,\y)\leqslant s\}$$ is unsatisfiable. So, ${\bf d}$ is independent of the choice of $M$.
This will be called the \emph{logic metric} on $K_n(T)$.
By definition, ${\bf d}(tp(\a),tp(\b))\leqslant d(\a,\b)$ for every $\a,\b\in M$.

A \emph{topometric} space is a Hausdorff topological space $X$ equipped with a metric
$d$ which refines the topology and is lower semi-continuous with respect to it, i.e.
for each $r\geqslant0$, $\{(x,y)\in X^2: d(x,y)\leqslant r\}$ is closed in $X^2$.
In such a structure, if the topology is compact, the metric is complete.

\begin{proposition}
\emph{(i)} The restriction maps $\pi_n:K_{n+1}(T)\rightarrow K_n(T)$ are logic continuous,
affine and $1$-Lipschitz with respect to both logic and norm metrics.

\emph{(ii)} $K_n(T)$ equipped with the logic topology and logic metric is a topometric space.

\emph{(iii)} $K_n(T)$ equipped with the logic topology and norm metric is a topometric space.
\end{proposition}
\begin{proof}
(i) Logic continuity is obvious. Let $\a b\vDash p(\x y)$ and $\c e\vDash q(\x y)$ be such that $\mathbf{d}(p,q)=d(\a b,\c e)$.
Then, $$\mathbf{d}(p|_{\x},q|_{\x})\leqslant d(\a,\c)\leqslant\mathbf{d}(p,q).$$
Also, assume $\|p(\x y)-q(\x y)\|=\lambda$. Then, for every $\phi(\x)$,
$$\|p|_{\x}(\phi)-q|_{\x}(\phi)\|\leqslant\lambda\|\phi\|_\infty.$$
So, $\|p|_{\x}-q|_{\x}\|\leqslant\lambda=\|p-q\|$.

(ii) Let $p$ be type, $\phi$ a formula and $\epsilon>0$.
Assume $\a\vDash p$ and $p(\phi)=r$. Then, since $\phi^M$ is $\lambda_\phi$-Lipschitz,
$$\{tp(\b): d(\a,\b)<\frac{\epsilon}{\lambda_\phi}\}\subseteq\{q:\ r-\epsilon<q(\phi)<r+\epsilon\}$$
Varying $\a$ over other realizations of $p$, one deduces that metric topology is finer than the logic one.
For lower semi-continuity of $d$ note that, up to a change of variables,
$$\{(p_1,p_2)\in (K_n(T))^2:\ {\bf d}(p_1,p_2)\leqslant r\}$$
is the image under the map $K_{2n}(T)\rightarrow K_n(T)\times K_n(T)$ of the set
$$\{q(\x,\y)\in K_{2n}(T):\ \ q(d(\x,\y))\leqslant r\}=\{q:\ \hat d(q)\leqslant r\}.$$

(iii) The norm topology is obviously finer that the logic one.
For lower semi-continuity of the norm metric, again, up to a change of variables,
$$\{(p_1,p_2)\in (K_n(T))^2:\ \|p_1-p_2\|\leqslant r\}$$
is the image under the map $K_{2n}(T)\rightarrow K_n(T)\times K_n(T)$ of the set
$$\bigcap_{\phi(\x)}\{q(\x,\y)\in K_{2n}(T):\ \ q(\phi(\x))-q(\phi(\y))\leqslant r\|\phi\|_\infty\}.$$
\end{proof}

In first order logic, every countable model in a countable language has a countable $\aleph_0$-homogeneous elementary extension.
If a similar property holds in AL, one can prove a criterion for the existence of separable saturated models.

\begin{proposition}
Assume for every separable $M\vDash T$ there is an $\aleph_0$-homogeneous separable $M\preccurlyeq N$.
Then, $T$ has a separable $\aleph_0$-saturated model if and only if every $K_n(T)$ is metrically separable.
\end{proposition}
\begin{proof}
The map $M^n\rightarrow K_n(T)$ defined by $\a\mapsto tp(\a)$ is $1$-Lipschitz.
So, if $M$ is separable and $\aleph_0$-saturated, the image of any countable dense subset of $M^n$ is metrically dense in $K_n(T)$.
Conversely, assume that every $K_n(T)$ is metrically separable.
Let $U_n\subseteq K_n(T)$ be countable and metrically dense for each $n$.
We may further assume that whenever $q_1(\x),q_2(\x)\in U_n$, there is $q(\x\y)\in U_{2n}$ such that $|\x|=|\y|$ and
$$q(\x\y)\ \supseteq\ q_1(\x)\cup q_2(\y)\cup\{d(x,y)\leqslant d(q_1,q_2)\}.$$
Using elementary chains, we can find a separable model $M$ realizing every type in every $U_n$.
By the assumption of the proposition, we may further assume that $M$ is $\aleph_0$-homogeneous.
We show that $M$ realizes every type of $T$. Take $p(x)\in K_1(T)$.
Let $p_k\rightarrow p$ in the metric topology where $p_k\in U_1$ and $a_k\in M$ realizes $p_k$.
Suppose that $d(p_k,p_{k+1})\leqslant 2^{-k}$.
Set $b_1=a_1$ and assume $b_1,...,b_k$ are defined such that $b_i\equiv a_i$ and
$d(b_i,b_{i+1})=d(p_i,p_{i+1})$ for $i<k$. Let $ce\in M$ realize a type
$$q(xy)\supseteq p_k(x)\cup p_{k+1}(y)\cup\{d(x,y)\leqslant d(p_k,p_{k+1})\}.$$
Then, $c\equiv a_k\equiv b_k$ and $e\equiv a_{k+1}$.
By homogeneity, there exists $b_{k+1}\in M$ such that $ce\equiv b_kb_{k+1}$.
Therefore, $b_{k+1}\equiv a_{k+1}$ and $d(b_k,b_{k+1})=d(p_k,p_{k+1})$.
The sequence $b_k$ is convergent to say $b\in M$. Then, $b\vDash p$.
Therefore, every type in $K_1(T)$ (and similarly, every type in $K_n(T)$) is realized in $M$.
We conclude by Proposition \ref{homogeneous 3} that $M$ is $\aleph_0$-saturated.
\end{proof}

\subsection{Facial types}\label{Facial types}
Let $T$ be a complete $L$-theory.
Extreme points of $K_n(T)$ are called \emph{extreme types}\index{extreme type}. The set of extreme $n$-types
of $T$ is denoted by $E_n(T)$\index{$E_n(T)$}.
By the Krein-Milman theorem, this is a non-empty set and its convex hull is dense in $K_n(T)$.

In first order logic (as well as continuous logic), the set of types realized in a model
$M\vDash T$ is dense in the set of all types. In affine continuous logic, the situation is different.
For example, $K_1(\textrm{PrA})$ (where PrA is the theory of probability algebras \S \ref{Examples})
is a $1$-dimensional simplex, hence isometric to the unit interval $[0,1]$.
While, $\{0,1\}$ is a model of PrA which realizes only the extreme types.
However, we have a similar result in this respect.

Let $M$ be a metric space and $\overline M$ be its Stone-\v{C}ech compactification.
Then, every normal positive linear functional $\Lambda:\mathbf{C}_b(M)\rightarrow\Rn$ extends to a positive linear
functional on $\mathbf{C}(\overline{M})$. So, by Riesz-Markov representation theorem (\S \ref{Appendix})
there is a regular Borel probability measure $\mu$ on $\overline M$
such that $$\ \ \ \ \Lambda(f)=\int \bar f d\mu\ \ \ \ \ \ \ \ \ \ \forall f\in \mathbf{C}_b(M)$$
where $\bar f$ denotes the continuous extension of $f$ to $\overline{M}$.
The set of all regular Borel probability measures on $M$ is denoted by $\mathcal P(M)$.

\begin{proposition} \label{dense-realization}
Let $M\vDash T$. Then, every $p\in E_n(T)$ is the limit (in the logic topology) of types realized in $M$.
\end{proposition}
\begin{proof}
Consider the case $n=1$. We may assume that $M$ (and hence $\overline{M}$) is separable.
Let $\zeta:\mathcal P(\overline{M})\rightarrow K_1(T)$ be the function defined by
$\zeta(\mu)=p_\mu$ where $p_\mu(\phi)=\int\overline{\phi^M} d\mu$ for all $\phi$.
Note that $\zeta$ is affine, i.e. $$\zeta(\lambda\mu+(1-\lambda)\nu)=\lambda\zeta(\mu)+(1-\lambda)\nu.$$
Let $p(x)\in K_1(T)$. By the Kantorovich extension theorem (\S \ref{Appendix}),
the map defined by $\Lambda(\overline{\phi^M})=p(\phi)$,
for all $\phi$, extends to a positive linear functional on $C(\overline{M})$.
So, there exists a regular Borel probability measure
$\mu$ on $\overline{M}$ such that for all $\phi(x)$
$$p(\phi)=\int\overline{\phi^M}d\mu.$$
This shows that $\zeta$ is surjective.
Let $p$ be extreme. Then, $\zeta^{-1}(p)$ is a face of $\mathcal P(\overline{M})$.
Let $\nu$ be an extreme point of $\zeta^{-1}(p)$.
Then, $\nu$ is an extreme point of $\mathcal P(\overline{M})$ too.
However, the extreme points of $\mathcal P(\overline{M})$ are pointed measures (\S \ref{Appendix}, \ref{regular-extreme}),
i.e. $\nu=\delta_a$ for some $a\in\overline{M}$ (see \cite{Aliprantis-Inf} Th 15.9).
We conclude that for every $\phi$
$$p(\phi)=\int\overline{\phi^M}\ d\delta_a=\overline{\phi^M}(a).$$
Now assume $a_k\in M$ and $a_k\rightarrow a$.
Let $p_k=tp(a_k)$. Then for each $\phi(x)$
$$p_k(\phi)=\phi^M(a_k)\rightarrow\overline{\phi^M}(a)=p(\phi).$$
Therefore, $p_k\rightarrow p$.
\end{proof}
\bigskip

\begin{corollary} \label{compact-realization1}
If $M\vDash T$ is compact, then every $p\in E_n(T)$ is realized in $M$.
\end{corollary}
\begin{proof}
The map $\a\mapsto tp(\a)$ is logic-continuous. Hence, its range is compact.
By Proposition \ref{dense-realization}, it must contain $E_n(T)$.
\end{proof}
\vspace{1mm}

Now, we extend a bit the framework of logic and allow arbitrary (maybe uncountable) sets of individual variables.
Every formula uses a finite number of variables as before.
If $\x$ is a (possibly infinite) tuple of variables, then $D_{\x}(T)$ is the normed space of all
formulas (up to $T$-equivalence) whose free variables are included in $\x$.
Also, $K_{\x}(T)$ is the compact convex set consisting of all complete types with free variables $\x$,
i.e. positive linear functionals $p:D_{\x}(T)\rightarrow\Rn$ with $p(1)=1$.
\bigskip

Let $\Gamma(\x)$ be a set of conditions satisfiable with $T$ (a partial type).
Complete types can be still regarded as partial types.
We say $\Gamma$ is a \emph{face} of $K_{\x}(T)$ (or is \emph{facial}\index{facial type})
if the set $$\{p\in K_{\x}(T):\ \Gamma\subseteq p\}$$ is a face of $K_{\x}(T)$.
Note that if $\Gamma(\x,\y)$ is satisfiable and $\a$ realizes $\Gamma|_{\x}$ then $\Gamma(\a,\y)$ is satisfiable.

\begin{proposition}\label{ext2ext1}
Let $\Gamma(\x,\y)$ be a face of $K_{\x\y}(T)$. Then $\Gamma|_{\x}$ is a face of $K_{\x}(T)$.
In particular, the image of an extreme type under the restriction map $K_{n+1}(T)\rightarrow K_n(T)$ is extreme.
\end{proposition}
\begin{proof} Let $M\vDash T$ be $\kappa$-saturated where $|\x|+|\y|+\aleph_0<\kappa$.
Assume that $$\frac{1}{2}tp(\a_1)+\frac{1}{2}tp(\a_2)=tp(\a)\in\Gamma|_{\x}.\ \ \ \ \ \ \ (*)$$
where $\a_1,\a_2,\a\in M$.
Since $\Gamma(\a,\y)$ is satisfiable, by saturation, there exists $\b$ such that
$(\a,\b)\vDash\Gamma(\x,\y)$. Let $\Sigma(\bar{u},\bar{v})$ be the set of all conditions
of the form $$\frac{1}{2}\phi(\a_1,\bar{u})+\frac{1}{2}\phi(\a_2,\bar{v})=\phi^M(\a,\b).$$
where $\phi(\x,\y)$ is an $L$-formula.
$\Sigma$ is closed under linear combinations. Let
$$\theta(\bar{u},\bar{v})=\frac{1}{2}\phi(\a_1,\bar{u})+\frac{1}{2}\phi(\a_2,\bar{v}).$$
By the assumption $(*)$,
$$\inf_{\bar{u}\bar{v}}\theta^M(\bar{u},\bar{v})=\frac{1}{2}\inf_{\bar{u}}\phi^M(\a_1,\bar{u})+\frac{1}{2}\inf_{\bar{v}}\phi^M(\a_2,\bar{v})
=\inf_{\y}\phi^M(\a,\y)\leqslant\phi^M(\a,\b)$$
Similarly, one has that $$\phi^M(\a,\b)\leqslant\sup_{\bar{u}\bar{v}}\theta^M(\bar{u},\bar{v}).$$
This shows that $\theta(\bar{u},\bar{v})=\phi^M(\a,\b)$ is satisfiable in $M$.
By affine compactness, $\Sigma$ is satisfiable in $M$. Let $(\bar{c},\bar{e})\vDash\Sigma$.
Then, one has that $$\frac{1}{2}tp(\a_1,\bar{c})+\frac{1}{2}tp(\a_2,\bar{e})=tp(\a,\b)\in\Gamma(\x,\y).$$
Since $\Gamma$ is a face, we must have that $tp(\a_1,\bar{c}), tp(\a_2,\bar{e})\in\Gamma(\x,\y)$.
Restricting to $\x$, we conclude that $tp(\a_1), tp(\a_2)\in\Gamma|_{\x}$.
\end{proof}
\bigskip

It is easy to verify that $\Gamma(\x)$ is a face if and only if $\Gamma|_{\y}$ is a face for every finite $\y\subseteq\x$.
In particular, $p(\x)$ is extreme if and only if $p|_{\y}$ is extreme for every finite $\y\subseteq\x$.

Exposed types may be regarded as the affine variant of principal types.
However, extreme types have more flexibility and behave similarly.

\begin{definition}
\emph{A tuple $\a\in M$ is \emph{extreme over}\index{extreme over} $A\subseteq M$ if $tp(\a/A)$ is extreme in $K_n(A)$.
In particular, $\a$ is extreme if it is extreme over $\emptyset$.}
\end{definition}

As stated above, if $\b$ realizes $p(\x,\y)|_{\y}$, then $p(\x,\b)$ is a type.
Also, if $(\a,\b)$ realizes $p(\x,\y)$, then
$$p(\x,\b)(\phi(\x,\b))=tp(\a/\b)(\phi(\x,\b))=\phi^M(\a,\b)=p(\phi(\x,\y)).$$

\begin{proposition} \label{extreme3}
$\a\b\in M$ is extreme if and only if $\b$ is extreme and $\a$ is extreme over $\b$.
In particular, if $\a$ is extreme, the restriction map $K_n(\a)\rightarrow K_n(\emptyset)$
takes extreme types to extreme types.
\end{proposition}
\begin{proof}
We may assume $M$ is $\aleph_0$-saturated and that $|\a|=|\b|=1$. Let $p(x,y)=tp(a,b)$.

$\Rightarrow$:\ \ That $b$ is extreme is a consequence of Proposition \ref{ext2ext1}.
Suppose that $$tp(a/b)=\frac{1}{2}p_1(x,b)+\frac{1}{2}p_2(x,b).$$
Assume $c,e\in M$ realize $p_1(x,b)$, $p_2(x,b)$ respectively.
Then, for each $\phi(x,y)$
$$p(\phi)=\phi^M(a,b)=tp(a/b)(\phi(x,b))=\frac{1}{2}\phi^M(c,b)+\frac{1}{2}\phi^M(e,b).$$
This means that 
$$p=\frac{1}{2}tp(c,b)+\frac{1}{2}tp(e,b).$$
Since, $p$ is extreme, $p(x,y)=tp(c,b)=tp(e,b)$ and hence $p(x,b)=p_1(x,b)=p_2(x,b)$.
\bigskip

$\Leftarrow$: Assume $$p(x,y)=\frac{1}{2}p_1(x,y)+\frac{1}{2}p_2(x,y).$$
Since $b$ is extreme, by restricting these types to $y$,
we conclude that $$tp(b)=p|_y=p_1|_y=p_2|_y.$$
Therefore, $p_1(x,b)$ and $p_2(x,b)$ are realizable types and one has that
$$tp(a/b)=p(x,b)=\frac{1}{2}p_1(x,b)+\frac{1}{2}p_2(x,b).$$
Since $a$ is extreme over $b$, one has that
$p(x,b)=p_1(x,b)=p_2(x,b)$. Since this type is realized by $a$,
we conclude that $p=p_1=p_2$ which means that $p$ is extreme.
\end{proof}

\begin{lemma} \label{facetoface}
Let $\Gamma(\x,\y)$ be facial and $\b\in M$. If $\Gamma(\x,\b)$ is satisfiable, it is a face of $K_{\x}(\b)$.
\end{lemma}
\begin{proof} We may assume $M$ is $\aleph_0$-saturated and $|\x|=1$. Let $a$ realize $\Gamma(x,\b)$.
Suppose that $a_1,a_2\in M$ and $$\frac{1}{2}tp(a_1/\b)+\frac{1}{2}tp(a_2/\b)=tp(a/\b)\in\Gamma(x,\b).$$
This means that $$\frac{1}{2}tp(a_1,\b)+\frac{1}{2}tp(a_2,\b)=tp(a,\b)\in\Gamma(x,\y).$$
Since $\Gamma(x,\y)$ is a face, we conclude that $tp(a_1,\b),\ tp(a_2,\b)\in\Gamma(x,\y)$.
Hence, $$tp(a_1/\b),\ tp(a_2/\b)\in\Gamma(x,\b).$$
\end{proof}

\begin{proposition}
Let $M$ be $\kappa$-saturated and $A\subseteq M$ where $|A|<\kappa$. Then, for each extreme $\b\in M$
there exists $\bar{c}\in M$ which is extreme over $A$ and $\bar{c}\equiv\b$.
\end{proposition}
\begin{proof}
Let $|\b|=n$. Consider the affine surjective map $\pi: K_n(A)\rightarrow K_n(\emptyset)$.
Since $tp(\b)$ is extreme, $\pi^{-1}(tp(\b))$ is a face in $K_n(A)$.
This face contains some $p(\x)\in E_n(A)$ which is realized by say $\bar{c}\in M$.
Then, $\bar{c}$ is extreme over $A$ and $\pi(p(\x))=tp(\b)$, i.e. $\bar{c}\equiv\b$.
\end{proof}

\subsection{Omitting non-extreme types} \label{omitting types}\index{omitting types}
In the previous subsection we proved that all extreme types are realized in compact models (if any).
In this subsection we show that there are models which omit all non-extreme types.
There is no limit on the cardinality of language.

\begin{theorem} \label{omitting}
Let $T$ be a complete theory in a language $L$ and $p(\z)\in K_{n}(T)$
be non-extreme. Then $p$ is omitted in a (complete) model $M$ of $T$.
\end{theorem}
\begin{proof}
For simplicity we assume $n=1$.
Let $\kappa\geqslant|L|$ be such that $\kappa^{\aleph_0}=\kappa$ and $N$ be a $\kappa$-saturated model of $T$.
Let $\{X,Y,Z,W\}$ be a partition of $\kappa$ into sets of cardinality $\kappa$.
Let $C=\{c_i|\ i\in Z\}$ be a set of distinct new constant symbols,
$\{\sigma_i|\ i\in X\}$ be an enumeration of $L(C)$-sentences
and $\{\phi_i(y)| \ i\in Y\}$ be an enumeration of $L(C)$-formulas with one free variable $y$.
Also, assume that all $\omega$-sequences of constant symbols of $C$ are enumerated by indices from $W$.
We construct a chain
$$\ \ \ \ T_{0}\subseteq T_1\subseteq\cdots \subseteq T_i\subseteq\cdots\ \ \ \ \ \ \ \ \ \ \ \ \ i<\kappa$$
of satisfiable extensions of $T$ such that for each $i$:

(I) $T_i=T\cup \Gamma_i(\bar{e})$ where $\bar{e}\in C$ (maybe of infinite length) and $|\Gamma_i|<\kappa$

(II) $\Gamma_i(\x)$ is a face in $K_{\x}(T)$ for $0<i$.
\bigskip

Set $T_{0}=T$ and for infinite limit $i$ set $\Gamma_i=\cup_{j<i} \Gamma_j$,\ \ $T_i=T\cup\Gamma_i$.
Note that (I), (II) hold.
Assume $T_i$ is defined and (I), (II) hold. We define $T_{i+1}$ according to the following cases:
\vspace{1mm}

- $i\in X$:\ There is a biggest $r$ such that $T_i,r\leqslant\sigma_i$
is satisfiable. So, $T_i\vDash\sigma_i\leqslant r$.
Let $\e$ be the tuple obtained by unifying the constants of $C$ used in $\Gamma_i$ and $\sigma_i$.
Let $$\Gamma_{i+1}(\e)=\Gamma_i(\e)\cup\{r\leqslant\sigma_i(\e)\},\ \ \ \ \ \ \ T_{i+1}=T\cup\Gamma_{i+1}(\e).$$
Note that $T\cup\Gamma_i(\x)\vDash\sigma_i(\x)\leqslant r$ so that $\Gamma_{i+1}(\x)$  is a face of $K_{\x}(T)$.
\bigskip

- $i\in Y$:\ Unifying the constants of $C$ used in $\Gamma_i$ and $\phi$, we may write $\Gamma_i=\Gamma_i(\e)$ and $\phi_i=\phi_i(\e,y)$.
Take a $c\in C$ not been used in $T_i,\phi_i$.
Let $$\Gamma_{i+1}(\e,c)=\Gamma_i(\e)\cup\{\phi_i(\e,c)\leqslant\inf_y\phi_i(y)\},\ \ \ \ \ \ \ T_{i+1}=T\cup\Gamma_{i+1}(\e,c).$$
Then, the conditions (I), (II) are satisfied for $T_{i+1}$ and $\Gamma_{i+1}$.
\bigskip

- $i\in Z$:\ Assume $T_i=T\cup\Gamma_i(\bar{e},c_i)$ and conditions (I), (II) are satisfied.
We claim that there are $\theta(z)\leqslant0\in p(z)$ and $\epsilon>0$
such that  $T_i\cup\{\epsilon\leqslant\theta(c_i)\}$ is satisfiable.
Suppose not. Then, we must have that $$T\cup\Gamma_i(\x,z)\vDash p(z)$$
By Proposition \ref{ext2ext1}, $p(z)$ is an extreme type. This is a contradiction.
We conclude that there are $\theta(z)\leqslant0\in p(z)$ and a greatest $\epsilon>0$ such that and
$T_i\cup\{\epsilon\leqslant\theta(c_i)\}$ is satisfiable.
Let $$\Gamma_{i+1}(\e,c_i)=\Gamma_i(\e,c_i)\cup\{\epsilon\leqslant\theta(c_i)\},\ \ \ \ \ \ \ T_{i+1}=T\cup\Gamma_{i+1}(\e,c_i).$$
Then, the conditions (I), (II) are satisfied for $T_{i+1}$ and $\Gamma_{i+1}$.
\bigskip

- $i\in W$:\ Suppose that $T_i=T\cup\Gamma_i(\e)$ is constructed and that the
index $i$ corresponds to the sequence $(c_{i_n})$.
Suppose that for all $n,k$ $$d(c_{i_n},c_{i_{n+k}})\leqslant2^{-n}\in T_i$$
(hence every $c_{i_k}$ is a component of $\e$).
Let $c\in C$ be a constant symbol different from the components of $\e$. Set
$$\Gamma_{i+1}(\e,c)=\Gamma_i(\e)\cup\big\{d(c_{i_n},c)\leqslant 2^{-n}:\ \ n=1,2,...\big\}.$$
Note that $\Gamma_{i+1}(\e,c)$ is finitely and hence totally satisfiable.
We show that $\Gamma(\x,y)$ is a face of $K_{\x y}(T)$.
Assume $$\frac{1}{2}q_1(\x,y)+\frac{1}{2}q_2(\x,y)=q(\x,y)\supseteq\Gamma_{i+1}(\x,y).$$
Since $\Gamma_i(\x)$ is a face, $q_1|_{\x}$ and $q_2|_{\x}$ must contain it.
Suppose that $c_{i_n}$ corresponds to the variable $x_n\in\x$.
So, the condition $d(x_{n},x_{n+k})\leqslant2^{-n}$ belongs to both $q_1|_{\x}$ and $q_2|_{\x}$ for each $n$.
Suppose that $q_1(\x,y)$ is realized by a tuple in $N$ where $x_n=a_n$, $y=a$ and $d(a_n,a)=r_n$.
Also, $q_2(\x,y)$ is realized by a tuple where $x_n=b_n$, $y=b$ and $d(b_n,b)=s_n$.
By the above assumption, we must have that $\frac{1}{2}r_n+\frac{1}{2}s_n\leqslant2^{-n}$ for each $n$.
This shows that $a_n\rightarrow a$ and $b_n\rightarrow b$.
Since $d(a_n,a_{n+k})\leqslant2^{-n}$ for each $k$, we must have that $d(a_n,a)\leqslant2^{-n}$.
Hence, $d(x_n,y)\leqslant2^{-n}$ belongs to $q_1(\x,y)$ for each $n$.
This implies that $q_1(\x,y)\supseteq\Gamma_{i+1}(\x,y)$. Similarly, $q_2(\x,y)\supseteq\Gamma_{i+1}(\x,y)$.
\vspace{2mm}

Finally, let $\overline T=\cup_i T_i$.
This is a complete $L(C)$-theory with the following properties:

- for every $\phi(y)$ in $L(C)$, there exists $c\in C$ such that $\phi(c)\leqslant\inf_y\phi(y)\in\overline T$

- for every $c\in C$, there exists $\theta(z)\leqslant0\in p(z)$ and $\epsilon>0$ such that
$\epsilon\leqslant\theta(c)\in\overline T$.

- for every sequence $(c_n)$ of constant symbols from $C$,
if $d(c_n,c_{n+k})\leqslant2^{-n}\in\overline T$ for every $k,n$, then there exists $c\in C$
such that $d(c_n,c)\leqslant2^{-n}\in\overline T$ for every $n$.
\vspace{2mm}

We define the canonical model of $\overline T$ as follows.
For $c,e\in C$ set $c\sim e$ if $d(c,e)=0\in\overline T$.
The equivalence class of $c$ under this relation is denoted by $\hat c$.
Let $M=\{\hat c:\ c\in C\}$. For $\hat c,\hat e\in M$ set $d^M(\hat c,\hat e)=r$
if $d(c,e)=r\in\overline T$. This defines a metric on $M$.
Define an $L(C)$-structure on $M$ by setting for all constant, relation and function symbols
(say unary for simplicity):

- $c^M=\hat c$\ \ if $c\in C$\ \ and \ \ $c^M=\hat e$\ \ if $c\in L$, $e\in C$ and $d(c,e)=0\in \overline T$

- $R^M(\hat c)=r$ \ if \ $R(c)=r\in\overline T$

- $F^M(\hat c)=\hat e$ \ if \ $d(F(c),e)=0\in\overline T$.
\vspace{1mm}

Note that for $c\in L$, since $\inf_x d(c,x)=0$ is satisfied in every model, there exists
$e\in C$ such that $d(c,e)=0\in\overline T$.
Similarly, for $c\in C$, there exists $e\in C$ such that $d(F(c),e)=0\in\overline T$.
It is not hard to verify that $M$ is a well-defined $L(C)$-structure.
It is routine to show by induction on the complexity of formulas that
for every $L$-formula $\phi(x_1,...,x_n)$, $r\in\Rn$ and $c_1,...,c_n\in C$
$$\phi^M(\hat{c}_1,...,\hat{c}_n)=r \ \ \ \ \ \ \
\mbox{iff}\ \ \ \ \ \ \ \phi(c_1,...,c_n)=r\in\overline T.$$
Another way to obtain the canonical model is to consider a model $M'\vDash\overline{T}$
and to show that $M=\{c^{M'}:\ c\in C\}$ is an elementary submodel of $M'$.

It is also clear by the construction of $\overline T$ that $M$ is metrically complete and that it omits $p(z)$.
\end{proof}
\bigskip

A stronger result is obtained by simultaneously omitting all non-extreme types.

\begin{theorem} \label{strong omitting}
There exists $M\vDash T$ omitting every non-extreme type in every $K_n(T)$.
\end{theorem}
\begin{proof}
There is nothing to prove if $K_n(T)$ is a singleton. Otherwise, $T$ has non-extreme types.
Let $\{p_\alpha(z_1,...,z_{n_\alpha}):\ \alpha\in\lambda\}$ be an enumeration of all
non-extreme types. Let $\kappa\geqslant|L|+\lambda$ be such that $\kappa^{\aleph_0}=\kappa$
and $C$ be a set of new constant symbols of cardinality $\kappa$.
Let $$\{X,Y,W\}\cup\{Z_\alpha:\ \alpha<\kappa\}$$ be a partition of $\kappa$
into $\kappa$ disjoint sets of cardinality $\kappa$.
Let $\{\sigma_i|\ i\in X\}$ be an enumeration of $L(C)$-sentences,
$\{\phi_i(y)| \ i\in Y\}$ be an enumeration of $L(C)$-formulas with one free variable $y$
and for each $\alpha$, $\{\bar{c}_i:\ i\in Z_\alpha\}$ be an enumeration of all
$n_\alpha$-tuples of constant symbols from $C$.
Also, assume that all $\omega$-sequences of constant symbols from $C$ are enumerated by indices from $W$.
Then the argument follows as in Theorems \ref{omitting} and the resulting theory
has a metrically complete canonical model omitting every $p_\alpha$.
\end{proof}
\vspace{1mm}

The model given by Theorem \ref{strong omitting} omits non-extreme types with infinite number of variables too.
This is because $p(\x)$ is extreme if and only if $p|_{\z}$ is extreme for every finite $\z\subseteq\x$.

\begin{definition}
\emph{A model $M\vDash T$ is \emph{extremal}\index{extremal model} if every $\a\in M$ has an extreme type.}
\end{definition}

So, the omitting types theorem states that every complete theory has an extremal model.
If $M$ is extremal, then for every $A\subseteq M$, the structure $(M,a)_{a\in A}$ is extremal.

Note that we have assumed $L$-structures have diameter at most $1$.
We call a theory $T$ \emph{ample} if there are $M\vDash T$
and an infinite sequence $a_n\in M$ such that $d(a_m,a_n)=1$ for all $m\neq n$.
If $T$ is ample, it has arbitrarily large models omitting the non-extreme type $p(\z)$.

\begin{proposition} \label{extremal3}
Let $T$ be a complete ample theory in a language $L$.
Let $\lambda$ be an infinite cardinal and $p(\z)\in K_n(T)$ be non-extreme.
Then $p$ is omitted in a model $M$ with $\lambda\leqslant |M|$.
\end{proposition}
\begin{proof}
Let $r>0$ be as above and $\kappa>\lambda+|L|$ be such that $\kappa^{\aleph_0}=\kappa$.
Take a partition of $\kappa$ as in the proof of Theorem \ref{omitting}.
Let $D$ be a subset of $C$ of cardinality $\lambda$.
Let $$\Gamma_{0}=\{1\leqslant d(c,c'):\ c,c'\in D\ \mbox{are\ distinct}\},
\ \ \ \ \ \ \ \ T_{0}=T\cup\Gamma_{0}.$$
Then, $T_0$ is consistent and $T_0$, $\Gamma_0$ satisfy the conditions (I), (II)
in the proof of Theorem \ref{omitting}.
We then continue the construction of $T_i$ and $\Gamma_i$ as before.
The resulting canonical model has cardinality at least $\lambda$ and omits $p(\z)$.
\end{proof}

Similarly, one can prove that

\begin{proposition} \label{extremal4}
Let $T$ be a complete ample theory in a language $L$ and $\lambda$ be an infinite cardinal.
Then, $T$ has an extremal model $M$ such that $\lambda\leqslant |M|$.
\end{proposition}
\vspace{1mm}

\begin{proposition} \label{extremal5}
Let $p(\z)$ be extreme. Then, there is an extremal $N\vDash T$ realizing $p$.
\end{proposition}
\begin{proof}
Let $\a\in M\vDash T$ realize $p(\z)$ and $\bar T=Th(M,\a)$.
Let $\bar{N}$ be an extremal model of $\bar T$ and $N$ be its reduction to the language of $T$.
Then, $N$ realizes $p(\z)$. On the other hand, for each $\b\in N$,\ \ $tp(\b/\a)$ is extreme.
So, since $\a$ is extreme, $\a\b$ is extreme. This implies that $\b$ extreme.
Therefore, $N$ omits every non-extreme type of $T$.
\end{proof}

The omitting types theorem can be extended to incomplete theories. Let $T$ be an incomplete theory.
A complete theory $\bar{T}$ is an \emph{extremal extension} of $T$ if it is an extreme
point of the compact convex set $$K(T)=\{T'\supseteq T:\ T'\ \mbox{is\ complete} \}.$$
The set of complete $n$-types for $T$ is a compact convex set which is denoted by $K_n(T)$.
We also extend the notion of extremal model for incomplete theories.
$M\vDash T$ is \emph{extremal for $T$} if for every $\a\in M$, $tp(\a)\in K_n(T)$ is extreme.

\begin{proposition}\label{incomplete-extreme}
Let $T$ be a (possibly incomplete) theory in $L$ and $\bar{T}$ an extremal extension of $T$.
Then, $M\vDash\bar{T}$ is extremal for $\bar{T}$ if and only if it is extremal for $T$.
In particular, $T$ has an extremal model.
\end{proposition}
\begin{proof}
Assume $M$ is extremal for $\bar{T}$. Let $\a\in M$ realize $p(\x)\in K_n(T)$ and
$p=\frac{1}{2}p_1+\frac{1}{2}p_2$ where $p_1,p_2\in K_n(T)$.
Let $T_1,T_2$ be the restrictions of $p_1,p_2$ to $L$-sentences.
Then, $T_1,T_2\in K(T)$ and $$\bar{T}=\frac{1}{2}T_1+\frac{1}{2}T_2.$$
Since $\bar{T}$ is extreme, we must have that $\bar{T}=T_1=T_2$.
This shows that $p_1,p_2$ are types of $\bar{T}$. Since $p$ is extreme for $\bar{T}$, we conclude that $p=p_1=p_2$.
Hence, $p$ is extreme for $T$.
Conversely, assume $M$ is extremal for $T$ and $\a\in M$ realizes $p=\frac{1}{2}p_1+\frac{1}{2}p_2$ where $p_1,p_2$ are types of $\bar{T}$.
Then, $p_1,p_2$ are types of $T$ too. So, $p=p_1=p_2$.
\end{proof}

One also shows that in fact $K_n(\bar{T})$ is a face of $K_n(T)$.

\subsection{Extremal saturation}\label{Extremal saturation}
The following remark is a consequence of general facts for compact convex sets (see \cite{Simon}).

\begin{remark} \label{facial}
If $\Gamma(\x)$ is a face of $K_{\x}(T)$ and $\y$ is disjoint from $\x$,
then $\Gamma(\x,\y)=\Gamma(\x)$ is a face of $K_{\x\y}(T)$.
If $\Gamma_i$ is a face of $K_{\x}(T)$ for each $i\in I$, then so is $\bigcup_{i\in I}\Gamma_i$ (if satisfiable).
If $\Gamma$ is a face and $\Gamma\vDash\theta(\x)\leqslant0$, then $\Gamma\cup\{0\leqslant\theta(\x)\}$
is a face (if satisfiable).
\end{remark}

It is well-known in CL that if $\mathcal F$ is a countably incomplete ultrafilter, then $\prod_{\mathcal F} M_i$ is $\aleph_1$-saturated.
More generally, if the ultrafilter is countably incomplete and $\kappa$-good, then the ultraproduct is $\alpha$-saturated.
The same proofs can be used in AL to show that such an ultraproduct is saturated if only the extreme types are considered.

\begin{definition}
\emph{A structure $M$ is \emph{extremally $\kappa$-saturated}\index{extremally saturated}
if for each $A\subseteq M$, with $|A|<\kappa$, every extreme type in $K_n(A)$ is realized in $M$.}
\end{definition}

Every compact model is extremally $\kappa$-saturated for all $\kappa$.
Also, every first order model whose first order theory is $\aleph_0$-categorical is extremally
$\aleph_0$-saturated (use Proposition \ref{dense-realization}).
Since every face contains an extreme type, an extremally $\kappa$-saturated model
realizes every facial type with less than $\kappa$ parameters.

Let $\mathcal F$ be an ultrafilter on a nonempty set $I$ and $\kappa,\lambda$ be infinite cardinals with $\lambda<\kappa$.
A function $f:S_\omega(\lambda)\rightarrow\mathcal F$ is monotonic if $f(\tau)\supseteq f(\eta)$
whenever $\tau\subseteq\eta$. It is additive if $f(\tau\cup\eta)= f(\tau)\cap f(\eta)$.\ \
$\mathcal F$ is $\kappa$-good if for every $\lambda<\kappa$ and every monotonic $f:S_\omega(\lambda)\rightarrow\mathcal F$
there is an additive $g$ such that $g(\tau)\subseteq f(\tau)$ whenever $\tau\in S_\omega(\lambda)$.

\begin{proposition} \label{sat1}
Let $\kappa$ be an infinite cardinal and $\mathcal F$ be a countably incomplete
$\kappa$-good ultrafilter on a set $I$. Let $|L|+\aleph_0<\kappa$ and for each $i\in I$,
$M_i$ be an $L$-structure. Then, $M=\prod_{\mathcal F}M_i$ is extremally $\kappa$-saturated.
\end{proposition}
\begin{proof}
The proof is an adaptation of the proof of Theorem 6.1.8 of \cite{CK1} for the present situation.
For any set $A\subseteq M$ with $|A|<\kappa$ one has that
$$(M,a)_{a\in A}\simeq\prod_{\mathcal F}(M_i,a_i)_{a\in A}.$$
So, we may forget the parameters and prove that every extreme type of $Th(M)$ is realized in $M$.
For simplicity assume $|\x|=1$ and let $p(x)$ be extreme.
Let $\mathbb U(M)$ be the set of ultracharges on $M$. This is a compact convex set whose extreme
points are ultrafilters. Let
$$V=\big\{\wp\in\mathbb{U}(M):\ \ \ p(\phi)=\int\phi^M(x)d\wp\ \ \ \ \forall\phi\big\}.$$
The type $p$ induces a positive linear functional on the space of functions $\phi^M(x)$.
By the Kantorovich extension theorem (\S \ref{Appendix})
it extends to a positive linear functional $\bar p$ on $\ell^{\infty}(M)$.
Then, $\bar p$ is represented by integration over an ultracharge on $M$
so that $V$ is non-empty.
Moreover, $V$ is a closed face of $\mathbb{U}(M)$.
In particular, assume for $r\in(0,1)$ one has that $r\mu+(1-r)\nu=\wp\in V$.
Define the types $p_\mu$, $p_\nu$ by setting for each $\phi(x)$
$$p_\mu(\phi)=\int\phi^M d\mu,\ \ \ \ \ \ \ \ p_\nu(\phi)=\int\phi^M d\nu.$$
Then, $rp_\mu+(1-r)p_\nu=p$.
We have therefore that $p_\mu=p_\nu=p$ and hence $\mu,\nu\in V$.

Let $\wp$ be an extreme point of $V$.
Then, $\wp$ is an extreme point of $\mathbb{U}(M)$ and hence it corresponds to
an ultrafilter (\S \ref{Appendix}, \ref{ultrafilters-extreme}), say $\mathcal D$
(not to be confused with the ultrafilter $\mathcal F$ on $I$).
We have therefore that $$\ \ \ \ p(\phi)=\int_M\phi^M(x)d\wp=
\lim_{\mathcal D,x}\phi^M(x) \ \ \ \ \ \ \ \forall\phi.$$
Since $|L|+\aleph_0<\kappa$, we may assume $p$ is axiomatized by a family of conditions
$$\{0\leqslant\phi(x):\ \ \phi\in\Sigma\}\equiv p(x)$$ where $\Sigma$ is a set of formulas with $|\Sigma|<\kappa$.
For this purpose, one may use formulas with rational coefficients.
Let $$\Sigma^+=\{\phi+r:\ \ \phi\in\Sigma,\ r>0\ \mbox{is\ rational}\}.$$
Let $I_1\supseteq I_2\supseteq\cdots$ be a chain such that
$I_n\in\mathcal F$ and $\bigcap_n I_n=\emptyset$.
Let $f:S_\omega(\Sigma^+)\rightarrow\mathcal F$ be defined as follows.
$f(\emptyset)=I$ and for nonempty $\tau\in S_\omega(\Sigma^+)$
$$f(\tau)=I_{|\tau|}\cap\big\{i\in I:\ \ \ 0<\sup_x\bigwedge_{\phi\in\tau}\phi^{M_i}(x)\big\}.\ \ \ \ \ \ \ \ \ (*)$$
Since $\mathcal D$ is an ultrafilter, there exists $a\in M$ such that
$0<\phi^M(a)$ for every $\phi\in\tau$.
We have therefore that $f(\tau)\in\mathcal F$.
Also, $f(\tau)\supseteq f(\eta)$ whenever $\tau\subseteq\eta$.
Since $\mathcal F$ is $\kappa$-good, there is an additive function $g:S_\omega(\Sigma^+)\rightarrow\mathcal F$
such that $g(\tau)\subseteq f(\tau)$ for every $\tau\in\Sigma^+$.
Let $$\tau(i)=\{\phi\in\Sigma^+:\ i\in g\{\phi\}\}.$$
If $\phi_1,...,\phi_n\in\tau(i)$ are distinct, then
$$i\in g\{\phi_1\}\cap\cdots\cap g\{\phi_n\}=g\{\phi_1,...,\phi_n\}\subseteq f\{\phi_1,...,\phi_n\}\subseteq I_n.$$
In particular, if $|\tau(i)|\geqslant n$ then $i\in I_n$
and hence $\tau(i)$ is finite for each $i$ as $\bigcap_n I_n=\emptyset$.
We have also that
$$i\in\bigcap\big\{g\{\phi\}:\ \phi\in\tau(i)\big\}=g(\tau(i))\subseteq f(\tau(i))\subseteq I_{|\tau(i)|}.$$

Now, we define $a\in M$ which realizes $p(x)$. By $(*)$, we may choose $a_i\in M_i$ such that
$$0<\bigwedge_{\phi\in\tau(i)}\phi^{M_i}(a_i).$$
Fix $\phi\in\Sigma^+$. For each $i\in g\{\phi\}\in\mathcal F$ one has that
$\phi\in\tau(i)$ and so $0<\phi^{M_i}(a_i)$. This shows that $0\leqslant\phi^M(a)$.
We conclude that $0\leqslant\phi^M(a)$ for every $\phi\in\Sigma$ and hence $a$ realizes $p(x)$.
\end{proof}

This gives an other proof that every compact structure is extremally $\kappa$-saturated for every $\kappa$.
Extremal forms of homogeneity and universality are meaningful and share some interesting properties with the standard case.
These will be studied briefly in section \ref{AEC}.

It is natural to ask for which ultracharges $\mu$,\ \ $\prod_{\mu}M_i$ is $\aleph_1$-saturated?
The following propositions may be a first step to find an answer to this question.

\begin{proposition}\label{mid}
Let $\mu$ be an ultracharge on $I=[0,1]\cap\mathbb Q$ with $\mu[r,s]=s-r$ and $N=\prod_\mu M_i$.
Then, $N$ has midpoints. Moreover, every $\phi^{N}(x)$ with parameters in $N$ takes on its maximum and minimum.
\end{proposition}
\begin{proof}
Assume $d(a,b)=r$ where $a,b\in N$. For $s\in[0,1]$ define
$$f(s)=\int_0^sd(a_i,b_i)d\mu-\int_s^1d(a_i,b_i)d\mu.$$
Then, $f$ is continuous and $f(0)=-r$,\ \ $f(1)=r$. So, there exists $s$ such that $f(s)=0$.
Let $c_i=a_i$ if $i\leqslant s$ and $c_i=b_i$ otherwise.
Then, one has that $d(c,a)=d(c,b)=\frac{r}{2}$.
On the other hand, assume for example that $\inf\phi^{N}(x,b)=0$.
So, $$\int\inf_x\phi^{M_i}(x,b_i)=0.$$
Let $g:I\rightarrow\{1,2,...\}$
be a bijection. For each $i\in I$ take $a_i\in M_i$ such that
$$\phi^{M_i}(a_i,b_i)\leqslant\inf_x\phi^{M_i}(x,b_i)+\frac{1}{g(i)}.$$
Then, by integrating, one has that $\phi^N(a,b)\leqslant0$. Note that $N$ may be incomplete.
\end{proof}

\begin{proposition}\label{sat1}
Let $T$ be complete and $M\vDash T$. Let $\wp$ be an ultracharge on $\mathbb N$ such that
$\wp(\{k\})=0$ for all $k$. Then $M^{\wp}$ realizes every type in $E_n(T)$.
\end{proposition}
\begin{proof}
(i) Assume $n=1$ and let $\mathcal{P}(\overline{M})$ be the set of regular Borel probability measures
on $\overline{M}$ (the Stone-\v{C}ech compactification of $M$).
For $\nu\in\mathcal{P}(\overline{M})$ define a type $p_\nu$ by $$p_\nu(\phi)=\int\overline{\phi^M} d\nu$$
where $\overline{\phi^M}$ is the natural extension of $\phi^M$ to $\overline{M}$. Then,
$$\zeta:\mathcal{P}(\overline{M})\rightarrow K_1(T), \ \ \ \ \ \ \ \ \ \ \zeta(\nu)=p_\nu$$ is an affine continuous function.
Conversely, let $p\in K_1(T)$ and for each $L$-formula $\phi(x)$ define $$\Lambda(\overline{\phi^M})=p(\phi).$$
This is a positive linear functional on the set of function $$\{\overline{\phi^M}:\ \phi(x)\ \mbox{is\ an $L$-formula}\}$$
which is a majorizing linear subspace of $\mathbf{C}(\overline{M})$.
So, by Kantorovich extension theorem (\S \ref{Appendix}), it extends to a
positive linear functional on $\mathbf{C}(\overline{M})$.
By Riesz representation theorem, there exists a regular Borel probability measure
$\nu$ on $\overline{M}$ such that $p(\phi)=\int\overline{\phi^M}d\nu$ for all $\phi(x)$.
This shows that $\zeta$ is surjective.
Let $p\in E_1(T)$. Then, $\zeta^{-1}(p)$ is a face of $\mathcal{P}(\overline{M})$.
Let $\mu$ be an extreme point of $\zeta^{-1}(p)$.
Then $\mu$ is an extreme point of $\mathcal{P}(\overline{M})$ too.
The extreme points of $\mathcal{P}(\overline{M})$ are of course pointed measures, i.e.
$\mu=\delta_a$ for some $a\in\overline{M}$ (\S \ref{Appendix}, \ref{regular-extreme}). 
We conclude that for each $\phi$
$$p(\phi)=\int\overline{\phi^M}\ d\delta_a=\overline{\phi^M}(a).$$
Let $a_k\rightarrow a$ where $a_k\in M$.
Then, for every $\phi$ one has that $\phi^M(a_k)\rightarrow\overline{\phi^M}(a)$ and hence
$$\phi^{M^\wp}([a_k])=\int\phi^M(a_k)d\wp=\overline{\phi^M}(a)=p(\phi)$$
which shows that $[a_k]\in M^{\wp}$ realizes $p$.
\end{proof}

\subsection{Extremal models}
There is a close relation between the topological nature of the extreme boundary of type spaces
and the model theoretic nature of the class of extremal models.
Generally, extremal models form an abstract elementary class.
They form an elementary class if and only if the extreme boundaries are closed.
In this case, the extreme types correspond to the CL-types of a complete CL-theory.

One verifies that if $f:M\rightarrow N$ is a partial elementary map with domain $A$ and $p(\x)\in K_n(A)$,
then its shift $$f(p)(\phi(\x,f(\a)))=p(\phi(\x,\a))$$
is a type over $f(A)$. If $p$ is extreme, then $f(p)$ is extreme.

In CL, conjunction and disjunction of formulas are allowed. So, the set of formulas then form
a normed Riesz space. For a CL-theory $\mathbb{T}$, the notion of type is defined similar to the AL case.
A $n$-type is a positive linear functional $p$ on the space of $\mathbb{T}$-equivalence classes
of CL-formulas $\phi(\x)$, $|\x|=n$, which preserves conjunction and $p(1)=1$.
The set of $n$-types of $\mathbb{T}$ is denoted by $S_n(\mathbb{T})$\index{$S_n(\mathbb{T})$}.
The logic topology and metric topology on $S_n(\mathbb{T})$
are defined similar to the AL case (see \cite{BBHU}).
For example, if $M\vDash\mathbb{T}$ is $\aleph_0$-saturated (in the CL sense),
$$\mathbf{d}(p,q)=\inf\{d(\a,\b):\ \a\vDash p,\ \b\vDash q,\ \a,\b\in M\}.$$
$S_n(\mathbb{T})$ is a compact Hausdorff space with the logic topology.

If $\mathbb{T}$ is a CL-theory, its \emph{affine part} \index{affine part of $\mathbb{T}$}
is the set of all affine consequences of $\mathbb{T}$, i.e. affine conditions $\sigma\leqslant\eta$ such that $\mathbb{T}\vDash\sigma\leqslant\eta$.
This is denoted by $\mathbb{T}_{\mathrm{af}}$\index{$\mathbb{T}_{\mathrm{af}}$}.
So, $\phi(\x)$ is an affine formula, then $\mathbb{T}\vDash\phi=0$ if and only if $\mathbb{T}_{\mathrm{af}}\vDash\phi=0$.
In particular, the restriction of a $n$-type of $\mathbb{T}$ to affine formulas is a $n$-type of its affine part.

\begin{theorem} \label{extremal embedding}
Let $T$ be a complete $L$-theory. Then, $E_n(T)$ is closed for all $n$ if and only if
the class of extremal models of $T$ is axiomatized by a CL-theory $\mathbb T$.
In this case, $\mathbb T$ is CL-complete.
\end{theorem}
\begin{proof}
Assume every $E_n(T)$ is closed. We show that the class of extremal models is closed under ultraproduct and CL-equivalence.
Let $\mathcal F$ be an ultrafilter on a set $I$ and $M_i\vDash T$ be extremal for each $i$.
Let $a=[a_i]\in M=\prod_{\mathcal F}M_i$ and $p_i(x)=tp(a_i)\in E_1(T)$.
Then $$p=\lim_{\mathcal{F},i}p_i\in E_1(T)$$ and for each $\phi(x)$
$$p(\phi)=\lim_{\mathcal{F},i} p_i(\phi)=\lim_{\mathcal{F},i}\phi^{M_i}(a_i)=\phi^{M}(a).$$
So, every $a\in M$ (and similarly every tuple $\a\in M$) has an extreme type.
On the other hand, assume $M\vDash T$ is extremal and $N\equiv_{\mathrm{CL}}M$.
Then, there is an ultrafilter $\mathcal F$ such that $N$ is elementarily embedded in $M^{\mathcal F}$. So, $N$ is extremal.
We conclude that there exists a CL-theory axiomatizing the extremal models of $T$.

Conversely, assume extremal models of $T$ are axiomatized by a CL-theory $\mathbb{T}$.
Since $T$ is complete, it is the affine part of $\mathbb{T}$.
So, the restriction map $$\zeta:S_n(\mathbb{T})\rightarrow E_n(T)$$ is well-defined.
Moreover, it is continuous. 
By Proposition \ref{extremal5}, every extreme type is realized in an extremal model.
So, $\zeta$ is surjective. We conclude that every $E_n(T)$ is closed.

For the second part of the theorem, assume $M,N\vDash\mathbb{T}$. We show that $M\equiv_{\mathrm{CL}}N$.
Take an appropriate ultrafilter $\mathcal F$ as above such that $M^{\mathcal F}$ and $N^{\mathcal F}$
are extremally $\aleph_1$-saturated models of $T$. They are also extremal.
We show that they are partially isomorphic.
For $\a\in M^{\mathcal F}$ and $\b\in N^{\mathcal F}$ of the same finite length set $\a\sim\b$
if $tp(\a)=tp(\b)$. Clearly, $\emptyset\sim\emptyset$.
Suppose $\a\sim\b$ and $c\in M^{\mathcal F}$ is given. Then, $tp(c/\a)$ is extreme.
So, its shift $$\{\phi(\b,x)=r:\ \ \phi^{M^{\mathcal F}}(\a,c)=r\}$$ is extreme.
Assume it is realized by $e\in N^{\mathcal F}$. Then, $\a c\equiv\b e$.
The back property is verified similarly. We conclude that
$M^{\mathcal F}\equiv_{\mathrm{CL}}N^{\mathcal F}$ and hence $M\equiv_{\mathrm{CL}}N$.
\end{proof}
\vspace{1mm}

One verifies that one direction of Proposition \ref{extremal embedding} holds for incomplete $T$ too.
That is, if every $E_n(T)$ is closed, then extremal models form an elementary class.
If every $E_n(T)$ is closed in $K_n(T)$, the CL-theory axiomatizing
the extremal models of $T$ is called the \emph{extremal theory} of $T$. It is denoted by $T^{ex}$\index{$T^{ex}$}.

\begin{proposition}
Let $T$ be an AL-theory which is CL-complete, i.e. $M\equiv_{\mathrm{CL}}N$
for any $M,N\vDash T$. If every $E_n(T)$ is closed, then $T$ is trivial, i.e. it satisfies $\sup_{xy} d(x,y)=0$.
\end{proposition}
\begin{proof}
$T^{ex}$ is equal to the CL-theory of any extremal model $M$ of $T$. By the assumption,
for every $N\vDash T$ one has that $N\equiv_{\mathrm{CL}}M$. So, every model of $T$ is extremal,
i.e. $T$ has no non-extreme type. This happens only if every $K_n(T)$ is a singleton
which means that $T$ is trivial.
\end{proof}

\begin{example}
\emph{Let APrAA be the theory probability algebras with an aperiodic automorphism \cite{BBHU}.
This is a CL-complete theory.
It is not hard to see that APrAA is axiomatized by affine axioms (see \S \ref{Examples} below).
So, regarding it as a non-trivial affine theory which is CL-complete, we conclude that
$E_n(\mathrm{APrAA})$ is not closed for some $n$.}
\end{example}
\vspace{1mm}

So, in general, $E_n(T)$ may not be closed.
Our next goal is to show that $E_n(T)$ coincides with $S_n(T^{ex})$ if $T^{ex}$ exists.

\begin{lemma}\label{extremal distance}
Let $p,q\in E_n(T)$ and $M\vDash T$ be extremally $\aleph_0$-saturated.
Then, $$\mathbf{d}(p,q)=\inf\{d(\a,\b):\ \a,\b\in M,\ \a\vDash p, \ \b\vDash q\}.$$
\end{lemma}
\begin{proof}
Let $r=\mathbf{d}(p,q)$ and $$\Sigma(\x,\y)=p(\x)\cup q(\y)\cup\{d(\x,\y)\leqslant r\}.$$
$\Sigma$ is satisfiable. We show that it is a face. It is clearly convex.
Suppose that $$\frac{1}{2}p_1(\x,\y)+\frac{1}{2}p_2(\x,\y)\supseteq\Sigma.$$
Restricting $p_1$ and $p_2$ to $\x$ (and then to $\y$), we conclude that $p_1$ and $p_2$ contain $p(\x)\cup q(\y)$.
Suppose that $(\a_1,\a_2)$ realizes $p_1$ and $(\b_1,\b_2)$ realizes $p_2$. Then, we must have that
$$\frac{1}{2}d(\a_1,\a_2)+\frac{1}{2}d(\b_1,\b_2)\leqslant r.$$
Since $\a_1,\b_1$ realize $p(\x)$ and $\a_2,\b_2$ realize $q(\y)$, we must have that
$$d(\a_1,\a_2)=d(\b_1,\b_2)=r.$$ Hence, $p_1$, $p_2$ contain $\Sigma$.
We conclude that $\Sigma(\x,\y)$ is realized in $M$.
\end{proof}

In particular, if $T$ has a first order model $M$, then all extreme types are realized in
$M^{\mathcal F}$ for some suitable $\mathcal F$.
So, $\mathbf{d}(p,q)=1$ for every distinct extreme types $p,q$.

\begin{theorem} \label{complete}
Assume every $E_n(T)$ is closed. Then, for each $n$, the restriction map $\zeta:S_n(T^{ex})\rightarrow E_n(T)$
is an isometry of the metric topologies and homeomorphism of the logic topologies.
\end{theorem}
\begin{proof}
Let $M\vDash T$ be extremal and $\mathcal F$ be a countably incomplete $\kappa$-good ultrafilter on a set $I$ where $|L|+\aleph_0<\kappa$.
Then, $M\preccurlyeq_{\mathrm{CL}} M^{\mathcal F}$ is extremal and extremally $\kappa$-saturated. Hence, $\zeta$ is surjective.
Moreover, if $\a,\b\in M$ and $tp(\a)=tp(\b)$, as in the proof of Proposition \ref{extremal embedding},
$(M^{\mathcal F},\a)$ and $(M^{\mathcal F},\b)$ are partially isomorphic. So, $\a$ and $\b$ have the same CL-type.
This shows that $\zeta$ is injective.
Now, by Lemma \ref{extremal distance}, $\zeta$ is an isometry. Also, since $\zeta$ is logic-continuous
and $S_n(T^{ex})$ is compact, $\zeta$ is a homeomorphism of the logic topologies.
\end{proof}

By a $n$-\emph{state}\index{state} for a CL-complete theory $\mathbb{T}$ in $L$ we mean a norm $1$
positive linear functional on the set of $\mathbb{T}$-equivalence classes of CL-formulas $\phi(\x)$ where $|\x|=n$.
Since the set of functions of the form $\hat\phi(p)=p(\phi)$ is dense in $\mathbf{C}(S_n(\mathbb{T}))$,
every $n$-state $\zeta$ corresponds in a unique way to a norm $1$ positive linear functional on $\mathbf{C}(S_n(\mathbb{T}))$
and hence to a unique regular Borel probability measure $\mu_\zeta$ on $S_n(\mathbb{T})$.
Therefore, by Theorem \ref{Bauer simplex}, the set of $n$-states is a Bauer simplex whose
extreme boundary is homeomorphic to $S_n(\mathbb{T})$ .

Let $T$ be an AL-complete theory for which $T^{\mathrm{ex}}$ exists.
The reduction of every $n$-state of $T^{\textrm{ex}}$ to AL-formulas is a $n$-type of $T$.
Conversely, every type of $T$ is obtained in this way.
In particular, by Choquet-Bishop-de Leeuv theorem \ref{Choquet-Bishop-de Leeuv}, every $p\in K_n(T)$
is represented by a Borel probability measure $\mu$ on $E_n(T)=S_n(T^{\mathrm{ex}})$ so that
$$f(p)=\int f(q)d\mu\ \ \ \ \ \ \ \ \ \ \forall f\in\mathbf{A}(K_n(T)).$$
Therefore, $$\zeta(\phi)=\int q(\phi)d\mu \hspace{15mm} \phi\ \textrm{any\ CL-formula\ in}\ L$$
extends $p$ to a $n$-state of $\mathbb{T}$. This extension is not unique.

\begin{theorem}
Assume every $E_n(T)$ is closed. Then, $K_n(T)$ is a Bauer simplex if and only if
every $n$-type of $T$ has a unique extension to a $n$-state of $T^{\textrm{ex}}$.
\end{theorem}
\begin{proof}
Assume $K_n(T)$ is a Bauer simplex and $p\in K_n(T)$ extends to distinct $n$-states $\zeta_1,\zeta_2$ of $T^{\textrm{ex}}$.
Then, regarding the above notation, we must have that $\mu_{\zeta_1}\neq\mu_{\zeta_2}$.
However, these measures both represent $p$. This is a contradiction.
For the converse, assume $p$ is represented by $\mu_1$ and $\mu_2$ so that
$$f(p)=\int f(q)d\mu_1,\ \ \ \ \ \ f(p)=\int f(q)d\mu_2\ \ \ \ \ \ \ \forall f\in K_n(T).$$
Since $\mu_1$ and $\mu_2$ define $n$-states which agree on $p$, they must define the same state.
Hence, $\mu_1=\mu_2$.
\end{proof}
\vspace{1mm}

So, $K_n(T)$ is a Bauer simplex if and only if every $n$-state is uniquely determined by its values on affine formulas.
We recall that by \cite{BBHU} Th. 12.10, a complete CL-theory $\mathbb T$ is $\aleph_0$-categorical
if and only if every $S_n(\mathbb{T})$ is compact in the metric topology. The theorem includes
the case where $\mathbb T$ has a (unique) compact model if we call such a theory $\aleph_0$-categorical.

\begin{proposition} \label{categoric case}
Assume $L$ is countable. Then, every $E_n(T)$ is compact in the metric topology if and only if
every $E_n(T)$ is closed in the logic topology and $T^{ex}$ is $\aleph_0$-categorical.
\end{proposition}
\begin{proof}
Assume every $E_n(T)$ is compact in the metric topology.
Then, it is compact in the logic topology (hence closed in $K_n(T)$).
Also, by Theorem \ref{complete}, $S_n(T^{ex})$ is isometric to $E_n(T)$ (hence metrically compact).
Therefore, $T^{ex}$ is $\aleph_0$-categorical.
Conversely, if every $E_n(T)$ is closed and $T^{ex}$ is $\aleph_0$-categorical,
then $E_n(T)$ is isometric to $S_n(T^{ex})$ which is compact in the metric topology.
\end{proof}

Similar to the CL case, a AL-complete theory is called \emph{$\kappa$-categorical} if
it has a unique model of density character $\kappa$ up to isomorphism.

\begin{proposition}\label{not categorical}
Let $T$ be a complete theory in a countable language $L$. Then $T$ is not $\aleph_0$-categorical.
\end{proposition}
\begin{proof}
There is nothing to prove if $T$ is trivial.
First, assume $T$ has no finite models. Let $a,b\in M\vDash T$ be distinct.
Then $tp(a,a)$ and $tp(a,b)$ are distinct. So, $S_2(T)$ has a non-extreme type say $p(x,y)$.
There is an infinite separable model of $T$ which omits $p$ and an infinite separable one which realizes $p$.
These models are not isomorphic.
Now, assume $T$ has a finite model $M$ (so, $|M|\geqslant2$).
Let $\mu$ be an arbitrary ultracharge on $\Nn$ with the property that $\mu(\{n\})>0$ for every $n$.
Then $M^{\mu}$ is a compact infinite separable model of $T$.
On the other hand, by the upward theorem \ref{upward}, $T$ has a noncompact model say $N$.
Then, using the CL variant of the downward L\"{o}wenheim-Skolem theorem, we obtain a non-compact
separable elementary submodel $N_0\preccurlyeq N$. Now, $M^\mu$ and $N_0$ are non-isomorphic models of $T$.
\end{proof}

\begin{proposition}
Assume $T$ is $\kappa$-categorical where $\kappa\geqslant|L|+\aleph_1$.
Then, all non-compact models of $T$ are CL-equivalent.
If $T$ has no compact model, it is CL-complete.
\end{proposition}
\begin{proof}
Let $M,N\vDash T$ be noncompact. Using the CL variant of the L\"{o}wenheim-Skolem theorems,
one obtains $M_0\equiv_{\mathrm{CL}}M$ and $N_0\equiv_{\mathrm{CL}}N$ such that
$M_0$ and $N_0$ have both density character $\kappa$.
Then, $M_0\equiv N_0$ and hence $M_0\simeq N_0$. Therefore, $M\equiv_{\mathrm{CL}}N$.
\end{proof}

\subsection{The boundary AEC} \label{AEC}
Although extremal models of a theory do not form an elementary class in general, it is easy to verify
that they form an abstract elementary class \cite{Vilaveces} with the relation $\preccurlyeq$.
By proposition \ref{extremal4}, if $T$ is ample, this class has arbitrarily large models.
It is natural to ask whether this class has the elementary joint embedding property and elementary amalgamation property.

\begin{proposition}\label{JEP-AP}
Let $T$ be a complete $L$-theory and $A,B,C\vDash T$ be extremal.

\emph{(i)} There exists an extremal $M$ such that $B\preccurlyeq M$ and $C\preccurlyeq M$.

\emph{(ii)} If $A\preccurlyeq B$ and $A\preccurlyeq C$, then there are extremal $M$ and
elementary embeddings $f:B\rightarrow M$, $g:C\rightarrow M$ such that $f|_A=g|_A$.
\end{proposition}
\begin{proof}
(i) Let $$\Sigma=\mbox{ediag}(B)\cup \mbox{ediag}(C).$$
By affine compactness, $\Sigma$ is satisfiable.
Let $\overline T$ be a complete extremal extension of $\Sigma$ (in the language $L(B\cup C)$).
Let $\overline M$ be an extremal model of $\overline T$ and $M$ be its restriction to $L$.
Then, there are elementary embeddings $f:B\rightarrow M$ and $g:C\rightarrow M$ which we may assume without harm that
$f,g$ are inclusion maps. So, $B\preccurlyeq M$ and $C\preccurlyeq M$. We show that $M$ is an extremal model of $T$.
Let $\e\in M$. Since $\overline{M}$ is extremal, $tp^{\overline{M}}(\e)=tp^M(\e/B\cup C)$ is extreme.
Therefore, $\e$ is extreme over $B\cup C$. We show that $B\cup C\subseteq M$ is extreme over $\emptyset$.
Then, by Proposition \ref{extreme3}, we conclude that $\e$ is extreme over $\emptyset$.
So, assume $B$ is enumerated by a tuple $\a$ and $C$ is enumerated by a tuple $\b$.
To prove that $\a\b$ is extreme, we have only to show that $\a$ is extreme over $\b$.
Assume $$tp(\a/\b)=p(\x,\b)=\frac{1}{2}p_1(\x,\b)+\frac{1}{2}p_2(\x,\b). \ \ \ \ \ \ \ (*)$$
Let $q, q_1,q_2$ be the restrictions of $p,p_1,p_2$ to $L$ respectively.
Since $\a$ is extreme, one has that $q=q_1=q_2$.
We may identify $a_i\in \a$ and $b_j\in\b$ with their corresponding constant symbols
and regard $p(\a,\b)$ as $L(B\cup C)$-theory. In this case, $p(\a,\b)$ is in fact equal to $\overline{T}\supseteq \Sigma$.
Similarly, $p_1(\a,\b)$ and $p_2(\a,\b)$ are other complete $L(B\cup C)$-theories containing $\Sigma$
which we may denote by $T_1$ and $T_2$ respectively.
Then, by $(*)$ above, one has that $\overline{T}=\frac{1}{2}T_1+\frac{1}{2}T_2$.
Since $\overline{T}$ is extreme, we must have that $\overline{T}=T_1=T_2$.
We then conclude that $$p(\x,\b)=p_1(\x,\b)=p_2(\x,\b)$$ and hence $\a$ is extreme over $\b$.

(ii) Similar.
\end{proof}

One verifies that part (ii) of Proposition \ref{JEP-AP} holds for incomplete $T$ too.

Let $T$ be a complete theory. It would be interesting to study saturation and homogeneity
inside the class of extremal models of $T$ using tools provided by AL.
We show that some results of \S \ref{Saturation and homogeneity} hold in this class too.

\begin{lemma}\label{realize2}
Assume $M\vDash T$ is extremal and $A\subseteq M$. Then, every $p(\x)\in E_n(A)$ is realized in some extremal $M\preccurlyeq N$.
\end{lemma}
\begin{proof}
The restriction map $K_n(M)\rightarrow K_n(A)$ is affine. So, $p$ is extended to some $q(\x)$ in $E_n(M)$.
By Proposition \ref{extremal5}, $q$ is realized in some extremal model $\bar{N}$ of $ediag(M)$ by say $\a\in\bar{N}$.
Let $N$ be the reduction of $\bar{N}$ to the language of $T$. So, without loss, we have that $M\preccurlyeq N$.
Also, for each $\b\in N$, $tp(\b/M)$ is extreme. Since $M$ is extremal, $\b M$ is extreme.
Therefore, $\b\in N$ is extreme. We conclude that $N$ is an extremal model of $T$ which realizes $p$.
\end{proof}

We have already defined extremally $\kappa$-saturated models in \S \ref{Extremal saturation}.

\begin{definition}
\emph{$M\vDash T$ is \emph{extremally $\kappa$-homogeneous}\index{extremally homogeneous}
if for every extreme tuples $\a$, $\b$ in $M$ with $|\a|=|\b|<\kappa$,
whenever $\a\equiv\b$ and $c$ is extreme over $\a$, there exists $e$ extreme over $\b$ such that $\a c\equiv\b e$.
It is \emph{extremally $\kappa$-universal} \index{extremally universal} if for every extremal $N\vDash T$ with $\mathsf{dc}(N)<\kappa$,
there is an elementary embedding $f:N\rightarrow M$.}
\end{definition}

It is clear that if $M$ is extremal, extremal $\kappa$-homogeneity is the same as $\kappa$-homogeneity.
So, by Proposition \ref{homogeneous}, separable extremal extremally $\aleph_0$-homogeneous models
are isomorphic if and only if they realize the same types in every $E_n(T)$.

\begin{proposition}
For every extremal $M\vDash T$ and infinite $\kappa$ there is $M\preccurlyeq N$ which is extremal
and extremally $\kappa$-saturated.
\end{proposition}
\begin{proof}
Assume $\kappa$ is regular. Let $\a\in M$ and $p(x,\a)$ be an extreme type where $|\a|<\kappa$.
Let $\{p_i:\ i<\lambda\}$ be an enumeration of all extreme types over all parameter sets $A\subseteq M$ where $|A|<\kappa$.
By Lemma \ref{realize2}, there is a chain $$M=M_0\preccurlyeq M_1\preccurlyeq\cdots\preccurlyeq M_i\preccurlyeq\cdots\ \ \ \ \ i<\lambda$$
of extremal models of $T$ such that $M_{i+1}$ realizes $p_i$.
We conclude that $M'=\cup_iM_i$ is an extremal model of $T$ which realizes every $p_i$.
Iterating the argument $\kappa$-many times, we obtain an extremal extremally $\kappa$-saturated $M\preccurlyeq N$.
\end{proof}

\begin{proposition} An extremal $M\vDash T$ is extremally $\kappa$-saturated if and only if it is
extremally $\kappa$-homogeneous and extremally $\kappa^+$-universal. For $\kappa\geqslant\aleph_1$,
$M$ is extremally $\kappa$-saturated if and only if it is extremally $\kappa$-homogeneous and extremally $\kappa$-universal.
\end{proposition}
\begin{proof} Assume $M$ is $\kappa$-homogeneous and extremally $\kappa^+$-universal.
Let $\a\in M$, $|\a|<\kappa$ and $p(x)\in E_1(\a)$. By Proposition \ref{extremal5}, there is a model $(N,\b)$ of $Th(M,\a)$
which realizes $p(x)$ (by say $c\in N$) and omits every non-extreme type of $Th(M,\a)$.
Since $\a\equiv\b\in N$ is extreme, $N|_L$ is an extremal model of $T$. We may further assume $\mathsf{dc}(N)\leqslant\kappa$.
So, by $\kappa$-universality, there is an elementary embedding $f:N\rightarrow M$.
By extremal $\kappa$-homogeneity, there exists $e\in M$ such that
$\a e\equiv f(\b)f(c)$. It is then clear that $e$ realizes $p(x)$. The reverse direction is obvious.
The second part is a rearrangement of the proof for uncountable $\kappa$.
\end{proof}

One also proves that extremal extremally $\kappa$-saturated models of $T$ having density character $\kappa$ are isomorphic.

\newpage\section{Powermean}\label{Powermean}
In this section we prove the AL-variant of Keisler-Shelah isomorphism theorem.
We first give some new forms of the powermean construction.
Recall that $L$-structures are always assumed to be complete.
However, the ultramean (or powermean) construction may produce incomplete structures.
In this case, we must apply Proposition \ref{exist} to complete them.

\subsection{Powermean constructions}
\noindent{\bf Measurable powermean}\index{measurable powermean}\\
It is always assumed in the ultramean construction that the given charge is maximal.
A basic reason for this choice is that formulas like $\phi^{M_i}(a_i)$ must be integrable with respect to the variable $i$.
If $M_i$'s are all the same, there are other interesting options for which this occurs.

Let $(I,\mathcal A,\mu)$ be a charge space and $M$ be an $L$-structure.
A map $a:I\rightarrow M$ is called $\mathcal{A}$-\emph{measurable} (or measurable for short)
if $a^{-1}(B)\in\mathcal A$ for every Borel $B\subseteq M$.
Generally, for a continuous $u:M^n\rightarrow\Rn$ and measurable
$a^1,...,a^n$, the real function $u(a^1_i,...,a^n_i)$ may not be $\mu$-integrable.
The topological \emph{weight} of $M$ is defined by
$$w(M)=\min\{|\mathcal{O}|:\ \mathcal{O}\ \mbox{is\ a\ basis\ for } M \}+\aleph_0.$$
In fact, for infinite $M$, one has that $w(M)=\mathsf{dc}(M)$.

\begin{definition}
\emph{An $L$-structure $M$ is \emph{$\mathcal A$-meanable}\index{meanable} if $M$ is finite or $\mathcal{A}$ is
$w(M)^+$-complete.} 
\end{definition}

In other words, $\mathcal A$ must be closed under the intersections of $\mathsf{dc}(M)$ many elements.
So, in particular, $M$ is $\mathcal A$-meanable in the following cases:
(i) $\mathcal A =P(I)$ (ii) $M$ is separable and $\mathcal A$ is a $\sigma$-algebra
(iii) $\mathcal A$ is a Boolean algebra and $M$ is finite.

Let $\mathcal{H}(\Rn)$\index{$\mathcal{H}(\Rn)$} be the Boolean algebra of subsets of
$\Rn$ generated by the half-intervals $[r,s)$. A function $u:I\rightarrow\Rn$ is measurable
if $u^{-1}(X)\in\mathcal A$ for every $X\in\mathcal H(\Rn)$ (see \S \ref{Appendix}).

\begin{lemma}
Let $M$ be $\mathcal A$-meanable.
Then, for every continuous $f:M^n\rightarrow\Rn$ and measurable $a^1,...,a^n:I\rightarrow M$,
the function $u(i)=f(a^1_i,...,a^n_i)$ is measurable.
Similarly, for every continuous $g:M^n\rightarrow M$,\ \ $a(i)=g(a^1_i,...,a^n_i)$ is measurable.
\end{lemma}
\begin{proof}
Consider the case $M$ is infinite. Let $\mathcal{O}$ be a base of topology for $M$
and $\mathcal{A}$ be $\kappa^+$-complete where $\kappa=|\mathcal{O}|+\aleph_0$.
By continuity, the inverse image of $[r,s)$ under $f$ is a countable intersection of open subsets of $M^n$, say $\bigcap_kU_k$.
Each $U_k$ is a union of at most $\kappa$ sets of the form $A_1\times\cdots\times A_n$
where $A_j\in\mathcal{O}$ for $j=1,...,n$. It is now clear that $u^{-1}([r,s))\in\mathcal{A}$.
The second part is similar.
\end{proof}

Assume $M$ is $\mathcal A$-meanable. For measurable $a,b:I\rightarrow M$
let $a\sim b$ if $\int d(a_i,b_i)d\mu=0$. The equivalence class of $a=(a_i)$ is denoted by $[a_i]$.
Let $M^{\mu}$ be the set of equivalence classes of measurable maps $a:I\rightarrow M$.
We set a metric on $M^{\mu}$ by $$d([a_i],[b_i])=\int d(a_i,b_i)d\mu.$$
Also, for $c,F,R\in L$ (assuming $F,R$ are unary for simplicity) we define:
$$c^{M^{\mu}}=[c^M]$$
$$F^{M^\mu}([a_i])=[F^M(a_i)]$$ $$R^{M^\mu}([a_i])=\int R^M(a_i)d\mu.$$
Then, $M^{\mu}$ is a (possibly incomplete) $L$-structure. If $\mu$ is an ultracharge, it coincides with
the maximal powermean defined in \S \ref{Affine compactness}.
To prove the powermean theorem for such a general $M^\mu$, we need a selection theorem.

A \emph{multifunction}\index{multifunction} $G:I\rightarrow M$ is a map which assigns to each $i$
a nonempty $G(i)\subseteq M$. It is \emph{closed-valued} if $G(i)$ is closed for each $i$.
It is $\mathcal A$-measurable if for every open $U\subseteq M$, the set
$$G^{-1}(U)=\{i\in I|\ G(i)\cap U\neq\emptyset\}$$ is $\mathcal A$-measurable.
A \emph{selection}\index{selection} for $G$ is a function $g:I\rightarrow M$ such that $g(i)\in G(i)$ for every $i$.

\begin{theorem} \emph{(Kuratowski, Ryll-Nardzewski)}
Let $M$ be a separable completely metrizable topological space
and $G:I\rightarrow M$ be an $\mathcal A$-measurable closed-valued multifunction.
If $\mathcal A$ is countably complete, then $G$ admits an $\mathcal{A}$-measurable selection.
\end{theorem}

A proof of the above theorem can be found in \cite{Srivastava}. 
It is however not hard to see that the same proof works for any complete $M$
if $\mathsf{dc}(M)\leqslant\kappa$ and $\mathcal A$ is $\kappa^+$-complete.
Also, existence of measurable selections for finite $M$ is obvious.
So, every $\mathcal A$-meanable $L$-structure $M$ has measurable selections.

Let $(I,\mathcal A,\mu)$ be a charge space and $M$ an $\mathcal A$-meanable $L$-structure.
Let $u:M^2\rightarrow\Rn^+$ be a $\lambda$-Lipschitz function.
Let $B_r(y)$ be the open ball of radius $r$ around $y$. Then, it is easy to verify that
$$\inf_{t\in B_r(y)}u(x,t)=\inf_z\big[u(x,z)+\lambda d(z,B_r(y))\big].$$
Let $a:I\rightarrow M$ be measurable.
Fix $0<\epsilon<1$ and assume the set
$$G(i)=\{t\in M|\ \ u(a_i,t)<\epsilon\}$$
is nonempty. Then, for each $r>0$,
$$d(y,G(i))<r \ \ \Longleftrightarrow\ \ \inf_{t\in B_r(y)}u(a_i,t)<\epsilon$$
We deduce that for each $y$, the map $i\mapsto d(y,\overline{G(i)})$ is measurable.
Let $D\subseteq M$ be a dense set with $|D|\leqslant\kappa$ and $U\subseteq M$ be open.
For each $y\in D\cap U$ choose $r_y$ such that
$$\frac{d(y, U^c)}{2}<r_y< d(y,U^c).$$ Then, $U=\bigcup_{y\in D\cap U} B_{r_y}(y)$ and
$$\{i|\ \overline{G(i)}\cap U\neq\emptyset\}=
\bigcup_y\{i|\ \overline{G(i)}\cap B_{r_y}(y)\neq\emptyset\}
=\bigcup_y\{i|\ d(y,\overline{G(i)})<r_y\}\in\mathcal A.$$
This shows that the multifunction $i\mapsto \overline{G(i)}$ is measurable.
Applying this for the function $$u(x,y)=[\sup_y\phi^M(x,y)]-\phi^M(x,y),$$ we conclude
that there is a measurable $b:I\rightarrow M$ such that
$$\sup_y\phi^M(a_i,y)-\epsilon\leqslant\phi^M(a_i,b_i)\hspace{12mm}\forall i\in I.$$

\begin{theorem} \label{powermean} \emph{(Powermean)}
Let $(I,\mathcal A ,\mu)$ be a charge space and $M$ be an $\mathcal A $-meanable structure.
Then, for each $L$-formula $\phi(\x)$ and $[a^1_i],...,[a^n_i]\in M^{\mu}$,
$$\phi^{M^{\mu}}([a^1_i],...,[a^n_i])=\int\phi^M(a^1_i,...,a^n_i)d\mu.$$
\end{theorem}
\begin{proof}
Clearly, the claim holds for atomic formulas.
Also, if it holds for $\phi,\psi$, it holds for $r\phi+s\psi$ too.
Assume the claim is proved for $\phi(\x,y)$. For simplicity assume $|\x|=1$.
Let $[a_i]\in M^{\mu}$ and $0<\epsilon<1$. As $M$ is complete, by the above argument,
there is a measurable $b$ such that
$$\sup_y\phi^M(a_i,y)-\epsilon\leqslant\phi^M(a_i,b_i)\hspace{14mm}\forall i\in I.$$
So,
$$\int\sup_y\phi^M(a_i,y)d\mu-\epsilon\leqslant\int\phi^M(a_i,b_i)d\mu=
\phi^{M^{\mu}}([a_i],[b_i])\leqslant\sup_y\phi^{M^{\mu}}([a_i],y)$$
and hence $$\int\sup_y\phi^M(a_i,y)d\mu\leqslant\sup_y\phi^{M^{\mu}}([a_i],y).$$
The inverse inequality is obvious. So, the claim holds for $\sup_y\phi(x,y)$ too.
\end{proof}

\begin{corollary} \label{diagonal}
The diagonal map $a\mapsto[a]$ is an elementary embedding of $M$ into $M^\mu$.
\end{corollary}

It is well-known that if every $M_i$ is complete and $\mathcal{F}$ is an ultrafilter,
then $\prod_{\mathcal F}M_i$ is complete. In the powermean case we have the following.

\begin{proposition} \label{metric completeness}
Let $(I,\mathcal A ,\mu)$ be a probability measure space.
If $M$ is $\mathcal A $-meanable, then $M^{\mu}$ is complete.
\end{proposition}
\begin{proof}
Let $a^1,a^2,...$ be a Cauchy sequence in $M^\mu$ where $a^k=[a^k_i]$.
Without loss of generality assume that for each $k$,\ \ $d(a^k,a^{k+1})<2^{-2k}$.
Let $$A_k=\{i: d(a^k_i,a^{k+1}_i)\geqslant 2^{-k}\}.$$
We then have that $\mu(A_k)< 2^{-k}$ since otherwise
$$d(a^k, a^{k+1})\geqslant\int_{A_k}d(a_i^k, a_i^{k+1})d\mu\geqslant 2^{-2k}.$$
Moreover, $B_n=\bigcup_{k=n}^\infty A_k$ is descending and by computation
$\mu (B_n)\leqslant2^{-n+1}$.
Therefore, $\mu(\bigcap_n B_n)=0$ and for each fixed
$i\not\in\bigcap_n B_n$, the sequence $a_i^k$ is Cauchy.
We set $a_i=\lim_k a^k_i$ if $i$ is outside $\bigcap_n B_n$ and arbitrary otherwise.
Then, $(a_i)$ is measurable and by Proposition \ref{convergence}
$$\int d(a^n_i,a_i)d\mu\leqslant\int\sum_{k=n}^\infty
d(a^k_i, a^{k+1}_i)d\mu=\sum_{k=n}^\infty d(a^k, a^{k+1})\leqslant 2^{-n+1}.$$
Hence, $a^n$ tends to $[a_i]\in M^\mu$.
\end{proof}

Turning back to Proposition \ref{mid}, one shows similarly that if $M$ is separable and $\mu$ is the Lebesgue measure on $[0,1]$,
then $M^\mu$ is a geodesic metric space. So, every theory in a countable language has a model which is metrically geodesic.

There are other families of functions for which a similar powermean theorem holds.
For example, assuming $\mathcal A$ is countably complete and $M$ is arbitrary.
Let $N$ be the set of equivalence classes of measurable maps $a:I\rightarrow M$ for which $\Phi(a)$
holds where $\Phi(a)$ is one of the following properties:
(i) $a$ has a finite range (ii) $a$ has a countable range (iii) the range of $a$ is contained in a compact set
(iv) the range of $a$ is contained in a separable set.
The first case has the advantage that the cardinality is controlled.
We use this case to find large groups of automorphisms on separable models.
\vspace{5mm}

\noindent{\bf Discrete powermean and automorphisms}\index{discrete powermean}\\
Let $(I,\mathcal A,\mu)$ be charge space and $M$ be an arbitrary $L$-structure.
A measurable map $a:I\rightarrow M$ with finite range is called \emph{simple}.
As before, simple $a,b$ are identified if $$\int d(a_i,b_i)d\mu=0.$$
Let $M^{\mu{\mathrm s}}$ be the set of equivalence classes of simple measurable maps.
Then $M^{\mu{\mathrm s}}$ is an (incomplete) $L$-structure as in the measurable powermean case
and the ultramean theorem holds for it. Note also that if $|\mathcal A|\leqslant\kappa$ and $\mathsf{dc}(M)\leqslant\lambda$
(where $\kappa,\lambda$ are infinite), then $\mathsf{dc}(M^{\mu{\mathrm s}})\leqslant\kappa\lambda$.

It is well-known that the only automorphism of an ultrafilter is the identity map.
Despite the rigidity of ultrafilters, ultracharges have usually many
automorphisms which raise correspondingly automorphisms of the powermeans.

\begin{proposition}\label{automorphism}
Let $(I,\mathcal A ,\mu)$ be a charge space and $|M|\geqslant2$.
Let $G$ be a group of measurable measure preserving bijections $g:I\rightarrow I$
such that for each $g\neq id$ there is $A$ with $\mu(A\triangle gA)>0$. Then, $G$ is embedded
in the automorphism group of $M^{\mu{\mathrm s}}$ (or its completion).
Similarly property holds for $M^{\mu}$ if $M$ is $\mathcal A $-meanable.
\end{proposition}
\begin{proof}
Fix $g\in G$. By the change of variables theorem, for each $\phi(\x)$
(assume $|\x|=1$ for simplicity) and $[a_i]\in M^{\mu{\mathrm s}}$ one has that
$$\phi^{M^{\mu{\mathrm s}}}([a_{g(i)}])=\int\phi^M(a_{g(i)})d\mu=\int\phi^M(a_i)d\mu=\phi^{M^{\mu{\mathrm s}}}([a_{i}]).$$
So, the map $\bar{g}([a_i])=[a_{g(i)}]$ is an automorphism of $M^{\mu{\mathrm s}}$.
It is easy to verify that $\overline{gh}=\bar{g}\bar{h}$.
Hence, $g\mapsto\bar{g}$ is a homomorphism from $G$ to $aut(M^{\mu{\mathrm s}})$.
We show that it is injective. For each $id\neq g\in G$ take $A\subseteq I$ such that $\mu(A\triangle g(A))>0$.
Clearly, $0<\mu(A)<1$. Let $a,b\in M$ be distinct and define $a_i=a$ if $i\in A$ and $a_i=b$ if $i\not\in A$.
Then, $\bar{g}([a_i])\neq[a_i]$. Note that every automorphism of $M^{\mu{\mathrm s}}$ extends uniquely to it completion.
The second part is proved similarly.
\end{proof}

Indeed, $G$ is embedded in $aut(M^{\mu{\mathrm s}},M)$ (they fix every $a\in M$).

\begin{example}
\em{(1) Recall that a left-invariant mean on a group $G$ is a
finitely additive probability measure (i.e. an ultracharge) $\mu:P(G)\rightarrow[0,1]$
such that $\mu(A)=\mu(gA)$ for every $A\subseteq G$ and $g\in G$.
If there is such a $\mu$, $G$ is called \emph{amenable}.
Let $G$ be amenable and $\mu$ a left-invariant mean on it.
Suppose that $G$ is residually finite, i.e. the intersection of subgroups of finite index is trivial.
Identify every $g\in G$ with the corresponding translation $x\mapsto gx$ on $G$.
Then, $G$ satisfies the conditions of Proposition \ref{automorphism}. Hence, it is embedded in $aut(M^{\mu{\mathrm s}})$.

(2) Let $G$ be a compact Hausdorff group and $\mu$ its Haar measure.
Note that every neighborhood of the identity has positive measure.
So, $G$ satisfies the conditions of the proposition if $g\in G$
is identified with the translation $x\mapsto gx$.
Hence, $G$ is (algebraically) embedded in $aut(M^{\mu{\mathrm s}})$.

(3) Let $(G,\mu)$ be as in (2). Every continuous automorphism $f$ of $G$ is measure preserving.
Also, if $f\neq id$, then there is a nonempty open $U$ such that $f(U)\cap U=\emptyset$.
So, the group of continuous automorphisms of $G$
is embedded in $aut(M^{\mu{\mathrm s}})$.}
\end{example}

It is well-known that every infinite first order model $M$ in a
countable signature has an elementary extension $N$ of the same
cardinality such that $aut(\mathbb Q,<)$ is embedded in $aut(N)$.

\begin{proposition}\label{permutation}
Assume $2\leqslant\mathsf{dc}(M)\leqslant\kappa$ and $\aleph_0\leqslant\kappa=|I|$.
Then there exists $M\preccurlyeq N$ such that $\mathsf{dc}(N)\leqslant\kappa$ and $per(I)$
(the group of permutations of $I$) is embedded in $aut(N)$.
\end{proposition}
\begin{proof}
Let $G=\mathbb Z_2^I$ be endowed with the product topology.
For finite $\tau\subseteq I$, let $G_\tau$ be the subgroup consisting of those $(x_i)\in G$
for which $x_i=0$ when $i\in\tau$.
Then, $\{G_{\tau}\}_{\tau\in J}$ is a local base consisting of open subgroups of
finite index where $J=S_\omega(I)$.
Let $\mathcal A $ be the Boolean algebra generated by cosets of these
subgroups and $\mu$ be the restriction of the Haar measure to $\mathcal A$.
Then $\mathsf{dc}(M^{\mu{\mathrm s}})\leqslant\kappa$ and every permutation $f$ of $I$ induces a continuous
(hence measure preserving) automorphism $\bar{f}$ of $G$. Clearly, if $f\neq id$, then $\bar{f}\neq id$
and the condition mentioned in Proposition \ref{automorphism} holds.
Therefore, $per(I)$ is embedded in the automorphism group of $M^{\mu{\mathrm s}}$
and hence in the automorphism group of its completion $N$. Obviously, $\mathsf{dc}(N)\leqslant\kappa$.
\end{proof}
\bigskip

\noindent{\bf Continuous powermean}\index{continuous powermean}\\
Assume $I$ is a paracompact zero-dimensional space and $\mu$ is a probability charge on the Borel algebra of $I$.
Note that a locally compact Hausdorff space is zero-dimensional if and only if it is totally disconnected.
Let $M$ be an $L$-structure. A multifunction $f:I\rightarrow M$ is called \emph{lower semi-continuous}
if for each open $U\subseteq M$ the set $$G^{-1}(U)=\{i\in I|\ G(i)\cap U\neq\emptyset\}$$ is open.

\begin{theorem} \emph{(\cite{Parthasarathy-selection} Th. 1.2)} \label{parthasarathy-selection}
Let $I$ be a paracompact zero-dimensional space and $M$ a complete metric space.
Then, every lower semi-continuous closed valued multifunction $f:I\rightarrow M$
has a continuous selection\index{continuous selection}.
\end{theorem}

Let $u:M^2\rightarrow\Rn^+$ be $\lambda$-Lipschitz and $a:I\rightarrow M$ be continuous.
Fix $0<\epsilon<1$ and for $i\in I$ assume
$$G(i)=\{t\in M|\ \ u(a_i,t)<\epsilon\}\neq\emptyset.$$ Then, for each $r>0$ and $y$,
$$d(y,G(i))<r \ \ \Longleftrightarrow\ \ \inf_{t\in B_r(y)}u(a_i,t)<\epsilon$$
We deduce that for each $y$, the set $\{i:\ d(y,\overline{G(i)}<r\})$ is open.
Let $U\subseteq M$ be open and for each $y\in U$ choose $r_y$ such that
$$\frac{d(y, U^c)}{2}<r_y< d(y,U^c).$$ Then, $U=\bigcup_{y\in U} B_{r_y}(y)$ and
$$\{i|\ \overline{G(i)}\cap U\neq\emptyset\}=
\bigcup_{y\in U}\{i|\ \overline{G(i)}\cap B_{r_y}(y)\neq\emptyset\}
=\bigcup_{y\in U}\{i|\ d(y,\overline{G(i)})<r_y\}$$
which is open. This shows that the multifunction $i\mapsto \overline{G(i)}$ is lower semi-continuous.
Applying this for the function $$u(x,y)=[\sup_y\phi^M(x,y)]-\phi^M(x,y),$$ we conclude
that there is a continuous $b:I\rightarrow M$ such that
$$\sup_y\phi^M(a_i,y)-\epsilon\leqslant\phi^M(a_i,b_i)\hspace{12mm}\forall i\in I.$$

Note that for every $\phi(\x)$ and continuous $a^1,...,a^n:I\rightarrow M$,
the map $i\mapsto\phi(a^1_i,...,a^n_i)$ is continuous (hence integrable).
For continuous $a,b:I\rightarrow M$ let $a\sim b$ if $$\int d(a_i,b_i)d\mu=0.$$
The equivalence class of $a=(a_i)$ is denoted by $[a_i]$.
Let $M^{\mu{\mathrm c}}$ be the set of equivalence classes of all continuous maps $a:I\rightarrow M$.
The metric on $M^{\mu{\mathrm c}}$ is defined by $$d(a,b)=\int d(a_i,b_i)d\mu.$$
Also, for $e,F,R\in L$ (say unary) and $a\in M^{\mu{\mathrm c}}$ define
$$e^{M^{\mu{\mathrm c}}}=[e^M], \hspace{8mm} F^{M^{\mu{\mathrm c}}}([a_i])=[F^M(a_i)], \hspace{8mm} R^{M^{\mu{\mathrm c}}}([a_i])=\int R^M(a_i)d\mu.$$
Then, $M^{\mu{\mathrm c}}$ is an (incomplete) $L$-structure. The following is then proved as before.

\begin{theorem} \label{powermean-continuous} \emph{(Continuous powermean)}
Let $I$ be a paracompact zero-dimensional space.
Let $\mu$ be a charge on the Borel algebra of $I$ and $M$ be an $L$-structure.
Then, for each $L$-formula $\phi(\x)$ and $a^1,...,a^n\in M^{\mu{\mathrm c}}$
$$\phi^{M^{\mu{\mathrm c}}}(a^1,...,a^n)=\int\phi^M(a^1_i,...,a^n_i)d\mu.$$
\end{theorem}

\subsection{Operations on ultracharges}\label{Operations on ultracharges}
{\bf Product and inverse limit}
\vspace{2mm}

Let $(I,\mathcal A ,\mu)$ be a charge space and $f:I\rightarrow J$ a map.
Define a charge $\nu$ on $J$ by setting
$$\mathcal B=\{X\subseteq J:\ f^{-1}(X)\in\mathcal A\}$$
$$\hspace{10mm} \nu(X)=\mu(f^{-1}(X))\hspace{12mm}\mbox{for}\ \ X\in\mathcal B.$$
In this case, one writes $\nu=f(\mu)$.
If $\nu=f(\mu)$ for some $f$, one writes $\nu\leqslant\mu$.

Let $J$ be an infinite index set and for each $r\in J$, $(I_r,\mu_r)$ be an ultracharge space.
Let $\mathcal{I}=\prod_{r\in J}I_r$. A subset $\prod_{r\in J}X_r\subseteq \mathcal{I}$
is called a \emph{cylinder} if $X_r=I_r$ for all except finitely many $r$.
Let $\mathcal C$ be the Boolean algebra generated by cylinders.
Equivalently, $\mathcal C$ is generated by the sets of the form $\pi^{-1}_r(X)$
where $\pi_r:\mathcal I\rightarrow I_r$ is the projection map and $X\subseteq I_r$.
Define a charge $\mu$ by first setting
$$\mu(\prod_{r\in J}X_r)=\prod_{r\in J}\mu_r(X_r)$$
and then extending it to $\mathcal C$ in the natural way.
We call $\mu$ the \emph{cylinder charge}.
It is clear that $\mu_r\leqslant\mu$ for all $r$.
Extending $\mu$ to an ultracharge (denoted again by $\mu$), we obtain the following.

\begin{lemma} \label{charge product}
For each $r\in J$, let $\mu_r$ be an ultracharge on a set $I_r$.
Then there is an ultracharge $\mu$ on $\mathcal I=\prod_{r\in J}I_r$ such that
$\mu_r\leqslant\mu$ for every $r\in J$.
\end{lemma}

Inverse limit of measures is a well-studied notion in the literature.
Here, we deal with a similar case, the inverse limit of ultracharges.
We consider a special case. Let $(J,<)$ be a affinely ordered
set and $\mu_r$ be an ultracharge on $I_r$ for each $r\in J$.
Assume for each $r\leqslant s$ there is a surjective map $f_{rs}:I_s\rightarrow I_r$
such that $\mu_r=f(\mu_s)$. Also assume that $f_{rr}=id$ and
$f_{rt}=f_{rs}\circ f_{st}$ whenever $r\leqslant s\leqslant t$.
Let $$\mathbf{I}=\{(i_r)_{r\in J}\in\mathcal I\ :\ \ \ f_{rs}(i_s)=i_r\ \ \ \ \forall\ r\leqslant s\}.$$

\begin{lemma} \label{inverse limit}
There exists a charge $(\mathcal A ,\mu)$ on $\mathbf I$
such that $\mu_r\leqslant\mu$ for all $r\in J$.
\end{lemma}
\begin{proof}
Let $\pi_r: \mathbf{I}\rightarrow I_r$ be the projection map.
Let $\mathcal A_r$ be the subalgebra of $P(\mathbf I)$ consisting of sets of the form
$\pi_r^{-1}(X)$ where $X\subseteq I_r$ and define $$\nu_r(\pi_r^{-1}(X))=\mu_r(X).$$
Then, $\nu_r$ is a charge on $\mathcal A_r$.
By the assumptions, for $r\leqslant s$ and $X\subseteq I_s$
one has that $$\pi_s^{-1}(X)=\pi_r^{-1}(f_{rs}^{-1}(X)).$$
Hence, $\mathcal A_r\subseteq\mathcal A_s$ and $\nu_r=\nu_s|_{\mathcal A_r}$.
Let $\mathcal A=\cup_{r\in J}\mathcal A_r$ and $\mu=\cup_{r\in J}\nu_r$.
Then, $\mu_r=\pi_r(\mu)$ and hence $\mu_r\leqslant\mu$ for each $r$.
\end{proof}
\bigskip

\noindent{\bf Fubini's product}
\vspace{2mm}

Let $(I,\mathcal{A},\mu)$ be a charge space and $(J,\mathcal{B},\nu)$ be an ultracharge space.
For $A\subseteq I\times J$ and $j\in J$ let $A_j=\{i:\ (i,j)\in A\}$ and define
$$\mathcal C=\{A\subseteq I\times J:\ \ \ \forall j\  A_j\in\mathcal A \}.$$
Then, $\mathcal C$ is a Boolean algebra of subsets of $I\times J$ and it is $\kappa$-complete
(resp. the whole power set) if $\mathcal A $ is so.
Since $\nu$ is an ultracharge, we may define a probability charge
on $(I\times J,\ \mathcal C)$ by setting
$$\wp(A)=\int\mu(A_j)d\nu\ \hspace{10mm} \forall A\in\mathcal C.$$
We denote $\wp$ by $\mu\otimes\nu$. For ultrafilters $\mathcal D$ on $I$ and $\mathcal F$ on $J$,
$\mu_{\mathcal D}\otimes\mu_{\mathcal F}$ corresponds to their Fubini product defined by
$$\mathcal D\times\mathcal F=\{A\subseteq I\times J:\ \{i:\ \{j: (i,j)\in A\}\in\mathcal F\}\in\mathcal D\}.$$
Note that $\mu,\nu\leqslant\mu\otimes\nu$ via the projection maps.
Then, a one sided Fubini theorem holds.
\bigskip

\begin{lemma}\label{Fubini}
For every bounded $\mathcal{C}$-measurable $f:I\times J\rightarrow\Rn$,
$$\int f(i,j)\ d(\mu\otimes\nu)=\iint f(i,j)\ d\mu d\nu.$$
\end{lemma}
\begin{proof} By definition, the claim holds for every $\chi_A$ where $A\in\mathcal C$.
So, it holds for simple functions too.
Let $f$ be as above with range contained in the interval $(-u,u)$.
Let $$\hspace{10mm} f_n(i,j)=\sum_{k=-n}^{n}\frac{k}{n}u\cdot\chi_{A_{nk}}(i,j) \hspace{14mm}
\mbox{where}\ \ A_{nk}=f^{-1}[\frac{k}{n}u, \frac{k+1}{n}u)\in\mathcal C.$$
Then, $f_n$ tends to $f$ uniformly.
Also, for each fixed $j$, $f_n(i,j)$ tends to $f(i,j)$ uniformly and
$$\Big|\int f_n(i,j)d\mu-\int f(i,j)d\mu\Big|\leqslant
\int\big|f_n(i,j)-f(i,j)\big|d\mu\leqslant\frac{u}{n}$$
which shows that $\int f_n(i,j)d\mu$ tends to $\int f(i,j)d\mu$ uniformly on $J$.
So, by Proposition \ref{convergence}
$$\int f(i,j)d\wp=\lim_n\int f_nd\wp=\lim_n\iint f_n(i,j)d\mu d\nu
=\iint\lim_n f_n(i,j)d\mu d\nu=\iint fd\mu d\nu.$$
\end{proof}

A consequence of the lemma is that for any ultracharges $\mu,\nu,\wp$
$$\mu\otimes(\nu\otimes\wp)=(\mu\otimes\nu)\otimes\wp.$$

\subsection{The isomorphism theorem}\label{The isomorphism theorem}
The existence of $\kappa$-saturated models is easily proved by iterated realizations of types.
It is however important to know whether there exists an ultracharge $\mu$ for which $\prod_{\mu}M_i$ is $\kappa$-saturated.
We answer to this question in the powermean case.

\begin{proposition}\label{compose}
Let $(I,\mathcal A,\mu)$ be a charge and $(J,\mathcal B,\nu)$ be an ultracharge on $J$.
Let $(I\times J,\mathcal C,\mu\otimes\nu)$ be their Fubini product. Assume $M$ is $\mathcal A$-meanable.
Then, $M$ is $\mathcal C$-meanable and $M^{\mu\otimes\nu}\simeq(M^{\mu})^{\nu}$.
\end{proposition}
\begin{proof}
Let $\wp=\mu\otimes\nu$ and $[a_{ij}]\in M^{\wp}$. By definition, for each fixed $j$,\ \
$a^j=(a_{ij})_{i\in I}$ is a measurable tuple and its class, which we denote by
$[a^j]_{\mu}$, belongs to $M^{\mu}$.
It is not hard to see that the map $a_\mu\mapsto[[a^j]_{\mu}]_{\nu}$ is
a well-defined bijection. We check that it preserves all formulas.
Let $\phi(\x)$ be a formula (assume $|x|=1$). Then by Lemma \ref{Fubini}
$$\phi^{M^{\wp}}(a_\wp)=\int\phi^{M}(a_{ij})d\wp=\int\int\phi^M(a_{ij})d\mu d\nu$$
$$=\int\phi^{M^{\mu}}([a^j]_{\mu})d\nu=\phi^{(M^\mu)^\nu}([[a^j]_{\mu}]_{\nu}).$$
\end{proof}

If $M$ is separable and $\mu,\nu$ are measures, one may use the classical
Tonelli theorem for the product measure $\mu\times\nu$ to show that
$(M^\mu)^\nu\simeq M^{\mu\times\nu}\simeq (M^\nu)^\mu$.

Recall that for every positive linear function
$\Lambda:\ell^\infty(X)\rightarrow\Rn$ with $f(1)=1$ there is a unique
ultracharge $\mu$ such that $\Lambda(f)=\int fd\mu$ for every $f$
(see \cite{Rao}, Th. 4.7.4).

\begin{proposition}\label{sat3}
For each $M$, there is an ultracharge $\mu$ such that $M^\mu$
realizes all types in every $K_n(M)$. 
\end{proposition}
\begin{proof}
Let $M$ be a $L$-structure. We only need to realize the types in $K_1(M)$ in some $M^\mu$.
The types in $K_n(M)$ are then automatically realized in it.
By Lemma \ref{charge product}, there is a set $I$ and an ultracharge $\wp$ on $I$
such that for every ultracharge $\mu$ on $M$ one has that $\mu\leqslant\wp$.
Given $p(x)\in K_1(M)$, there is an ultracharge $\mu$
on $M$ such that for every $\phi(x)$ $$p(\phi)=\int\phi^M(x)d\mu.$$
Let $f:I\rightarrow M$ be such that $f(\wp)=\mu$.
Let $a=[a_i]_\wp$ where $a_i=f(i)$.
Then, $$p(\phi)=\int_M\phi^M(x)d\mu=\int_I\phi^M(a_i)d\wp=\phi^{M^\wp}(a).$$
\end{proof}

\begin{theorem} \label{saturated}
Assume $\mathsf{dc}(M)\leqslant\kappa$. Then there is a charge space
$(I,\mathcal A,\wp)$ such that $M$ is $\mathcal A $-meanable and $M^{\wp}$ is $\kappa^+$-saturated.
\end{theorem}
\begin{proof}
First assume $\kappa=\aleph_0$. By a repeated use of Proposition \ref{sat3},
we obtain a countable chain
$$M\preccurlyeq M^{\mu_1}\preccurlyeq (M^{\mu_1})^{\mu_2}\preccurlyeq\cdots$$
where $\mu_n$ is an ultracharge on a set $I_n$.
In the light of Proposition \ref{compose}, we may rewrite it as
$$M\preccurlyeq M^{\wp_1}\preccurlyeq M^{\wp_2}\preccurlyeq\cdots$$
where $\wp_n=\mu_1\otimes\cdots\otimes\mu_n$ is an ultracharge on $J_n=I_1\times\cdots\times I_n$
and $M^{\wp_{n+1}}$ realizes all types in $K_1(M^{\wp_n})$.
It is clear that $\{(J_n,\wp_n), f_{mn}\}$, where $f_{mn}:J_n\rightarrow J_m$ is the
projection map, is an inverse system of ultracharges.
The inverse limit of this system is a charge space which can be completed
to an ultracharge space, say $(J_\omega,\wp_\omega)$. Then, $\wp_n\leqslant\wp_\omega$ and
$$M\preccurlyeq M^{\wp_1}\preccurlyeq M^{\wp_2}\preccurlyeq\cdots\preccurlyeq M^{\wp_\omega}.$$
Iterating the argument, we obtain an inverse system
$\{(J_\alpha, \wp_\alpha), f_{\alpha\beta}\}_{\alpha<\beta<\omega_1}$
of ultracharges and a chain
$$M\preccurlyeq M^{\wp_1}\preccurlyeq\cdots\preccurlyeq
M^{\wp_{\alpha}}\preccurlyeq\cdots\ \hspace{14mm} \alpha\in\omega_1$$
such that every $M^{\wp_{\alpha+1}}$ realizes all types in $K_1(M^{\wp_\alpha})$.

Let $(\mathbf J,\mathcal A,\wp)$ be the inverse limit of
$\{(J_\alpha,\wp_\alpha), f_{\alpha\beta}\}_{\alpha<\beta<\omega_1}$
given by Lemma \ref{inverse limit}.
By regularity of $\aleph_1$, $\mathcal A$ is $\aleph_1$-complete (since every $\wp_\alpha$ is an ultracharge).
It is clear that $N=\bigcup_{\alpha<\omega_1} M^{\wp_\alpha}$ is $\aleph_1$-saturated and that
$N\preccurlyeq M^{\wp}$. We will show that $N=M^\wp$.

Recall that the embedding $M^{\wp_\alpha}\preccurlyeq M^{\wp_\beta}$
takes place via the map $[a_i]\mapsto[b_j]$ where $b_j=a_{f_{\beta\alpha}(j)}$ for every $j\in J_\beta$.
In this way, we identify $[b_j]$ with $[a_i]$.
An element of $M^\wp$ is of the form $[a_{\mathbf r}]$ where ${\mathbf r}=(r_\gamma)\in\mathbf{J}$
and $f_{\beta\gamma}(r_\gamma)=r_{\beta}$.
We show that for every such $[a_{\mathbf r}]$ there exists $\alpha<\omega_1$
such that $a_{\mathbf r}$ does not depend on $r_\gamma$ when $\alpha\leqslant\gamma$.
In other words, for each ${\mathbf r},{\mathbf s}\in\mathbf{J}$, if $r_\gamma=s_\gamma$
for all $\gamma\leqslant\alpha$ (or equivalently if $r_\alpha=s_\alpha$),
then $a_{\mathbf r}=a_{\mathbf s}$. This clearly implies that $[a_{\mathbf r}]\in M^{\wp_\alpha}$.

Fix a countable base $\{U_k\}_{k\in\omega}$ for $M$.
For every $\alpha$, let $\mathcal A_\alpha$ be the $\sigma$-algebra of subsets of $\mathbf{J}$ consisting
of sets of the form $\pi_\alpha^{-1}(X)$ where $X\subseteq J_\alpha$
(recall that $\pi_\alpha:\mathbf{J}\rightarrow J_\alpha$ is the projection map).
For $\alpha\leqslant\beta$ and $X\subseteq J_\beta$,
one has that $$\pi_\beta^{-1}(X)=\pi_\alpha^{-1}(f_{\alpha\beta}^{-1}(X)).$$
So, $\mathcal{C}_\alpha\subseteq\mathcal{C}_\beta$.
By regularity of $\omega_1$, we must have that $\mathcal A=\bigcup_{\alpha<\omega_1}\mathcal A_\alpha$.
In particular, there exists $\alpha<\omega_1$ such that every $a^{-1}(U_k)$
belongs to $\mathcal A_\alpha$.
Suppose that $a_{\mathbf r}\neq a_{\mathbf s}$.
Then, ${\mathbf r}\in a^{-1}(U)$ and ${\bf s}\in a^{-1}(V)$ for some disjoint basic open subsets $U$, $V$ of $M$.
Hence, $$r_\alpha=\pi_\alpha({\mathbf r})\neq \pi_\alpha({\mathbf s})=s_\alpha.$$
We conclude that if $r_\alpha=s_\alpha$, then $a_{\mathbf r}=a_{\mathbf s}$.
For arbitrary $\kappa$, one must use an elementary chain of length $\kappa^+$.
The algebra $\mathcal A$ is then $\kappa^+$-complete and $M$ is $\mathcal A $-meanable.
\end{proof}
\vspace{2mm}

A \emph{partial isomorphism}\index{partial isomorphism} between $M$ and $N$ is a relation
$\sim$ between tuples $\a\in M$ and $\b\in N$ of the same (finite) length such that:

(i) $\emptyset\sim\emptyset$

(ii) if $\a\sim\b$, then $\theta^M(\a)=\theta^N(\b)$ for every atomic formula $\theta$

(iii) if $\a\sim\b$, then for every $c\in M$ there exists $e\in N$ such that $\a c\sim\b e$

(iv) if $\a\sim\b$, then for every $e\in N$ there exists $c\in M$ such that $\a c\sim\b e$.
\vspace{1mm}

\begin{lemma}
If $M$ and $N$ are partially isomorphic then $M\equiv_{\mathrm{CL}} N$.
\end{lemma}
\begin{proof} We show by induction on the complexity of the CL-formula $\phi(\x)$ in $L$
that whenever $\a\sim\b$, one has that $\phi^M(\a)=\phi^N(\b)$.
The atomic and connective cases $+,r., \wedge,\vee$ are obvious.
Assume the claim is proved for $\phi(\x,y)$. Let $\sup_y\phi^M(\a,y)=r$ and $\a\sim\b$.
Given $\epsilon>0$, take $c\in M$ such that $r-\epsilon<\phi^M(\a,c)$.
Take $e\in N$ such that $\a c\sim\b e$.
By the induction hypothesis, $\phi^M(\a,c)=\phi^N(\b,e)$.
Since $\epsilon$ is arbitrary, one has that $r\leqslant\sup_y\phi^N(\b,y)$.
Similarly, one has that $\sup_y\phi^N(\b,y)\leqslant\sup_y\phi^M(\a,y)$
and hence they are equal. We conclude that $M\equiv_{\mathrm{CL}} N$.
\end{proof}

The isomorphism theorem in first order logic (as well as continuous logic)
states that for $L$-structures $M$ and $N$, if $M\equiv_{\mathrm{CL}} N$, then there is an
ultrafilter $\mathcal F$ such that $M^{\mathcal F}\simeq N^{\mathcal F}$ (see \cite{BBHU} Theorem 5.7).
To prove a similar result in AL we need the following.

\begin{proposition} \label{10}
Assume $M\equiv N$ and they are $\aleph_0$-saturated. Then $M\equiv_{\mathrm{CL}} N$.
\end{proposition}
\begin{proof}
For $\a\in M$ and $\b\in N$ set $\a\sim\b$ if $tp^M(\a)=tp^N(\b)$.
We show that this is a partial isomorphism.
Obviously, $\emptyset\sim\emptyset$. Let $\a\sim\b$ and $c\in M$. Let
$$p(x)=\{\phi(\b,x)=r|\ \ \phi^M(\a,c)=r\}.$$
$p(x)$ is closed under linear combinations.
Assume $\phi^M(\a,c)=r$. Then $N\vDash\inf_x\phi(\b,x)\leqslant r$.
So, there exists $e_1\in N$ such that $\phi^N(\b,e_1)\leqslant r$.
Similarly, there exists $e_2$ such that $r\leqslant\phi^N(\b,e_2)$.
By affine compactness (or by Proposition \ref{connected range}), $\phi(\b,x)=r$ is satisfied in $N$.
So, $p(x)$ is satisfied in $N$. Let $e\in N$ realize $p(x)$. Then, $\a c\sim\b e$.
This is the forth property and the back property is verified similarly.
\end{proof}

For every affinely complete theory $T$ let $T^{\mathrm{sat}}$ be the common theory
(in the CL sense) of affinely $\aleph_0$-saturated models of $T$.
Proposition \ref{10} states that $T^{\mathrm{sat}}$ is CL-complete.

\begin{theorem} \label{isomorphism} \emph{(Isomorphism)}\index{isomorphism theorem}
Assume $M\equiv N$. Then, there are charge spaces $(I,\mathcal A,\mu)$, $(J,\mathcal B, \nu)$ such $M$
is $\mathcal A$-meanable, $N$ is $\mathcal B$-meanable and $M^{\mu}\simeq N^{\nu}$.
\end{theorem}
\begin{proof}
By Theorem \ref{saturated}, there are charges $\wp_1,\wp_2$ such that
$M^{\wp_1}$, $N^{\wp_2}$ are $\aleph_1$-saturated (hence complete).
Therefore, $M^{\wp_1}\equiv N^{\wp_2}$.
By Proposition \ref{10}, $M^{\wp_1}\equiv_{\mathrm{CL}} N^{\wp_2}$.
By the CL variant of the isomorphism theorem, there is an ultrafilter $\mathcal F$
such that $(M^{\wp_1})^{\mathcal F}\simeq(N^{\wp_2})^{\mathcal F}$.
We conclude by Proposition \ref{compose} that $M^{\wp_1\otimes\mathcal F}\simeq N^{\wp_2\otimes\mathcal F}$.
\end{proof}

In the proof of proposition \ref{sat3},
if $M\equiv N$, one can find (using Lemma \ref{charge product})
a $\mu$ such that $M^\mu$ and $N^\mu$ realize all types over $M$
and $N$ respectively.  As a consequence, it is possible to arrange
in Theorem \ref{isomorphism} to have that $\mu=\nu$.

\subsection{Approximations}
Recall that $\mathbb{D}(L)$ is the vector space of $L$-sentences (see \S \ref{Affine compactness}).
A complete AL-theory is just a positive linear functional
$T:\mathbb{D}(L)\rightarrow\Rn$ with $T(1)=1$ and $M\vDash T$ means that $T(\sigma)=\sigma^M$
for every $\sigma$.
Let $K=K(L)$ be the set of all complete $L$-theories. Then, for $T_1,T_2\in K$ and $0\leqslant\gamma\leqslant1$,
$$\gamma T_1+(1-\gamma)T_2\in K.$$ So, $K$ is convex.
It is also a closed subset of the unit ball of $\mathbb{D}(L)^*$ and hence compact by the Banach-Alaoglu theorem.
So, $K$ is compact convex.
Recall that for a compact convex set $K$, $\mathbf{A}(K)$ is the set of affine continuous real functions.

Let $\Gamma$ be a set of AL-formulas in the language $L$. A CL-formula $\phi(\x)$ is approximated by formulas in $\Gamma$
if for each $\epsilon>0$, there is a formula $\theta(\x)$ in $\Gamma$ such that
$$|\phi^M(\a)-\theta^M(\a)|\leqslant\epsilon\ \ \ \ \ \ \ \ \ \forall M\ \ \forall\a\in M.$$
A CL-sentence $\phi$ is preserved by ultramean if for every ultracharge space
$(I,\mu)$ and $L$-structures $M_i$, if $M=\prod_\mu M_i$, then $$\phi^{M}=\int\phi^{M_i}d\mu.$$
Similarly, $\phi$ is preserved by powermean if for every charge space $(I,\mathcal A,\mu)$ and
model $M$ which is $\mathcal A$-meanable, one has that $\phi^{M^\mu}=\phi^M$.
Affine sentences are preserved by ultramean and powermean.

\begin{theorem} \label{char1}
A CL-sentence $\phi$ is preserved by ultramean and powermean if and only if it is approximated by affine sentences.
\end{theorem}
\begin{proof}
The `if' part is easy. Conversely, assume $\phi$ is preserved by ultramean and powermean.
For each affine sentence $\sigma$ define a function
$$f_\sigma:K\rightarrow\Rn$$ $$f_\sigma(T)=T(\sigma).$$
Clearly, $f_\sigma$ is affine and continuous. Let $$X=\{f_\sigma:\ \ \sigma\ \mbox{an\ AL-sentence\ in}\ L\}.$$
$X$ is a linear subspace of $\mathbf{A}(K)$ which contains constant functions.
Moreover, if $T_1\neq T_2$, there is an AL-sentence $\sigma$ such that
$T_1(\sigma)\neq T_2(\sigma)$. So, $f_\sigma(T_1)\neq f_\sigma(T_2)$.
This shows that $X$ separates points. Hence, it is dense in $\mathbf{A}(K)$.

Similarly, for $T\in K$ define $f_\phi(T)=\phi^M$ where $\phi$ is the sentence mentioned in the theorem and $M\vDash T$.
By the isomorphism theorem, if $M\equiv N$, for some $\mu,\nu$ one has that $M^{\mu}\simeq N^{\nu}$.
Hence, $$\phi^M=\phi^{M^{\mu}}=\phi^{N^{\nu}}=\phi^N.$$
So, $f_\phi$ is well-defined. We show that it is an affine map.
Let $\gamma\in[0,1]$ and $T_1,T_2\in K$. Let $M_1\vDash T_1$ and $M_2\vDash T_2$.
Then, $M=\gamma M_1+(1-\gamma)M_2$ is a model of the theory $\gamma T_1+(1-\gamma)T_2$.
Moreover, since $\phi$ is preserved by ultramean, one has that
$$f_\phi(\gamma T_1+(1-\gamma)T_2)=\phi^M=\gamma\phi^{M_1}+(1-\gamma)\phi^{M_2}
=\gamma f_\phi(T_1)+(1-\gamma)f_\phi(T_2).$$
Note also that $f_\phi$ is continuous, i.e. for each $r$ the sets
$$\{T\in K:\ f_\phi(T)\leqslant r\},\ \ \ \ \ \ \ \ \{T\in K:\ f_\phi(T)\geqslant r\}$$
are closed. For example, assume $T_k\rightarrow T$ in the weak* topology and $f_\phi(T_k)\leqslant r$ for each $k$.
We show that $f_\phi(T)\leqslant r$. Take a nonprincipal ultrafilter $\mathcal F$ on $\Nn$.
Let $M_k\vDash T_k$ and $M=\prod_{\mathcal F}M_k$. Then, one has that $M\vDash T$. As a consequence,
$$f_\phi(T)=\phi^M=\lim_{k,\mathcal F}\phi^{M_k}=\lim_{k,\mathcal F} f_\phi(T_k)\leqslant r.$$

We conclude that $f_\phi\in\mathbf{A}(K)$.
So, since $X$ is dense, for each $\epsilon>0$ there is an affine sentence $\sigma$ such that for every $T\in K$,\ \
$|f_\phi(T)-f_\sigma(T)|\leqslant\epsilon$. In other words, for every $M$, \  $|\phi^M-\sigma^M|\leqslant\epsilon$.
\end{proof}
\bigskip

A CL-sentence $\sigma$ in $L$ is preserved by AL-equivalence if for every
$M,N$, whenever $M\equiv N$, one has that $\sigma^M=\sigma^N$.
Note that if $\sigma,\eta$ are preserved by $\equiv$ then so does
$\sigma\wedge\eta$ and $\sigma\vee\eta$.
In fact, every sentence in the Riesz space generated by the set of affine sentences
is preserved by AL-equivalence. We denote this Riesz space by $\Lambda$.

\begin{theorem}
A CL-sentence $\phi$ is preserved by AL-equivalence if and only if it is
approximated by the Riesz space $\Lambda$ generated by the set of affine sentences.
\end{theorem}
\begin{proof}
For the non-trivial direction, as in the proof of Theorem \ref{char1},
for each $\sigma\in\Lambda$,
$$f_\sigma:K\rightarrow\Rn$$
$$f_\sigma(T)=\sigma^M$$ where $M\vDash T$ is arbitrary. It is well-defined.
Let $$X=\{f_\sigma:\ \sigma\in\Lambda\}.$$
Then, $X$ is a sublattice of $\mathbf{C}(K)$ which contains $1$ and separates points.
In particular, $$-f_\sigma=f_{-\sigma}, \ \ \ \ \ f_\sigma+f_\eta=f_{\sigma+\eta},\ \ \ \ \
f_\sigma\wedge f_\eta=f_{\sigma\wedge\eta}.$$
By the assumption, the function $f_\phi(T)=\phi^M$ for $M\vDash T$ is well-defined.
Since $\phi$ is preserved by ultraproducts, it is shown similar to the proof
of Proposition \ref{char1} that $f_\phi$ is continuous.
So, by the lattice version of Stone-Weierstrass theorem (see \cite{Aliprantis-Inf}
Th. 9.12), $f_\phi$ is approximated by elements of $X$.
In other words, for each $\epsilon>0$, there is $\sigma\in\Lambda$ such that $|\phi^M-\sigma^M|\leqslant\epsilon$ for every $M$.
\end{proof}

\subsection{Rudin-Keisler ordering}
Let $\mu$ and $\nu$ be ultracharges on $I$ and $J$ respectively.
As in \S \ref{Operations on ultracharges}, we write $\nu\leqslant\mu$ if there is a
map $f:I\rightarrow J$ such that for every $X\subseteq J$
$$\nu(X)=\mu(f^{-1}(X)).$$
In this case we write $\nu=f(\mu)$. This defines a partial pre-order on the class of ultracharges
which generalizes the Rudin-Keisler ordering on ultrafilters \cite{Comfort-Negrepontis}.
By the change of variables formula, if $\nu=f(\mu)$, then for each bounded integrable
$h:J\rightarrow\Rn$ one has that $$\int h\ d\nu=\int h\circ f\ d\mu.$$

Let $\nu\leqslant\mu$ via the function $f:I\rightarrow J$ and $M$ be an $L$-structure.
Define a map $f^*:M^{\nu}\rightarrow M^{\mu}$ as follows:
$$f^*([a]_{\nu})=[a\circ f]_{\mu}\ \ \ \ \ \ \ \ \ \forall a:J\rightarrow M.$$

\begin{lemma} $f^*$ is an elementary embedding.
\end{lemma}
\begin{proof}
Let $\phi(x_1,...,x_n)$  be a formula and $a^1,...,a^n \in\prod_J M$.
Then by the ultramean theorem
$$\phi^{M^{\nu}}([a^1]_\nu ,...,[a^n]_\nu)=\int\phi^M(a^1(j) ,...,a^n(j))\ d\nu.$$
$$=\int\phi^M\big((a^1\circ f)(i) ,...,(a^n\circ f)(i)\big)\ d\mu$$
$$=\phi^{M^{\mu}}([a^1\circ f ]_{\mu},...,[a^n\circ f]_{\mu})$$
$$=\phi^{M^{\mu}}(f^*([a^1]_{\nu}),...,f^*([a^n]_{\nu})).$$
\end{proof}
\vspace{1mm}

Let $J$ be a nonempty set and $L$ be the first order (hence Lipschitz) language consisting
of relation and function symbols for every relation (i.e. a subset on $J^n$) and operation on $J$.
Then, $J$ is endowed with a first order $L$-structure $M$ in the natural way called the \emph{complete structure} on $J$.
Every $a\in M$ is the interpretation of a constant symbol (i.e. a $0$-ary function symbol).
Therefore, every structure affinely equivalent to $M$
contains an elementary substructure isomorphic to $M$.

\begin{lemma}
Let $\mu$ and $\nu$ be ultracharges on $I$ and $J$ respectively.
Let $M$ be the complete structure on $J$. Then for every elementary embedding
$$\zeta:{M}^{\nu}\rightarrow {M}^{\mu}$$
there exists a unique (up to $\mu$-null sets) $f:I\rightarrow J$ such that $\nu=f(\mu)$
and $\zeta=f^*$.
\end{lemma}
\begin{proof} Let $\sf id$ be the identity map on $J$. Then $\sf id$ determines an element
$[{\sf id}]_{\nu}$ of $M^{\nu}$ and $\zeta([{\sf id}]_{\nu})\in {M}^{\mu}$.
Let $f:I\rightarrow M$ be such that $[f]_\mu=\zeta([{\sf id}]_{\nu})$.
We show that $\nu=f(\mu)$ and $\zeta=f^*$.

Let $A\subseteq J$. Then we have
$$\chi_A^{{M}^{\nu}}([{\sf id}]_{\nu})=\int_J \chi_A^{M}(j)\ d\nu=\nu(A).$$
On the other hand, since $\zeta$ is an elementary embedding, one has that
$$\chi_A^{{M}^{\nu}}([{\sf id}]_{\nu})=
\chi_A^{{M}^{\mu}}(\zeta([{\sf id}]_{\nu}))=\chi_A^{{M}^{\mu}}([f]_{\mu})$$
$$=\int_I \chi_A^{{M}}(f(i))\ d\mu=\int_J \chi_ A^{M}(j)\ d(f(\mu))=f(\mu)(A).$$
This shows that $\nu=f(\mu)$.
Now, we show that $\zeta=f^*$.
Let $[a]_{\nu}\in M^{\nu}$. Then, $a$ is also a unary operation on $J$ and
$$[a]_{\nu}=[a\circ {\sf id}]_{\nu}=[a^M(j)]_{\nu}=a^{M^{\nu}}([{\sf id}]_{\nu}).$$
Again, since $\zeta$ is elementary, one has that
$$\zeta([a]_{\nu})=\zeta(a^{M^{\nu}}([{\sf id}]_{\nu}))=a^{M^{\mu}}(\zeta([{\sf id}]_{\nu}))=a^{M^{\mu}}([f]_{\mu})$$$$=[a^M(f(i))]_{\mu}=[a\circ f]_{\mu}=f^*([a]_{\nu}).$$
Therefore, $\zeta=f^*$.

For uniqueness, let $g:I\rightarrow J$ be another function such that $\zeta=g^*$.
Then $$[f]_{\mu}=\zeta([{\sf id}]_{\nu})=g^*([{\sf id}]_{\nu})=[g]_{\mu}.$$
This means that $f=g$ a.e.,\ \ i.e. $\mu\{i: |f(i)-g(i)|>\epsilon\}=0$
for every $\epsilon>0$.
\end{proof}

\begin{corollary} \label{embedding} Let $\mu$ and $\nu$ be ultracharges
on $I$ and $J$ respectively. Then $\nu\leqslant\mu$ if and only if $M^{\nu}$ is
elementarily embedded in $M^{\mu}$ for every $M$.
\end{corollary}

Let $M$ be the complete structure on $J$. Assume $\zeta:M^{\nu}\rightarrow M^{\mu}$
is an isomorphism and $\eta:M^{\mu}\rightarrow M^{\nu}$ is its inverse.
Let $f:I\rightarrow J$ and $g:J\rightarrow I$ be such that $f^*=\zeta$ and $g^*=\eta$.
Then $(g\circ f)^*=f^*\circ g^*= ({\sf id}_I)^*$ and hence $g\circ f={\sf id}_I$ a.e.
Similarly, $f\circ g={\sf id}_J$.
We deduce that if $M^\mu\simeq M^\nu$ for all $M$, then $\mu\equiv\nu$ (i.e. $\mu\leqslant\nu\leqslant\mu$).
The converse this observation holds for ultrafilters. This is essentially because ultrafilters are
rigid, i.e. have no nontrivial automorphism \cite{Comfort-Negrepontis}. 
This property does not hold for ultracharges.
It is natural to ask whether it is true that $\mu\equiv\nu$ implies $M^\mu\simeq M^\nu$ for all $M$.
\vspace{1mm}

\newpage\section{Consistency and interpolation} \label{Robinson et al.}
As in first order logic, Robinson's consistency theorem can be proved by using the compactness theorem.

\begin{lemma} \label{Robinson lemma}
Let $M$ be an $L_1$-structure, $N$ be an $L_2$-structure and $L=L_1\cap L_2$.\\
(i) If $M|_{L}\equiv N|_{L}$, there is a model $K$ in $L_2$
such that $M|_{L}\preccurlyeq K|_{L}$ and $N\equiv K$.\\
(ii) If $M|_{L}\preccurlyeq N|_{L}$, there is a model $K$ in $L_1$
such that $M\preccurlyeq K$ and $N|_{L}\preccurlyeq K|_{L}$.
\end{lemma}
\begin{proof} (i) By linear compactness, it is sufficient to prove that
$ediag(M|_{L})\cup Th(N)$ is linearly satisfiable.
Let $M\vDash0\leqslant\phi(\a)$ and $N\vDash0\leqslant\sigma$ where
$\phi(\x)$ is an $L$-formula and $\sigma$ is an $L_2$-sentence.
Then $M$ and $N$ satisfy $0\leqslant\sup_{\x}\phi(\x)$.
So, for each $\epsilon>0$, the condition $-\epsilon\leqslant\phi(\a)+\sigma$ is
satisfiable in $N$ with a suitable interpretation of $\a$.
We conclude that $0\leqslant\phi(\a)+\sigma$ is satisfiable.

(ii) It is sufficient to prove that $ediag(M)\cup ediag(N|_{L})$ is linearly satisfiable.
Again, let $M\vDash0\leqslant\phi(\a)$ and $N\vDash0\leqslant\psi(\a,\b)$ where
$\phi(\x)$ is an $L_1$-formula, $\psi$ is an $L$-formula and $\a\in M, \b\in N-M$.
Then, by the assumption, $M$ satisfies $0\leqslant\sup_{\y}\psi(\a,\y)$.
Hence, for each $\epsilon>0$, it satisfies the condition
$-\epsilon\leqslant\psi(\a,\b')$ with a suitable interpretation of $\b'$ in $M$.
So, $M$ satisfies $-\epsilon\leqslant\phi(\a)+\psi(\a,\b')$.
We conclude that $0\leqslant\phi(\a)+\psi(\a,\b)$ is satisfiable.
\end{proof}

\begin{theorem} \label{Robinson th} \emph{(Affine Robinson's consistency)}\index{Robinson's consistency}
Let $T_1$ and $T_2$ be satisfiable theories in $L_1$
and $L_2$ respectively. Assume $T=T_1\cap T_2$ is complete
in $L=L_1\cap L_2$. Then $T_1\cup T_2$ is satisfiable.
\end{theorem}
\begin{proof} Let $M_1\vDash T_1$ and $N\vDash T_2$.
Then $M_1|_{L}\equiv N|_{L}$.
By Lemma \ref{Robinson lemma} (i), there is a model $M_2$ in $L_2$
such that $M_1|_{L}\preccurlyeq M_2|_{L}$ and $N\equiv M_2$. So, $M_2\vDash T_2$.
Also, by \ref{Robinson lemma} (ii), there is a model $M_3$ in $L_1$
such that $M_1\preccurlyeq M_3$ and $M_2|_{L}\preccurlyeq M_3|_{L}$.
So, $M_3\vDash T_1$. Repeating (ii), one finds
a sequence $M_n$ of models ($M_{2n}$ in $L_2$ and $M_{2n+1}$ in $L_1$) such that
$$M_{2n}\vDash T_2, \ \ \ \ \ \ \ \ M_{2n+1}\vDash T_1$$
$$\ \ \ M_n\preccurlyeq M_{n+2},\ \ \ \ \ \ \ \ M_n|_{L}\preccurlyeq M_{n+1}|_{L}.$$
Let $K_1=\bigcup M_{2n+1}$ and $K_2=\bigcup M_{2n}$.
By elementary chain theorem, $K_1\vDash T_1$ and $K_2\vDash T_2$.
Furthermore, $K_1|_{L}=K_2|_{L}$. It follows that $K_1$, $K_2$
have a common expansion to a model of $T_1\cup T_2$.

{\sc A shorter proof} is obtained by using the isomorphism theorem.
Let $M_i$ be a model of $T_i$ and $N_i$ be its reduction to $L$ for $i=1,2$.
Then $N_1\equiv N_2$. By the isomorphism theorem, there are charge spaces
$(I,\mathcal A,\mu)$ and $(J,\mathcal B,\nu)$ such that $N_1$ is $\mathcal A$-meanable,
$N_2$ is $\mathcal B$-meanable and $N^\mu_1\simeq N^\nu_2$.
Let $f:N_1^\mu\rightarrow N_2^\nu$ be an isomorphism.
$M_1^\mu$ and $N_1^\mu$ have the same ambient set.
So, for each $c,F,R\in L_2$ and $\a\in M_1^\mu$ define
$$c^{M_1^\mu}=f^{-1}(c^{N_2^\nu})$$
$$F^{M_1^\mu}(\a)=f^{-1}(F^{N_2^\nu}(f(\a)))$$
$$R^{M_1^\mu}(\a)=R^{N_2^\nu}(f(\a)).$$
Then, $M^\mu_1$ is a structure in $L_1\cup L_2$ and a model of $T_1\cup T_2$.
A better imagination is obtained if we suppose that $f$ is the identity map.
\end{proof}

By $\Gamma\leqslant\phi$ is meant the set $\{\theta\leqslant\phi:\theta\in\Gamma\}$.
Similar notations have their obvious meanings.

\begin{proposition} \label{interpolation1} \emph{(Affine Craig interpolation)\index{Craig's interpolation}}
Assume $\vDash\phi\leqslant\psi$ where $\phi$
is an $L_1$-sentence and $\psi$ is an $L_2$-sentence.
Let $L=L_1\cap L_2$.
Then for each $\epsilon>0$ there exists $\theta$ in $L$ such that
$$\vDash\phi-\epsilon\leqslant\theta\leqslant\psi.$$
\end{proposition}
\begin{proof} Let $$\Gamma=\{\theta:\ \ \theta\ \mbox{is\ an}\ L\mbox{-sentence}\ \ \ \&\ \ \vDash \theta\leqslant\psi\}.$$
We claim that for each $\epsilon>0$ there exists $\theta\in\Gamma$ such that $\vDash\phi-\epsilon\leqslant\theta$.
Assume not. Then, by affine compactness,
$\Gamma\leqslant\phi-\epsilon$ is satisfiable for some $\epsilon>0$.
Let $T_1$ be a completion of this theory in $L_1$ and $T$ be its restriction to $L$.
There is unique real number $r$ such that $\phi-\epsilon=r\in T_1$.
Also, $T$ is a complete $L$-theory. Below, we show that $T\cup\{\psi\leqslant r\}$ is satisfiable.
Then, we conclude by Theorem \ref{Robinson th} that
$$T\cup\{\phi-\epsilon=r,\ \psi\leqslant r\}$$ is satisfiable which is a contradiction.

Assume $T\cup\{\psi\leqslant r\}$ is not satisfiable.
Then, there are $\alpha,\beta\geqslant0$ and $0\leqslant\sigma$ in $T$
such that $\beta\psi\leqslant\alpha\sigma+\beta r$ is not satisfiable.
So, for some $\delta>0$, $$\vDash\alpha\sigma+\beta r\leqslant\beta\psi-\delta$$
So, since $\beta\neq0$, $$\vDash\frac{\alpha}{\beta}\sigma+r+\frac{\delta}{\beta}\leqslant\psi.$$
Therefore, $\frac{\alpha}{\beta}\sigma+r+\frac{\delta}{\beta}\in\Gamma$ and hence
$\frac{\alpha}{\beta}\sigma+r+\frac{\delta}{\beta}\leqslant\phi-\epsilon=r$ belongs to $T_1$.
This is a contradiction. Now, $T_1$ and $T\cup\{\psi\leqslant r\}$ are satisfiable.
\end{proof}

\begin{corollary} \label{interpolation2}
Let $\phi$ be an $L_1$-sentence and $\psi$ be an $L_2$-sentence.
Assume $0\leqslant\phi\vDash0\leqslant\psi$.
Then, for each $\epsilon>0$ there exists $\theta$ in $L_1\cap L_2$
such that $$0\leqslant\phi\vDash0\leqslant\theta, \ \ \ \ \&\ \ \ \ \
0\leqslant\theta\vDash0\leqslant\phi+\epsilon.$$
\end{corollary}
\begin{proof} By Lemma \ref{proofs}, there exists $\alpha\geqslant0$ such that
$\vDash\alpha\phi\leqslant\psi+\frac{\epsilon}{2}$. By Proposition \ref{interpolation1}, there exists $\theta$ in $L_1\cap L_2$ such that
$\vDash\alpha\phi\leqslant\theta\leqslant\psi+\epsilon$. Therefore,
$$0\leqslant\phi\ \vDash\ 0\leqslant\theta\ \vDash\ 0\leqslant\phi+\epsilon.$$
\end{proof}
\vspace{2mm}

Let $\Sigma(P)$ be an $L\cup\{P\}$-theory.
$\Sigma(P)$ defines $P$ \emph{implicitly} if
$$\Sigma(P)\cup\Sigma(P')\vDash P(\x)=P'(\x).$$
$\Sigma(P)$ defines $P$ \emph{explicitly} if for each $\epsilon>0$ there exists an
$L$-formula $\theta(\x)$ such that
$$\Sigma(P)\vDash 0\leqslant P-\theta\leqslant\epsilon.$$

\begin{proposition} \label{Beth} \emph{(Affine Beth's definability)}\index{Beth's definability}
Assume $\Sigma(P)$ is satisfiable. Then $\Sigma(P)$ defines $P$ explicitly
if and only if $\Sigma(P)$ defines $P$ implicitly.
\end{proposition}
\begin{proof} Let us prove the nontrivial direction. Assume $\Sigma(P)$ defines $P$ implicitly.
Then, $$\Sigma(P)\cup\Sigma(P')\cup\{P\leqslant P'-\epsilon\}$$ is unsatisfiable for each $\epsilon>0$.
We may further assume $\Sigma(P)$ is affinely closed.
So, for some $0\leqslant\phi$ in $\Sigma(P)$ and $0\leqslant\psi$ in $\Sigma(P')$, the condition
$P\leqslant\phi+\psi+P'-\epsilon$
is unsatisfiable.
This means that $$\vDash\ \psi+P'-\epsilon\leqslant P-\phi.$$
By the Proposition \ref{interpolation1}, there is a formula $\theta(\x)$ in $L$ such that
$$\vDash\ \psi+P'-2\epsilon\leqslant\theta\leqslant P-\phi.$$
Let $\psi_0$ be the result of replacing $P'$ with $P$ in $\psi$.
Then $0\leqslant\psi_0$ belongs to $\Sigma(P)$ and $$\vDash\psi_0+P-2\epsilon\leqslant\theta\leqslant P-\phi.$$
Hence $$0\leqslant\phi,\ 0\leqslant\psi_0\ \vDash\ 0\leqslant P-\theta\leqslant2\epsilon$$
$$\Sigma(P)\ \vDash\ 0\leqslant P-\theta\leqslant2\epsilon.$$
\end{proof}
\vspace{2mm}

A class $\mathcal K$ of $L$-structures is called a \emph{projective class}\index{projective class}
if there is a language $\bar L\supseteq L$ and a set $\Sigma$ of $\bar L$-conditions such that $\mathcal K$ is
the class of $L$-reductions of models of $\Sigma$. If $L=\bar L$, it is an \emph{elementary class}
and a \emph{basic} one if furthermore $\Sigma$ consists of a single condition.
Two projective classes $\mathcal K_1, \mathcal K_2$ (coming from $L_1,L_2\supseteq L$ respectively)
are \emph{separated}\index{separated} by a basic elementary class in $L$ if there is a sentence $\sigma$ in $L$ and $r<s$ such that
$$\mathcal K_1\subseteq \textrm{Mod}(\sigma\leqslant r)\ \ \ and \ \ \ \mathcal K_2\subseteq\textrm{Mod}(s\leqslant\sigma).$$

For each sentence $\sigma$ define a functional $f_{\sigma}$ on the set of complete theories by
$$f_{\sigma}(T)=T(\sigma).$$

\begin{corollary}
Any two disjoint projective classes $\mathcal K_1$ and $\mathcal K_2$
are separated by a basic elementary class in the common language.
\end{corollary}
\begin{proof}
Let $\Sigma$ be an $L_1$-theory, $\Delta$ be an $L_2$-theory, $L=L_1\cap L_2$ and
$$\mathcal K_1=\{M|_L:\ M\vDash\Sigma\}, \ \ \ \ \mathcal K_2=\{M|_L:\ M\vDash\Delta\}.$$
Let $A$ be the set of complete $L$-theories $T$ such that $T\cup\Sigma$ is satisfiable
and $B$ be the set of complete $L$-theories $T$ such that $T\cup\Delta$ is satisfiable.
Clearly, $A,B\subseteq\mathbb{D}(L)^*$ are compact convex.
Moreover, by Theorem \ref{Robinson th}, they are disjoint
since otherwise $\mathcal K_1\cap \mathcal K_2$ would be nonempty.
By Theorems \ref{Rudin} and \ref{Conway}, there exist an $L$-sentence
$\sigma$ and $r,s\in\Rn$ such that
$$\ \ \ \ f_{\sigma}(T_1)\leqslant r<s\leqslant f_{\sigma}(T_2)
\ \ \ \ \ \ \ \ \ \ \ \forall T_1\in A\ \ \forall T_2\in B.$$
In particular, $\Sigma\vDash\sigma\leqslant r$ and
$\Delta\vDash s\leqslant\sigma$ and hence
$$\mathcal K_1\subseteq\mbox{Mod}(\sigma\leqslant r), \ \ \ \ \
\mathcal K_2\subseteq\mbox{Mod}(s\leqslant\sigma).$$
\end{proof}

Now, we come to the proof of Lyndon interpolation theorem.
Here, the form stated in Proposition \ref{interpolation2} is intended.
We modify the first order proof stated in \cite{CK1} for the present situation.

A relation symbol is said to occur \emph{positively} (resp. \emph{negatively})
in the formula $\phi$ if it is within the scope of an even
(resp. odd) number of minus symbols.
A formula has \emph{negative normal form} (nnf) if it is built up
from atomic and negative atomic formulas
using the connectives $+$ and $r\cdot$ where $r\geqslant0$.
Every formula is equivalent to an nnf formula.
$\sigma^*$ denotes the nnf of $-\sigma$.
For a set $\Theta$ of formulas $0$$\leqslant$$\Theta$ denotes
$\{0\leqslant\sigma|\ \sigma\in\Theta\}$.

Lyndon interpolation holds for relational languages.
Let $\phi,\psi$ be sentences in the relational language $L$.
Let $C$ be a countable set of new constant symbols and $\bar{L}=L\cup C$.
Let $\Phi$ be the set of all rational nnf sentences $\sigma$ in $\bar L$ such that
every relation symbol (excluding $d$) which occurs positively (resp. negatively)
in $\sigma$, occurs positively (resp. negatively) in $\phi$.
The set $\Psi$ is defined in the same way for $\psi$.
Let $\Psi^*=\{\sigma^*|\ \ \sigma\in\Psi\}$.
Then, $\Phi,\Psi,\Psi^*$ are closed under $+$ and $r\cdot$ for $r\geqslant0$.
Two theories $T\subseteq(0$$\leqslant$$\Phi)$ and $U\subseteq(0$$\leqslant$$\Psi^*)$
are called \emph{separable}
if there are $\theta\in\Phi\cap\Psi$ and $\delta>0$ such that
$T\vDash0\leqslant\theta$ and $U\vDash\theta\leqslant-\delta$.
Otherwise, they are called \emph{inseparable}.

\begin{theorem} \label{main}
Let $0\leqslant\phi\vDash0\leqslant\psi$.
Then for each $\epsilon>0$ there exists $\theta$ such that:\\
- $0\leqslant\phi\ \vDash\ 0\leqslant\theta$ and
$0\leqslant\theta\ \vDash\ 0\leqslant\psi+\epsilon$\\
- every relation symbol (excluding $d$) which occurs positively (resp. negatively)
in $\theta$ occurs positively (resp. negatively) in both $\phi$ and $\psi$.
\end{theorem}
\begin{proof} Let $\epsilon>0$ and assume there is no $\theta$ as required.
We will construct a model for $\{0\leqslant\phi,\ \psi\leqslant-\epsilon\}$.
We may assume $\phi$ and $\psi$ are rational nnf sentences.
Enumerate $\Phi$ and $\Psi$ defined above as $\phi_0,\phi_1,...$ and
$\psi_0,\psi_1,...$ respectively (where each sentence is repeated infinitely many times).
We will construct two increasing sequences of theories
$$\{0\leqslant\phi\big\}=T_0 \subseteq T_1\subseteq\cdots\subseteq\
\ (0\mbox{$\leqslant$}\Phi)$$
$$\{0\leqslant\psi^*-\epsilon\}=U_0\subseteq U_1\subseteq\cdots\subseteq\
(0\mbox{$\leqslant$}\Psi^*)$$
such that
\begin{enumerate}
 \item[(1)] $T_m$ and $U_m$ are inseparable finite sets of closed conditions,
 \item[(2)] If $T_m\cup\{0\leqslant\phi_m\}$ and $U_m$ are
 inseparable, then $0\leqslant\phi_m\in T_{m+1}$

 If $T_{m+1}$ and $U_{m}\cup\{0\leqslant\psi_m^*\}$ are inseparable,
 then $0\leqslant\psi_m^*\in U_{m+1}$,

\item[(3)] If $\phi_m=\sup_x\sigma(x)+r$ and $0\leqslant\phi_m\in T_{m+1}$,
then $0\leqslant\sigma(c)+r\in T_{m+1}$ for some $c$

If $\psi_m^*=\sup_x\eta(x)+r$ and $0\leqslant\psi_m^*\in U_{m+1}$,
then $0\leqslant\eta(c)+r\in U_{m+1}$ for some $d$
\end{enumerate}

In (3), we use constant symbols $c,d$ which have not been used yet.
The argument of construction is clear and inseparability is preserved
in each step. For example, suppose that
$T_m\cup\{0\leqslant\sigma(c)\}$ and $U_m$ are separable.
Then, for some $\theta\in\Phi\cap\Psi$ and $\delta>0$
$$T_m,\ 0\leqslant\sigma(c)\vDash0\leqslant\theta,\ \ \ \ \
\ \ \ \ \ U_m\vDash\theta\leqslant-\delta.$$
So, $$T_m,\ 0\leqslant\sup_x\sigma(x)\vDash0\leqslant\theta$$
and hence $T_m\vDash0\leqslant\theta$.
Thus, $T_m$ and $U_m$ are separable which is a contradiction.

Let $$T_\omega=\bigcup_{m<\omega}T_m,\ \ \ \ \ \ \ \ U_\omega=\bigcup_{m<\omega}U_m.$$
We show that $T_\omega$ and $U_\omega$ are inseparable.
Suppose $T_\omega\vDash0\leqslant\theta$ and $U_\omega\vDash\theta\leqslant-\delta$
for some $\theta\in\Phi\cap\Psi$ and $\delta>0$.
Then, there exists $m$ such that
$$T_m\vDash-\frac{\delta}{3}\leqslant\theta, \ \ \ \ \ \ \ \ \
U_m\vDash\theta\leqslant-\frac{2\delta}{3}$$
which contradicts the inseparability of $T_m$ and $U_m$.
We also conclude that $T_\omega,U_\omega$ are both satisfiable.
We will show that $T_\omega\cup U_\omega$ is satisfiable too.
We first prove the following claims:
\bigskip

\noindent{\sc Claim 1:}
For $\sigma,\eta\in\Phi$,\ if $0\leqslant\sigma+\eta\in T_{\omega}$
then either $0\leqslant\sigma\in T_{\omega}$ or $0\leqslant\eta\in T_{\omega}$.\\
Similarly, for $\sigma,\eta\in\Psi^*$,\ if $0\leqslant\sigma+\eta\in U_{\omega}$
then either $0\leqslant\sigma\in U_{\omega}$ or $0\leqslant\eta\in U_{\omega}$.\\
{\sc Proof:}
We prove the first one. Suppose not.
Then for some $\theta_1,\theta_2\in\Phi\cap\Psi$ and $\delta>0$ one has that
$$T_{\omega}, 0\leqslant\sigma\vDash0\leqslant\theta_1, \ \ \ \ \ \ \
U_{\omega}\vDash\theta_1\leqslant-\delta$$
$$T_{\omega}, 0\leqslant\eta\vDash0\leqslant\theta_2, \ \ \ \ \ \ \
U_{\omega}\vDash\theta_2\leqslant-\delta.$$
By Corollary \ref{proofs}, for some $\alpha,\beta\geqslant0$
$$T_{\omega}\vDash\alpha\sigma\leqslant\theta_1+\frac{\delta}{2},\ \ \ \ \ \ \
T_{\omega}\vDash\beta\eta\leqslant\theta_2+\frac{\delta}{2}.$$
By inseparability of $T_\omega$ and $U_\omega$,\ \ $\alpha,\beta$ are nonzero.
Hence,
$$T_{\omega}\vDash0\leqslant\alpha\beta(\sigma+\eta)\leqslant
(\alpha+\beta)\frac{\delta}{2}+\beta\theta_1+\alpha\theta_2.$$
On the other hand,\\
$$U_{\omega}\vDash\beta\theta_1+\alpha\theta_2\leqslant-\delta(\beta+\alpha)$$
This means that $T_\omega$ and $U_\omega$ are separable.
\bigskip

\noindent{\sc Claim 2:} For each $a,b\in\bar{L}$, there is a unique real $s$
such that $T_\omega\vDash d(a,b)=s$ and $U_\omega\vDash d(a,b)=s$.\\
{\sc Proof:} For each $a,b\in\bar{L}$ and rational $r$, one of
$0\leqslant d(a,b)-r$ and $0\leqslant r-d(a,b)$ belongs to $T_\omega$.
This is because the condition $$0\leqslant(d(a,b)-r)+(r-d(a,b))$$ belongs to $T_\omega$.
Similarly, one of these conditions belongs to $U_\omega$.
Since $T_\omega$ and $U_\omega$ are inseparable, for each $r$,
we must have that $T_\omega\vDash r\leqslant d(a,b)$ if and only if
$U_\omega\vDash r\leqslant d(a,b)$ and similarly for $d(a,b)\leqslant r$.
This proves the claim.
\bigskip

Now, we show that $T_\omega\cup U_\omega$ is satisfiable.
Members of this set are of the form $0\leqslant\theta$.
For simplicity, we write $0\leqslant\theta-r$ as $r\leqslant\theta$.
Clearly, $T_\omega\vDash d(a,b)=0$ is an equivalence relation on
the set of constant symbols of $\bar L$. Let $M$ be the set of all these classes.
Of course, $U_\omega$ defines the same set $M$.
Note also that the set $\Delta$ of atomic and negative atomic
conditions, i.e. $r\leqslant\theta$, $r\leqslant-\theta$ where $\theta$ is atomic,
belonging to $T_\omega\cup U_\omega$ is satisfiable.
Since, otherwise, for some atomic $\theta$, $r\in\Rn$ and $\delta>0$ we must have that
$r\leqslant\theta\in T_\omega$ and $\delta-r\leqslant-\theta\in U_\omega$ (or conversely).
Then $\theta\in\Phi\cap\Psi$. This contradicts the inseparability of
$T_\omega$ and $U_\omega$. The converse case is similar.
We conclude that $\Delta$ has a model based on the set $M$.
We prove by induction on the complexity of $\theta$ that every
$r\leqslant\theta\in T_\omega$ is satisfied in $M$.
Suppose the claim is proved for $\theta_1$ and $\theta_2$. Let
$$r=\sup\{s|\ s\leqslant\theta_1+\theta_2\in T_\omega\}, \ \ \ \
r_i=\sup\{s|\ s\leqslant\theta_i\in T_\omega\}.$$
Suppose $r,r_1,r_2$ are rational.
Then $r_1\leqslant\theta_1,$ $r_2\leqslant\theta_2$, $r\leqslant\theta_1+\theta_2$
as well as $r_1+r_2\leqslant\theta_1+\theta_2$ belong to $T_\omega$.
So, $r_1+r_2\leqslant r$. Suppose $r-r_1-r_2=\epsilon>0$.
Then, at least one of $$r_1+\frac{\epsilon}{2}\leqslant\theta_1, \ \ \ \ \ \ r_2+\frac{\epsilon}{2}\leqslant\theta_2$$
must belong to $T_\omega$. This is a contradiction.
We conclude that $r_1\leqslant\theta^M_1$, $r_2\leqslant\theta^M_2$ and hence $r\leqslant\theta^M_1+\theta^M_2$.
If $r,r_1,r_2$ are not rational, one does the argument using approximations.
The case $r\theta$ for $r\geqslant0$ is obvious.
Suppose $r\leqslant\sup_x\theta(x)\in T_\omega$. Then, there exists $c$ such that
$r\leqslant\theta(c)\in T_\omega$. So, $M\vDash r\leqslant\theta(c)$ and hence
$M\vDash r\leqslant\sup_x\theta(x)$.
Also, if $r\leqslant\inf_x\theta(x)\in T_\omega$ then
$r\leqslant\theta(c)\in T_\omega$ for every $c$.
So, $M\vDash r\leqslant\theta(c)$ for every $c$ and hence
$M\vDash r\leqslant\inf_x\theta(x)$.
One proves by a similar argument that $M\vDash U_\omega$.
\end{proof}
\bigskip

We end the paper by an affine form of Herbrand interpolation theorem.
A sentence of the form $\inf_{\x}\phi$ (resp. $\sup_{\x}\phi$)
where $\phi$ is quantifier-free is called an $\inf$-sentence (resp. $\sup$-sentence).

\begin{proposition}
Let $\phi$ be an $\inf$-sentence and $\psi$ be a $\sup$-sentence.
If $0\leqslant\phi\vDash0\leqslant\psi$, then for each $\epsilon>0$
there exists a quantifier-free $\theta$ such that
$$0\leqslant\phi\ \vDash\ 0\leqslant\theta\ \vDash\ 0\leqslant\psi+\epsilon.$$
\end{proposition}
\begin{proof}
We just give the sketch of the proof which is similar to that of Theorem \ref{main}.
Assume $0\leqslant\phi$ and $0\leqslant\psi+\epsilon$
have no quantifier-free interpolant and construct a model of
$\{0\leqslant\phi,\ \psi\leqslant-\epsilon\}$ as follows.
This time, $\Phi$ is the set of rational nnf sentences $\sigma$ such that
every quantifier which occurs in $\sigma$ also occurs in $\phi$.
$\Psi$ is defined similarly.
Then construct $T_{\omega}\subseteq(0$$\leqslant$$\Phi)$
and $U_{\omega}\subseteq(0$$\leqslant$$\Psi^*)$ as before.
The rest of the proof is similar to the previous argument.
\end{proof}

\newpage\section{Definability}\label{Definability}

\subsection{Definable predicates}
As before, $T$ is a complete theory in $L$. Unless otherwise stated, by definable we mean without parameters.
We assume all parameters needed to define a notion are already named in the language.
By a \emph{predicate}\index{predicate} we mean a bounded function $P:M^n\rightarrow\Rn$
which is uniformly continuous, i.e. for all $\epsilon>0$ there is $\delta>0$ such that
$$d(\x,\y)<\delta\ \Longrightarrow\ |P(\x)-P(\y)|\leqslant\epsilon\ \ \ \ \ \ \ \forall\x,\y.$$

\begin{definition}
\emph{A predicate $P:M^n\rightarrow\Rn$ is \emph{definable}\index{definable predicate} if there is a sequence $\phi_k(\x)$
of formulas such that $\phi_k^M\rightarrow P$ uniformly on $M^n$.}
\end{definition}

Note that a definable predicate $P$ is not Lipschitz.
However, the sequence $\phi_k$ determines a definable predicate on every $N\vDash T$ which is denoted by $P^N$.
We can treat definable predicates as interpretations of new relation symbols added to the language.
Although affine compactness theorem does not hold in such a language, notions such as
$(M,P)\preccurlyeq (N,P^N)$ are meaningful. The proof of the following proposition is routine.

\begin{proposition} \label{definitional expansion}
Let $P:M^n\rightarrow\Rn$ be definable. If $N\preccurlyeq M$ then $P^N=P|_N$ and
$(N,P^N)\preccurlyeq(M,P)$. If $M\preccurlyeq N$, then $(M,P)\preccurlyeq(N,P^N)$.
\end{proposition}

Given a formula $\phi(\x)$, the function defined on $K_n(T)$ by $\hat\phi(p)=p(\phi)$
is affine, logic-continuous and $\lambda_{\phi}$-Lipschitz.
Clearly, $\hat\phi=\hat\psi$ if and only if $\phi$ and $\psi$ are $T$-equivalent.
Also, for each $M\vDash T$, $\phi^M_k(\a)=\hat\phi_k(tp(\a))$ and hence $\phi^M_k$ is Cauchy if and only if $\hat\phi_k$ is Cauchy.

\begin{proposition} \label{Lipschitz}
The following are equivalent for every $\xi: K_n(T)\rightarrow\Rn$:

\emph{(i)} $\xi\in\mathbf{A}(K_n(T))$ (i.e. $\xi$ is affine and continuous)

\emph{(ii)} There is a sequence $\phi_k$ of formulas such that $\hat\phi_k$ converges to $\xi$ uniformly.
\end{proposition}
\begin{proof} (i)$\Rightarrow$(ii): Every $\hat\phi$ is clearly affine and
logic-continuous. Moreover, the subspace of $\mathbf{A}(K_n(T))$
consisting of these functions contains constant maps and separates points.
Hence, it is dense in $\mathbf{A}(K_n(T))$. 
The reverse direction is obvious.
\end{proof}

The set of definable predicates (on $M$ or on any other model of $T$) is denoted by
$\mathbf{D}_n(T)$\index{$\mathbf{D}_n(T)$}. This is the completion of $\mathbb{D}_n(T)$.
Note that if $M$ realizes all types, then
$$\sup_{\a\in M}|\phi^M(\a)|=\sup_{p\in K_n(T)}|\hat\phi(p)|.$$
In particular, $\|\phi\|=\|\hat\phi\|$.
We deduce by Proposition \ref{Lipschitz} that $\mathbf{D}_n(T)$ and $\mathbf{A}(K_n(T))$ are isometrically isomorphic.

\begin{proposition} \label{Svenonius}\index{Svenonius theorem}
\emph{(Affine Svenonius)} Let $M\vDash T$ and $P:M^n\rightarrow\Rn$ be a predicate.
Then $P$ is definable if and only if for each $(M,P)\preccurlyeq(M',P')$ and
automorphism $f$ of $M'$ one has that $P'=P'\circ f$.
\end{proposition}
\begin{proof} The `only if' direction is obvious. We prove the reverse direction.
Since $P$ is uniformly continuous, there is a sequence $P_k$ of Lipschitz functions on $M^n$ uniformly convergent to $P$.
Using the language $\bar{L}=L\cup\{P_1,P_2,..\}$, we can find $(M,P)\preccurlyeq (N,P^N)$
which we may further assume by Proposition \ref{saturated strong homogeneous} that it is $\aleph_0$-saturated and
strongly $\aleph_0$-homogeneous on every sublanguage of $\bar{L}$.
Given $p(\x)\in K_n(T)$, define $$\xi(p)=P^N(\a)$$
where $\a\in N$ realizes $p$. If $\b$ is another realization of $p$, there is an automorphism $f$ of $N$
such that $f(\a)=\b$. By the assumption, we must have that $P^N(\a)=P^N(\b)$.
This shows that $\xi$ is well-defined.
For continuity of $\xi$, suppose $\xi(p)=r$. 
Then, for each $\epsilon>0$ there must exist $0\leqslant\phi$
in $p$ and $\delta>0$ such that
$$-\delta<\phi^N(\a) \ \ \Rightarrow\ \ P^N(\a)< r+\epsilon \ \ \ \ \ \ \ \ \ \ \forall\a\in N.$$
Since, otherwise, for some $\epsilon>0$, the type
$p(\x),\ r+\epsilon\leqslant P(\x)$ is satisfied in $(N,P^N)$ which is impossible.
Therefore, for each type $q$
$$-\delta<q(\phi)\ \ \Rightarrow\ \ \xi(q)<r+\epsilon.$$
Similarly, there are $0\leqslant\psi\in p$ and $\delta'>0$ such that
$$-\delta'<q(\psi)\ \ \Rightarrow\ \ r-\epsilon<\xi(q).$$
This shows that $\xi$ is logic-continuous.

Now, we show that $\xi$ is affine.
First, note that if $(N,P^N)$ is replaced with any $(M,P)\preccurlyeq(N_1,P^{N_1})$
which realizes $p(\x)$ (by say $\a\in N_1$), then $\xi(p)=P^{N_1}(\a)$
(take an elementary extension of $(N_1,P^{N_1})$ which is $\aleph_0$-saturated and
strongly $\aleph_0$-homogeneous on sublanguages).
Let $(N_1,P^{N_1})$, $(N_2,P^{N_2})$ be elementary extensions of $(M,P)$ and containing $\a\vDash p$ and $\b\vDash q$ respectively.
Let $0\leqslant\gamma\leqslant1$.
Then, (up to isomorphism) $$(M,P)\preccurlyeq\gamma(N_1,P^{N_1})+(1-\gamma)(N_2,P^{N_2})=(N_3,P^{N_3})$$
contains $(\a,\b)$ which realizes $\gamma p+(1-\gamma)q$.
Moreover, $$\xi(\gamma p+(1-\gamma)q)=P^{N_3}((\a,\b))=\gamma P^{N_1}(\a)+(1-\gamma)P^{N_2}(\b)=
\gamma \xi(p)+(1-\gamma)\xi(q).$$
By Proposition \ref{Svenonius}, there is a sequence $\phi_k$ of formulas such that $\hat\phi_k$ converges to $\xi$ uniformly.
Therefore, $\phi_k\stackrel{u}{\longrightarrow}\xi\circ tp=P$ and hence $P$ is definable.
\end{proof}

\begin{proposition}
Let $M\vDash T$ be $\aleph_1$-saturated strongly $\aleph_1$-homogeneous and $P:M^n\rightarrow\Rn$ be definable with parameters.
Then, $P$ is $\emptyset$-definable if and only if it is preserved by every automorphism of $M$.
\end{proposition}
\begin{proof} We prove the nontrivial part.
Assume $|\phi^M_k(\x,\a)-P(\x)|\leqslant\frac{1}{k}$ for every $\x$ where $\a$ is a countable tuple of parameters in $M$.
Let $\b\vDash p(\y)=tp(\a)$. Then, there is an automorphism $f$ of $M$ such that $f(\b)=\a$.
Since $f$ preserves both $P$ and $\phi_k(\x,\y)$, for all $\x\in M^n$
$$|\phi_k^M(\x,\a)-\phi_k^M(\x,\b)|\leqslant|\phi_k^M(\x,\a)-P(\x)|+|P(\x)-\phi_k^M(\x,\b)|\leqslant\frac{2}{k}.$$
This shows that the set
$$p(\y),\ \phi_k(\x,\a)+\frac{3}{k}\leqslant\phi_k(\x,\y)$$
is not satisfied in $(M,\a)$. So, there are $0\leqslant\theta(\y)$ in $p$ and $r,s\geqslant0$ such that for all $\x,\y\in M$
$$r(\phi^M_k(\x,\a)+\frac{3}{k})\leqslant r\phi^M_k(\x,\y)+s \theta^M(\y)$$
is not satisfied in $M$ (since $M$ is $\aleph_1$-saturated).
Indeed, we may assume $r=1$. So, for all $\x,\y\in M$
$$\phi^M_k(\x,\y)+s\theta^M(\y)\leqslant\phi^M_k(\x,\a)+\frac{3}{k}.$$
Therefore, since $0\leqslant\theta^M(\a)$, for all $\x\in M^n$
$$P(\x)-\frac{1}{k}\leqslant\phi^M_k(\x,\a)\leqslant\sup_{\y}(\phi^M_k(\x,\y)+s\theta^M(\y))
\leqslant\phi^M_k(\x,\a)+\frac{3}{k}\leqslant P(\x)+\frac{4}{k}.$$
We conclude that $P(\x)$ is $\emptyset$-definable.
\end{proof}
\vspace{2mm}

The \emph{epigraph}\index{epigraph} of a function $f:X\rightarrow\Rn$ is the set
$$\mbox{epi}(f)=\{(x,r)\ : \ f(x)\leqslant r\}.$$
We say $P:M^n\rightarrow\Rn$ has a \emph{type-definable epigraph}
if there is a set $\Phi(\x)$ of formulas such that
$$\mbox{epi}(P)=\{(\a,r)\ : \  \phi^M(\a)\leqslant r\ \ \mbox{for\ every}\ \phi\in\Phi\}.$$

Let $K\subseteq V$ be convex where $V$ is a topological vector space.
A function $f:K\rightarrow\Rn$ is \emph{convex} if for every $p,q\in K$
and $0\leqslant\gamma\leqslant 1$, one has that
$$f(\gamma p+(1-\gamma)q)\leqslant\gamma f(p)+(1-\gamma)f(q).$$
Then, $f$ convex if and only if its epigraph is a convex set (\cite{Aliprantis-Inf} Lem. 5.39).

\begin{proposition} \label{dfncriterion2}
Let $M$ be $\aleph_0$-saturated. Then, a predicate $P:M^n\rightarrow\Rn$ is
definable if and only if both $P$ and $-P$ have type-definable epigraphs.
\end{proposition}
\begin{proof} Assume $P$ is definable.
Take a sequence $\phi_k$ of formulas such that $\|P-\phi_k\|\leqslant\frac{1}{k}$ for all $k$.
Then, $$\mbox{epi}(P)=\big\{(\a,r)\ \big|\ \phi_k^M(\a)-\frac{1}{k}\leqslant r\ \ \ \ \ \forall k<\omega\big\}.$$
Similarly, the epigraph of $-P$ is type-definable.
Conversely assume the epigraphs of $P$ and $-P$ are type-definable.
Define a map $$\xi: K_n(T)\rightarrow\Rn$$
by $\xi(p)=P(\a)$ if $\a\vDash p$.
It is clear that $\xi$ is well-defined and logic-continuous.
We show that it is affine. Assume $$\mbox{epi}(P)=\{(\a,r)\ :\ \forall\phi\in\Phi,\ \ \phi^M(\a)\leqslant r\}.$$
Then, for every $p$ and $\phi\in\Phi$ one has that $p(\phi)\leqslant\xi(p)$.
Fix $p_1$, $p_2$ and let $\bar{c}$ realizes $\gamma p_1+(1-\gamma)p_2$.
Then, for all $\phi\in\Phi$ one has that
$$\phi^M(\bar{c})=\gamma p_1(\phi)+(1-\gamma)p_2(\phi)\leqslant\gamma\xi(p_1)+(1-\gamma)\xi(p_2).$$
Therefore, by the assumption
$$\xi(\gamma p_1+(1-\gamma)p_2)=P(\bar{c})\leqslant\gamma\xi(p_1)+(1-\gamma)\xi(p_2)$$
which shows that $\xi$ is convex. Similarly, $-\xi$ is convex.
We conclude that $\xi$ is affine. By Proposition \ref{Lipschitz},
$\hat\phi_k\stackrel{u}\rightarrow\xi$ for some sequence $\phi_k$.
Therefore, $\phi^M_k\stackrel{u}\rightarrow P$.
\end{proof}

\subsection{Definable functions}
A function $f:M^m\rightarrow M^n$ is \emph{definable}\index{definable function} if $d(f(\x),\y)$ is definable where $|\y|=n$.

\begin{lemma} \label{P-D}
If $f$ is definable, then for each definable $P(u,\y)$,\ \ $P(f(\x),\y)$ is definable.
\end{lemma}
\begin{proof}
First assume $P$ is the formula $\phi$ and show that
$$\phi^M(f(\x),\z)\leqslant\inf_{\y}[\phi^M(\y,\z)+\lambda_{\phi} d(f(\x),\y)]\leqslant\phi^M(f(\x),\z).$$
Then, assume $\phi_k^M\stackrel{u}{\rightarrow}P$ and deduce that $\phi_k^M(f(\x),\y)\stackrel{u}{\rightarrow}P(f(\x),y)$.
\end{proof}

As a consequence, if $f$ and $g$ are definable, then so is $g\circ f$.
Also, since projections are definable, $f=(f_1,...,f_n)$ is definable if and only if $f_1,...,f_n$ are so.

Proof of the following proposition is similar to Propositions 9.7, 9.8 and 9.25 of \cite{BBHU}.

\begin{proposition}
If $f:M^n\rightarrow M$ is definable and $N\preccurlyeq M$ then $f|_N$ maps
$N^n$ into $N$ and it is definable in $N$. Similarly, if $M\preccurlyeq N$ then
there is a unique definable extension $f\subseteq\bar f:N^n\rightarrow N$.
\end{proposition}

The following equality shows that if $f$ is definable, then if its graph $G_f$ is definable:
$$d((\x,\y),G_f)=\inf_{\bar{u}}[d(\x,\bar{u})+d(\y,f(\bar{u}))]$$
We have also the following stronger result which holds if $M$ is $\aleph_0$-saturated.
A set $X\subseteq M^n$ is \emph{type-definable} if it is the set of common solutions of
a family of conditions.

\begin{proposition} \label{graph}
Let $M$ be $\aleph_0$-saturated. Then, $f:M^n\rightarrow M$ is definable if and only if $G_f$ is type-definable.
\end{proposition}
\begin{proof} For the nontrivial direction, assume $\Gamma(\x,u)$ is a set of conditions
of the form $\phi(\x,u)\leqslant0$ which type-defines $G_f$. Let
$$\Lambda_r(\x,y)=\big\{\inf_u[\alpha\phi(\x,u)+d(u,y)]\leqslant r\ :\ \ \
\phi(\x,u)\leqslant0\in\Gamma,\ \alpha\geqslant0\big\}.$$
Clearly, if $d(f(\a),b)\leqslant r$, then $(\a,b)$ satisfies $\Lambda_r(\x,y)$ (set $u=f(\a)$).
Conversely, if $(\a,b)$ satisfies $\Lambda_r(\x,y)$, then the type
$$\{\phi(\a,u)\leqslant 0\ : \ \phi(\x,u)\leqslant0\in\Gamma\}\cup\{d(u,b)\leqslant r\}$$
is affinely satisfied in $M$. So, by saturation, it is satisfied by some $c\in M$.
Then $f(\a)=c$ and $d(f(\a),b)\leqslant r$. We therefore have that
$$d(f(\a),b)\leqslant r\ \ \Leftrightarrow\ \ (\a,b)\vDash\Lambda_r(\x,y)\ \ \ \ \ \ \ \forall\a,b$$
and hence the epigraph of $d(f(\x),y)$ is type-definable.
Similarly, one shows that the epigraph of $-d(f(\x),y)$ is type-definable.
We conclude by Proposition \ref{dfncriterion2} that $d(f(\x),y)$ is a definable predicate.
\end{proof}

In particular, in the case $M$ is $\aleph_0$-saturated, if $f$ is definable and invertible, then $f^{-1}$ is definable.
The following is a definable variant of the existence of invariant probability measures
for continuous functions on compact metric spaces.

\begin{proposition}\label{invariant types}
Let $f:M\rightarrow M$ be definable. Then there exists a type $p(x)\in K_1(T)$
which is $f$-invariant, i.e. $p(\phi(x))=p(\phi(f(x))$ for every $\phi(x)$.
In particular, if $M$ is $\aleph_0$-saturated, then there exists $c\in M$ such that $c\equiv f(c)$.
\end{proposition}
\begin{proof} $f$ induces an affine continuous map
$\hat f:K_1(T)\rightarrow K_1(T)$ by $$\hat f(p)(\phi(x))=p(\phi(f(x))).$$
By the Schauder-Tychonoff fixed point theorem (\cite{Conway} p.150),
$\hat f$ has a fixed point $p$. Then, for each $\phi(x)$,
$$p(\phi(x))=p(\phi(f(x)).$$
If $M$ is $\aleph_0$-saturated, let $c\in M$ realize $p$. Then $c\equiv f(c)$.
\end{proof}

The set of $f$-invariant types forms a compact convex subset of $K_1(T)$.
Invariant types exist for Abelian groups of definable bijections too
(use Markov-Kakutani theorem (\cite{Conway} p.151).

\begin{proposition} \label{Herbrand1} \emph{(Affine Herbrand's theorem)}\index{Herbrand's theorem}
Let $T$ be a satisfiable universal theory admitting quantifier-elimination.
Let $\x=x_1\ldots x_m$ and $\y=y_1\ldots y_n$. Then for each formula $\phi(\x,\y)$
and $\epsilon>0$ there are tuples of $L$-terms $\t_1(\x),...,\t_k(\x)$ of length $n$
and $r_1,...,r_k\geqslant0$ with $\sum_ir_i=1$
such that $$T\vDash\sup_{\y}\phi(\x,\y)-\epsilon\leqslant\sum_{i=1}^kr_i\phi(\x,\t_i(\x))$$
\end{proposition}
\begin{proof} Let
$$\Gamma(\x)=T\cup\big\{\phi(\x,\t(\x))\leqslant\sup_{\y}\phi(\x,\y)-\epsilon:
\ \ \t(\x)\ \mbox{is\ a\ tuple\ of}\ L\mbox{-terms}\big\}.$$
We first show that $\Gamma(\x)$ is unsatisfiable. Suppose $(M,\a)$ is a model of $\Gamma(\a)$
and $N$ is the substructure of $M$ generated by $\a$.
Then, by assumptions of the proposition, $N\preccurlyeq M$. Let $r=\sup_{\y}\phi^N(\a,\y)$
and $\b\in N$ be such that $r-\frac{\epsilon}{2}\leqslant\phi^N(\a,\b)$.
We may assume without loss of generality that $b_i=t_i(\a)$.
We therefore have that
$$r-\frac{\epsilon}{2}\leqslant\phi^N(\a,\t(\a))=\phi^M(\a,\t(\a))
\leqslant\sup_{\y}\phi^M(\a,\y)-\epsilon=r-\epsilon.$$
This is a contradiction.

Now, since $\Gamma(\x)$ is not satisfiable, by affine compactness,
there are $\t_1(\x),...,\t_k(\x)$ and $r_1,...,r_k\geqslant0$ such that
$$T,\ \ \sum_{i=1}^kr_i\phi(\x,\t_i(\x))\leqslant\sum_{i=1}^k r_i(\sup_{\y}\phi(\x,\y)-\epsilon)$$
is unsatisfiable. Since $T$ is satisfiable, $r_i\neq0$ for at least one $i$.
So, multiplying the inequality by a suitable coefficient, we may further assume $\sum_i r_i=1$.
Hence,
$$T\vDash\sup_{\y}\phi(\x,\y)-\epsilon\leqslant\sum_{i=1}^kr_i\phi(\x,\t_i(\x)).$$
\end{proof}

\begin{corollary} \label{Herbrand2}
Let $T$ be a universal theory admitting quantifier-elimination. Then, every
parametrically definable $f:M^n\rightarrow M$ is approximated piecewise by terms,
i.e. given $\epsilon>0$, there are $L$-terms $t_1,...,t_k$ such that
$$\forall\a\in M\ \ \exists j\leqslant k\ \ \ \ \ \ d(f(\a),t_j^M(\a))<\epsilon.$$
\end{corollary}
\begin{proof} Given $\epsilon>0$, let $$|d(f(\x),y)-\phi(\x,y)|<\frac{\epsilon}{3}.$$
Applying Herbrand's theorem for $-\phi$ and $\frac{\epsilon}{3}$ we obtain terms
$t_1,...,t_k$ and $r_1,...,r_k\geqslant0$ with $\sum_ir_i=1$ such that
$$T\vDash\sum_{i=1}^kr_i\phi(\x,t_i(\x))\leqslant\inf_{\y}\phi(\x,y)+\frac{\epsilon}{3}.$$
So, $$T\vDash\sum_{i=1}^kr_i d(f(\x),t_i(\x))\leqslant\inf_{\y} d(f(\x),y)+\epsilon=\epsilon.$$
Now, for each $\a\in M$ there is a $j$ such that for all $i$,
$$d(f(\a),t^M_j(\a))\leqslant d(f(\a),t^M_i(\a)).$$
Therefore, $d(f(\a),t^M_j(\a))\leqslant\epsilon$.
\end{proof}
\bigskip

\subsection{Definable sets}
A closed $D\subseteq M^n$ is definable if $d(\x,D)=\inf_{\a\in D}d(\x,\a)$ is definable\index{definable set}.
We use the convention $\inf_{\a\in\emptyset}P(\a)=\|P\|$.
Definable sets are not closed under Boolean combinations. However, if $D,E$ are definable, then so are
$D\times E$ and $\{\x:\ \exists y\ \x y\in D\}$.

\begin{remark} \label{3conditions}
Let $D\subseteq M^n$ be definable and set $P(\x)=d(\x,D)$.
Then the following properties hold for every $\x,\y\in M^n$:

\emph{(i)} $0\leqslant P(\x)$

\emph{(ii)} $P(\x)-P(\y)\leqslant d(\x,\y)$

\emph{(iii)} $0\leqslant\inf_{\x}\sup_{\y}[sP(\x)-rP(\y)-s\hspace{.6mm} d(\x,\y)]$
\ \ \ \ \ \ \ \ $\forall r,s\geqslant0$.
\end{remark}

The inequality (iii) states that for each $\a$,
$\{P(\y)\leqslant0,\ d(\a,\y)\leqslant P(\a)\}$ is affinely approximately satisfiable in $M$.
The properties (i)-(iii) characterize definable sets.

\begin{proposition} \label{dfnconditions}
Let $M$ be extremally $\aleph_0$-saturated and $P:M^n\rightarrow\Rn$ be definable.
If $P$ satisfies (i)-(iii) above, then $P(\x)=d(\x,D)$ where $D=Z(P)=\{\a:P(\a)=0\}\neq\emptyset$.
\end{proposition}
\begin{proof} As stated in the previous subsection, $P$ can be regarded as a formula so that $(M,P)$ is extremally saturated.
Then, taking $s=0$, $r=1$ in (iii) and using extremal saturation,
one checks that $D$ is nonempty. By (ii), one has that $P(\x)\leqslant d(\x,\y)$
for all $\y\in D$. Hence $P(\x)\leqslant d(\x,D)$. For the inverse inequality, fix $\a\in M$.
By (iii) (and using extremal saturation),
$$\{P(\y)\leqslant 0,\ d(\a,\y)\leqslant P(\a)\}$$ is affinely satisfiable in $M$.
By Remark \ref{facial}, this is a facial type since $P(\y)\leqslant0$ implies that $\y\in D$.
So, by (i) and extremal saturation, there exists $\b$ such that
$$P(\b)=0, \ \ d(\a,\b)\leqslant P(\a).$$
Therefore, $d(\a,D)\leqslant d(\a,\b)\leqslant P(\a)$
and hence $d(\x,D)\leqslant P(\x)$ for all $\x\in M$.
\end{proof}

\begin{proposition} \label{dfnrestriction}
Let $M$ be extremally $\aleph_0$-saturated and $M\preccurlyeq N$.
If $D\subseteq N^n$ is definable, then $C=D\cap M^n$ is definable and for each
$\x\in M$, $d(\x,D)=d(\x,C)$. In particular, $(M,d(\x,C)\preccurlyeq(N,d(\x,D))$.
If $D\neq\emptyset$ then $C\neq\emptyset$.
\end{proposition}
\begin{proof} By Proposition \ref{definitional expansion}, $Q(\x)=d(\x, D)|_{M^n}$ is definable in $M$
and $(M,Q)\preccurlyeq (N,d(\x,D))$.
Note that $Q$ satisfies conditions (i)-(iii) in Remark \ref{3conditions}.
So, since the zeroset of $Q$ is $C$, by Proposition \ref{dfnconditions}, $Q(\x)=d(\x,C)$.
For the last part, use the fact that $\inf_xd(x,D)<1$.
\end{proof}

Similarly, if $M\preccurlyeq N$ and $N$ is extremally $\aleph_0$-saturated, one promotes a definable
$C\subseteq M^n$ to a definable $D\subseteq N^n$ such that $C=D\cap M^n$.

Assume $D\subseteq M^n$ and $$P(\x,\y)\leqslant P(\x,\z)+\lambda d(\z,\y)\ \ \ \ \ \ \ \forall \x,\y,\z.$$
Take the infimum first over $\y\in D$ and then over $\z\in M^n$ to obtain
$$\inf_{\y\in D} P(\x,\y)\leqslant \inf_{\z}[P(\x,\z)+\lambda d(\z,D)].$$
Allowing $\z\in D$, we see that
$$\inf_{\y\in D} P(\x,\y)=\inf_{\z}[P(\x,\z)+\lambda d(\z,D)].\ \ \ \ \ \ \ (*)$$

\begin{proposition} \label{dfnprojection}
$D\subseteq M^n$ is definable if and only if for each definable
$P:M^{m+n}\rightarrow\Rn$, the predicate $\inf_{\y\in D}P(\x,\y)$ is definable.
\end{proposition}
\begin{proof} For the if part, take $P=d(\x,\y)$. For the converse, use the equality $(*)$
above if $P$ is $\lambda$-Lipschitz.
If $P$ is arbitrary definable, let $$\phi_k^M(\x,\y)\stackrel{u}{\longrightarrow}P(\x,\y).$$
Then, verify that $$\inf_{\y\in D}\phi_k^M(\x,\y)\stackrel{u}{\longrightarrow}\inf_{\y\in D}P(\x,\y).$$
\end{proof}

In particular, if $\a$ is definable, every predicate which is definable using $\a$ as parameters
is $\emptyset$-definable. Also, if $f$ is a definable function and $D$ is a definable set then $f(D)$
is definable: $$d(x,f(D))=\inf_{t\in D} d(x,f(t)).$$

\begin{corollary} \label{restriction}
Assume $M\preccurlyeq N$,\ $D\subseteq N^m$ is definable and $d(\x,D)|_M=d(\x,C)$ where $C\subseteq M^m$.
Then, for each definable predicate $P:N^{n+m}\rightarrow\Rn$ and $\x\in M^n$ one has that\ \
$\inf_{\y\in D}P(\x,\y)=\inf_{\y\in C}P|_M(\x,\y)$. In particular, $C$ and $D$ have the same diameter.
\end{corollary}
\begin{proof}
For the first part consider Lipschitz and non Lipschitz cases as in the proof of the preceding proposition.
Also, the diameter of $D$ is obtained by $\sup_{\x\y\in D}d(\x,\y)$.
\end{proof}

\begin{proposition} \label{dfn3conditions}
For a closed $D\subseteq M^n$ the following are equivalent:

\emph{(i)} $D$ is definable.

\emph{(ii)} There exists a definable predicate $P:M^n\rightarrow\Rn^+$ such that
$$\forall\x\in D,\ P(\x)=0\ \ \ \ \ \ \mbox{and}
\ \ \ \ \ \ \forall\x\in M^n, \ d(\x,D)\leqslant P(\x).$$

\emph{(iii)} For each $k$ there exists a definable predicate $P_k:M^n\rightarrow\Rn$ such that
$$\forall\x\in D,\ P_k(\x)\leqslant0\ \ \ \ \ \ \mbox{and} \ \ \ \ \ \ \forall\x\in M^n, \ d(\x,D)\leqslant P_k(\x)+\frac{1}{k}.$$

\end{proposition}
\begin{proof} (i)$\Rightarrow$(ii): Take $P(\x)=d(\x,D)$.

(ii)$\Rightarrow$(iii): Take $P_k=P$.

(iii)$\rightarrow$(i): For each $k$ set $$Q_k(\x)=\inf_{\y}\ [d(\x,\y)+P_k(\y)].$$
We then have that
$$Q_k(\x)\leqslant\inf_{\y\in D}\ [P_k(\y)+d(\x,\y)]\leqslant\inf_{\y\in D}d(\x,\y)=d(\x,D).$$
On the other hand, $d(\x,D)\leqslant d(\x,\y)+d(\y,D)$.
So, using the assumption,
$$d(\x,D)-\frac{1}{k}\leqslant\inf_{\y}\ [d(\x,\y)+d(\y,D)-\frac{1}{k}]\leqslant Q_k(\x).$$
We conclude that $d(\x,D)$ is the uniform limit of $Q_k(\x)$, hence definable.
\end{proof}

\begin{corollary}\label{first order definable}
If $M$ is first order, $D\subseteq M^n$ is affinely definable if and only if it is the zero-sets of
a definable predicate $P:M^n\rightarrow\Rn^+$ such that $$\inf\{P(\a):\ \ \a\in M,\ 0<P(\a)\}=r>0.$$
\end{corollary}
\begin{proof}
Let $P$ be as above.
Then $$d(\x,Z(P))\leqslant\frac{n}{r}\ P(\x)\ \ \ \ \ \ \ \forall\x\in M.$$
So, by Part (ii) of Proposition \ref{dfn3conditions}, $Z(P)$ is definable in $M$.
The converse is obvious.
\end{proof}

\begin{lemma} \label{zerosets}
Let $P,Q:M^n\rightarrow\Rn^+$ be definable where $M$ is extremally $\aleph_0$-saturated.
Then, $Z(P)\subseteq Z(Q)$ if and only if for each $\epsilon>0$ there is $\lambda\geqslant0$
such that for all $\x\in M$ one has that $Q(\x)\leqslant\lambda P(\x)+\epsilon$.
\end{lemma}
\begin{proof}
We prove the non-trivial part. First assume $M$ is $\aleph_0$-saturated.
Assume the claim does not hold. So, there exists $\epsilon>0$ such that
the set $$\{Q(\x)\geqslant\lambda P(\x)+\epsilon:\ \ \lambda\geqslant0\}$$
is satisfiable in $N$ by say $\bar{c}\in N$. We must therefore have that $P(\c)=0$ and hence
$Q(\c)=0$ which is impossible.

Now, assume $M$ is extremally $\aleph_0$-saturated and let $M\preccurlyeq N$ where $N$ is $\aleph_0$-saturated.
We have only to show that $Z(P^N)\subseteq Z(Q^N)$. Assume not.
Then, for some $r>0$, there is $b\in N$ such that $P^N(\b)=0$ and $r\leqslant Q^N(\b)$.
We may assume $r$ is the biggest real number with this property.
So, $$T,P(\x)\leqslant0\vDash Q(\x)\leqslant r.$$
We conclude that $\{P(\x)\leqslant0,\ r\leqslant Q(\x)\}$ is a satisfiable facial type,
hence satisfiable in $M$. This is a contradiction.
\end{proof}

The following corollary gives a simpler condition for definability of an end-set if the model is extremally saturated.

\begin{corollary} \label{satzeroset}
Let $M$ be extremally $\aleph_0$-saturated and $P:M^n\rightarrow\Rn^+$ be definable.
Then $D=Z(P)$ is definable if and only if for each $\epsilon>0$,
there exists $\lambda\geqslant0$ such that $$d(\x,D)\leqslant\lambda P(\x)+\epsilon\ \ \ \ \ \ \ \ \forall\x\in M.$$
\end{corollary}
\begin{proof} If $D$ is definable, the mentioned condition holds by Lemma \ref{zerosets} since
$P$ and $d(\x,D)$ have the same zeroset. Conversely, assume the above condition holds.
For each $k$ take $\lambda_k$ such that $$d(\x,D)\leqslant\lambda_k P(\x)+\frac{1}{k}.$$
Then, part (iii) of Proposition \ref{dfn3conditions} holds for the predicate $P_k=\lambda_k P(\x)$.
Hence, $D$ is definable.
\end{proof}

In general, zerosets have little chance to be definable.
The following proposition shows that in big models, type-definable sets are either trivial or big.

\begin{proposition} \label{dfncompact}
Let $M$ be $\aleph_0$-saturated and $D\subseteq M^n$ a nonempty compact type-definable set.
Then $D$ is a singleton.
\end{proposition}
\begin{proof} Assume $n=1$. Let $D$ be type-defined by $\Gamma(x)$ and $d(a,b)=r>0$ for $a,b\in D$.
We first show that the partial type
$$\Sigma=\{\frac{r}{2}\leqslant d(x_i,x_j)\ : \ \ \  i<j<\omega\}$$
is affinely realized in the set $\{a,b\}$. Take a condition
$$\ \ \ \ \ \ \ A_n=\frac{r}{2}\sum_{i<j\leqslant n}\alpha_{ij}\leqslant\sum_{i<j\leqslant n}\alpha_{ij} d(x_i,x_j)
=\sigma_{n}(x_0,...,x_{n})\ \ \ \ \ \ \ \ \alpha_{ij}\geqslant0$$
and assume by induction that
$A_{n-1}\leqslant\sigma_{n-1}^M(e_0,...,e_{n-1})$ where $e_i$ is either $a$ or $b$.
Verify that for one of $e_n=a$ and $e_n=b$ we must have that
$$\frac{r}{2}\sum_{i<n}\alpha_{in}\leqslant\sum_{i<n}\alpha_{in}d(e_i,e_n).$$
Hence $$A_n=A_{n-1}+\frac{r}{2}\sum_{i<n}\alpha_{in}\leqslant\sigma^M_{n-1}(e_0,...,e_{n-1})+\sum_{i<n}\alpha_{in}d(e_i,e_n)=\sigma_n^M(e_0,...,e_n).$$
We have just shown that $\Sigma\cup\Gamma(x_1)\cup\Gamma(x_2)\cup\cdots$
is affinely satisfiable in $M$.
By $\aleph_0$-saturation, any finite part of $\Sigma$ is satisfied by elements of $D$.
This contradicts the compactness of $D$.
\end{proof}

In particular, the notion of algebraic closure defined in CL or first order logic is meaningless in AL.
We can however define definable closure of a set.
The notion definable set over a set $A\subseteq M$ of parameter is defined in the usual way.
A tuple $\a\in M^n$ is said to be $A$-definable if $d(\x,\a)$ is $A$-definable.
As stated above, the projection of a definable set is definable.
So, if $\a$ is $A$-definable, then every $a_i$ is $A$-definable.
Conversely, if every $a_i$ is $A$-definable, then
$$d(\x,\a)=\sum_{i=1}^n d(x_i,a_i)$$ which shows that $\a$ is $A$-definable.
For $A\subseteq M$,\ \ $\textrm{dcl}_M(A)$ denotes the set of points which are $A$-definable.
Clearly, it is topologically closed.

\begin{proposition} \label{well-dcl}
Let $A\subseteq M\preccurlyeq N$. Then $dcl_M(A)=dcl_N(A)$.
\end{proposition}
\begin{proof} It is sufficient to prove the claim for the case where $N$ is $\aleph_0$-saturated.
There is no harm if we further assume $A=\emptyset$. Let $a\in dcl_M(\emptyset)$.
The unique definable extension of $d(x,a)$ to $N$ satisfies the conditions (i)-(iii) of Remark \ref{3conditions}.
So, for some definable $D\subseteq N$ one has that $d(x,D)|_M=d(x,a)$.
By Corollary \ref{restriction}, $D=\{a\}$ and hence $a\in dcl_N(\emptyset)$.
Conversely assume $a\in dcl_N(\emptyset)$. Let $P(x)= d(x,a)|_M$.
Then $(M,P)\preccurlyeq (N,d(\cdot,a))$
and hence $$\inf_{x\in M} P(x)=\inf_{x\in N}d(x,a)=0.$$
For each $k$ take $a_k\in M$ such that $0\leqslant P(a_k)\leqslant\frac{1}{k}$.
Then, $d (a_k,a)=P(a_k)\leqslant\frac{1}{k}$ which means that $a_k\rightarrow a$.
Therefore, $a\in M$ and $P(a)=0$.
We have also that $$d^M(x,a)=d^N(x,a)=P(x) \ \ \ \ \ \ \ \forall x\in M$$
which shows that $a$ is definable in $M$.
\end{proof}

So, $\textrm{dcl}_M(A)$ does not depend on $M$ and we may simply denote it
by $\textrm{dcl}(A)$. The following properties are also proved easily:

\indent 1. $A\subseteq \textrm{dcl}(A)$.\\
\indent 2. If $A\subseteq\textrm{dcl}(B)$ then $\textrm{dcl}(A)\subseteq\textrm{dcl}(B)$.\\
\indent 3. If $a\in\textrm{dcl}(B)$ then $a\in\textrm{dcl}(A)$ for some countable $A\subseteq B$.\\
\indent 4. If $A$ is a dense subset of $B$ then $\textrm{dcl}(A)=\textrm{dcl}(B)$.\\
\indent 5. If $h:M^n\rightarrow M$ is $A$-definable and $\a\in\textrm{dcl}(A)$ then $h(\a)\in\textrm{dcl}(A)$.

\subsection{Principal types}
Despite full continuous logic, it is not true that if the logic and metric topologies coincide at a type $p$,
then $p$ is realized in every model.
For example, for the theory of probability algebras, $K_1(\textrm{PrA})=[0,1]$
and the two topologies coincide.
However, only the extreme types are realized in the model $\{0,1\}$.
For a complete type $p(\x)$ set $$p(M)=\{\a\in M^n:\ tp(\a)=p\}.$$

\begin{proposition} \label{principal type}
Assume $p(M)$ is nonempty definable for some $M\vDash T$.
Then $p(N)$ is nonempty definable for any $N\vDash T$ which is extremally $\aleph_0$-saturated.
\end{proposition}
\begin{proof} First, suppose that $M\preccurlyeq N$ where $N$ is $\aleph_0$-saturated.
Let $P(\x)=d(\x,p(M))$. So, $(M,P)\preccurlyeq (N,P^N)$
and $P^N$ satisfies the conditions (i)-(iii) of Remark \ref{3conditions}.
Hence $P^N(\x)=d(\x,D)$ where $D$ is the zeroset of $P^N$. We show that $D=p(N)$.

Take a condition $\phi(\x)\leqslant0$ in $p(\x)$.
For each $\a\in M$ and $\b\in p(M)$
$$\phi^M(\a)\leqslant\phi^M(\a)-\phi^M(\b)\leqslant\lambda_{\phi}\ d(\a,\b).$$
So, $$\phi^M(\a)\leqslant\lambda_{\phi}\ d(\a,p(M))
=\lambda_{\phi}\ P(\a) \ \ \ \ \ \ \ \forall \a\in M$$
and hence $$\phi^N(\a)\leqslant\lambda_{\phi}P^N(\a)\ \ \ \ \ \ \ \forall \a\in N.$$
In particular, $\phi^N(\a)\leqslant0$ for each $\a$ with $P^N(\a)=0$.
We conclude that $D\subseteq p(N)$ .
For the reverse inclusion, assume $M\vDash |\phi_k(\x)-P(\x)|\leqslant\frac{1}{k}$ for each $k$.
Since $p(M)$ is nonempty, $-\frac{1}{k}\leqslant\phi_k(\x)\leqslant\frac{1}{k}$ must belong to $p(\x)$.
Therefore, for any $k$ and $\b\in p(N)$
$$0\leqslant P^N(\b)\leqslant|P^N(\b)-\phi_k^N(\b)|+|\phi_k^N(\b)|\leqslant\frac{2}{k}.$$
This shows that $P^N(\b)=0$ for each $\b\in p(N)$.

Now assume $N$ is extremally $\aleph_0$-saturated and take an $\aleph_0$-saturated $K$ such that $M\preccurlyeq K$ and $N\preccurlyeq K$.
Then, $p(K)$ is definable. So, $p(N)=p(K)\cap N^n$ is definable by Proposition \ref{dfnrestriction}.
\end{proof}

A type $p(\x)$ is called \emph{principal}\index{principal type} if $p(M)$ is nonempty definable for some $M\vDash T$.
Every principal type $p$ is extreme since it is exposed by $d(\x,p(M))$.
A consequence of Proposition \ref{principal type} is that if $p(\x,\y)$ is principal then so is $q(\x)=p|_{\x}$.
In fact, $q(M)$ is the projection on $p(M)$ on $M^n$ if $M$ is $\aleph_0$-saturated and $|\x|=n$.

As stated before, if $P$ is a definable predicate we may treat it as a formula in the
definitional expansion of $T$ to $L\cup\{P\}$.
Also, every type $p:\mathbb{D}_n(T)\rightarrow\Rn$ has a natural extension to $\mathbf{D}_n(T)$
so that $p(Q)=\lim p(\phi_k)$ whenever $\phi_k\stackrel{u}\longrightarrow Q$.
So, $p(\x)\vDash P(\x)\leqslant0$ if and only if for each $\a\in M\vDash T$,\
$\a\vDash p$ implies that $P(\a)\leqslant0$.
The following proposition states that $p$ is principal if and only if logic and metric topologies coincide at $p$.

\begin{proposition} \label{princip}
Let $p\in K_n(T)$. Then the following are equivalent:

\emph{(i)} $p$ is principal

\emph{(ii)} For each $k$ there is a definable predicate $P_k(\x)$ such that
$$T\vDash0\leqslant P_k(\x),\ \ \ \ \ \ p(\x)\vDash P_k(\x)\leqslant0\ \ \ \ \textrm{and}
\ \ \ \ \ [P_k<1]\subseteq B(p,\frac{1}{k}).$$
\end{proposition}
\begin{proof} Let $M$ be $\aleph_0$-saturated.

(i)$\Rightarrow$(ii): The requirement holds with $P_k(\x)=kd(\x,p(M))$.

(ii)$\Rightarrow$(i):
For each $\a\in M$, we have either $P_k(\a)<1$ or $1\leqslant P_k(\a)$.
In the first case one has that $d(tp(\a),p)<\frac{1}{k}$.
So, by saturation $$d(\a,p(M))<\frac{1}{k}\leqslant P_k(\a)+\frac{1}{k}.$$
In the second case, $d(\a,p(M))\leqslant1\leqslant P_k(\a)$.
So, for any $\a\in M$, $d(\a,p(M))\leqslant P_k(\a)+\frac{1}{k}$.
We conclude by part (iii) of Proposition \ref{dfn3conditions} that $p(M)$ is definable.
\end{proof}

\subsection{Compact models} \label{Definability in compact models}
In this subsection we assume $T$ has a compact model.
In the framework of CL, if $M$ is compact, zerosets of $\emptyset$-definable predicates are definable and
if the language is countable, they are the only type-definable sets.
The situation is different in AL.
In a compact model, a type-definable set need not be a zeroset and a zeroset need not be definable.
Moreover, definable sets are exactly the end-sets of definable predicates.
By an \emph{end-set} we mean a set of the form $\{\x:\ P(\x)=r\}$
where $r$ is either $\inf_{\x}P(\x)$ or $\sup_{\x}P(\x)$.

\begin{theorem} \label{minimal sets}
Let $M\vDash T$ be extremally $\aleph_0$-saturated.
Then a nonempty $D\subseteq M^n$ is definable if and only if there is a definable
$P:M^n\rightarrow\Rn^+$ such that $D=Z(P)$.
\end{theorem}
\begin{proof} We prove the non-trivial part.
Let $P$, $D$ be as above and consider the case $n=1$.
First assume $M$ is compact and $\epsilon>0$ is fixed.
Then there must exist $\lambda\geqslant0$ such that
$d(x,D)\leqslant\lambda P(x)+\epsilon$ for all $x\in M$.
Otherwise, for each $\lambda\geqslant0$, the set
$$X_\lambda=\{x\ : \ d(x,D)\geqslant\lambda P(x)+\epsilon\}$$
is nonempty closed. Since $M$ is compact,
there exists $b\in\cap_\lambda X_\lambda$. Clearly then $P(b)=0$ and hence
$b\in D$ and $d(b,D)\geqslant\epsilon$. This is a contradiction.
Therefore, by Corollary \ref{satzeroset}, $D$ is definable.

Now, assume $M$ is $\aleph_0$-saturated.
An easy back and forth argument shows that every compact model of $T$ can be elementarily embedded in $M$.
Let $K\preccurlyeq M$ where $K$ is compact. We then have that $(K,P^K)\preccurlyeq(M,P)$.
Moreover, $D_0=Z(P^K)$ is nonempty definable in $K$.
By Propositions \ref{dfnconditions} and \ref{definitional expansion},
for some definable $D_1\subseteq M$
$$(K,P^K,d(x,D_0))\preccurlyeq(M,P,d(x,D_1)).$$
Since $P^K$ and $d(x,D_0)$ have the same zeroset, by Lemma \ref{zerosets},
$P$ and $d(x,D_1)$ must have the same zeroset. We conclude that $D=D_1$.

Finally, assume $M$ is just extremally $\aleph_0$-saturated. Let $M\preccurlyeq N$ where $N$
is $\aleph_0$-saturated. Then $Z(P^N)$ is definable. Hence $Z(P)=Z(P^N)\cap M^n$
is definable by Proposition \ref{dfnrestriction}.
\end{proof}

We proved in subsection \ref{omitting types} that a theory having a compact model has a unique compact extremal model.
Such a model is elementarily embedded in every extremally $\aleph_0$-saturated model of the theory.
Proof of the above theorem could be then shortened a bit by using this fact.

A consequence of Theorem \ref{minimal sets} is that in an extremally $\aleph_0$-saturated model $M$
if $f:M^n\rightarrow M^m$ and $D\subseteq M^m$ are definable then $f^{-1}(D)$ definable.
It is the zeroset of $d(f(\x),D)$.
Also, writing $P(x)=\sum_k 2^{-k}d(x,D_k)$ one checks that countable intersections of definable sets are definable.
Definable sets are not closed under finite unions. We have however the following.

\begin{proposition} \label{union}
Let $D_1\subseteq D_2\subseteq\ldots$ be a chain of definable sets in $M$.
Assume $M$ contains a compact elementary submodel. Then $D=\overline{\cup_n D_n}$ is definable.
\end{proposition}
\begin{proof} It is clear that $d(\x,D_n)$ converges to $d(\x,D)$ pointwise.
Assume $K\preccurlyeq M$ where $K$ is compact. Let $P_n(\x)=d(\x,D_n)|_K$ and $P(\x)=d(\x,D)|_K$.
Then $P_n$ is monotone and converges to $P$ pointwise.
Since $P$ is continuous, the convergence is uniform (by Dini's theorem). For each $\epsilon>0$ take $\ell$
such that
$$|P_m(\x)-P_n(\x)|\leqslant\epsilon\ \ \ \ \ \ \forall m,n\geqslant \ell,\ \ \forall \x\in K.$$
Then we must similarly have that
$$|d(\x,D_m)-d(\x,D_n)|\leqslant\epsilon\ \ \ \ \ \ \forall m,n\geqslant \ell,\ \ \forall \x\in M.$$
This shows that the convergence of $d(\x,D_n)$ to $d(\x,D)$ is uniform.
\end{proof}

It is proved in \cite{Ibarlucia} that if $T$ has a compact model, its unique extremal model is prime, i.e. embeds in every model of $T$.
So, the assumption of the above proposition is superfluous.

As in the case of formulas, $\hat Q(p)=p(Q)$ is an affine logic-continuous function on $K_n(T)$.
A partial type $\Sigma(\x)$ is exposed if the set $$[\Sigma]=\{p\in K_n(T):\ \Sigma\subseteq p\}$$
is a face exposed by $\hat Q$ for some definable predicate $Q$.

\begin{proposition} \label{exposed types}
Let $M$ be $\aleph_0$-saturated and $\Sigma(\x)$ be a partial type. Then the following are equivalent:

\emph{(i)} $\Sigma(M)=\{\a: \ \a\vDash\Sigma(\x)\}$ is definable

\emph{(ii)} $[\Sigma]$ is either $K_n(T)$ or an exposed face

\emph{(iii)} There exists a definable predicate $Q(\x)$ such that
$$T\vDash0\leqslant Q(\x) \ \ \ \ \ \ \& \ \ \ \ \
\Sigma\equiv\{Q(\x)=0\}.$$
\end{proposition}
\begin{proof} (i)$\Rightarrow$(ii) Let $Q(\x)=d(\x,\Sigma(M))$.
We show that for each $p$,\ \ $\Sigma\subseteq p$ if and only if $p(Q(\x))=0$.
Fix $\a\vDash p$. If $\Sigma\subseteq p$ then $p(Q)=Q(\a)=0$.
Conversely, if $p(Q)=0$, then $Q(\a)=p(Q)=0$ and hence $\a\vDash\Sigma$.
This implies (by saturation of $M$) that $\Sigma\subseteq p$.
Now, one has that $$[\Sigma]=\{p\in K_n(T):\ \hat Q(p)=0\}.$$

(ii)$\Rightarrow$(iii): If $[\Sigma]=K_n(T)$, the required conditions hold with $Q=0$.
Otherwise, there exists $Q(\x)\in\mathbf{D}_n(T)$ such that $\hat Q$ is nonnegative nonconstant on $K_n(T)$
and $$[\Sigma]=\{p\in K_n(T)\ :\ \hat{Q}(p)=0\}.$$
In this case, the required conditions hold with $Q(\x)$.

(iii)$\Rightarrow$(i): The assumption implies that $\Sigma(M)=Z(Q)$. By Theorem \ref{minimal sets}, this is a definable set.
\end{proof}

As a consequence, complete principal types are exactly the exposed ones.
Also, regarding Theorem \ref{minimal sets}, there is an order preserving correspondence
between nonempty definable sets in $M^n$ and exposed faces of $K_n(T)$
so that bigger sets correspond to bigger exposed faces.
\bigskip

\begin{example} \label{Examples1}
\em{In an extremally $\aleph_0$-saturated structure finite or countable intersections
of sets determined by equations $t_1(\x)=t_2(\x)$, where $t_1,t_2$ are terms, are definable.
In an extremally $\aleph_0$-saturated metric group (assuming the operations are Lipschitz),
closure of the torsion subgroup is definable.
In a dynamical system $(M,f)$ (assuming $f$ is Lipschitz), the closure of the set of
periodic points, $\overline{\textrm{per}(f)}$, is definable.
In the sphere $\mathbb{S}^2$ equipped with the geodesic metric, every point
is defined by means of its antipode as parameter.
Also, all arcs of length less that $\pi$ (and hence by Proposition
\ref{union} all arcs of length at most $\pi$) are definable with parameters.
In the closed unit disc, the boundary as well as the center is definable.
Some interesting subsets of acute triangles (with the Euclidean metric) are definable.
For example, the circumcenter and centroid are definable.}
\end{example}

A model $M$ is called \emph{minimal}\index{minimal model} if it has no proper elementary submodel.

\begin{lemma}
Every compact model $M$ has a minimal elementary submodel.
\end{lemma}
\begin{proof}
Let $\{N_i\}_{i\in I}$ be a maximal decreasing elementary chain of submodels of $M$.
In fact, $I$ is countable and there is no harm if we assume $I=\Nn$ (take a cofinal subset).
Let $N=\cap_nN_n$.
We show by induction on the complexity of formulas that $\phi^N=\phi^M$ for every $\phi$ with parameters in $N$.
The atomic and connective cases are obvious. Consider the case $\sup_x\phi(x)$.
Since the chain is elementary, there are $r,s$ such that for all $n$
$$r=\sup_x\phi^N(x)\leqslant\sup_x\phi^{N_n}(x)=s.$$
If $r<s$, there exists $a_n\in N_n$ such that $\frac{r+s}{2}\leqslant\phi^{N_n}(a_n)$.
Then $(a_n)$ has a subsequence converging to say $a$.
Then, $a\in N$ and $\frac{r+s}{2}\leqslant\phi^M(a)$. This is a contradiction.
We have therefore that $N\preccurlyeq M$ and $N$ is minimal.
\end{proof}
\vspace{1mm}

\begin{example}\label{sphere}
\em{(i) We show that the unit circle $\mathbf{S}\subseteq\Rn^2$ equipped with the Euclidean metric
is minimal (temporarily, we allow metric spaces of diameter 2).
Let $M\preccurlyeq\mathbf{S}$ be minimal and $a\in M$. Then $d(x,a)$ is maximized by $-a$.
So, $-a\in M$. Also, there is $b\in M$ which maximizes $d(x,a)+d(x,-a)$ (with value $2\sqrt{2}$).
So, $-b\in M$. Generally, for each distinct $a,b\in M$, the midpoints of the arcs joining $a$ to $b$
are obtained as the maximizer of $d(x,a)+d(x,b)$ and its antipode.
We conclude that $M$ contains a dense subset of $\mathbf{S}$.
Since $M$ is closed, we must have that $M=\mathbf{S}$.
\vspace{1mm}

\noindent(ii) Similar arguments show that the unit disc $\mathbf B\subseteq\Rn^2$ as well as
the unit sphere $\mathbf{S}^2\subseteq\Rn^3$ equipped with the Euclidean metrics are minimal.
\vspace{1mm}

\noindent(iii) $[0,1]$ with the Euclidean metric is not minimal.
A minimal submodel must contain $\{0,1\}$ and $\{\frac{1}{2}\}$ (as the end-sets of $\sup_y d(x,y)$).
We show that $M=\{0,\frac{1}{2},1\}$ is the minimal elementary submodel of $[0,1]$.
For this purpose, we must prove that every formula $\phi(x,\a)$, where $\a\in M$, takes its maximum in $M$.
Indeed, we only need to prove this for every linear combination of the formulas
$$d(x,0),\ \ d(x,1),\ \ d(x,\frac{1}{2}),\ \ \sup_yd(x,y).$$
Note that the first two formulas are linearly dependent and the fourth formula is
equivalent to $\frac{1}{2}+d(x,\frac{1}{2})$.
Therefore, we have only to verify the claim for any formula of the form $rd(x,0)+sd(x,\frac{1}{2})$.
The value of this function for $x$ in $[0,\frac{1}{2}]$ or in $[\frac{1}{2},1]$
is of the form $rx+s$ 
which takes its maximum in $\{0,1\}$ in any case.
\vspace{1mm}

\noindent(iv) Let $C$ be the Cantor ternary set with the Euclidean metric and $M$ be the minimal elementary submodel of $C$.
Then, $M$ contains $\{0,1\}$ and $\{\frac{1}{3},\frac{2}{3}\}$ (as the end-sets of $\sup_y d(x,y)$).
We claim that $M=\{0,\frac{1}{3},\frac{2}{3},1\}$.
As in example (iii), we must verify that every linear combination of the formulas $\sup_yd(x,y)$
and $d(x,a)$, where $a\in M$, takes its maximum in $M$.
One verifies that such a function takes its maximum in $\{0,1\}$ in both cases $x\in[0,\frac{1}{3}]$ and $x\in[\frac{2}{3},1]$.
\vspace{1mm}

\noindent(v) A minimal dynamical system $(M,f)$, where $f$ is Lipschitz, is minimal in the present sense.}
\end{example}

\begin{corollary} \label{minimal model}
Assume $L$ is countable and $T$ has a compact model. Then $M\vDash T$ extremal
if and only if it is compact and minimal. Moreover, $T$ has a unique such model.
\end{corollary}
\begin{proof}
Assume $M$ is separable extremal and $N\vDash T$ is compact minimal.
We show that $M\simeq N$. Let $\{a_0,a_1,...\}\subseteq M$ be dense.
Let $f_{-1}=\emptyset$ and assume a partial elementary map $f_i:\{a_0,...,a_i\}\rightarrow N$ is defined.
Since $(a_0,...,a_{i+1})$ is extreme, $p(x)=tp(a_{i+1}/a_0...a_i)$ is extreme.
So, its shift $f_i(p)$ is an extreme type over $f_i(a_0)... f_i(a_i)$ realized by say $b\in N$.
Therefore, $f_{i+1}=f_i\cup\{(a_{i+1},b)\}$ is elementary.
Let $f$ be the continuous extension of $\cup_i f_i$ to all $M$.
Then, $f:M\rightarrow N$ is elementary. Since $N$ is minimal, one has that $M\simeq N$.
It is now clear by Corollary \ref{compact-realization1} that every $E_n(T)$ is closed.
Since, $T^{ex}$ is CL-complete, $M=N$ is its unique model. In particular, every extremal model
of $T$ is separable and the claim is proved.
\end{proof}

An other description of the compact minimal model is the following.

\begin{proposition}
Let $N\vDash T$ be compact and $A\subseteq N$ be extreme.
Then, there is an extremal elementary submodel $A\subseteq M\preccurlyeq N$
such that every $a\in N-M$ is non-extreme over $M$. This is the unique (up to isomorphism) extremal model of $T$.
\end{proposition}
\begin{proof}
Construct an increasing chain $A_i$,\ \ $i\in ord$, of extremal subsets of $N$ as follows.
Let $A_0=A$ and assume $A_i$ is defined. If $N-A_i$ does not contains any extreme point over $A_i$, set $A_{i+1}=A_i$.
Otherwise, choose $a\in N-A_i$ extreme over $A_i$. Then, $A_{i+1}=A_i\cup\{a\}$ is extreme.
Also, for limit $i$ set $A_i=\cup_{j<i} A_j$.
There is a least $i$ such that $A_i=A_{i+1}$. Clearly, $M=A_i\neq\emptyset$.
Assume a sequence $b_k\in M$ tends to $b\in N$. Then, $\frak{tp}(b/M)$ is extreme
and $b$ is its unique realization. Therefore, $b\in M$ which shows that $M$ is closed.
Indeed, $M$ is a substructure. We show that $M\preccurlyeq N$.
Let $\phi(x)$ be a formula with parameters in $M$.
Then, the map $\hat\phi: S_1(M)\rightarrow\Rn$, \ $\hat\phi(p)=p(\phi)$, is affine continuous.
Hence, it takes its maximum in an extreme type $p\in S_1(M)$.
Since $N$ is compact, $p$ is realized in it. By definition, $M$ contains a realization $b$ of $p$
and $p(\phi)=\phi(b)=\sup_x\phi^N(x)$. Also, every $a\in N-M$ is non-extreme over $M$.
\end{proof}

The minimal model is not unique inside a compact model. For example,
consider the unit circle $\mathbf{S}$ with the Euclidian metric and let $N=\frac{1}{2}\mathbf{S}+\frac{1}{2}\mathbf{S}$.
Then, $\{(x,x)\in x\in\mathbf{S}\}$ and $\{(x,-x): x\in\mathbf{S}\}$, where $-x$ is the antipode of $x$,
are isometric to $\mathbf{S}$ hence minimal.

By maximality of compact Hausdorff topologies, if $T$ has a compact model,
metric and logic topologies coincide on $E_n(T)$, since the map $\a\mapsto tp(\a)$ (defined on the minimal model) is Lipschitz.

\begin{proposition}
The compact minimal $M\vDash T$ is strongly homogeneous.
\end{proposition}
\begin{proof}
Let $\a,\b,c\in M$ be finite tuples and $\a\equiv\b$. Since $M$ is extremal, $tp(c/\a)$
and hence the type $$\{\phi(x,\b)=r:\phi^M(c,\a)=r\}$$ is extremal.
Assume it is realized by $e$. Then, $\a c\equiv\b e$.
Now, by separability and Proposition \ref{strong homogeneity}, $M$ is strongly homogeneous.
\end{proof}
Let us call a type $p\in K_n(T)$ \emph{compact}\index{compact type} if it is realized in a compact model.

\begin{proposition} \label{counatble extremes}
Assume $T$ has a compact model.

\emph{(i)} The set of compact types is dense in $K_n(T)$.

\emph{(ii)} If $E_n(T)$ is countable, then every $p\in K_n(T)$ is compact.

\emph{(iii)} If $M\vDash T$ is finite with at least two elements, then every
$p\in K_n(T)$ is realized in a finite model of the form $M^{\mu}$.
In particular, the number of non-isomorphic finite models of $T$ is $2^{\aleph_0}$.

\end{proposition}
\begin{proof}
(ii) We know that $I=E_n(T)$ is compact. Let $M\vDash T$ be compact and $p\in K_n(T)$.
By the Choquet-Bishop-de Leeuw theorem, there is a probability measure $\mu$
on $I$ which represents $p$. Clearly, $\mu$ is an ultracharge and $N=\prod_{\mu}M$ is compact.
For each $q\in I$, let $a_q\in M$ realize $q$ and $a=[a_q]\in N$. Then, for every $\phi$
$$p(\phi)=\hat\phi(p)=\int_I\hat\phi(q)\ d\mu=\int_I\phi^M(a_q)d\mu=\phi^N(a).$$

(i) $I=E_n(T)$ is a compact metric space. By the density theorem (\cite{Aliprantis-Inf} Th 15.10)
the set of probability measures on $I$ with finite support is dense in $\mathcal P(I)$ (the set of all probability measures on $I$).
Let $p\in K_n(T)$ be represented by $\mu\in\mathcal P(I)$.
Let $\mu_k$ converge to $\mu$ (also written by $\mu_k\stackrel{*}{\longrightarrow}\mu$)
where every $\mu_k\in\mathcal P(I)$ has finite support.
So, for every $f\in \mathbf{C}_b(M)$ one has that $\int fd\mu_k\rightarrow\int fd\mu$ (\cite{Aliprantis-Inf} Th. 15.3).
As in the proof of (ii), the type defined by $p_k(\phi)=\int\hat\phi d\mu_k$ is realized in some compact model.
Moreover, one has that $p_k\rightarrow p$ in the $w^*$-topology.

(iii) Every $E_n(T)$ is finite. So, as in the proof of (ii), every type is realized in a model of the form
$M^{\mu}$ for some ultracharge $\mu$ on $E_n(T)$.
Since every finite model realizes a countable number of types and $E_2(T)$ has cardinality $2^{\aleph_0}$,
we conclude that the number of non-isomorphic finite models has cardinality $2^{\aleph_0}$.
\end{proof}

\newpage\section{Examples} \label{Examples}
Abstract results are expected to lead to the better understanding of concrete examples.
In this section, we give some easy examples or non-examples in this respect.
More analytic examples can be found in \cite{Ibarlucia}.

\vspace{4mm}\noindent{\bf Affinic theories}\\
A natural way of finding examples is affinization.
Let $\mathbb{T}$ be a CL-theory. Recall that $\mathbb{T}_{\textrm{af}}$ is its affine part.
$\mathbb{T}$ is called \emph{affinic}\index{affinic} if for every $M\vDash\mathbb{T}_{\mathrm{af}}$
there exists $M\preccurlyeq N\vDash\mathbb{T}$. Every affine theory is affinic.
By proposition \ref{10}, the $\aleph_0$-saturated models of an affinely complete theory $T$ are CL-equivalent.
The common CL-theory of such models is then an affinic theory.

\begin{proposition} \label{13}
Assume $\mathbb{T}$ is CL-complete affinic and inductive. Then, $\mathbb{T}$ it is the
common theory of $\aleph_0$-saturated models of $\mathbb{T}_{\mathrm{af}}$.
\end{proposition}
\begin{proof}
By Proposition \ref{10}, we have only to find an $\aleph_0$-saturated model of $\mathbb{T}_{\mathrm{af}}$
which is also a model of $\mathbb{T}$. Let $M_0\vDash \mathbb{T}_{\mathrm{af}}$ be $\aleph_0$-saturated.
Then, there is a chain $$M_0\preccurlyeq M_1\preccurlyeq M_2\preccurlyeq\ \cdots$$
such that $M_i\vDash \mathbb{T}_{\mathrm{af}}$ is $\aleph_0$-saturated when $i$ is even and
$M_i\vDash\mathbb{T}$ when $i$ is odd. Let $M=\cup_i M_i$.
Since $\mathbb{T}$ is inductive, $M\vDash\mathbb{T}$.
On the other hand, $M\vDash \mathbb{T}_{\mathrm{af}}$ is $\aleph_0$-saturated.
\end{proof}

Affinic theories transfer their properties to their affine part.

\begin{proposition}\label{comp-to-lin}
Let $\mathbb{T}$ be an affinic theory.
If $\mathbb{T}$ has any of the following properties, then so has $\mathbb{T}_{\mathrm{af}}$.
(i) Quantifier-elimination (ii) model-completeness (iii) universal
(iv) existential  (v) amalgamation property (vi) joint embedding property.
\end{proposition}
\begin{proof}
(i): Use \ref{quant-el-loc} and the corresponding property in CL.

(ii) Take $A,B\vDash\mathbb{T}_{\mathrm{af}}$ such that $A\subseteq B$.
Let $A\preccurlyeq A'\vDash\mathbb{T}$ and $\Sigma=diag(A')\cup ediag(B)$.
Let $0\leqslant\phi^{A'}(\a,\a')$ where $\phi$ quantifier-free and $\a\in A$, $\a'\in A'-A$.
Then, $A'$ and hence $A$, $B$ satisfy $0\leqslant\sup_{\y}\phi(\a,\y)$.
So, $\Sigma$ is affinely satisfiable. Let $B_0\vDash\Sigma$.
Then $A'\subseteq B_0$ and $B\preccurlyeq B_0$.
Let $B_0\preccurlyeq B'\vDash\mathbb{T}$. So, $A'\subseteq B'$ and $A',B'\vDash\mathbb{T}$.
Since $\mathbb{T}$ is model-complete, $A'\preccurlyeq_{\mathrm{CL}} B'$. We conclude that $A\preccurlyeq B$.

(iii) Let $A\subseteq B\vDash \mathbb{T}_{\mathrm{af}}$. Let $B\preccurlyeq C\vDash\mathbb{T}$.
Then $A\vDash\mathbb{T}$ and hence $A\vDash \mathbb{T}_{\mathrm{af}}$.

(iv) Let $A\vDash\mathbb{T}_{\mathrm{af}}$ and $A\subseteq B$. We show that $B\vDash\mathbb{T}_{\mathrm{af}}$.
Let $A\preccurlyeq C\vDash\mathbb{T}$. Note that $ediag(B)\cup diag(C)$ is affinely satisfiable.
Assume $B\preccurlyeq D$ and $C\subseteq D$. Since $\mathbb{T}$ is existential, $D\vDash\mathbb{T}$ and hence $B\vDash \mathbb{T}_{\mathrm{af}}$.

(v) Assume $A\subseteq B$ and $A\subseteq C$ where $A,B,C\vDash \mathbb{T}_{\mathrm{af}}$.
Let $A\preccurlyeq A'\vDash\mathbb{T}$. Then, $diag(A')\cup ediag(B)$ is affinely satisfiable.
Let $B_0$ be a model of this theory. Then $A'\subseteq B_0$ and $B\preccurlyeq B_0$.
Let $B_0\preccurlyeq B'\vDash\mathbb{T}$.
Similarly, there exists $A'\subseteq C'$, $C\preccurlyeq C'\vDash\mathbb{T}$.
Assume $D\vDash\mathbb{T}$ is an amalgamation of $B'$ and $C'$ over $A'$.
Then $D$ is an amalgamation of $B$ and $C$ over $A$.

(vi) Similar to (v).
\end{proof}

In particular, if $T$ is an AL theory which has any of the above properties in the CL sense,
it has that property in the AL sense.

\begin{definition}
\em{An AL-complete (resp. CL-complete) theory $T$ is \emph{$\kappa$-stable} if for every
$A\subseteq M\vDash T$ with $|A|\leqslant\kappa$,\ \ $K_n(A)$ (resp. $S_n(A)$)
has density character at most $\kappa$ in the metric topology. $T$ is stable\index{stable}
if it is $\kappa$-stable for some $\kappa$.}
\end{definition}

\begin{proposition} \label{stability}
Let $\mathbb{T}$ a CL-complete affinic theory in $L$. If $\mathbb{T}$ is $\kappa$-stable,
then $\mathbb{T}_{\mathrm{af}}$ is $\kappa$-stable in the AL sense.
Moreover, if $\mathbb{T}$ has quantifier-elimination, then $\mathbb{T}$ and $\mathbb{T}_{\mathrm{af}}$
have the same type spaces over any set of parameters from models of $\mathbb{T}$.
\end{proposition}
\begin{proof} For the first part, take $A\subseteq M\vDash\mathbb{T}_{\mathrm{af}}$ where $M$ is
$\kappa^+$-saturated (in the AL sense) and $|A|\leqslant\kappa$.
Let $M\preccurlyeq N\vDash\mathbb{T}$ be $\kappa^+$-saturated in the CL sense. Then every $p\in K_n(A)$ is realized in $N$.
This shows that the map $$\textrm{af}:S_n(A)\rightarrow K_n(A)$$
which sends every CL-type to its restriction to affine formulas is surjective and $1$-Lipschitz.
Let $Z\subseteq S_n(A)$ be dense with cardinality $\kappa$.
Then $\textrm{af} [Z]\subseteq K_n(A)$ is dense. So, $\mathbb{T}_{\mathrm{af}}$ is $\kappa$-stable.

For the second part, note that every type $p\in S_n(A)$ is uniquely determined by its
values on the normal form formulas, i.e. those of the form $\theta=\bigvee_{i}\bigwedge_{j}\phi_{ij}$
where each $\phi_{ij}$ is a quantifier-free affine formula (see \cite{Aliprantis} for properties of vector lattices).
Since $p$ is a lattice homomorphism, it is uniquely determined by its values on the affine formulas.
We conclude that the map $\mathrm{af}$ is injective. Moreover, $p$ and $\textrm{af}(p)$ have the same
realizations in the model $N$ above. Therefore, $\textrm{af}$ is an isometry.
\end{proof}
\bigskip

\noindent{\bf Probability algebras}\\
Let $L=\{\wedge,\vee,\ ' ,0,1,\mu\}$ be the language of probability algebras.
It is easily seen that the interpretations of function symbols $\wedge, \vee,\ '$ and $\mu$ are Lipschitz.
The theory of probability algebras\index{probability algebras}, denoted by PrA, consists of the following set of axioms:
\bigskip

- Boolean algebra axioms

- $\mu(0)=0$ and $\mu(1)=1$

- $\mu(x)\leqslant\mu(x\vee y)$

- $\mu(x\wedge y)+\mu(x\vee y)=\mu(x)+\mu(y)$

- $d(x,y)=\mu(x\triangle y)$.
\bigskip

Adding a further non-affine axiom stating atomlessness, one obtains the theory of
atomless probability algebras APrA$=\mathbb{T}$ which is a complete $\aleph_0$-stable CL theory
admitting quantifier-elimination \cite{BBHU}.
Here, we show that similar properties hold for PrA as an AL-theory.
We may use the Proposition \ref{comp-to-lin}. For this purpose, we need to show that $\mathbb{T}_{\mathrm{af}}\equiv$ PrA.
Clearly, $\mathbb{T}_{\mathrm{af}}\supseteq$ PrA. Conversely, given a model $M$ of PrA an atomless ultracharge
$\wp$ on a set $I$, $M\preccurlyeq M^{\wp}$ is atomless. This shows that PrA is complete hence equivalent to
$\mathbb{T}_{\mathrm{af}}$. In particular, APrA is affinic.

\begin{proposition}\label{elimination}
PrA admits quantifier-elimination and is complete.
\end{proposition}
\begin{proof}
We saw that it is complete. Its quantifier-elimination is a consequence of Proposition \ref{comp-to-lin}.
We can however prove both claims directly.

$L$-terms have the normal form $\vee_i\wedge_j z_{ij}$ where $z_{ij}=x_{ij}$
or $z_{ij}=x'_{ij}$ and $x_{ij}$ is a variable. Atomic formulas may be written in
the form $\mu(t(\x))$ where $t(\x)$ is an $L$-term.
By inclusion-exclusion formula, each $\mu(t(\x))$ can be written as a finite sum of
conjunctive summands, i.e. those of the form $\mu(z_1\wedge\ldots\wedge z_n)$ where $z_i=x_i$ or $x_i'$
and $x_i$ is a variable.
Now, let $\phi(\x,y)$ be a quantifier-free formula.
We show that $\sup_y\phi$ is equal to a quantifier-free formula. Assume
$$\phi(\x,y)=r_1\mu(t_1(\x,y))+\cdots+r_m\mu(t_m(\x,y))$$
where the terms $t_k(\x,y)$ are conjunctive.
Using the equality $\mu(u)=\mu(u\wedge v)+\mu(u\wedge v')$,
we may further assume that each term $t_k(\x,y)$ has the form $z_1\wedge\ldots\wedge z_n\wedge y$
where $z_i=x_i$ or $x_i'$, i.e. $y'$ does not appear while one of $x_i$, $x'_i$ appears.
So, the claim reduces to expressions of the form
$$\phi(\x,y)=\sum_k r_k\cdot\mu(z_1^k\wedge\ldots\wedge z_n^k\wedge y)$$ where,
$z_i^k=x_i$ or $x_i'$, and the terms appearing in distinct summands are `disjoint'.

Now, to obtain the supremum, for the summands with positive coefficient we must have
$y\geq z^k_1\wedge\ldots\wedge z^k_n$ and for summands with negative coefficient we must
have $y\wedge z^k_1\wedge...\wedge z^k_n=0$.
So the supremum is obtained when $y$ is the disjunction of those terms
$z^k_1\wedge\ldots\wedge z^k_n$
appearing positively, and $\sup_y\phi$ is equivalent to the sum of the
corresponding summands. For the second part of the proposition note that
$\{0,1\}$ is a prime model of PrA.
\end{proof}
\bigskip

Let us determine the space of 1-types. Fix $0\leqslant r\leqslant1$.
Then the set $$\{r\leqslant\mu(x),\ \ \mu(x)\leqslant r\}$$
is affinely satisfiable and hence a partial type.
Indeed, by quantifier-elimination, $\mu(x)=r$ determines
a complete type. So, $K_1(\mbox{PrA})$ (with the logic topology) is homeomorphic to $[0,1]$.
As stated above, if $\wp$ is atomless and $M\vDash$ PrA, then $M^{\wp}$ is atomless.
So, models of PrA embed in saturated models of APrA (indeed, affine and full saturation coincide
by Proposition \ref{stability}).
Therefore, by $\aleph_0$-categoricity of APrA, the logic-topology coincides with the metric topology
in $K_n(\mbox{PrA})$.
Note also that, only the extreme points of the type space are realized in the prime model $\{0,1\}$.
This shows that the types realized in a model may in general not be dense.
Let us compute the distance between types in $K_1$.
The collection of events in $[0,1]$ is a model of PrA. By quantifier elimination,
$\mathbb{D}_1(PrA)$ is the real vector space generated by $\mu(x)$ and $1$. Hence,
formulas are of the form $\alpha\mu(x)+\beta$.
Types are determined by their value on $\mu(x)$.
Let $p(\mu(x))=r$ and $q(\mu(x))=s$. For example, for the events
$a=[0,r]$ and $b=[0,s]$ one has that $a\vDash p$ and $b\vDash q$.
Then, $\mathbf{d}(p,q)\leqslant\mu(a\triangle b)=|r-s|$
and since $\mu$ is $1$-Lipschitz, $\mathbf{d}(p,q)\geqslant|p(\mu(x))-q(\mu(x))|=|r-s|$.
Hence $\mathbf{d}(p,q)=|r-s|$.
Also, a formula $\mu(x)+\beta$ has sup-norm $1+\beta$ if
$\beta\geqslant-\frac{1}{2}$ and $|\beta|$ if $\beta\leqslant-\frac{1}{2}$.
So, $$\|p-q\|=\sup_{\phi}\frac{|p(\phi)-q(\phi)|}{\|\phi\|_\infty}=
\sup_{\beta\in\Rn}\frac{|p(\mu(x))-q(\mu(x))|}{\|\mu(x)+\beta\|_\infty}=2|r-s|.$$
So, the three topologies coincide. Finally, by Proposition \ref{stability},

\begin{corollary}
PrA is $\aleph_0$-stable.
\end{corollary}

For each $A\subseteq M\vDash$ PrA, let $\bar A$ be the topological closure of
the probability algebra generated by $A$. Then, $\bar A$ is a model of PrA.
We conclude by Proposition \ref{well-dcl} that $\textrm{dcl}(A)=\bar A$.
There is also an easy description of parametrically definable subsets of $M$.
We recall some definitions from \cite{Fremlin}. A bounded functions
$f:M\rightarrow\Rn$ is additive if $f(x\vee y)=f(x)+f(y)$
whenever $x\wedge y=0$. Countable additivity is defined similarly.
$f$ is said to be positive on $a$ if $0\leqslant f(t)$ for each $t\leqslant a$.
It is negative on $a$ if $-f$ is positive on $a$. For each $a$, the function
$\mu(x\wedge a)$ is countably additive. Additive functions form a vector space.
By inclusion-exclusion principle, the formula $\mu(t(\x))$ is equivalent to a finite sum of
formulas of the form $\mu(z_1\wedge\ldots\wedge z_n)$ where $z_i$ is either $x_i$ or $x_i'$.
Therefore, for each quantifier-free formula $\phi(x,\a)$, the function $\phi^M(x)-\phi^M(0)$ is countably additive.
For $a,b\in M$ let $[a,b]=\{x\in M: a\leqslant x\leqslant b\}$.

\begin{proposition} \label{PrA-def}
A closed $D\subseteq M\vDash$ PrA is definable with parameters if and only if $D=[a,b]$ for some $a,b$.
\end{proposition}
\begin{proof} If $a\leqslant b$, then $d(x,[a,b])=\mu(x\wedge b')+\mu(a\wedge x')$
(the minimum distance is obtained at $a\vee(b\wedge  x)$).
So, $[a,b]$ is definable for every $a,b$. Conversely assume $D\subseteq M$
is nonempty and definable.
First assume $M$ is $\aleph_0$-saturated. Suppose that $D$ is the maximum-set
(points at which $P$ takes its maximum) of a definable function $P:M\rightarrow\Rn$.
Let $\phi_k^M\rightarrow P$ uniformly.
Then $\phi_k^M(x)-\phi_k^M(0)$ tends to $f(x)=P(x)-P(0)$ and hence $f(x)$ is finitely additive.
In fact, since $f$ is continuous, it is countably additive (see \cite{Fremlin} 327B).
$D$ is the maximum-set of $f$ too. We must determine $D$.
By the Hahn decomposition theorem (\cite{Fremlin} 326I), there exists $a$
such that $f$ is positive on $a$ and negative on $a'$.
By completeness of $M$, we may further assume that $a$ is maximal with this property.
Also, there is a maximal $b$ such that $f$ is negative on $b$ and positive on $b'$.
So, $b'\leqslant a$ and $f(t)=0$ for every $t\leqslant a\wedge b$.
Moreover, by maximality of $a$ and $b$,\ \ $f(t\wedge a')<0$ whenever $t\wedge a'>0$ and
$f(t\wedge b')>0$ whenever $t\wedge b'>0$. Now, by additivity of $f$, for each $t$
$$f(t)=f(t\wedge a')+f(t\wedge a\wedge b)+f(t\wedge b')=f(t\wedge a')+f(t\wedge b').$$
We conclude that, $f$ takes its maximum value at $t$ if and only if
$b'\leqslant t\leqslant a$.

Now, assume $M$ is arbitrary. Let $M\preccurlyeq N$ be $\aleph_0$-saturated
and $Q(x)$ be the definable extension of $d(x,D)$ to $N$.
By Proposition \ref{dfnconditions}, $Q(x)=d(x,\bar D)$ where $\bar D=Z(Q)$.
Let $\bar D=[a,b]$ where $a,b\in N$. We show that $M\cap[a,b]$ is an interval in $M$. Since $M$ is Dedekind complete,
$a_1=\inf\{t\in M:\ a\leqslant t\}$ and $b_1=\sup\{t\in M:\ t\leqslant b\}$ belong to $M$.
Clearly, then $D=M\cap[a,b]=[a_1,b_1]$.
\end{proof}

Finally, since $\{0,1\}$ is a first order model of PrA, $K_n(PrA)$ is a simplex. 
In fact, it is the standard $(2^n-1)$-simplex.
\vspace{4mm}

\noindent{\bf Probability algebras with an automorphism}\\
Let $L$ be the language of probability algebras augmented with a unary function
symbol $\tau$. Then the theory of atomless probability algebras equipped with an
aperiodic automorphism, $T=$ APrAA, is a complete stable CL-theory with quantifier-elimination \cite{Aw}.
It turns out that, thanks to the automorphism, $T$ is completely stated
by affine axioms. The axioms stating that $\tau$ is a measure preserving
(hence $1$-Lipschitz) automorphism are obviously affine.
Aperiodicity is also expressed by the Rokhlin property where for probability spaces
states that for every $n\geqslant1$ and $\epsilon>0$, there is a set $X$ such that
$X,\tau(X),...,\tau^{n-1}(X)$ are disjoint and $\mu(\cup_{i<n}\tau^i(X))>1-\epsilon$.

\begin{lemma} Assume $\mu(x)=\frac{1}{n}+\epsilon$ where $\epsilon>0$. Then for each $n\geqslant2$
$$\alpha(x)=\mu(x)-\sum_{0\leqslant i<j<n}\mu(\tau^i(x)\wedge\tau^j(x))
\leqslant\frac{1}{n}-\epsilon.$$
\end{lemma}
\begin{proof} By the inclusion-exclusion principle
$$\sum_{i<j<n}\mu(\tau^i(x)\cap\tau^j(x))\geqslant\sum_{i<n}\mu(\tau^i(x))-
\mu(x\cup\ldots\cup\tau^{n-1}(x))\geqslant n(\frac{1}{n}+\epsilon)-1=n\epsilon.$$
So,
$$\alpha(x)\leqslant\frac{1}{n}+\epsilon-n\epsilon\leqslant\frac{1}{n}-\epsilon.$$
\end{proof}

So, for any $x$, if $|\mu(x)-\frac{1}{n}|\geqslant\epsilon$ then $\alpha(x)\leqslant\frac{1}{n}-\epsilon$.
For probability algebras, Rokhlin's property is expressed by the family of affine conditions
$$\sup_x\big[\mu(x)-\sum_{0\leqslant i<j<n}\mu(\tau^i(x)\wedge\tau^j(x))\big]
=\frac{1}{n} \ \ \ \ \ \ \ \ \ \ n\geqslant1.$$
These conditions state that for each $n$  and $\epsilon>0$ there exists
$x$ such that $|\mu(x)-\frac{1}{n}|<\epsilon$ and
$\mu(\tau^i(x)\wedge\tau^j(x))<\epsilon$ whenever $0\leqslant i<j<n$.
It is clear that these properties imply atomlessness.
So, $T$ is an AL-theory which is CL-complete.
Since $T$ has quantifier-elimination in CL, by Proposition \ref{comp-to-lin}
it has quantifier-elimination in AL as well.
Stability of $T$ is a result of Proposition \ref{stability}.

\begin{proposition}
$T$ is a complete stable theory with quantifier-elimination.
\end{proposition}

Also, by Proposition \ref{stability}, $K_n(T)$ is isomorphic to $S_n(T)$. Note that $T^{\textrm{ex}}$ does not exist.
\bigskip

\noindent{\bf Vector spaces over a finite field}\\
Let $\mathbb F$ be a finite field with $|\mathbb F|=q$ and
$L=\{+,0, \alpha\cdot\}_{\alpha\in\mathbb F}$. Let $T$ be the following theory where $|x|=d(x,0)$:

- Axioms of vector spaces over $\mathbb F$

- $d(x,y)=|x-y|$

- $|\alpha x|=|x|$ \ \ \ for $\alpha\neq0$


- $\sup_{x_1\cdots x_q}|x_i-x_j|=\binom{q}{2}$.

- $\sup_{x_1\cdots x_{q+1}}|x_i-x_j|\leqslant\binom{q+1}{2}-1$.
\bigskip

$\mathbb F$ with the trivial metric is the unique first order model of $T$.
Non-trivial continuous models can be obtained by powermeans $\mathbb F^\mu$ where
$\mu$ is any maximal probability measure on a set say $[0,1]$.
The theory $T$ is complete. We show that it has elimination of quantifiers.

For $\a,\b\in\mathbb F^n$ let $\a\cdot\b=\sum_{i=1}^n a_ib_i$.
Every quantifier-free $L$-formula $\phi(\x)$ is a linear combination
of formulas of the form $|\a\cdot\x|$, i.e.
$$\ \ \ \ \ \phi(\x)=r+\sum r_\ell\ |\a_\ell\cdot\x|,
\ \ \ \ \ \ \ \ \a_\ell\in\mathbb F^n, \ \ r, r_\ell\in\Rn.$$

\begin{proposition} \label{qe}
$T$ has elimination of quantifiers.
\end{proposition}
\begin{proof}
We show that every formula is equivalent to a quantifier-free formula.
It is sufficient that this equivalence occur in $\mathbb F$.
Regard $\mathbb F^n$ as a vector space over $\mathbb F$ and let $\b_1,...,\b_m\in\mathbb F^n$
be a maximal list of pairwise linearly independent elements of $\mathbb F^n$.
Clearly, $m=\frac{q^n-1}{q-1}$. This is the number of
$1$-dimensional subspaces of $\mathbb F^n$ which is also equal
to the number of $(n-1)$-dimensional subspaces of $\mathbb F^n$.
Let us assume
$$\sup_y\sum r_\ell\ |\a_\ell\cdot\x-y|=s_0+\sum_{\ell=1}^m s_\ell\ (1-|\b_\ell\cdot\x|)$$
and find the coefficients $s_\ell$ so that the condition holds for every $\x\in\mathbb F^n$.
Clearly, it is sufficient that the condition hold for $\x=0,\b_1,...,\b_m$.
Putting these tuples in the condition, we obtain $m+1$ linear equations with
indeterminates $s_0,s_1,...,s_{m}$. We must show that this system of equations has a solution.
For this purpose, we have only to show that the $(m+1)\times(m+1)$ matrix with entries:
\[ A_{\ell k}= \left\{
  \begin{array}{ll}
    1 \ \ \ & \hbox{if}\ k=1 \\
    1-|\b_\ell\cdot\b_k| \ \ \ \ \ \ & \hbox{if}\ k\neq 1
  \end{array}
\right.\]
is invertible. In fact, we only need to show that the $m\times m$ matrix
$U=[u_{\ell k}]$ where $u_{\ell k}=1-|\b_\ell\cdot\b_k|$ is invertible.
This is the matrix of incidence between $1$-dimensional and $(n-1)$-dimensional
subspaces of $\mathbb F^n$. By a result of W. M. Kantor \cite{W.Kantor}, $U$ has rank $m$.
\end{proof}
\bigskip

\noindent{\bf Affine theories of first order structures}\\
It is natural to ask which first order theories have quantifier-free (or model complete) affine parts.
Here we use a fact proved in \cite{Ibarlucia} that if $T$ has a first order model then
$T^{\textrm{ex}}$ exists and equals to the first order theory of any its first order models.
Also, that $M\equiv_{\textrm{AL}}N$ implies $M\equiv_{\textrm{CL}}N$ for any first order $M,N$.

\begin{proposition}
A theory $T$ has quantifier-elimination if and only if every type is determined by quantifier-free formulas,
i.e. if $p(\phi)=q(\phi)$ for every quantifier-free (or equivalently atomic) $\phi$, then $p=q$.
Similarly, $T$ is model-complete if and only if types are determined by infimal formulas.
\end{proposition}
\begin{proof}
Assume for every distinct $p,q\in K_n(T)$ there exists a quantifier-free $\phi$ such that
$p(\phi)\neq q(\phi)$.
Let $B$ be the vector space of functions $\hat\phi$ where $\phi(x)$ is quantifier-free.
Then $B$ is a subspace of $\mathbf{A}(K_n(T))$ which contains the constant functions and separates the points.
So, it is dense in $\mathbf{A}(K_n(T))$. We conclude that $T$ has quantifier-elimination. The other direction is obvious.

For the second part note that $T$ is model-complete if and only if every formula
is approximated by infimal formulas.
\end{proof}

Equivalently, $T$ has quantifier-elimination if and only if quantifier-free formulas separate types,
i.e. if $p\neq q$, then there is a quantifier-free formula $\phi$ such that $\hat\phi(p)\neq\hat\phi(q)$.
Also, $T$ is model-complete if and only if infimal formulas separate types.
Similar results holds in full continuous logic where one uses Stone-Weierstrass theorem in the proof.
In particular, a first order theory is model-complete if and only if universal formulas separate types.

Let $\mathbb{T}$ be a complete first order theory and $\mu$ be a regular Borel probability measure on $S_n(\mathbb{T})$.
For any first order $\phi(\x)$ set $$\mu(\phi)=\mu\{u\in S_n(\mathbb{T}): \phi\in u\}=\int\hat\phi(u)d\mu.$$
By regularity, $\mu$ is uniquely determined by its values on such sets.
If $\mathbb{T}=T^{\textrm{ex}}$ where $T$ is an AL-complete theory, then $S_n(\mathbb{T})=E_n(T)$.
In this case, and we say that $\mu$ and $\nu$ coincide on an affine formula $\phi(\x)$ if
$$\int\hat\phi(u)d\mu=\int\hat\phi(u)d\nu.$$

By the inclusion-exclusion principle, for every first order formulas $\phi_1,...,\phi_n$ one has that
$$\mu(\bigvee_{i=1}^n\phi_i)=\sum_{\emptyset\neq J\subseteq\{1,...,n\}}(-1)^{|J|+1}\mu(\bigwedge_{j\in J}\phi_j).$$
In fact, this is a consequence of the more general equality holding for any $f_1,...,f_n$ in a Riesz space:
$$\bigvee_{i=1}^nf_i=\sum_{\emptyset\neq J\subseteq\{1,...,n\}}(-1)^{|J|+1}\bigwedge_{j\in J}f_j.$$
The duals of these equalities hold similarly.

\begin{lemma}\label{atomic values}
Let $\mathbb T$ be a complete first order theory which has quantifier-elimination and $\eta\vee\theta$ is
$\mathbb{T}$-equivalent to an atomic formula whenever $\eta$ and $\theta$ are atomic.
Let $\mu$ and $\nu$ be regular Borel probability measures on $S_n(\mathbb{T})$ which coincide on atomic formulas.
Then, $\mu=\nu$.
\end{lemma}
\begin{proof}
By the assumptions and the normal form theorem, every formula is $\mathbb{T}$-equivalent to a conjunction of formulas of the form
$$\eta\vee\neg\theta_1\vee\cdots\vee\neg\theta_n$$
where $\eta,\theta_i$ are atomic. So, by the inclusion-exclusion principle, we must show that $\mu$ and $\nu$ coincide
on disjunctions of such formulas (which is again equivalent to one of the above form).
Again by the inclusion-exclusion principle, we must show that they coincide on formulas of the form
$$\eta\wedge\neg\theta_1\wedge\cdots\wedge\neg\theta_n.$$
This one is equivalent to a formula of the form $\eta\wedge\neg\theta$ where $\eta$ and $\theta$ are atomic.
For such a formula we have that
$$\mu(\eta\wedge\neg\theta)=\mu(\eta\vee\theta)-\mu(\theta)=\nu(\eta\vee\theta)-\nu(\theta)=\nu(\eta\wedge\neg\theta).$$
\end{proof}

Similar result holds if $\eta\wedge\theta$ is $\mathbb{T}$-equivalent to an atomic formula for every atomic $\eta,\theta$.

\begin{lemma}\label{existential values}
Let $\mathbb T$ be a complete and model-complete first order theory. Assume every disjunction of atomic formulas is
$\mathbb{T}$-equivalent to an atomic formula. Let $\mu$, $\nu$ be regular Borel probability measures on
$S_n(\mathbb{T})$ which coincide on affine infimal formulas.
Then, $\mu=\nu$.
\end{lemma}
\begin{proof}
By the assumptions, every first order formula is $\mathbb{T}$-equivalent to a conjunction of formulas of the form
$$\forall\y(\eta\vee\neg\theta_1\vee\cdots\vee\neg\theta_n)\ \ \ \ \ \ \ \ (*)$$
where $\eta(\x,\y)$, $\theta_i(\x,\y)$ are atomic formulas.
So, by the inclusion-exclusion principle, we have to prove that $\mu$ and $\nu$ coincide on
disjunctions of such formulas, which is again equivalent to one of the form $(*)$.
Identify $\neg\theta$ with $1-\theta$.
Then, by the Riesz space variant of the inclusion-exclusion principle and the assumption of the lemma,
the formula $$\eta\vee\neg\theta_1\vee\cdots\vee\neg\theta_n$$
is $\mathbb{T}$-equivalent to a linear combination of formulas of the form $$\eta\wedge\neg\theta_1\wedge\cdots\wedge\neg\theta_n$$
This later formula is equivalent to one of the form $\eta\wedge\neg\theta\equiv(\eta\vee\theta)-\theta$
where $\eta,\theta$ are atomic.
Again, $\eta\vee\theta$ is equivalent to an atomic formula.
We conclude that the formula $\forall\y(\eta\vee\neg\theta_1\vee\cdots\vee\neg\theta_n)$ is $\mathbb{T}$-equivalent
to a formula of the form $\inf_{\y}\psi$ where $\psi$ is affine and quantifier-free.
Since $\mu$ and $\nu$ coincide on affine infimal formulas, we conclude that $\mu=\nu$.
\end{proof}

\begin{proposition}
Let $T$ be a complete affine theory such that $T^{\textrm{ex}}$ is first order.

(i) If $T$ has quantifier-elimination (resp. is model-complete) in the AL sense, then $T^{\textrm{ex}}$
has quantifier-elimination (resp. is model-complete) in the first order sense.

(ii) If $T^{\textrm{ex}}$ has quantifier-elimination and every disjunction of atomic formulas is
$T^{\textrm{ex}}$-equivalent to an atomic formula, then $T$ has quantifier-elimination in the AL-sense.

(iii) If $T^{\textrm{ex}}$ is model-complete and every disjunction of atomic formulas is
$T^{\textrm{ex}}$-equivalent to an atomic formula, then $T$ is model-complete in the AL-sense.
\end{proposition}
\begin{proof}
(i) For every distinct types $p,q\in E_n(T)=S_n(T^{\textrm{ex}})$, there is an atomic formula $\phi$
such that $p(\phi)\neq q(\phi)$. So, $T^{\textrm{ex}}$ has quantifier-elimination.
For the second part, assume $p(\phi)\neq q(\phi)$ where $\phi=\inf_{\y}\theta$ and $\theta$ is affine.
Then, $\theta$ is $T^{\textrm{ex}}$-equivalent to a quantifier-free first order formula say $\psi$.
So, $\inf_{\y}\theta$ is equivalent to $\forall\y\psi$. Hence, $p(\forall\y\psi)\neq q(\forall\y\psi)$.
Alternatively, for first order $M,N\vDash T$, if $M\subseteq N$, then
$M\preccurlyeq_{\mathrm{AL}}N$ and hence $M\preccurlyeq_{\mathrm{CL}}N$.

(ii) By Lemma \ref{atomic values}, every regular Borel probability measure $\mu$ on
$S_n(T^{\textrm{ex}})$ is uniquely determined by its values on atomic formulas.
Let $p,q\in K_n(T)$ be distinct. By the Choquet-Bishop-de Leeuw theorem, $p,q$ are represented
by regular boundary measures $\mu$ and $\nu$ respectively. So, for every affine formula $\phi$ one has that
$$p(\phi)=\int_{u\in E_n(T)}\hat\phi(u)d\mu, \ \ \ \ \ \ \ q(\phi)=\int_{u\in E_n(T)}\hat\phi(u)d\nu$$
Clearly, $\mu\neq\nu$ and hence there is an atomic formula $\theta$ such that $\mu(\theta)\neq\nu(\theta)$.
We conclude that $p(\theta)=\mu(\theta)\neq\nu(\theta)=q(\theta)$ and hence $T$ has quantifier-elimination.

(iii) As in part (ii), let $p,q\in K_n(T)$ be distinct types represented by boundary measures $\mu$ and $\nu$ respectively.
By Lemma \ref{existential values}, $\mu$ and $\nu$ differ on an affine infimal formulas.
So, $p,q$ differ on that formula. We conclude that $T$ is model-complete.
\end{proof}
\vspace{1mm}

\begin{corollary}
(i) Let $\mathbb{T}$ be a complete (first order) theory of fields in the language of rings $\{+,-,\times,0,1\}$.
If $\mathbb{T}$ has quantifier-elimination (resp. is model-complete) in the first order sense,
then $\mathbb{T}_{\textrm{af}}$ has quantifier-elimination (resp. is model-complete) in the affine logic sense.

(ii) Let $\mathbb{T}$ be a complete theory of Boolean algebras which has quantifier-elimination (resp. is model-complete).
Then, $\mathbb{T}_{\textrm{af}}$ has quantifier-elimination (resp. is model-complete).
\end{corollary}

\begin{example}
\em{(i) Algebraically closed fields and finite fields are the only rings which have quantifier-elimination.
So, the affine part of ACF$_p$ as well as the affine theory of any finite field has quantifier-elimination.
The affine part of RCF is model-complete.

(ii) The affine part of DCF$_0$ has quantifier-elimination in the language of differential rings $\{+,-,\times,\delta,0,1\}$.

(iii) The affine part of the theory of atomless Boolean algebras has quantifier-elimination.
The affine theory of any finite Boolean algebra has quantifier-elimination.}
\end{example}
\bigskip

The classical theories stated above are decidable. So, their affine parts are decidable too.
This is because affine formulas form a computable part of the classical formulas.
It is natural to ask what is a complete axiomatization of the affine part of ACF$_0$.
A model of this theory is a metric commutative ring with unity and without nonzero nilpotent elements.
Moreover, setting $|x|=d(x,0)$, one has that:

(i) $|x-y|=d(x,y)$

(ii) $|xy|\leqslant|x|$

(iii) $|xy|\leqslant\frac{|x|+|y|}{2}$

(iv) $d(x,0)+d(y,0)\leqslant d(xy,0)+1$

(v) $d(nx,0)=d(x,0)$ for all $n\geqslant1$

(vi) $\inf_y d(xy,1)\leqslant1-d(x,0)$

(vii) $\inf_x d(y_nx^{n}+\cdots+y_1x+y_0,0)\leqslant 1-d(y_n,0)\ \ \ \ \ \ \ \ \ \ \ \ \forall n\geqslant1$.
\bigskip

This is however not a complete axiomatization.
On the other hand, given a type $p(\x)$ over a model $A$ of the affine part,
$$I_p=\{f(\x)\in A[\x]:\ p(d(f(\x),0))=0\}$$
is an ideal which is not prime and does not characterize $p$.
Indeed, $p$ defines a metric on $\frac{A[\x]}{I_p}$ making it a metric ring having all the properties stated above except (vi-vii),
where we may set $|f+I_p|=p(d(f,0))$.





\bigskip\noindent{\bf Urysohn space}\\
Let $\mathbb T$ be the CL-theory of the Urysohn space $\mathbb{U}$ of diameter $1$.
Then $\mathbb{T}$ is $\aleph_0$-categorical. Hence, by Proposition \ref{not categorical},
$\mathbb{T}\not\equiv\mathbb{T}_{\mathrm{af}}$.
We can further show that $\mathbb{T}$ is not affinic. Let $M=\frac{1}{2}\mathbb U+\frac{1}{2}\mathbb U$.
Let $a,b\in\mathbb U$ be such that $d(a,b)=1$. Take $x_1=aa$, $x_2=ab$, $x_3=ba$, $x_4=bb$ in $M$
and let $r_1=r_4=\frac{1}{2}$, $r_2=r_3=1$.
Let $$\phi(x)=[d(x,aa)-d(x,ab)-d(x,ba)+d(x,bb)].$$
Suppose that $M\preccurlyeq N\vDash\mathbb{T}$.
Then, there must exist $x\in N$ such that $d(x,x_i)=r_i$ for each $i$.
In particular, $N$ satisfies the condition $$\inf_x\phi(x, aa, ab, ba, bb)\leqslant-1.$$
So, $M$ (approximately) satisfies it by say $x=ce\in M$. However, computation shows that $$\phi^M(ce,aa,ab,ba,bb)=0.$$
\vspace{5mm}

\noindent{\bf Metric groups}\\
A natural way to obtain interesting examples of theories is to find model-companion of incomplete theories.
It is well-known that the theory of groups has no model-companion.
The proof (see \cite{CK1}) can be adopted for the theory of metric groups.

\begin{lemma} \label{conjugate}
For each metric group $G$ and $a,b\in G$ of the same order there exists
a metric group extension $G\subseteq H$ (of the same diameter)
such that $a$, $b$ are conjugate, i.e. $a=c^{-1}bc$ for some $c\in H$.
\end{lemma}
\begin{proof}
We may assume $G$ is presented as $\langle X,R\rangle$, i.e. $G=\langle X\rangle/N$
where $\langle X\rangle$ is the free group generated by $X$ and $N$ is
the normal closure for $R$ (the smallest normal subgroup of $\langle X\rangle $
containing elements $uv^{-1}$ where $u=v\in R$). Let $z$ be a new element.
Let $u,v$ be representatives for $a,b$ in $\langle X\rangle$ respectively.
Let $H=\langle X\cup\{z\}, S\rangle$ where $S=R\cup\{z^{-1}uz=v\}$.
Then, $H$ is an extension of $G$ in which $c^{-1}ac=b$ where $c=zK$ and $K$
is the normal closure for $S$.

The above argument is well-known (see \cite{Lyndon} p. 188).
Now, assume $d$ is a metric on $G$ and $\alpha:\langle X\rangle\rightarrow G$
is the corresponding homomorphism.
Let $\hat d$ be the pseudometric on $\langle X\rangle$ defined by $\hat d(u,v)=d(uN,vN)$.
Every element of $H$ can be presented as $x_1^{k_1}z^{\ell_1}\cdots x_n^{k_n}z^{\ell_n}K$
where $x_i\in X$ and $k_i,\ell_i\in\mathbb Z$.
Define
$$\rho(x_1^{k_1}z^{\ell_1}\cdots x_n^{k_n}z^{\ell_n}K,\ x_1^{k'_1}z^{\ell'_1}\cdots x_n^{k'_n}z^{\ell'_n}K)=
\left\{
  \begin{array}{ll}
    \hat d(x_1^{k_1}\cdots x_n^{k_n},\ x^{k'_1}\cdots x_n^{k'_n}) & \mbox{if}\ \sum_i \ell_i=\sum_i\ell_i' \\
    1 & \mbox{otherwise.}
  \end{array}
\right.$$
It is easy to check that $\rho$ is a metric on $H$ which extends $d$.
\end{proof}

\begin{proposition}
The theory of metric groups of diameter $1$ has no model-companion.
\end{proposition}
\begin{proof}
Let $G_0$ be a pure group having elements of arbitrarily large finite order and
$G$ be an existentially closed extension of $G_0$ (in the AL setting).
Then, $G_0$ (and hence $G$) has pairs of non-conjugate elements of sufficiently large orders which have distance $1$.
This is stated by the conditions
$$3\leqslant\sup_{xy}\inf_z[d(x^{n!},1)+d(y^{n!},1)+d(x,z^{-1}yz)].$$
By affine compactness, there exists $G\preccurlyeq H$ which contains a pair of non-conjugate elements of infinite order.
Applying Lemma \ref{conjugate}, we see that $H$ can not existentially closed.
Since $G$ is existentially closed, we conclude that the class of existentially closed metric groups (of diameter at most $1$) is not elementary.
Therefore, the theory of metric groups of diameter $1$ has not model-companion.
\end{proof}

A similar result hold for the theories of metric groups equipped with a left/right invariant metric.

\newpage\section{Proof system} \label{Proof system}
\def\Sc{\mathcal {S}}
\def\Tc{\mathcal {T}}
Since the space of truth values is indiscrete, a proof system for AL must be either approximate or infinitary.
The one presented here is infinitary and countable length proofs are allowed. It is a sound and complete system.
We just give the main definitions and results. Further details can be  found in \cite{Safari-Bagheri}.
Logical axioms and rules of inference of AL are as follows.
\bigskip

{\bf Linearity axioms}:

(A1) $r_1+r_2=r$ \ \ \ if $\Rn\vDash r_1+r_2=r$

(A2) $r_1r_2=r$ \hspace{8mm} if $\Rn\vDash r_1r_2=r$

(A3) $r\leqslant s$ \hspace{13mm} if $\Rn\vDash r\leqslant s$

(A4) $\phi+(\psi+\theta)=(\phi+\psi)+\theta$

(A5) $\phi+\psi=\psi+\phi$

(A6) $0+\phi=\phi$

(A7) $r(\phi+\psi)=r\phi+r\psi$

(A8) $(r+s)\phi=r\phi+s\phi$

(A9) $r(s\phi)=(rs)\phi$

(A10) $1\phi=\phi$

(A11) $0\phi=0$
\vspace{3mm}

{\bf Quantifier axioms}:

(A12) $\phi[t/x]\leqslant(\sup_x\phi)$ \ \ \ \ if substitution of the term $t$ in place of $x$ is correct

(A13) $\sup_x(\phi+\psi)=\sup_x\phi(x)+\psi$ \ \ \ where $x$ is not free in $\psi$

(A14) $\sup_x(\phi+\psi)\leqslant\sup_x\phi+\sup_x\psi$

(A15) $\sup_x(r\phi)=r\sup_x\phi$ where $r\geqslant 0$

(A16) $\sup_x\phi=-\inf_x-\phi$
\vspace{3mm}

{\bf Pseudometric axioms}:

(A17)  $d(x,x)=0$

(A18) $d(x,y)=d(y,x)$

(A19) $d(x,z)\leqslant d(x,y)+d(y,z)$
\vspace{3mm}

{\bf Bound and Lipschitz axioms}:\vspace{1mm}

(A20) $d(F\bar{x},F\bar{y})\leqslant\lambda_Fd(\bar{x},\bar{y})$ \ \ for each function symbol $F$

(A21) $R\bar{x}-R\bar{y}\leqslant\lambda_Rd(\bar{x},\bar{y})$ \ \ \ for each relation symbol $R$

(A22) $0\leqslant R(\x)\leqslant1$ \hspace{17mm} for each relation symbol $R$ (including $d$)
\vspace{3mm}

{\bf Deduction rules}: \vspace{2mm}

(R1) $\frac{\phi\leqslant\psi,\ \psi\leqslant\theta}{\phi\leqslant\theta}$ \vspace{2mm}

(R2) $\frac{\phi\leqslant\psi}{\phi+\theta\leqslant\psi+\theta}$ \vspace{2mm}

(R3) $\frac{0\leqslant r,\ \phi\leqslant\psi}{r\phi\leqslant r\psi}$ \vspace{2mm}

(R4) $\frac{\phi\leqslant\psi}{\sup_x\phi\leqslant\sup_x\psi}$ \vspace{2mm}
\bigskip

In (R3), $r$ is allowed to be negative. The point is that assuming
the conditions in the numerator are proved, we deduce the condition in the denominator.
To define the notion of deduction, we define an increasing
chain of length $\omega_1$ of relations $\vdash_\alpha$.
Below, $\Sc,\Sc_1,...$ denote affine conditions.

\begin{definition} \label{proofdfn}
{\em We write $\Gamma\vdash_0\Sc$ if $\Sc$ belongs to $\Gamma$ or is a logical axiom.
Suppose $0<\alpha<\omega_1$ and that $\vdash_\beta$ is defined for all $\beta<\alpha$.
Then, we write $\Gamma\vdash_\alpha\Sc$\index{$\vdash_\alpha$} if any one of the following clauses holds:
\begin{quote}

- there exists $\beta<\alpha$ such that $\Gamma\vdash_\beta\Sc$

- there exist $\beta<\alpha$ and $\Sc_1, \Sc_2$ such that
$\Gamma\vdash_\beta\Sc_1$,\ \ $\Gamma\vdash_\beta\Sc_2$ and
$\frac{\Sc_1,\ \Sc_2}{\Sc}$ is an instance of one of the rules
(R1-R3)

- $\Sc$ is the condition $\sup_x\phi\leqslant\sup_x\psi$,\ \ $x$ is not free
in $\Gamma$ and there exists $\beta<\alpha$
such that $\Gamma\vdash_\beta\phi\leqslant\psi$, and

- $\Sc$ is the condition $\phi\leqslant\psi$ and there exists a dense set
$A\subseteq\Rn$ such that for each $r\in A$, there is $\beta_r<\alpha$ with
$\Gamma,r\leqslant\phi\vdash_{\beta_r}r\leqslant\psi$.
\end{quote}}
\end{definition}


\begin{definition}
{\em  $\Sc$ is \emph{provable}\index{provable} (or deducible) from
$\Gamma$, denoted by $\Gamma\vdash\Sc$, if there exists
$\alpha<\omega_1$ such that $\Gamma\vdash_\alpha\Sc$. A set of
conditions $\Gamma$ is \emph{inconsistent} if
$\Gamma\vdash 1\leqslant 0$. Otherwise, it is \emph{consistent}\index{consistent}.}
\end{definition}

It is easy to prove by induction on the complexity of $\phi$ that
$\emptyset\vdash\ -\mathbf{b}_\phi\leqslant\phi\leqslant\mathbf{b}_\phi$
where $\mathbf{b}_\phi$ is the bound given by Definition \ref{formulas and bounds}.
To have an example of proof, we show that $r=0\vdash r\phi=0$.
By (R3), $0\leqslant r\vdash r\phi\leqslant r\mathbf{b}_\phi$.
By (R1-R3) and linearity axioms,
$$r\leqslant 0\ \vdash\ 0\leqslant -r\ \ \ \ \ \ \ \ \ \ $$
$$\hspace{22mm} \vdash\ (-r)(-\phi)\leqslant (-r)\mathbf{b}_\phi$$
$$\hspace{7mm} \vdash\ r\phi\leqslant -r\mathbf{b}_\phi.$$
So, $$r=0\ \ \vdash \ \ 2r\phi\leqslant r\mathbf{b}_\phi - r \mathbf{b}_\phi \ \ \vdash \ \ r\phi\leqslant 0.$$
Similarly, we have that $r=0\vdash 0\leqslant r\phi$ which yields the claim.
\vspace{1mm}

\begin{lemma}
Assume $\Gamma\vdash0\leqslant\phi$. Then, for each $\epsilon>0$ there is a finite
$\Delta\subseteq\Gamma$ such that $\Delta\vdash-\epsilon\leqslant\phi$.
\end{lemma}

The following lemma is some sort of cut elimination.

\begin{lemma}
Assume  $\Gamma,0\leqslant\theta\vdash\Sc$ and $\Gamma,\theta\leqslant0\vdash\Sc$.
Then $\Gamma\vdash\Sc$.
\end{lemma}

\begin{theorem} \emph{(Soundness and completeness)}\index{completeness theorem}
A set $\Gamma$ of conditions is satisfiable if and only if it is consistent.
Also, for each sentence $\phi$,\ \ $\Gamma\vDash0\leqslant\phi$ if and only if $\Gamma\vdash0\leqslant\phi$.
\end{theorem}

\newpage\section{Maximality of AL} \label{Maximality of AL}
In this section we give a brief review of a Lindstr\"{o}m type theorem for AL.
The arguments are essentially adopted according to \cite{Iovino}. More details can be found in \cite{Malekghasemi}.
Roughly speaking, an abstract logic is a pair $(\mathscr{L},\vDash_{\mathscr{L}})$,
where $\mathscr{L}$ is a mapping defined on languages $L$ such that
$\mathscr{L}(L)$ is a set (or even a class, called the class of $\mathscr{L}$-sentences of $L$)
and $\vDash_{\mathscr{L}}$ is a relation between first order $L$-structures
and $L$-sentences (see \cite{Barwise}).
For continuous structures, we may adopt the following definition.
By language we mean a Lipschitz language and by structure we mean a Lipschitz structure.

\begin{definition} \label{logic}
\emph{A {\it logic} is a pair $(\mathscr{L},\vDash_{\mathscr{L}})$,
where $\mathscr{L}$ is a mapping defined on languages $L$ such that
$\mathscr{L}(L)$ is a set and $\vDash_{\mathscr{L}}$ is a function
which assigns to every $L$-structure $M$ and every
$\sigma\in\mathscr{L}(L)$ a real number $\sigma^M\in\Rn$ such that:
\begin{itemize}
\item[(i)] For each $r\in\Rn$, $r\in\mathscr{L}(L)$ and $r^M=r$.
\item [(ii)] For every $\sigma\in\mathscr{L}(L)$, there is a bounded interval $I_{\sigma}\subseteq \mathbb R$
such that $\sigma^{M}\in I_{\sigma} $ for every $L$-structure $M$.
\item[(iii)] If $L\subseteq L'$, then $\mathscr{L}(L)\subseteq\mathscr{L}(L')$.
\item[(iv)] Occurrence property: For each sentence $\sigma$, there is a finite language
$L_\sigma$ such that $\sigma\in\mathscr{L}(L_\sigma)$.
\item[(v)] Isomorphism property: If $M$ and $N$ are  $L$-structures and $M\simeq N$,
then for every $\sigma\in\mathscr{L}(L)$, \ \ $\sigma^{M}=\sigma^{N}$.
\item[(vi)] Reduction property: If $L\subseteq L^{'}$, $M$ is a $L^{'}$-structure and
$\sigma\in\mathscr{L}(L)$, then $\sigma^{M}=\sigma^{M|_L}$.
\item[(vii)] Renaming property: If $f:L\rightarrow L'$ is a renaming, then for each $\sigma\in\mathscr{L}(L)$
there exists $\sigma^{f}\in\mathscr{L}(L')$ such that for each $L$-structure $M$, one has $\sigma^{M}=(\sigma^f)^{M^f}$.
\end{itemize}}
\end{definition}

We also need some sort of regularity for logics, i.e. closure under
linear connectives and quantifiers. A logic $\mathscr{L}$ has {\it addition}
if for every pair of sentences $\sigma,\eta\in\mathscr{L}(L)$
there is a sentence $\xi\in\mathscr{L}(L)$ such that for any $L$-structure $M$,
one has $\xi^{M}=\sigma^{M}+\eta^{M}$. $\mathscr{L}$ has {\it scalar multiplication}
if for every $\sigma$ and $r\in\Rn$ there is a sentence $\eta$ such that $\eta^{M} =r\sigma^{M}$ for every $M$.
These are unique up to $\mathscr{L}$-equivalence.
If there is no confusion, we denote them by $\sigma+\eta$ and $r\sigma$ respectively.
So, $\mathscr{L}(L)$ may be regarded as a vector space over $\Rn$
which includes a copy of $\Rn$.
$\mathscr{L}$ has \emph{supremal quantification} if for each $\sigma$ and $c\in L_{\sigma}$,
there is a sentence $\sup_{x}\text{$\sigma(x)$}$ in $L_{\sigma}-\{c\}$ such that
for every structure $M$ in $L_{\sigma}-\{c\}$ one has that
$$(\sup_{x}\text{$\sigma(x)$})^{M}=\mbox{sup}\{\sigma^M(a)\mid a\in M\}.$$
Here, $\sigma^M(a)$ is the value of $\sigma$ in the structure $(M,c)$ where $c^M=a$.
Infimal quantification $\inf_{x}\text{$\sigma(x)$}$ is defined similarly.

A logic having the above properties is called \emph{regular}.
The logics AL and CL are regular. Below, the logic AL is denoted by $\mathscr{L}^1$ (see also \S \ref{ALp}).
By logic we mean a regular one.
For a logic $\mathscr{L}$, the concepts such as condition, theory, elementary chain,
affine compactness etc. are all defined as in $\mathscr{L}^1$.

\begin{definition}\label{compare2}
\emph{Let $\mathscr{L}$, $\mathscr{L}'$ be two logics. $\mathscr{L}'$ is an {\it extension} of $\mathscr{L}$
(or $\mathscr{L}$ is reducible to $\mathscr{L}'$), denoted by $\mathscr{L}\leqslant\mathscr{L}'$, if for every
language $L$, sentence $\phi\in\mathscr{L}(L)$ and $\epsilon>0$ there exists $\sigma\in\mathscr{L}'(L)$ such that
$$\vDash\ |\sigma-\phi|\leqslant\epsilon.$$
$\mathscr{L}$ and $\mathscr{L}'$ are {\it equivalent} if $\mathscr{L}\leqslant\mathscr{L}'$
and $\mathscr{L}'\leqslant\mathscr{L}$. This is denoted by $\mathscr{L}\equiv\mathscr{L}'$.}
\end{definition}

In \cite{Iovino}, a weaker notion of extension is defined for logics having the approximation property.
For the present context, it can be restated as follows.

\begin{definition} \label{stronger2}
\emph{$\mathscr{L}\vartriangleleft\mathscr{L}'$ if for every language $L$,
$\phi\in\mathscr{L}(L)$ and $\epsilon>0$, there exists $\sigma\in\mathscr{L}'(L)$
such that for every $L$-structure $M$\\
- if $M\vDash0\leqslant\phi$ then  $M\vDash0\leqslant\sigma$\\
- if $M\vDash-\epsilon\leqslant\sigma$ then $M\vDash-\epsilon\leqslant\phi$.}
\end{definition}

\begin{proposition}
Assume $\mathscr{L}'$ satisfies the affine compactness theorem.
Then, $\mathscr{L}\vartriangleleft\mathscr{L}'$ if and only if for every
$L$ and $\mathscr{L}(L)$-theory $T$ there is $\mathscr{L}'(L)$-theory $T'$
such that Mod$(T)=$Mod$(T')$.
\end{proposition}

\begin{theorem}\label{2}
Let $\mathscr{L}^1\leqslant\mathscr{L}$.
Assume the affine compactness theorem and the elementary chain
property hold in $\mathscr{L}$. Then $\mathscr{L}^1\equiv\mathscr{L}$.
\end{theorem}

In particular, any sublogic of CL including $\mathscr{L}^1$ and satisfying the affine
compactness theorem and the elementary chain theorem is equivalent to $\mathscr{L}^1$.
The arguments leading to the proof of the main theorem can be easily modified to
prove an other variant of maximality of $\mathscr{L}^1$.

\begin{theorem}
Let $\mathscr{L}^1\vartriangleleft\mathscr{L}$.
Assume the approximate affine compactness theorem and the elementary chain
property hold in $\mathscr{L}$. Then $\mathscr{L}\vartriangleleft\mathscr{L}^1$.
\end{theorem}

\newpage\section{The logic $\mathscr{L}^p$} \label{ALp}
In the framework AL, every $\aleph_0$-saturated model $M$ has midpoints, i.e.
for every $x,y$ there exists $t$ such that
$$d(x,t)=d(y,t)=\frac{1}{2}d(x,y).$$
A metric space has approximate midpoints if for every $x,y$ and $\epsilon>0$ there exists $t$ such that
$$\max\{d(x,t),d(y,t)\}\leqslant\frac{1}{2}d(x,y)+\epsilon.$$
For complete metric spaces, this is equivalent to saying that for each $\epsilon>0$
$$\forall xy\exists z\ \ \ \ \ \ \ \ d(x,z)^2+d(y,z)^2\leqslant\frac{1}{2}d(x,y)^2+\epsilon$$
which is also equivalent to being a \emph{length space} (see \cite{Bacak}).
We may rewrite this by the condition
$$\sup_{xy}\inf_z \big[d(x,z)^2+d(y,z)^2-\frac{1}{2}d(x,y)^2\big]\leqslant0.$$
Since every metric structure has an $\aleph_0$-saturated elementary extension, this condition
is not expressible in the framework of AL.

There are other interesting CL theories which are affine in some sense but
not expressible by AL conditions.
A complete metric space is a Hadamard space if for each $\epsilon>0$
$$\forall xy\exists m\forall z \ \ \ \ \ \ \ \ d(z,m)^2+\frac{d(x,y)^2}{4}\leqslant\frac{d(z,x)^2+d(z,y)^2}{2}+\epsilon.$$
Hilbert spaces have similar axioms (the parallelogram law).
Also, the theory of abstract $L^p$-spaces is stated by
$$\|x\wedge y\|^p\leqslant\|x\|^p+\|y\|^p\leqslant\|x+y\|^p.$$
Such theories are generally formalizable in the logics $\mathscr{L}^p$
defined below.

Let $L$ be a Lipschitz language and $1\leqslant p<\infty$.
The set of formulas of $\mathscr{L}^p$ is inductively defined as follows:
$$r,\ \ d(t_1,t_2)^p,\ \ R(t_1,...,t_n),\ \ r\phi,\ \ \phi+\psi,\ \ \sup_x\phi,\ \ \inf_x\phi.$$
So, exceptionally, the formula $d(t_1,t_2)$ in AL is replaced with $d(t_1,t_2)^p$.
If $M$ is an $L$-structure, $M^n$ is equipped with the metric
$$d_n(\x,\y)=\Big(\sum_{k=1}^n d(x_k,y_k)^p\Big)^{\frac{1}{p}}.$$
As before, for $n$-ary $F,R\in L$, we require $F^M:M^n\rightarrow M$ to be
$\lambda_F$-Lipschitz and $R^n:M^n\rightarrow[0,1]$ to be $\lambda_R$-Lipschitz.
In particular, the function $d(x,y)^p$ is $2p$-Lipschitz.
In fact, by the mean value theorem, $|r^p-s^p|\leqslant p|r-s|$ whenever $0\leqslant r,s\leqslant1$ and $p\geqslant1$.
Hence, using the fact that for $\mathbf r\in\Rn^2$
$$\ \ \ \ \ \ \|\mathbf{r}\|_1\leqslant 2^{1-\frac{1}{p}}\|\mathbf{r}\|_p \ \ \ \ \ \
(\textrm{where}\ \|\mathbf{r}\|_p=(|r_1|^p+|r_2|^p)^{\frac{1}{p}}),$$
one has that
$$|d(x,y)^p-d(x',y')^p|\leqslant p|d(x,y)-d(x',y')|$$$$\leqslant p [d(x,x')+d(y,y')]\leqslant2p\ d_2((x,y),(x',y')).$$
It is also verified that for each formula $\phi$ there is a $\lambda_\phi\geqslant0$ and a bound $\mathbf{b}_{\phi}$
such that $\phi^M$ is $\lambda_\phi$-Lipschitz and $|\phi^M(\a)|\leqslant\mathbf{b}_{\phi}$ for every $\a\in M$.

Logical notions such as condition, theory, elementary equivalence etc. are defined as before.
Also, the ultramean construction can be defined analogously.
Let $\mu$ be an ultracharge on $I$ and $(M_i,d_i)$ be a
$L$-structure for each $i$. For $a,b\in \prod_i M_i$ set
$$d(a,b)=\|d_i(a_i,b_i)\|_p=\Big(\int d_i(a_i,b_i)^p d\mu\Big)^{\frac{1}{p}}.$$
Then, by Minkowski's inequality, $d$ is a pseudometric on $\prod_iM_i$ and $d(a,b)=0$
defines an equivalence relation on it.
The equivalence class of $(a_i)$ is denoted by $[a_i]$ and the resulting
quotient set by $M=\prod_{\mu}^p M_i$.
The metric induced on $M$ is denoted again by $d$.
We also define an $L$-structure on $M$ as follows.
For each $c,F\in L$ and non-metric $R\in L$ (unary for simplicity)
and $(a_i)\in\prod_i M_i$ set
$$c^M=[c^{M_i}]$$    $$F^M([a_i])=[F^{M_i}(a_i)]$$    $$R^M([a_i])=\int R^{M_i}(a_i)d\mu.$$
Then, $F^M$ is $\lambda_F$-Lipschitz and $R^M$ is $\lambda_R$-Lipschitz.
In particular, if $\a=([a^1_i],...,[a^n_i])$ and $\a_i=(a^1_i,...,a^n_i)$, then
$$R^{M_i}(\a_i)-R^{M_i}(\b_i)\leqslant\lambda_R\ d^{M_i}_n(\a_i,\b_i)\hspace{12mm} \forall i\in I.$$
Since $\|f\|_1\leqslant\|f\|_p$ for any integrable $f:I\rightarrow\Rn$, by integrating
$$R^M(\a)-R^M(\b)\leqslant\lambda_R\ \|d^{M_i}_n(\a_i,\b_i)\|_1\leqslant
\lambda_R\ \|d^{M_i}_n(\a_i,\b_i)\|_p$$
$$=\lambda_R\Big[\int\sum_{k=1}^n \big(d^{M_i}(a_i^k,b_i^k)\big)^p\ \Big]^{\frac{1}{p}}d\mu=\lambda_R\ d^M_n(\a,\b).$$
Hence $M$ is an $L$-structure.

\begin{theorem}
For each $\mathscr{L}^p$-formula $\phi(x_{1}, \ldots, x_{n})$ in $L$
and $[a^{1}_{i}], \ldots, [a^{n}_{i}]\in M$
$$\phi^{M}([a^{1}_{i}],\ldots, [a^{n}_{i}])=
\int\phi^{M_{i}}(a^{1}_{i},\ldots, a^{n}_{i})d\mu.$$
\end{theorem}

\begin{theorem}
Any affinely satisfiable set of $\mathscr{L}^p$-conditions in a language $L$ is satisfiable.
\end{theorem}
Similarly, powermean results as well as the isomorphism theorem hold in the framework of $\mathscr{L}^p$.
Most results of sections \ref{Basic technics}, \ref{Types}, \ref{Powermean}, \ref{Robinson et al.},
\ref{Maximality of AL} hold in $\mathscr{L}^p$. However, there are arguments which are
not easily applicable in $\mathscr{L}^p$ if $p>1$.
For example, since the metric axioms are not considered as $\mathscr{L}^p$-conditions, Henkins's method does not apply.
The same thing is true for the proof system given in \S \ref{Proof system}.
On the other hand, results concerning definability notions require H\"{o}lder conditions
$$\phi(\x)-\phi(\y)\leqslant\lambda_{\phi} d(\x,\y)^p$$
which does not hold in the metric structures.
One may suggest replacing Lipschitzness with the H\"{o}lder conditions.
However, such conditions are highly restrictive when $p>1$.
For example, if $M$ is an interval in $\Rn$, this condition implies that $\phi$ is constant.

\newpage\section{Affine integration logic}\label{Appendix}

\newpage\section{Appendix}\label{Appendix}
\noindent{\bf\large Charges and integration}\\
Let $I$ be a nonempty index set. A family $\mathcal D$ of subsets of $I$ is a \emph{filter} if

- $I\in\mathcal D$, \ \ \ $\emptyset\not\in\mathcal D$

- $A,B\in\mathcal D$ implies $A\cap B\in\mathcal D$

- $B\supseteq A\in\mathcal D$ implies $B\in\mathcal D$.

\noindent It is an \emph{ultrafilter} if $A\in\mathcal D$ or $I-A\in\mathcal D$ for every $A\subseteq I$.
By Tarski's ultrafilter theorem, every filter is extended to an ultrafilter.
Let $\mathcal D$ be an ultrafilter on $I$ and $X$ be a topological space.
The $\mathcal D$-limit of a sequence $\{x_i\}_{i\in I}$ of elements in $X$ is $x$ if for every open $U\ni x$
one has that $\{i:x_i\in U\}\in\mathcal D$. In this case one writes $\lim_{i,\mathcal D}x_i=x$.

\begin{fact}
$X$ is compact Hausdorff if and only if for every sequence $\{x_i\}_{i\in I}$ of elements of $X$,
$\lim_{i,\mathcal D}x_i$ exists and is unique.
\end{fact}

$\mathcal D$-limit preserves continuous operations. Let $F:X^n\rightarrow X$ be continuous
and $\{x^k_i\}_{i\in I}$ be a sequence in $X$ for $k=1,...,n$.
Then, $$F(\lim_{i,\mathcal D}x^1_i,...,\lim_{i,\mathcal D}x^n_i)=\lim_{i,\mathcal D}F(x^1_i,...,x^n_i).$$
Chang and Keisler introduced continuous model theory with truth values in a compact Hausdorff space $X$
using continuous operations $F:X^n\rightarrow X$ as connectives (see \cite{CK}).
A final refinement of this logic was then obtained in \cite{BBHU} by assuming $X=[0,1]$.
In the affine fragment of continuous logic, ultrafilters are replaced with maximal
finitely additive probability measures defined below.

A \emph{charge space}\index{charge space} is a triple $(I,\mathcal A,\mu)$ where $\mathcal A$
is a Boolean algebra of subsets of $I$ and $\mu:\mathcal A\rightarrow\Rn^+$ is a function with the following properties:

- $\mu(\emptyset)=0$

- $\mu(A\cup B)=\mu(A)+\mu(B)$ whenever $A\cap B=\emptyset$.
\bigskip

$\mu$ is a probability charge if $\mu(I)=1$. In this text, by a charge we always mean a probability charge.
If $\mathcal A=P(I)$, $\mu$ is called an \emph{ultracharge}\index{ultracharge}.
Any filter $\mathcal D$ on $I$ defines a probability charge as follows:
$$\mathcal A_{\mathcal D}=\{A:\ A\in\mathcal D\ \mbox{or}\ A^c\in\mathcal D\}$$
\[\mu_{\mathcal D}(A)=
\left\{
  \begin{array}{ll}
    1 & \hbox{\ if $A\in\mathcal D$} \\
    0 & \hbox{\ if $A^c\in\mathcal D$.}
  \end{array}
\right.
\]

If $\mathcal D$ is an ultrafilter, $\mu_{\mathcal D}$ is an ultracharge.
Tarski's ultrafilter theorem has an ultracharge variant.

\begin{theorem} \emph{(\cite{Rao}, Th. 3.3.)} \label{ultracharge-extension}
Let $(I,\mathcal A,\mu)$ be a charge space and $B\subseteq I$.
Assume $$\sup\{\mu(A):\ A\subseteq B, \ A\in\mathcal A\}\leqslant r\leqslant
\inf\{\mu(A):\ B\subseteq A,\ A\in\mathcal A\}.$$
Then, there is an extension of $\mu$ to an ultracharge $\bar\mu$
on $I$ such that $\bar\mu(B)=r$.
\end{theorem}
\bigskip

If $(I,\mathcal A,\mu)$ is a charge space, a set $A\subseteq I$ is \emph{null}\index{null set} if
$$\mu^*(A)=\inf\{\mu(B):\ A\subseteq B\in\mathcal A\}=0.$$
A function $f:I\rightarrow\Rn$ is null if for each $\epsilon>0$ $$\mu^*(\{x\in I: \epsilon<|f(x)|\})=0.$$

We also recall some properties of integration with respect to charges (see \cite{Aliprantis-Inf, Rao}).
Let $(I,\mathcal A,\mu)$ be a charge space and $\mathcal{H}(\Rn)$\index{$\mathcal{H}(\Rn)$} be the
Boolean algebra of subsets of $\Rn$ generated by the half-intervals $[r,s)$.
A function $f:I\rightarrow\Rn$ is called $(\mathcal A,\mathcal H(\Rn))$-\emph{measurable}
(or $\mathcal A$-measurable for short) if $f^{-1}(X)\in\mathcal A$ for every $X\in\mathcal H(\Rn)$.
Bounded $\mathcal A$-measurable functions are always integrable. 
The integral of $f$ with respect to $\mu$ is denoted by $\int fd\mu$.
In particular, if $\mu$ is an ultracharge on $I$, every bounded $f:I\rightarrow\Rn$ is integrable.
If $\mathcal D$ is an ultrafilter on $I$ and $f:I\rightarrow\Rn$ is bounded,
then $$\lim_{i,\mathcal D}f(i)=\int f\ d\mu_{\mathcal D}.$$
Integration with respect to a charge is a positive linear functional.
However, since $\mu$ is not sigma-additive, Fubini's theorem and Fatou's lemma do not hold.
A sequence $f_n$ of integrable real functions on $I$ converges to $f$
hazily (or in probability) if for every $\epsilon>0$
$$\lim_n\mu\{i\in I:\ |f_n(i)-f(i)|>\epsilon\}=0.$$

\begin{proposition} \label{convergence} \emph{(\cite{Rao}, Th. 4.4.20)}
Let $f_n$ be a sequence of integrable functions on $I$ such that
$\lim_{m,n\rightarrow \infty}\int |f_n-f_m| d\mu=0$.
Assume $f_n$ converges hazily to $f$. Then $f$ is integrable and
$\lim_n\int |f_n-f|d\mu=0$. In particular, $\lim_n\int f_n=\int f$.
\end{proposition}
\bigskip

\noindent{\bf\large Convexity}\\
A topological vector space (tvs) is locally convex if its topology is generated by a family of seminorms.
Every Banach space as well as $V^*$ for any normed space $V$ (equipped with the weak* topology) is locally convex.
A subset $K$ of a tvs is \emph{convex}\index{convex set} if for every $p,q\in K$ and $\gamma\in[0,1]$,
one has that $\gamma p+(1-\gamma)q\in K$.
Let $K$ be a compact convex subset of a locally convex tvs $V$.

A function $f:K\rightarrow\Rn$ is \emph{affine}\index{affine function} if for all $p,q\in K$ and $\gamma\in[0,1]$,
$$f(\gamma p+(1-\gamma)q)=\gamma f(p)+(1-\gamma)f(q).$$
$\mathbf{A}(K)$ \label{affine function}\index{$\mathbf{A}(K)$} denotes the set of affine continuous functions on $K$.
This is a Banach space with the sup-norm.
It is well-known that if $X\subseteq\mathbf{A}(K)$ contains constant functions and separates points,
it is dense (\cite{Wickstead} I.1.12).

A non-empty convex $F\subseteq K$ is called a \emph{face}\index{face} if for every $p,q\in K$ and $0<\gamma<1$,
$\gamma p+(1-\gamma)q\in F$ implies that $p,q\in F$ (though it is sufficient to use $\gamma=\frac{1}{2}$).
A point $p$ is \emph{extreme}\index{extreme point} if $\{p\}$ is a face.
Also, $p$ is \emph{exposed}\index{exposed} if it is the unique maximizer of a non-zero affine continuous
function. Every exposed point is extreme.
The extreme boundary of $K$, denoted by $Ext(K)$, is the set of extreme points of $K$.
By the Krein-Milman theorem, if $K$ is nonempty, then so is $Ext(K)$.
Moreover, the topological closure of the convex hull of $Ext(K)$ is equal to $K$.
If $K$ is metrizable, $Ext(K)$ is a $G_\delta$ set.

The Baire sigma-algebra $\mathcal B_0(K)$ of $K$ is the sigma-algebra generated by the family of closed $G_\delta$ subsets of $K$.
This is also the smallest sigma-algebra for which every continuous $f:K\rightarrow\Rn$ is measurable.
One has that $\mathcal{B}_0(K)\subseteq\mathcal{B}(K)$ where $\mathcal{B}$ is the Borel sigma-algebra $K$ (generated by the open subsets).
A measure on $\mathcal B_0(K)$ (resp. $\mathcal B(K)$) is called a Baire (resp. Borel) measure.
A Baire probability measure $\mu$ on a compact Hasudorff space is always regular, i.e. for every $A\in\mathcal B_0(K)$
$$\mu(A)=\sup\{\mu(E):\ E\subseteq A, \ E\ \mbox{is\ closed}\}$$
$$=\inf\{\mu(U):\ A\subseteq U\ \mbox{is\ open}\}.$$
Also, any Baire probability measure on a compact Hausdorff space has a unique extension to a regular Borel measure,
i.e. a measure for which the above property holds for every $A\in\mathcal B(K)$.
On compact metric spaces, Borel sets are the same as Baire sets (see \cite{Dudley} for more details).

For every set $X$ let $\mathbb{U}(X)$ be the set of probability ultracharges on $X$.
This is a compact convex set.

\begin{lemma}\label{ultrafilters-extreme}
$\wp\in\mathbb{U}(X)$ is extreme if and only if it is an ultrafilter.
\end{lemma}
\begin{proof}
Assume $\wp$ is extreme and $\wp(Y)=\gamma\in(0,1)$ for some $Y\subseteq X$.
For $A\subseteq X$ set
$$\wp_1(A)=\frac{\wp(A\cap Y)}{\gamma},\ \ \ \ \ \ \ \wp_2(A)=\frac{\wp(A\cap Y^c)}{1-\gamma}.$$\
Then, $\wp=\gamma\wp_1+(1-\gamma)\wp_2$ and $\wp\neq\wp_1,\wp_2$.
This is a contradiction. Conversely assume $\wp$ is 2-valued and $\wp=\gamma\wp_1+(1-\gamma)\wp_2$ where $0<\gamma<1$.
Let $A\subseteq X$. Then both $\wp(A)=0$ and $\wp(A)=1$ imply that $\wp_1(A)=\wp_2(A)=\wp(A)$.
\end{proof}

If $X$ is compact Hasudorff and $\mathcal{P}(X)$ is the set of regular Borel probability
measures on $X$, then $\mathcal{P}(X)$ is compact convex.

\begin{lemma}\label{regular-extreme}
Let $X$ be compact Hausdorff. Then $\mu\in\mathcal{P}$ is extreme if and only if $\mu=\delta_x$ for some $x\in X$.
\end{lemma}
\begin{proof}
Clearly, every pointed measure $\delta_x$ is extreme. Conversely, assume $\mu$ is extreme.
As in Lemma \ref{ultrafilters-extreme}, $\mu$ must be $\{0,1\}$-valued.
Assume for each $x\in X$ there is an open $U\ni x$ such that $\mu(X)=0$.
Then, by compactness, $\mu(X)=0$ which is impossible. Hence, there exists $x$ such that
$\mu(U)=1$ for every $U\ni x$. Now, by regularity, $\mu\{x\}=1$.
\end{proof}

Let $\mu$ be a Baire probability measure on $K$.
The \emph{barycenter} of $\mu$ is the unique $p\in K$ such that
$$\ \ f(p)=\int f\ d\mu\ \ \ \ \ \ \ \ \ \forall f\in \mathbf{A}(K).$$
In this case, one says that $\mu$ represents $p$. Every Baire probability measure on $K$ has a barycenter.

\begin{theorem} \emph{(Choquet-Bishop-de Leeuw)}\label{Choquet-Bishop-de Leeuv}
Let $K$ be a compact convex subset of a locally convex tvs and $p\in K$.
Then there is a Baire probability measure $\mu$ which represents $p$
and which vanishes on every Baire $A\subseteq K$ with $A\cap Ext(K)=\emptyset$.
\end{theorem}

\begin{theorem} \emph{(Bishop-de Leeuw)}
For every point $p$ in a compact convex set $K$ there is a measure $\mu$
on the sigma-algebra $Ext(K)\cap\mathcal B_0(K)$ such that $\mu(Ext(K))=1$ and
$$f(p)=\int_{Ext(K)}f\ d\mu\ \ \ \ \ \ \ \ \forall f\in \mathbf{A}(K).$$
\end{theorem}

A compact convex set $K$ is called a \emph{Choquet simplex}\index{simplex} if every $p\in K$ is represented by a unique Baire measure.
A Choquet simplex is called a \emph{Bauer simplex} if $Ext(K)$ is closed.

\begin{theorem} \label{Bauer simplex} (\cite{Alfsen} p.104)
Let $X$ be a compact Hausdorff space. Then, $M_1^+(X)$ (the set of regular Borel probability measures on $X$)
is a Bauer simplex and its extreme boundary is homeomorphic to $X$.
\end{theorem}

More details on compact convex sets can be found in \cite{Alfsen, Phelps, Simon}.
\bigskip

\noindent{\bf\large Extension theorems}
\begin{theorem} \emph{(Kantorovich, \cite{Aliprantis-Inf} Th. 8.32)} \label{ext1}
Let $G$ be a majorizing vector subspace of a Riesz space $E$ (i.e. for each $x\in E$ there is $x\leqslant y\in G$).
Let $\Lambda:G\rightarrow\Rn$ be a positive linear function.
Then $\Lambda$ has an extension to a positive linear function on $E$.\end{theorem}

\begin{theorem} (Riesz representation theorem \cite{Aliprantis-Inf}, Th.14.9)\label{Riesz representation}
Let $X$ be a normal Hausdorff topological space and $\Lambda:{\mathbf C}_b(X)\rightarrow\Rn$ be a positive linear functional.
Then there exists a unique finite regular charge $\mu$ on the algebra $\mathcal{A}(X)$ generated by the open sets such that
$\mu(X)=\|\Lambda\|=\Lambda(1)$ and $$\Lambda(f)=\int fd\mu\ \ \ \ \ \ \ \ \ \ \forall f\in\mathbf{C}_b(X).$$
\end{theorem}

\begin{theorem} \label{Riesz-Markov noncompact} (Riesz-Markov representation theorem \cite{Royden})
Let $X$ be a compact Hausdorff space and $\Lambda:C(M)\rightarrow\Rn$ a positive linear functional with $\Lambda(1)=1$.
Then there is a unique Baire probability measure $\mu$ on $X$ such that
$$\ \ \ \ \Lambda(f)=\int f d\mu\ \ \ \ \ \ \ \ \ \forall f\in C(M).$$
\end{theorem}

Every Baire probability measure on a compact Hausdorff $X$ has a unique extension to a
regular Borel measure. So, we may replace Baire with Borel in the above theorem.

\newpage\noindent{\bf\large Questions}
\begin{enumerate}
\item Find general ultracharges $\mu$ for which $M^\mu$ is $\kappa$-saturated (or complete if $M$ is so).
\item Prove the compact variant of the joint embedding property:
every two compact $M\equiv N$ have a common compact elementary extension
(similarly for the elementary amalgamation property).
Prove that if $T$ has a compact model, every type is realized in a compact model.
\item Prove that if $M$ is compact connected and minimal (has no proper elementary submodel),
then every $N\equiv M$ is connected.
\item Does there exist an AL-complete $\kappa$-categorical theory for $\kappa\geqslant\aleph_1$?
\item Study the affine part of a first order theory such as PA, ACF, RCF etc.
Study ultrameans of $\Nn$, $\Rn$ etc. similar to the non-standard analysis methods.
\item Study computability aspects of affine logic. Is it true that the affine part PA is still undecidable?
\item Let $\mu,\nu$ be ultracharges on $I$ such that $\mu\leq\nu$ and $\nu\leq\mu$.
Prove that there is a bijection $f:I\rightarrow I$ such that $f(\mu)=\nu$.
Prove that $\mu\equiv\nu$ implies $M^\mu\simeq M^\nu$ for all $M$.
\item Study stability and related notions in the framework of affine logic.
\item Study the model theoretic relations between an affine theory and its extremal theory.
For example, prove that $T$ is stable (has QE) if and only if its extremal part is so.
\item Regarding Proposition \ref{invariant types}, prove that the extreme points of the set of invariant types have
properties similar to ergodic measures. Does there exist a model theoretic variant of ergodic theory
where probability measures are replaced with types?
\item Reduced mean with respect to an arbitrary probability charge is defined similar to ultramean
by replacing all integrals with upper integrals.
In particular, $d(a,b)=\bar{\int\ }d(a_i,b_i)d\mu$. Prove an affine variant for Feferman-Vaught
theorem and deduce that reduced means preserve affine elementary equivalence.
\end{enumerate}

{\small\printindex}

\newpage\begin{thebibliography}{5}
\bibitem{Alfsen} E.M. Alfsen, \emph{Compact convex sets and boundary integrals}, Springer-Verlag (1971).
\bibitem{Aliprantis-Inf} C.D. Aliprantis, K.C. Border, \textit{Infinite dimensional analysis}, third edition, Springer (2006).
\bibitem{Aliprantis} C.D. Aliprantis, O. Burkinshaw, \textit{Positive operators}, Springer (2006).
\bibitem{Bacak} M. Ba\v{c}\'{a}k, \emph{Convex analysis and optimization in Hadamard spaces},
De Gruyter Series in Nonlinear Analysis and Applications 22, De Gruyter (2014).
\bibitem{Bagheri-Lip} S.M. Bagheri, \emph{Linear model theory for Lipschitz structures}, Arch. Math. Logic 53:897-927 (2014).
\bibitem{Bagheri-Extreme} S.M. Bagheri, \textit{Extreme types and extremal models}, Annals of Pure and Applied Logic 175 (7), 103451 (2024).
\bibitem{Bagheri-Definability} S.M. Bagheri, \textit{Definability in affine logic}, submitted.
\bibitem{Bagheri-Safari} S.M. Bagheri, R. Safari, \emph{Preservation theorems in linear continuous logic}, Math. Log. Quart. 60, No.3, 168-176 (2014).
\bibitem{Barwise} J. Barwise and S. Feferman, \textit{Model-theoretic logics}, Springer, 1985.
\bibitem{BBHU} I. Ben-Yaacov, A. Berenstein, C.W. Henson, A. Usvyatsov,
\textit{Model theory for metric structures}, Model theory with Applications
to Algebra and Analysis, volume 2 (Zoe Chatzidakis, Dugald Macpherson,
Anand Pillay, and Alex Wilkie, eds.), London Math Society Lecture Note Series,
vol. 350, Cambridge University Press, (2008), pp. 315-427.
\bibitem{ben0} I. Ben-Yaacov, A.P. Pedersen, \textit{A proof of completeness for continuous first order logic},
Journal of Symbolic Logic 75 (2010), no. 1, 168-190.
\bibitem{Ibarlucia} I. Ben-Yaacov, T. Ibarluc\'{i}a, T. Tsankov,
\textit{Extremal models and direct integrals in affine logic}, Preprint arXiv:2407.13344 (2024).
\bibitem{Aw} A. Berenstein, C. Ward Henson, \textit{Model theory of probability spaces with an automorphism}, Preprint arXiv:math/0405360v1 (2004).
\bibitem{Rao} K.P.S. Bhaskara Rao, M. Bhaskara Rao, \textit{Theory of charges}, Academic Press (1983).
\bibitem{Buechler} S. Buechler, \emph{Essential stability theory}, Springer-Verlag (1996).
\bibitem{CK1} C.C. Chang and H.J. Keisler, \textit{Model theory }, North-Holland (1990).
\bibitem{CK} C.C. Chang and H.J. Keisler, \textit{Continuous model theory}, Princeton University Press (1966).
\bibitem{Morris}J. Cleary, S. Morris, D. Yost, \emph{Numerical Geometry - Numbers for Shapes}, The American Mathematical Monthly, vol. 93 (1987), 260-275.
\bibitem{Comfort-Negrepontis} W.W. Comfort, S. Negrpontis, \textit{The theory of ultrafilters}, Springer-Verlag (1974).
\bibitem{Conway} J.B. Conway, \textit{A course in functional analysis}, second edition, Springer (1990).
\bibitem{Dudley} R.M. Dudley, \textit{Real analysis and probability}, Cambridge University Press (2004).
\bibitem{Fadai} F. Fadai, S.M. Bagheri, \textit{Linear formulas in continuous logic}, Iranian journl of mathematical sciences and informatics,
Vol. 17 No.2,  75-86 (2002).
\bibitem{Ziegler} J. Flum, M. Ziegler, \emph{Topological model theory}, Lecture Notes in Mathematics, 769, Springer (1980).
\bibitem{Fremlin} D.H. Fremlin, \emph{Measure theory, volume 3: Measure algebras}, Internet file (2004).
\bibitem{Goldbring} I. Goldbring, \emph{An approximate Herbrand's theorem and definable functions in metric structures}, Math. Logic Quart. 58, No.3 (2012).
\bibitem{Henson} C.W. Henson, J. Iovino, \emph{Ultraproducts in analysis}, in Analysis and Logic,
London Mathematical Society Lecture Notes Series, vol 262 (2002).
\bibitem{Hoogland} E. Hoogland, \emph{Definability and interpolation.
Model-theoretic investigations}, PhD Thesis, Amsterdam (2001).
\bibitem{Lyndon} R.C. Lyndon, P.E. Schupp, \textit{Combinatorial group theory}, Springer-Verlag (1977).
\bibitem{Iovino} J. Iovino. \textit{On the maximality of logics with approximations,}
The Journal of Symbolic Logic 66.04 (2001): 1909-1918.
\bibitem{W.Kantor} W. M. Kantor, \emph{On incidence matrices of finite projective and affine spaces},
Math. Z. 124, 315-318, $\copyright$ by Springer-Verlag (1972).
\bibitem{Malekghasemi} M. Malekghasemi, S.M. Bagheri, \emph{Maximality of linear continuous logic}, Math. Log. Quart. 64, No. 3 185-191 (2018).
\bibitem{Marker} D. Marker, \emph{Model theory, an introduction}, Springer-Verlag (2002).
\bibitem{Parthasarathy-selection} T. Parthasarathy, \emph{Selection theorems and their applications}, Lectures Notes in Mathematics 163, Springer-Verlag 1972
\bibitem{Phelps} R.R. Phelps, \textit{Lectures on Choquet's theorem}, Springer-Verlag (2001).
\bibitem{Poizat} B. Poizat, A. Yeshkeyev, \textit{Positive Jonsson theories}, Log. Univers. 12 (2018).
\bibitem{Poizat-Model} B. Poizat, \emph{Cours de Th\'{e}orie des mod\`{e}les}, Villeurbannes (1985).
\bibitem{Royden} H.L. Royden \textit{Real analysis}, Third edition, Macmillan (1988).
\bibitem{Rudin} W. Rudin, \emph{Functional analysis}, McGraw-Hill Inc. (1991).
\bibitem{Safari-Bagheri} R. Safari, S.M. Bagheri, \emph{Completeness for linear continuous logic},
Journal of Logic and Computation 27 (4), 985-998 (2015).
\bibitem{Simon} B. Simon, \emph{Convexity: An Analytic Viewpoint}, Cambridge University Press (2011).
\bibitem{Srivastava} S.M. Srivastava, \textit{A course on Borel sets}, Springer-Verlag (1998).
\bibitem{Vilaveces} A. Villaveces, P. Zambrano, \emph{Limit models in metric abstract elementary classes: the categorical case},
Math. Log. Quart. 62, No 4-5, p. 319-334 (2016).
\bibitem{Wickstead} A. W. Wickstead, \emph{Affine functions on compact convex sets}, Internet file.
\end{thebibliography}
\end{document}